\documentclass [12pt]{amsart}
\usepackage[utf8]{inputenc}

\usepackage{amsmath,amssymb,amsthm,amsfonts}
\usepackage{mathtools}

\usepackage[english]{babel}
\usepackage{color}
\usepackage{comment}
\usepackage{hyperref}
\usepackage{bbm}
\usepackage{mathrsfs}

\usepackage{tikz,graphicx,color}
\usetikzlibrary{calc}
\usetikzlibrary{arrows}
\usetikzlibrary{shapes}
\usetikzlibrary{patterns}
\usetikzlibrary{positioning}
\usepackage[pdf]{pstricks}
\usepackage{pst-node}
\usepackage{pst-tree}
\usepackage{pgffor}

\usepackage{enumerate}
\usepackage{diagbox}
\usepackage{subfig}
\usepackage{cite}
\usepackage{chngcntr} 
\usepackage{arcs}
\usepackage{xcolor}
\usetikzlibrary{tikzmark}

\usetikzlibrary{arrows.meta}

\newtheorem{theorem}{Theorem}[section]

\newtheorem{lemma}[theorem]{Lemma}
\newtheorem{proposition}[theorem]{Proposition}
\newtheorem{corollary}[theorem]{Corollary}
\theoremstyle{definition}
\newtheorem{remark}[theorem]{Remark}
\theoremstyle{definition}
\newtheorem{definition}[theorem]{Definition}
\newtheorem{conjecture}[theorem]{Conjecture}

\theoremstyle{definition}

\theoremstyle{definition}
\newtheorem{example}[theorem]{Example}

\def\Acal{\mathcal{A}}\def\Bcal{\mathcal{B}}\def\Ccal{\mathcal{C}}\def\Dcal{\mathcal{D}}\def\Ecal{\mathcal{E}}\def\Fcal{\mathcal{F}}\def\Gcal{\mathcal{G}}\def\Hcal{\mathcal{H}}\def\Ical{\mathcal{I}}\def\Jcal{\mathcal{J}}\def\Kcal{\mathcal{K}}\def\Lcal{\mathcal{L}}\def\Mcal{\mathcal{M}}\def\Ncal{\mathcal{N}}\def\Ocal{\mathcal{O}}\def\Pcal{\mathcal{P}}\def\Qcal{\mathcal{Q}}\def\Rcal{\mathcal{R}}\def\Scal{\mathcal{S}}\def\Tcal{\mathcal{T}}\def\Ucal{\mathcal{U}}\def\Vcal{\mathcal{V}}\def\Wcal{\mathcal{W}}\def\Xcal{\mathcal{X}}\def\Ycal{\mathcal{Y}}\def\Zcal{\mathcal{Z}}

\def\abf{\mathbf{a}}\def\bbf{\mathbf{b}}\def\cbf{\mathbf{c}}\def\dbf{\mathbf{d}}\def\ebf{\mathbf{e}}\def\fbf{\mathbf{f}}\def\gbf{\mathbf{g}}\def\hbf{\mathbf{h}}\def\ibf{\mathbf{i}}\def\jbf{\mathbf{j}}\def\kbf{\mathbf{k}}\def\lbf{\mathbf{l}}\def\mbf{\mathbf{m}}\def\nbf{\mathbf{n}}\def\obf{\mathbf{o}}\def\pbf{\mathbf{p}}\def\qbf{\mathbf{q}}\def\rbf{\mathbf{r}}\def\sbf{\mathbf{s}}\def\tbf{\mathbf{t}}\def\ubf{\mathbf{u}}\def\vbf{\mathbf{v}}\def\wbf{\mathbf{w}}\def\xbf{\mathbf{x}}\def\ybf{\mathbf{y}}\def\zbf{\mathbf{z}}

\def\afr{\mathfrak{a}}
\def\Sfr{{ \mathfrak{S}}}

\def\one{{\mathbbm{1}}}
\def\R{\mathbb{R}}
\def\N{\mathbb{N}}
\def\Z{\mathbb{Z}}
\def\Q{\mathbb{Q}}
\def\P{\mathbb{P}}

\def\Ahat{\hat{A}} 
\def\ipar{{(i)}}
\def\kpar{{(k)}}
\newcommand\parr[1]{{({#1})}}

\def\<{{\langle}}
\def\>{{\rangle}}
\def\e{{\epsilon}}
\def\l{{\lambda}}
\def\m{{\mu}}
\def\lm{{\l/\m}}
\def\RP{{\R P}}
\def\CP{{\C P}}
\def\multiset#1#2{\left(\!\left({#1\atopwithdelims..#2}\right)\!\right)}
\def\toi{{\xhookrightarrow{i}}}
\def\weakMap{\leadsto}

\def\Vol{\operatorname{Vol}}
\def\det{{ \operatorname{det}}}
\def\tr{{ \operatorname{tr}}}
\def\Ker{{ \operatorname{Ker}}}
\def\Im{{ \operatorname{Im}}}
\def\diag{{ \operatorname{diag}}}
\def\rank{{ \operatorname{rank}}}
\def\rk{{\mathrm{rk}}}
\def\cork{ \operatorname{cork}}
\def\codim{ \operatorname{codim}}
\def\Vert{{ \operatorname{Vert}}}
\def\op{{ \operatorname{op}}}
\def\type{{\operatorname{type}}}
\def\sh{{ \operatorname{sh}}}
\def\Conv{ \operatorname{Conv}}
\def\Span{ \operatorname{Span}}
\def\proj{ \operatorname{proj}}

\def\affA{{\hat A}}
\def\affD{{\hat D}}
\def\affE{{\hat E}}
\def\affL{{\hat \Lambda}}
\def\L{{\Lambda}}
\def\VertL{{\Vert(\affL)}}
\def\VertQ{{ \Vert(Q)}}
\def\eig{\v}
\def\speed{ \operatorname{SPEED}}
\def\summ{ \operatorname{SUM}}
\newcommand{\plusOne}[1]{%
\number\numexpr#1+1\relax%
}
\newcommand{\bl}[1]{[fillstyle=solid,fillcolor=lightgray,mnode=circle]#1}
\newcommand{\wh}[1]{[fillstyle=solid,fillcolor=white,mnode=circle]#1}
\newcommand{\blova}[1]{[fillstyle=solid,fillcolor=lightgray,mnode=oval]#1}
\newcommand{\whova}[1]{[fillstyle=solid,fillcolor=white,mnode=oval]#1}
\newcommand{\ra}[2]{\ncline[linecolor=red]{#1}{#2}}
\newcommand{\ba}[2]{\ncline[linecolor=blue]{#1}{#2}}
\newcommand{\arr}[2]{\ncline[linecolor=black]{#1}{#2}}

\def\bt{$\boxtimes$}
\def\affinite{{affine $\boxtimes$ finite }}
\def\affaff{{affine $\boxtimes$ affine }}
\def\t{ \mathfrak{t}}
\def\Ttr{ \t^\l}
\def\Ytr{ y^\l}
\def\i{{\mathbf{i}}}
\def\Edges{ \operatorname{Edges}}
\def\v{\vbf}
\newcommand{\col}[2]{\begin{pmatrix}
                         #1\\#2
                        \end{pmatrix}} 
\def\scf{ \operatorname{scf}}

\def\Aut{ \operatorname{Aut}}
\def\id{ \operatorname{id}}

\begin{document}
\numberwithin{equation}{section}

\title[Quivers with additive labelings]{Quivers with additive labelings: classification and algebraic entropy}

\author{Pavel Galashin}
\address{Department of Mathematics, Massachusetts Institute of Technology,
Cambridge, MA 02139, USA}
\email{{\href{mailto:galashin@mit.edu}{galashin@mit.edu}}}

\author{Pavlo Pylyavskyy}
\address{Department of Mathematics, University of Minnesota,
Minneapolis, MN 55414, USA}
\email{{\href{mailto:ppylyavs@umn.edu}{ppylyavs@umn.edu}}}

\date{\today}

\thanks{P.~P. was partially supported by NSF grants  DMS-1148634, DMS-1351590, and Sloan Fellowship.}



\begin{abstract}
We show that Zamolodchikov dynamics of a recurrent quiver has zero algebraic entropy 
only if the quiver has a weakly subadditive labeling, and conjecture the converse. 
By assigning a pair of generalized Cartan matrices of affine type to each quiver with an additive labeling, we completely classify such quivers, obtaining $40$ infinite families and $13$ exceptional quivers. This completes the program of classifying Zamolodchikov periodic and integrable quivers. 
\end{abstract}

\maketitle

\setcounter{tocdepth}{1}
\tableofcontents

\def\e{{\epsilon}}
\def\VertL{{\operatorname{Vert}(\affL)}}
\def\map{\u}
\def\coeff{\theta}
\def\eig{\v}
\def\mutmat{\omega}
\def\coxmat{\mathbf{C}}
\def\bwmat{A}

\def\speed{ \operatorname{SPEED}}
\def\summ{ \operatorname{SUM}}

\def\rec{J}

\def\Vert{{ \operatorname{Vert}}}
\def\VertQ{{ \Vert(Q)}}
\def\x{{ \mathbf{x}}}

\def\Ttr{ \mathfrak{t}^\l}
\def\Ytr{ y^\l}
\def\t{ \mathfrak{t}}

\def\i{{\mathbf{i}}}

\newcommand{\bg}[1]{\#\ref{#1}}

\section*{Introduction}
Given a \emph{bipartite quiver} $Q$ which is just a directed bipartite graph without directed cycles of length $1$ and $2$, one can define a certain discrete dynamical system called the \emph{$T$-system associated with $Q$}. It assigns a multivariate rational function $T_v(t)$ to each vertex $v$ of $Q$ and each integer $t$ and satisfies the following recurrence relation
\begin{equation}\label{eq:T_system}
T_v(t+1)T_v(t-1)=\prod_{u\to v} T_u(t)+\prod_{v\to w} T_w(t). 
\end{equation}
In certain cases, this relation specializes to various well studied objects such as the \emph{octahedron recurrence} of Speyer~\cite{Sp}. 

When a bipartite quiver $Q$ satisfies a certain simple local condition (we call such quivers \emph{recurrent}), the $T$-system dynamics can be viewed as a special case of a \emph{cluster algebra}, see~\cite{FZ}. Partially because of this connection to cluster algebras, $T$-systems have received much attention in the past two decades. One particular popular direction of research in this area is the so called \emph{Zamolodchikov periodicity}. It was conjectured by Zamolodchikov~\cite{Z} that if $Q$ is an orientation of an $ADE$ Dynkin diagram of finite type then the $T$-system is periodic. This conjecture was later generalized by Kuniba-Nakanishi~\cite{KN} and Ravanini-Valleriani-Tateo~\cite{RVT} to the case when $Q$ is a \emph{tensor product} of two finite $ADE$ Dynkin diagrams. The conjecture stayed open for around twenty years with various special cases being completed in~\cite{RVT,KN,KNS,FZy,GT,Vo,Sz}. It was finally resolved for all pairs of finite $ADE$ Dynkin diagrams by Keller~\cite{K}. 

For the connections of $T$-systems with thermodynamic Bethe ansatz \cite{Z} as well as their other appearances in physics and representation theory, see~\cite{KNS,KR,R,OW,FR,Kni,N}, and see~\cite{KNSi} for a survey.
Of special note is the work of Hernandez~\cite{H}, where he studied the occurrence of $T$-systems in representation theory
for simply-laced quivers beyond Dynkin quivers.

This is the third and final paper in the series~\cite{GP1,GP2} of works that classify bipartite recurrent quivers for which the $T$-system satisfies a certain algebraic property. In~\cite{GP1}, we have shown that the $T$-system associated to a bipartite recurrent quiver $Q$ is periodic if and only if $Q$ admits a \emph{strictly subadditive labeling}. Such quivers turn out to exactly correspond to commuting pairs of Cartan matrices which have been classified earlier by Stembridge~\cite{S}. In particular, tensor products of finite $ADE$ Dynkin diagrams belong to this family, so the result of~\cite{GP1} is a generalization of the result of~\cite{K}. 

Next, we showed in~\cite{GP2} that the values of the $T$-system satisfy a linear recurrence only if $Q$ admits a \emph{subadditive labeling}. We gave an analogous classification for quivers admitting a subadditive labeling. In particular, it includes tensor products of an affine $ADE$ Dynkin diagram with a finite $ADE$ Dynkin diagram. This classification allowed an extensive computational verification of the conjecture that the converse is also true, i.e., for every bipartite recurrent quiver $Q$ admitting a subadditive labeling, the associated $T$-system satisfies a linear recurrence. We proved this claim for the case when $Q$ has type $\affA\otimes A$ using a combinatorial formula due to Speyer~\cite{Sp} for the octahedron recurrence.

In this text, we classify (Section~\ref{sect:classif}) bipartite recurrent quivers admitting a \emph{weakly subadditive labeling}. We use \emph{algebraic entropy} as a motivating property of the $T$-system that conjecturally characterizes such quivers. Having zero algebraic entropy is a frequently used criterion for checking integrability of a discrete dynamical system. It was introduced in~\cite{FV} and further developed in~\cite{BV,HV}. Roughly speaking, the fact that the algebraic entropy of a discrete dynamical system is nonzero means that its values grow \emph{doubly exponentially}, i.e., as $\exp(\exp(ct))$ for some positive constant $c$. We show in Section~\ref{sect:entropy} that for any bipartite recurrent quiver that does not admit a weakly subadditive labeling, the $T$-system has nonzero algebraic entropy. Using our classification again we get rich computational evidence suggesting that for the remaining bipartite recurrent quivers (that is, the ones from our classification), the values of the $T$-system grow \emph{quadratic exponentially}, i.e. as $\exp(ct^2)$. Thus there seems to be a big gap in the rate of growth that separates the $T$-systems associated to the quivers in our classification from all other $T$-systems. 

For two special cases (quivers of type $\affA\otimes \affA$ and \emph{twists} $\affL\times \affL$ of affine $ADE$ Dynkin diagrams) we prove in Sections~\ref{sect:AA} and~\ref{sect:twists} respectively that the growth is quadratic exponential. The former again is a consequence of Speyer's formula for the octahedron recurrence. We finish the text (Section~\ref{sect:conj}) by giving some refinements of the rate of growth conjecture. In particular, we conjecture that the $Y$-systems associated to the quivers from our classification are \emph{Arnold-Liouville integrable}.

\section{Main results}
Let us start by introducing some notions necessary to formulate our results. As we have mentioned, a \emph{quiver} $Q$ is a directed graph without loops and pairs of arrows forming a directed $2$-cycle.

Given a quiver $Q$ with vertex set $\VertQ$, a vertex $v\in\VertQ$, and a family $T_\ast=(T_u)_{u\in\VertQ}$ of rational functions in some set $\x$ of variables, one can define the \emph{mutation} operation $\mu_v$ that produces a new quiver $\mu_v(Q)$ with the same set $\VertQ$ of vertices and a new family $\mu_v(T_\ast)=(T_u')_{u\in\VertQ}$ according to a certain set of combinatorial rules. The definition of the quiver $\mu_v(Q)$ is given in Definition~\ref{dfn:mutations} and $\mu_v(T_\ast)$ is defined as follows. For $u\neq v$, we set $T'_u:=T_u$, and for $u=v$ we put 
\begin{equation}\label{eq:T_mutation}
T_v'=\frac{\prod_{u\to v} T_u+\prod_{v\to w} T_w}{T_v}.
\end{equation}
Here the product is taken over all arrows in $Q$. It follows from the definition that the operations $\mu_v$ and $\mu_w$ commute when there are no arrows between $v$ and $w$ in $Q$. 

We say that a quiver $Q$ is \emph{bipartite} if there exists a map $ \e:\VertQ\to\{0,1\},\ v\mapsto\e_v$ called a \emph{bipartition} such that for every arrow $u\to v$ of $Q$ we have $\e_u\neq \e_v$. It follows that for a bipartite quiver, the operations 
 \begin{equation}\label{eq:mu_0_1}
 \mu_0=\prod_{u:\e_u=0} \mu_u;\quad \mu_1=\prod_{v:\e_v=1} \mu_v
 \end{equation}
are well defined since the results of products are independent of the order. We are now ready to introduce a crucial notion of a \emph{recurrent} quiver. 
\begin{definition}\label{dfn:recurrent}
	We say that a bipartite quiver $Q$ is \emph{recurrent} if $\mu_0(Q)=\mu_1(Q)=Q^\op$ where $Q^\op$ is the quiver obtained from $Q$ by reversing all of its arrows.
\end{definition}
We give an alternative simpler definition for bipartite recurrent quivers in Corollary~\ref{cor:recurrent_commuting}. 

Let us now define the $T$-system. The main part of the definition will be equation~\eqref{eq:T_system}. Note that for each of the terms $T_v(t+1),T_v(t-1),T_u(t),T_w(t)$ involved in~\eqref{eq:T_system}, the numbers 
\[\e_v+t+1,\e_v+t-1,\e_u+t,\e_w+t\]
all have the same parity. Thus it makes sense to restrict the values of the $T$-system to only pairs $(v,t)$ such that $t\equiv\e_v\pmod 2$.

\begin{definition}\label{dfn:T_system}
	Given a bipartite recurrent quiver $Q$ with a bipartition $\e$, the \emph{$T$-system} associated with $Q$ is a family $T_v(t)$ of rational functions in variables $\x=\{x_v\}_{v\in\VertQ}$ defined for any $v\in\VertQ$ and any $t\in\Z$ satisfying $t\equiv\e_v\pmod2$. For any $v\in\VertQ$ and $t\not\equiv\e_v\pmod2$, the values of the $T$-system are required to satisfy~\eqref{eq:T_system}. Finally, for each $v\in\VertQ$, we impose an initial condition
	\begin{equation}\label{eq:initial}
	T_v(\e_v)=x_v.
	\end{equation}
\end{definition}
One easily observes that~\eqref{eq:initial} and~\eqref{eq:T_system} determine $T_v(t)$ uniquely for any $t\equiv \e_v\pmod 2$.

Since the $T$-system is defined for only bipartite recurrent quivers, it can be viewed as a composition of mutations $\m_0\m_1$, see~\eqref{eq:T_mutation} and~\eqref{eq:mu_0_1}. As a consequence, it follows from the \emph{Laurent phenomenon} property of cluster algebras~\cite{FZ} that for any $t\equiv \e_v\pmod2$, the value $T_v(t)$ is actually a \emph{Laurent polynomial} in $\x$: $T_v(t)\in\Z[\x^{\pm1}]$.

\def\maxdeg{{ \operatorname{deg}_{\max}}}
\def\mindeg{{ \operatorname{deg}_{\min}}}

For $v,u\in\VertQ$ and $t\equiv\e_v\pmod2$, define $\maxdeg(x_u;T_v(t))$ to be the maximal degree of the variable $x_u$ in the Laurent polynomial $T_v(t)$. 

\begin{definition}
We say that $Q$ has \emph{algebraic entropy zero} if for any two vertices $u,v\in\VertQ$ we have 
\begin{equation}\label{eq:entropy}
\lim_{t\to \infty} \frac{\log\left(\maxdeg(x_u,T_v(\e_v+2t))\right)}{t}=0.
\end{equation}
\end{definition}

Before we state our main results, let us give one more definition. 
\begin{definition}\label{dfn:weakly_subadd}
	Given a quiver $Q$, we say that a map $\l:\VertQ\to \R_{>0}$ is a \emph{weakly subadditive labeling} if for any vertex $v\in\VertQ$, we have
	\begin{equation}\label{eq:weakly_subadd}
		2\l(v)\geq\max\left(\sum_{u\to v} \l(u),\sum_{v\to w} \l(v)\right).
	\end{equation}
\end{definition}

This is a generalization of Vinberg's \emph{additive functions}~\cite{V}. We recall the analogous definitions of \emph{strictly subadditive} and \emph{subadditive} labelings in Definition~\ref{dfn:subadditive}.

\begin{theorem}\label{thm:entropy}
	Suppose that $Q$ is a bipartite recurrent quiver. If $Q$ does not admit a weakly subadditive labeling then $Q$ does not have zero algebraic entropy.
\end{theorem}

Our second main result is the classification (Theorem~\ref{thm:classification}) of bipartite recurrent quivers that admit weakly subadditive labelings.

By Theorem~\ref{thm:entropy}, every quiver $Q$ that has algebraic entropy zero admits a weakly subadditive labeling and therefore is necessarily one of the quivers in our classification. According to our computer experiments, we give a precise conjecture describing the asymptotics of the $T$-system.

\begin{definition}
	Let $f(0),f(1),\dots$ be a sequence of positive real numbers. We say that $f$
	\begin{enumerate}
		\item is \emph{bounded} if there exists a constant $M$ such that $f(t)<M$ for all $t\geq 0$;
		\item \emph{grows exponentially} if there exists a positive limit of 
		\[\frac{\log(f(t))}t;\]
		\item \emph{grows quadratic exponentially} if there exists a positive limit of 
		\[\frac{\log(f(t))}{t^2};\]
		\item \emph{grows doubly exponentially} if there exists a positive limit of 
		\[\frac{\log\log(f(t))}t.\]
	\end{enumerate}
\end{definition}

\begin{conjecture}\label{conj:master}
	Let $Q$ be a bipartite recurrent quiver. Substitute some positive real numbers for each variable in $\x$. Let $f(t)=T_v(\e_v+2t), t\geq 0$ be the sequence of values of the $T$-system at vertex $v$ that we get after such a substitution. 
	\begin{enumerate}
		\item\label{item:periodic} If $Q$ admits a strictly subadditive labeling then $f$ is bounded (and in fact is periodic);
		\item\label{item:linearizable} otherwise, if $Q$ admits a subadditive labeling then $f$ grows exponentially (and satisfies a linear recurrence);
		\item\label{item:quadratic_exp} otherwise, if $Q$ admits a weakly subadditive labeling then $f$ grows quadratic exponentially;
		\item\label{item:doubly_exp} otherwise $f$ grows doubly exponentially.
	\end{enumerate}
\end{conjecture}

Some parts of this conjecture are already proven. For example, we proved part~\eqref{item:periodic} in~\cite{GP1}. A weaker version of part~\eqref{item:linearizable} was shown in~\cite{GP2}. In this paper, we prove part~\eqref{item:doubly_exp} and a weaker version of part~\eqref{item:quadratic_exp}. When all components of the bigraph associated with $Q$ (see Definition~\ref{dfn:bigraph}) are of type $A$ or $\affA$, we prove all parts of Conjecture~\ref{conj:master} in full generality. We do the same when $Q$ is a \emph{twist} (see Definition~\ref{dfn:twist}) of a finite or affine $ADE$ Dynkin diagram.

\begin{remark}
	Discrete dynamical systems exhibiting only bounded, linear, quadratic, or exponential growth of the degrees appear in surprisingly many diverse contexts. One example is the analogous result for cluster mutation-periodic quivers with period $1$, see~\cite[Theorem~3.12]{FH}. Other instances include~\cite{FM,DF}. We thank Andrew Hone for bringing these references to our attention.
\end{remark}

\section{Preliminaries}

\subsection{Bigraphs}
We follow very closely the exposition in~\cite{GP2}. We start by briefly recalling the correspondence between bipartite quivers and \emph{bipartite bigraphs} introduced by Stembridge~\cite{S}. 
\begin{definition}
	A \emph{bigraph} is a pair $G=(\Gamma,\Delta)$ of simple undirected graphs on the same vertex set that do not share any edges. A bigraph is called \emph{bipartite} if there is a map $\e:V\to\{0,1\}$ such that for every edge $(u,v)$ of $\Gamma$ or $\Delta$ we have $\e_u\neq \e_v$.
\end{definition}
The graphs $\Gamma$ and $\Delta$ are allowed to have multiple edges but no loops. Throughout, we assume all bigraphs to be bipartite.

Since our main result is a classification of certain bigraphs, we write down the obvious definition of an isomorphism between two such objects.
\begin{definition}
	Two bigraphs $G=(\Gamma,\Delta)$ and $G'=(\Gamma',\Delta')$ are called \emph{isomorphic} if there is a map $\phi:\Vert(G)\to\Vert(G')$ such that for any $u,v\in\Vert(G)$ we have 
	\begin{itemize}
		\item $(u,v)\in\Gamma\Leftrightarrow (\phi(u),\phi(v))\in\Gamma'$, and
		\item $(u,v)\in\Delta\Leftrightarrow (\phi(u),\phi(v))\in\Delta'$.
	\end{itemize}
\end{definition}

There is a simple correspondence between bipartite quivers and bipartite bigraphs that we now explain.

\begin{definition}\label{dfn:bigraph}
	Given a bipartite quiver $Q$ with bipartition $\e:\VertQ\to\{0,1\}$, define the bigraph $G(Q)=(\Gamma(Q),\Delta(Q))$ with the same vertex set as follows:
	\begin{itemize}
		\item For every directed edge $u\to v$ of $Q$ with $\e_u=0,\e_v=1$, $\Gamma(Q)$ contains an undirected edge $(u,v)$.
		\item For every directed edge $u\to v$ of $Q$ with $\e_u=1,\e_v=0$, $\Delta(Q)$ contains an undirected edge $(u,v)$.
	\end{itemize}
\end{definition}
Thus every arrow of $Q$ corresponds to precisely one edge of $G(Q)$. Note also that this is a bijection: given a bipartite bigraph $G$ with a bipartition $\e$, one can easily reconstruct the bipartite quiver $Q=Q(G)$ such that $G=G(Q)$.

We represent a bigraph $G=(\Gamma,\Delta)$ as a simple graph with the edges of $\Gamma$ colored \emph{red} and the edges of $\Delta$ colored \emph{blue}. 

Let us recall the definition of a \emph{tensor product} of two bipartite graphs:
\begin{definition}\label{dfn:tensor_product}
 Let $S$ and $T$ be two bipartite undirected graphs. Then their \emph{tensor product} $S\otimes T$ is a bipartite bigraph $G=(\Gamma,\Delta)$ with vertex set $\Vert(S)\times \Vert(T)$ and the following edge sets:
 \begin{itemize}
  \item for each edge $\{u,u'\}\in S$ and each vertex $v\in T$ there is an edge between $(u,v)$ and $(u',v)$ in $\Gamma$;
  \item for each vertex $u\in S$ and each edge $\{v,v'\}\in T$ there is an edge between $(u,v)$ and $(u,v')$ in $\Delta$;
 \end{itemize}
 An example of a tensor product is given in Figure~\ref{fig:tensor_product}.
\end{definition}

\begin{figure}
\scalebox{0.8}{
\begin{tikzpicture}
\coordinate (v0x0) at (0.00,-3.60);
\coordinate (v0x1) at (0.94,-1.90);
\coordinate (v0x2) at (0.94,1.70);
\coordinate (v0x3) at (0.00,3.60);
\coordinate (v0x4) at (-0.94,2.40);
\coordinate (v0x5) at (-0.94,-1.20);
\coordinate (v1x0) at (7.87,-3.60);
\coordinate (v1x1) at (8.81,-1.90);
\coordinate (v1x2) at (8.81,1.70);
\coordinate (v1x3) at (7.87,3.60);
\coordinate (v1x4) at (6.94,2.40);
\coordinate (v1x5) at (6.94,-1.20);
\draw[color=red,line width=0.75mm] (v0x1) to[] (v0x0);
\draw[color=red,line width=0.75mm] (v0x2) to[] (v0x1);
\draw[color=red,line width=0.75mm] (v0x3) to[] (v0x2);
\draw[color=red,line width=0.75mm] (v0x4) to[] (v0x3);
\draw[color=red,line width=0.75mm] (v0x5) to[] (v0x0);
\draw[color=red,line width=0.75mm] (v0x5) to[] (v0x4);
\draw[color=blue,line width=0.75mm] (v1x0) to[bend right=5] (v0x0);
\draw[color=blue,line width=0.75mm] (v1x0) to[bend left=5] (v0x0);
\draw[color=blue,line width=0.75mm] (v1x1) to[bend right=5] (v0x1);
\draw[color=blue,line width=0.75mm] (v1x1) to[bend left=5] (v0x1);
\draw[color=red,line width=0.75mm] (v1x1) to[] (v1x0);
\draw[color=blue,line width=0.75mm] (v1x2) to[bend right=5] (v0x2);
\draw[color=blue,line width=0.75mm] (v1x2) to[bend left=5] (v0x2);
\draw[color=red,line width=0.75mm] (v1x2) to[] (v1x1);
\draw[color=blue,line width=0.75mm] (v1x3) to[bend right=5] (v0x3);
\draw[color=blue,line width=0.75mm] (v1x3) to[bend left=5] (v0x3);
\draw[color=red,line width=0.75mm] (v1x3) to[] (v1x2);
\draw[color=blue,line width=0.75mm] (v1x4) to[bend right=5] (v0x4);
\draw[color=blue,line width=0.75mm] (v1x4) to[bend left=5] (v0x4);
\draw[color=red,line width=0.75mm] (v1x4) to[] (v1x3);
\draw[color=blue,line width=0.75mm] (v1x5) to[bend right=5] (v0x5);
\draw[color=blue,line width=0.75mm] (v1x5) to[bend left=5] (v0x5);
\draw[color=red,line width=0.75mm] (v1x5) to[] (v1x0);
\draw[color=red,line width=0.75mm] (v1x5) to[] (v1x4);
\draw[fill=black!20!white] (v0x0.center) circle (0.2);
\draw[fill=white] (v0x1.center) circle (0.2);
\draw[fill=black!20!white] (v0x2.center) circle (0.2);
\draw[fill=white] (v0x3.center) circle (0.2);
\draw[fill=black!20!white] (v0x4.center) circle (0.2);
\draw[fill=white] (v0x5.center) circle (0.2);
\draw[fill=white] (v1x0.center) circle (0.2);
\draw[fill=black!20!white] (v1x1.center) circle (0.2);
\draw[fill=white] (v1x2.center) circle (0.2);
\draw[fill=black!20!white] (v1x3.center) circle (0.2);
\draw[fill=white] (v1x4.center) circle (0.2);
\draw[fill=black!20!white] (v1x5.center) circle (0.2);
\end{tikzpicture}}
\caption{\label{fig:tensor_product} A tensor product of a hexagon (type $\affA_5$) and a $2$-cycle (type $\affA_1$). Tensor products are listed as family~\bg{tensor} in our classification in Section~\ref{sect:classif}.}
\end{figure}
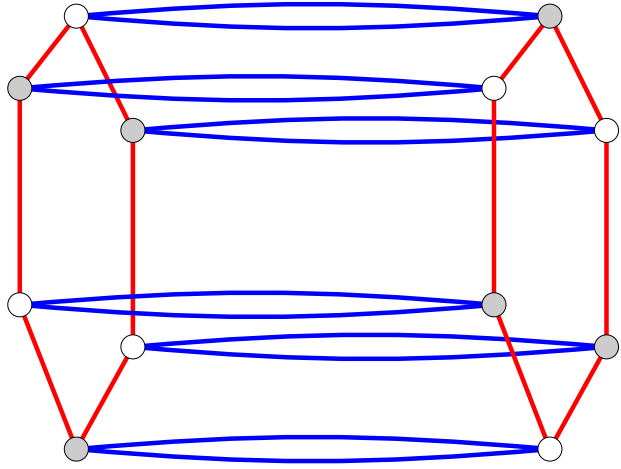

\subsubsection{$T$-systems for bipartite bigraphs}\label{subsec:bigraph_quiver_translation}
Just as in~\cite{GP1,GP2}, we reformulate the definition of the $T$-system in the language of bigraphs that is more convenient for us to work with.

\begin{definition}
  Let $G=(\Gamma,\Delta)$ be a bipartite bigraph with vertex set $V$. Then the associated $T$-system for $G$ is defined as follows:
 \begin{eqnarray*}
 T_v(t+1)T_v(t-1)&=&\prod_{(u,v)\in\Gamma} T_u(t)+\prod_{(v,w)\in\Delta} T_w(t);\\
 T_v(\e_v) &=& x_v.
 \end{eqnarray*}
 \end{definition}
 It is easy to see that we have $T_v(t)=T'_v(t)$ where $T'_v(t)$ is the value of the $T$-system associated with $Q(G)$ via Definition~\ref{dfn:T_system}.

\subsection{Finite and affine $ADE$ Dynkin diagrams}\label{sec:ADEDynkin} 

\def\affH{{h^\parr2}}

\begin{figure}

\begin{tabular}{ccc}
{\bf Name}
&
{\bf Finite diagram}
&
$h(\Lambda)$
\\

$A_\ell$ ($\ell\geq1$)
&
\scalebox{0.6}{
\begin{tikzpicture}
\coordinate (v0) at (0.00,0.00);
\coordinate (v1) at (1.50,0.00);
\coordinate (v2) at (3.00,0.00);
\coordinate (v3) at (4.50,0.00);
\coordinate (v4) at (6.00,0.00);
\draw[color=black,line width=0.75mm] (v1) to[] (v0);
\draw[color=black,line width=0.75mm] (v2) to[] (v1);
\draw[color=black,line width=0.75mm] (v3) to[] (v2);
\draw[color=black,line width=0.75mm] (v4) to[] (v3);
\draw[fill=black!20!white] (v0.center) circle (0.2);
\draw[fill=white] (v1.center) circle (0.2);
\draw[fill=black!20!white] (v2.center) circle (0.2);
\draw[fill=white] (v3.center) circle (0.2);
\draw[fill=black!20!white] (v4.center) circle (0.2);
\end{tikzpicture}}
&
$\ell+1$
\\

$D_\ell$ ($\ell\geq4$)
&
\scalebox{0.6}{
\begin{tikzpicture}
\coordinate (v0) at (0.00,0.30);
\coordinate (v1) at (0.00,-0.30);
\coordinate (v2) at (1.50,0.00);
\coordinate (v3) at (3.00,0.00);
\coordinate (v4) at (4.50,0.00);
\coordinate (v5) at (6.00,0.00);
\draw[color=black,line width=0.75mm] (v2) to[] (v0);
\draw[color=black,line width=0.75mm] (v2) to[] (v1);
\draw[color=black,line width=0.75mm] (v3) to[] (v2);
\draw[color=black,line width=0.75mm] (v4) to[] (v3);
\draw[color=black,line width=0.75mm] (v5) to[] (v4);
\draw[fill=white] (v0.center) circle (0.2);
\draw[fill=white] (v1.center) circle (0.2);
\draw[fill=black!20!white] (v2.center) circle (0.2);
\draw[fill=white] (v3.center) circle (0.2);
\draw[fill=black!20!white] (v4.center) circle (0.2);
\draw[fill=white] (v5.center) circle (0.2);
\end{tikzpicture}}
&
$2\ell-2$
\\

$E_6$
&
\scalebox{0.6}{
\begin{tikzpicture}
\coordinate (v0) at (3.00,0.90);
\coordinate (v1) at (0.00,0.00);
\coordinate (v2) at (1.50,0.00);
\coordinate (v3) at (3.00,0.00);
\coordinate (v4) at (4.50,0.00);
\coordinate (v5) at (6.00,0.00);
\draw[color=black,line width=0.75mm] (v2) to[] (v1);
\draw[color=black,line width=0.75mm] (v3) to[] (v0);
\draw[color=black,line width=0.75mm] (v3) to[] (v2);
\draw[color=black,line width=0.75mm] (v4) to[] (v3);
\draw[color=black,line width=0.75mm] (v5) to[] (v4);
\draw[fill=black!20!white] (v0.center) circle (0.2);
\draw[fill=white] (v1.center) circle (0.2);
\draw[fill=black!20!white] (v2.center) circle (0.2);
\draw[fill=white] (v3.center) circle (0.2);
\draw[fill=black!20!white] (v4.center) circle (0.2);
\draw[fill=white] (v5.center) circle (0.2);
\end{tikzpicture}}
&
$12$
\\

$E_7$
&
\scalebox{0.6}{
\begin{tikzpicture}
\coordinate (v0) at (3.00,0.90);
\coordinate (v1) at (0.00,0.00);
\coordinate (v2) at (1.50,0.00);
\coordinate (v3) at (3.00,0.00);
\coordinate (v4) at (4.50,0.00);
\coordinate (v5) at (6.00,0.00);
\coordinate (v6) at (7.50,0.00);
\draw[color=black,line width=0.75mm] (v2) to[] (v1);
\draw[color=black,line width=0.75mm] (v3) to[] (v0);
\draw[color=black,line width=0.75mm] (v3) to[] (v2);
\draw[color=black,line width=0.75mm] (v4) to[] (v3);
\draw[color=black,line width=0.75mm] (v5) to[] (v4);
\draw[color=black,line width=0.75mm] (v6) to[] (v5);
\draw[fill=black!20!white] (v0.center) circle (0.2);
\draw[fill=white] (v1.center) circle (0.2);
\draw[fill=black!20!white] (v2.center) circle (0.2);
\draw[fill=white] (v3.center) circle (0.2);
\draw[fill=black!20!white] (v4.center) circle (0.2);
\draw[fill=white] (v5.center) circle (0.2);
\draw[fill=black!20!white] (v6.center) circle (0.2);
\end{tikzpicture}}
&
$18$
\\

$E_8$
&
\scalebox{0.6}{
\begin{tikzpicture}
\coordinate (v0) at (3.00,0.90);
\coordinate (v1) at (0.00,0.00);
\coordinate (v2) at (1.50,0.00);
\coordinate (v3) at (3.00,0.00);
\coordinate (v4) at (4.50,0.00);
\coordinate (v5) at (6.00,0.00);
\coordinate (v6) at (7.50,0.00);
\coordinate (v7) at (9.00,0.00);
\draw[color=black,line width=0.75mm] (v2) to[] (v1);
\draw[color=black,line width=0.75mm] (v3) to[] (v0);
\draw[color=black,line width=0.75mm] (v3) to[] (v2);
\draw[color=black,line width=0.75mm] (v4) to[] (v3);
\draw[color=black,line width=0.75mm] (v5) to[] (v4);
\draw[color=black,line width=0.75mm] (v6) to[] (v5);
\draw[color=black,line width=0.75mm] (v7) to[] (v6);
\draw[fill=black!20!white] (v0.center) circle (0.2);
\draw[fill=white] (v1.center) circle (0.2);
\draw[fill=black!20!white] (v2.center) circle (0.2);
\draw[fill=white] (v3.center) circle (0.2);
\draw[fill=black!20!white] (v4.center) circle (0.2);
\draw[fill=white] (v5.center) circle (0.2);
\draw[fill=black!20!white] (v6.center) circle (0.2);
\draw[fill=white] (v7.center) circle (0.2);
\end{tikzpicture}}
&
$30$
\\

\end{tabular}

\caption{\label{fig:finADE} Finite $ADE$ Dynkin diagrams and their Coxeter numbers. Each diagram whose name contains index $\ell$ has $\ell$ vertices.}
\end{figure}

 \def\zoom{0.8}
 
\begin{figure}

\begin{tabular}{ccc}
{\bf Name}
&
{\bf Affine diagram}
&
$\affH(\affL)$
\\

&
 
&
 
\\

$\affA_{\ell}$  ($\ell\geq1$)
&
\scalebox{0.6}{
\begin{tikzpicture}
\node[draw,circle,fill=black!20!white] (v0) at (-4.80,0.00) {1};
\node[draw,circle,fill=white] (v1) at (-3.47,0.34) {1};
\node[draw,circle,fill=black!20!white] (v2) at (-0.08,0.48) {1};
\node[draw,circle,fill=white] (v3) at (3.32,0.34) {1};
\node[draw,circle,fill=black!20!white] (v4) at (4.80,0.00) {1};
\node[draw,circle,fill=white] (v5) at (3.99,-0.34) {1};
\node[draw,circle,fill=black!20!white] (v6) at (0.60,-0.48) {1};
\node[draw,circle,fill=white] (v7) at (-2.79,-0.34) {1};
\draw[color=black,line width=0.75mm] (v1) to[] (v0);
\draw[color=black,line width=0.75mm] (v2) to[] (v1);
\draw[color=black,line width=0.75mm] (v3) to[] (v2);
\draw[color=black,line width=0.75mm] (v4) to[] (v3);
\draw[color=black,line width=0.75mm] (v5) to[] (v4);
\draw[color=black,line width=0.75mm] (v6) to[] (v5);
\draw[color=black,line width=0.75mm] (v7) to[] (v0);
\draw[color=black,line width=0.75mm] (v7) to[] (v6);
\end{tikzpicture}}
&
$\ell+1$
\\

&
 
&
 
\\

$\affD_{\ell}$  ($\ell\geq4$)
&
\scalebox{0.6}{
\begin{tikzpicture}
\node[draw,circle,fill=white] (v0) at (-4.44,-0.30) {1};
\node[draw,circle,fill=white] (v1) at (-3.60,0.30) {1};
\node[draw,circle,fill=black!20!white] (v2) at (-1.20,0.00) {2};
\node[draw,circle,fill=white] (v3) at (1.20,0.00) {2};
\node[draw,circle,fill=black!20!white] (v4) at (3.60,0.00) {2};
\node[draw,circle,fill=white] (v5) at (6.84,-0.30) {1};
\node[draw,circle,fill=white] (v6) at (6.00,0.30) {1};
\draw[color=black,line width=0.75mm] (v2) to[] (v0);
\draw[color=black,line width=0.75mm] (v2) to[] (v1);
\draw[color=black,line width=0.75mm] (v3) to[] (v2);
\draw[color=black,line width=0.75mm] (v4) to[] (v3);
\draw[color=black,line width=0.75mm] (v5) to[] (v4);
\draw[color=black,line width=0.75mm] (v6) to[] (v4);
\end{tikzpicture}}
&
$4(\ell-2)$
\\

&
 
&
 
\\

$\affE_6$
&
\scalebox{0.6}{
\begin{tikzpicture}
\node[draw,circle,fill=black!20!white] (v0) at (9.60,0.00) {1};
\node[draw,circle,fill=white] (v1) at (7.20,0.00) {2};
\node[draw,circle,fill=black!20!white] (v2) at (0.72,-0.30) {1};
\node[draw,circle,fill=white] (v3) at (3.12,-0.30) {2};
\node[draw,circle,fill=black!20!white] (v4) at (4.80,0.00) {3};
\node[draw,circle,fill=white] (v5) at (1.68,0.30) {2};
\node[draw,circle,fill=black!20!white] (v6) at (-0.72,0.30) {1};
\draw[color=black,line width=0.75mm] (v1) to[] (v0);
\draw[color=black,line width=0.75mm] (v3) to[] (v2);
\draw[color=black,line width=0.75mm] (v4) to[] (v1);
\draw[color=black,line width=0.75mm] (v4) to[] (v3);
\draw[color=black,line width=0.75mm] (v5) to[] (v4);
\draw[color=black,line width=0.75mm] (v6) to[] (v5);
\end{tikzpicture}}
&
$24$
\\

&
 
&
 
\\

$\affE_7$
&
\scalebox{0.6}{
\begin{tikzpicture}
\node[draw,circle,fill=white] (v0) at (0.00,0.00) {2};
\node[draw,circle,fill=white] (v1) at (8.88,-0.30) {1};
\node[draw,circle,fill=black!20!white] (v2) at (6.48,-0.30) {2};
\node[draw,circle,fill=white] (v3) at (4.08,-0.30) {3};
\node[draw,circle,fill=black!20!white] (v4) at (2.40,0.00) {4};
\node[draw,circle,fill=white] (v5) at (5.52,0.30) {3};
\node[draw,circle,fill=black!20!white] (v6) at (7.92,0.30) {2};
\node[draw,circle,fill=white] (v7) at (10.32,0.30) {1};
\draw[color=black,line width=0.75mm] (v2) to[] (v1);
\draw[color=black,line width=0.75mm] (v3) to[] (v2);
\draw[color=black,line width=0.75mm] (v4) to[] (v0);
\draw[color=black,line width=0.75mm] (v4) to[] (v3);
\draw[color=black,line width=0.75mm] (v5) to[] (v4);
\draw[color=black,line width=0.75mm] (v6) to[] (v5);
\draw[color=black,line width=0.75mm] (v7) to[] (v6);
\end{tikzpicture}}
&
$48$
\\

&
 
&
 
\\

$\affE_8$
&
\scalebox{0.6}{
\begin{tikzpicture}
\node[draw,circle,fill=black!20!white] (v0) at (2.40,1.20) {3};
\node[draw,circle,fill=white] (v1) at (0.00,-0.00) {2};
\node[draw,circle,fill=black!20!white] (v2) at (1.20,-0.00) {4};
\node[draw,circle,fill=white] (v3) at (2.40,-0.00) {6};
\node[draw,circle,fill=black!20!white] (v4) at (3.60,-0.00) {5};
\node[draw,circle,fill=white] (v5) at (4.80,-0.00) {4};
\node[draw,circle,fill=black!20!white] (v6) at (6.00,-0.00) {3};
\node[draw,circle,fill=white] (v7) at (7.20,-0.00) {2};
\node[draw,circle,fill=black!20!white] (v8) at (8.40,-0.00) {1};
\draw[color=black,line width=0.75mm] (v2) to[] (v1);
\draw[color=black,line width=0.75mm] (v3) to[] (v0);
\draw[color=black,line width=0.75mm] (v3) to[] (v2);
\draw[color=black,line width=0.75mm] (v4) to[] (v3);
\draw[color=black,line width=0.75mm] (v5) to[] (v4);
\draw[color=black,line width=0.75mm] (v6) to[] (v5);
\draw[color=black,line width=0.75mm] (v7) to[] (v6);
\draw[color=black,line width=0.75mm] (v8) to[] (v7);
\end{tikzpicture}}
&
$120$
\\

\end{tabular}

\caption{\label{fig:affADE} Affine $ADE$ Dynkin diagrams together with their additive functions and McKay numbers. Each diagram whose name contains index $\ell$ has $\ell+1$ vertices.}
\end{figure}

 \begin{definition}\label{dfn:Vinberg_affine}
 	Given an undirected graph $G=(V,E)$ with possibly multiple edges, we say that a map $\l:V\to\R_{>0}$ is an \emph{additive function} if for all $v\in V$ we have 
 	 \begin{equation}\label{eq:Vinberg_affine} 
  2\l(v)=\sum_{(u,v)\in E} \l(u).
 \end{equation}
 \end{definition}
 
The following characterization of affine $ADE$ Dynkin diagrams is due to Vinberg~\cite{V}:
\begin{theorem}\label{thm:Vinberg}
 Let $G=(V,E)$ be an undirected graph with possibly multiple edges. Then $G$ is an affine $ADE$ Dynkin diagram if and only if there exists an additive function for $G$.
\end{theorem}
 Finite and affine $ADE$ Dynkin diagrams are given in Figures~\ref{fig:finADE} and~\ref{fig:affADE} respectively. The affine diagrams are drawn together with the values of their additive functions. We scale the values of the additive function so that they are relatively prime positive integers. Note that the only affine $ADE$ Dynkin diagram that is not a bipartite graph is $\affA_{2n}$ for $n\geq 1$.

For each finite $ADE$ Dynkin diagram $\Lambda$ there is an associated integer $h(\Lambda)$ called the \emph{Coxeter number}, see e.g.~\cite[Chapter~V, \S6]{Bourbaki}. We list Coxeter numbers of finite $ADE$ Dynkin diagrams in Figure~\ref{fig:finADE}. If $\affL$ is an affine $ADE$ Dynkin diagram, we set $h(\affL)=\infty$.

The Coxeter number has various nice interpretations, we list some of them below.
\begin{proposition}\label{prop:coxeter_eigenvalues}\leavevmode
 \begin{itemize}
  \item If $\Lambda$ is a finite $ADE$ Dynkin diagram then the dominant eigenvalue of its adjacency matrix equals $2\cos(\pi/h(\Lambda))$;
  \item if $\affL$ is an affine $ADE$ Dynkin diagram then the dominant eigenvalue of its adjacency matrix equals $2$.
  \item The Coxeter element of the Coxeter group associated with $\L$ has period $h(\L)$.
  \item If $\affL$ is the affine $ADE$ Dynkin diagram corresponding to a finite $ADE$ Dynkin diagram $\L$ then $h(\Lambda)$ equals the sum of the values of the additive function for $\affL$.
 \end{itemize} \qed
\end{proposition}
In particular, the first three claims justify setting $h(\affL):=\infty$. Motivated by the last claim, we introduce the following affine analog of the Coxeter number that will come into play in the proof of our classification in Section~\ref{sect:classif_proof}.

\begin{definition}
	Given an affine $ADE$ Dynkin diagram $\affL$, its \emph{McKay number} $\affH(\affL)$ is the sum of squares of the values of the additive function for $\affL$.
\end{definition}

The values of $\affH(\affL)$ are given in Figure~\ref{fig:affADE}. The motivation for the name comes from the fact that $\affH(\affL)$ is the size of the subgroup of $SU(2)$ associated with $\affL$ via the \emph{McKay correspondence}, see~\cite{McKay} or~\cite{Stek}.\footnote{We thank Christian Gaetz for pointing out this connection to us.}

\subsection{Generalized Cartan matrices}

\newcommand{\UA}{\text{---}}

\newcommand{\DOTS}{\cdots}
\newcommand{\DRA}{\Rightarrow}
\newcommand{\DLA}{\Leftarrow}
\newcommand{\DLRA}{\Leftrightarrow}

\newcommand{\TRA}{\Rrightarrow}
\newcommand{\TLA}{\Lleftarrow}

\newcommand{\QRA}{%
\begingroup
\tikzset{every path/.style={}}%
\tikz \draw[line width=0.13mm] (0,3pt) -- ++(1em,0) (0,1pt) -- ++(1em+1pt,0) (0,-1pt) -- ++(1em+1pt,0) (0,-3pt) -- ++(1em,0) (1em-1pt,5pt) to[out=-75,in=135] (1em+2pt,0) to[out=-135,in=75] (1em-1pt,-5pt);
\endgroup
}

\newcommand{\QLA}{%
\begingroup
\tikz
\draw[shorten >=0pt,shorten <=0pt,line width=0.13mm] (0,3pt) -- ++(-1em,0) (0,1pt) -- ++(-1em-1pt,0) (0,-1pt) -- ++(-1em-1pt,0) (0,-3pt) -- ++(-1em,0) (-1em+1pt,5pt) to[out=-105,in=45] (-1em-2pt,0) to[out=-45,in=105] (-1em+1pt,-5pt);
\endgroup
}

\newcommand{\loops}[1]{\frac12#1}

In this section, we review Kac's classification~\cite{Kac} of generalized Cartan matrices of affine type. We will however need to consider a slightly more general class of matrices.

\begin{definition}
	An $n\times n$ matrix $A=(a_{ij})_{i,j=1}^n$ is called a \emph{weak generalized Cartan matrix} if it satisfies the following axioms:
	\begin{enumerate}[(C1)]
		\item $a_{ii}\in\Z$ and $a_{ii}\leq 2$ for $i=1,\dots,n$;
		\item $a_{ij}$ are non-positive integers for $i\neq j$;
		\item $a_{ij}=0$ implies $a_{ji}=0$.
	\end{enumerate}
\end{definition}

Thus a \emph{generalized Cartan matrix} is a weak generalized Cartan matrix satisfying $a_{ii}=2$ for all $i\in [n]:=\{1,2,\dots,n\}$. Following~\cite[\S4.7]{Kac}, to each weak generalized Cartan matrix $A$ we associate its \emph{Dynkin diagram} $S(A)$ with vertex set $[n]$ as follows. We connect each vertex $i$ to itself by $2-a_{ii}$ self-loops. Two vertices $i\neq j\in[n]$ such that $|a_{ij}|\geq|a_{ji}|$ are connected by $|a_{ij}|$ lines in $S(A)$ which are equipped with an arrow pointing toward $i$ if $|a_{ij}|>1$. 
 We say that $A$ is \emph{indecomposable} if $S(A)$ is a connected graph. For a column vector $u$ with coordinates $u^T=(u_1,u_2,\dots)$, we write $u>0$ (resp., $u\geq 0$) if all $u_i>0$ (resp., all $u_i\geq 0$). 

\makeatletter
\newcommand{\mylabel}[2]{\normalfont #2\def\@currentlabel{#2}\label{#1}}
\makeatother 

\begin{theorem}[{\cite[Theorem~4.3]{Kac}}]\label{thm:fin_aff_ind}
	Let $A$ be a real $n\times n$ indecomposable weak generalized Cartan matrix. Then exactly one of the following holds:
	\begin{enumerate}
		\item[\mylabel{item:Fin}{(Fin)}] There exists $u>0$ such that $Au>0$.
		\item[\mylabel{item:Aff}{(Aff)}] There exists $u>0$ such that $Au=0$.
		\item[\mylabel{item:Ind}{(Ind)}] There exists $u>0$ such that $Au<0$.
	\end{enumerate}
\end{theorem}

In cases \ref{item:Fin}, \ref{item:Aff}, \ref{item:Ind}, we will say that $A$ is of \emph{finite, affine,} or \emph{indefinite type}, respectively.

\begin{theorem}\label{thm:Kac}
	\leavevmode
	\begin{enumerate}
		\item If $A$ is an indecomposable weak generalized Cartan matrix of finite type then $S(A)$ is one of the diagrams in Figure~\ref{fig:fin}.
		\item If $A$ is an indecomposable weak generalized Cartan matrix of affine type then $S(A)$ is one of the diagrams shown in Figure~\ref{fig:aff}. 
		\item The labels in Figure~\ref{fig:aff} are the coordinates of the unique vector $\delta=(\delta_1,\dots,\delta_n)$ such that $A\delta=0$ and the $\delta_i$ are positive relatively prime integers.
	\end{enumerate}
\end{theorem}
\begin{proof}
	Most of the statements follow from~\cite[Theorem~4.8]{Kac}. The only additional work one needs to do is the case where we have $a_{ii}<2$ for some $i\in[n]$. Suppose that $A$ is a weak generalized Cartan matrix of finite (resp., affine) type such that $a_{ii}<2$ for some $i\in[n]$. Introduce a $2n\times 2n$ weak generalized Cartan matrix $B$ with indexing set $[n]\cup[n']=\{1,2,\dots,n,1',2',\dots,n'\}$ is obtained from $A$ as follows. For $i\neq j\in[n]$, put $b_{ij}=b_{i'j'}=a_{ij}$ and put $b_{ij'}=b_{i'j}=0$. For $i\in[n]$, set $b_{ii}=b_{i'i'}=2$ and $b_{ii'}=b_{i'i}=2-a_{ii}$. It follows that $B$ is a generalized Cartan matrix of finite (resp., affine) type. Thus $S(A)$ is obtained from $S(B)$ by taking a quotient with respect to a fixed-point-free involutive automorphism of order $2$, and it is straightforward to check that the only Dynkin diagrams in Figures~\ref{fig:fin} and~\ref{fig:aff} that admit such an automorphism are $A_{2n}$, $A_{2n+1}^\parr1$, $C_{2n+1}^\parr1$, $D_{2n+1}^\parr1$, and $D_{2n+3}^\parr2$. We denote the corresponding diagram $S(A)$ by $\loops{A_{2n}}$ (Figure~\ref{fig:fin}), $\loops{A_{2n+1}^\parr1}$, $\loops{C_{2n+1}^\parr1}$, $\loops{D_{2n+1}^\parr1}$, and $\loops{D_{2n+3}^\parr2}$ (Figure~\ref{fig:aff}) respectively.
\end{proof}

Note that a $1\times 1$ matrix $A$ with $a_{11}\leq 2$ is an indecomposable weak generalized Cartan matrix for any $a_{11}$ including $a_{11}=0$. This zero $1\times 1$ matrix corresponds to the diagram $\loops{A_{2\ell+1}^\parr1}$ for $\ell=0$ which is a single vertex with two self-loops. We also denote this diagram by $A_0^\parr1$.

\def\noode{\circ}
\newcommand{\nodeZ}[1]{\begin{tikzpicture}\node () at (0,0) {#1};\end{tikzpicture}}
\def\nd{1}

\newcommand{\chain}[1]{
	\node (A1) at (0,0) {};
	\foreach [count=\i] \j in {#1}{
	    \newcommand{\prev}{(A\i.east)};
	    \node[anchor=west] (A\plusOne{\i}) at \prev {$\j$} ;
	}
}

\def\lw{0.09mm}

\newcommand{\abv}[2]{
	\node[anchor=south] (ABV#1) at (#1.north) {$\rotatebox{90}{\UA}$};
	\node[anchor=south] (ABVV#1) at (ABV#1.north) {$#2$};
}
\newcommand{\lp}[1]{
	\draw[line width=\lw] (#1) to[in=0,out=75] ++(0,0.75*\nd) to[in=105,out=180] (#1);
}
\newcommand{\lplp}[1]{
		\draw[line width=\lw] (#1) to[in=25,out=100] ++(115:0.75*\nd) to[in=130,out=205] (#1);
		\draw[line width=\lw] (#1) to[in=-25,out=50] ++(65:0.75*\nd) to[in=80,out=155] (#1);
}
\newcommand{\arcc}[2]{
		\draw[line width=\lw] (#1) to[in=160,out=20] (#2);
}

	
	\newcommand{\ZA}[4]{\begin{tikzpicture}
		\chain{#1,\UA,#2,\UA,\DOTS,\UA,#3,\UA,#4}
	\end{tikzpicture}}
	\newcommand{\ZB}[4]{\begin{tikzpicture}
		\chain{#1,\UA,#2,\UA,\DOTS,\UA,#3,\DRA,#4}
	\end{tikzpicture}}  
	\newcommand{\ZC}[4]{\begin{tikzpicture}
		\chain{#1,\UA,#2,\UA,\DOTS,\UA,#3,\DLA,#4}
	\end{tikzpicture}}
	\newcommand{\ZD}[5]{\begin{tikzpicture}
		\chain{#1,\UA,#2,\UA,\DOTS,\UA,#4,\UA,#5}
		\abv{A4}{#3}
	\end{tikzpicture}}
	\newcommand{\ZEsix}[6]{\begin{tikzpicture}
		\chain{#1,\UA,#2,\UA,#3,\UA,#5,\UA,#6}
		\abv{A6}{#4}
	\end{tikzpicture}}
	\newcommand{\ZEseven}[7]{\begin{tikzpicture}
		\chain{#1,\UA,#2,\UA,#3,\UA,#4,\UA,#6,\UA,#7}
		\abv{A8}{#5}
	\end{tikzpicture}}
	\newcommand{\ZEeight}[8]{\begin{tikzpicture}
		\chain{#1,\UA,#2,\UA,#3,\UA,#5,\UA,#6,\UA,#7,\UA,#8}
		\abv{A6}{#4}
	\end{tikzpicture}}
	\newcommand{\ZF}[4]{\begin{tikzpicture}
		\chain{#1,\UA,#2,\DRA,#3,\UA,#4}
	\end{tikzpicture}}
	\newcommand{\ZG}[2]{\begin{tikzpicture}
		\chain{#1,\TRA,#2}
	\end{tikzpicture}}
	\newcommand{\ZloopsA}[4]{\begin{tikzpicture}
		\chain{#1,\UA,#2,\UA,\DOTS,\UA,#3,\UA,#4}
		\lp{A2}
	\end{tikzpicture}}
	
	\newcommand{\Zedge}[3]{\begin{tikzpicture}
		\chain{#1,#2,#3}
	\end{tikzpicture}}
	
	\newcommand{\chaintp}[1]{\begin{tikzpicture}
		\chain{#1}
	\end{tikzpicture}}
	
	\newcommand{\Zloop}[1]{\begin{tikzpicture}
		\chain{#1}
		\lp{A2}
	\end{tikzpicture}}
	\newcommand{\Zlooploop}[1]{\begin{tikzpicture}
		\chain{#1}
		\lplp{A2}
	\end{tikzpicture}}
	
	\newcommand{\ZAA}[4]{\begin{tikzpicture}
		\chain{#1,\UA,#2,\UA,\DOTS,\UA,#3,\UA,#4}
		\arcc{A2}{A10}
	\end{tikzpicture}}
	\newcommand{\ZDD}[7]{\begin{tikzpicture}
		\chain{#1,\UA,#2,\UA,#4,\UA,\DOTS,\UA,#5,\UA,#7}
		\abv{A4}{#3}
		\abv{A10}{#6}
	\end{tikzpicture}}
	\newcommand{\ZEEsix}[7]{\begin{tikzpicture}
		\chain{#1,\UA,#2,\UA,#3,\UA,#6,\UA,#7}
		\abv{A6}{#4}
		\abv{ABVVA6}{#5}
	\end{tikzpicture}}
	\newcommand{\ZEEseven}[8]{\begin{tikzpicture}
		\chain{#1,\UA,#2,\UA,#3,\UA,#4,\UA,#6,\UA,#7,\UA,#8}
		\abv{A8}{#5}
	\end{tikzpicture}}
	\newcommand{\ZDDD}[4]{\begin{tikzpicture}
		\chain{#1,\DLA,#2,\UA,\DOTS,\UA,#3,\DRA,#4}
	\end{tikzpicture}}
	\newcommand{\ZEEeight}[9]{\begin{tikzpicture}
		\chain{#1,\UA,#2,\UA,#3,\UA,#5,\UA,#6,\UA,#7,\UA,#8,\UA,#9}
		\abv{A6}{#4}
	\end{tikzpicture}}
	\newcommand{\ZAAone}[2]{\begin{tikzpicture}
		\chain{#1,\DLRA,#2}
	\end{tikzpicture}}
	\newcommand{\ZAAtwo}[2]{\begin{tikzpicture}
		\chain{#1,\QLA,#2}
	\end{tikzpicture}}
	\newcommand{\ZGG}[3]{\begin{tikzpicture}
		\chain{#1,\UA,#2,\TRA,#3}
	\end{tikzpicture}} 
	\newcommand{\ZDDfour}[3]{\begin{tikzpicture}
		\chain{#1,\UA,#2,\TLA,#3}
	\end{tikzpicture}}
	\newcommand{\ZBB}[5]{\begin{tikzpicture}
		\chain{#1,\UA,#2,\UA,\DOTS,\UA,#4,\DRA,#5}
		\abv{A4}{#3}
	\end{tikzpicture}}
	\newcommand{\ZAAodd}[5]{\begin{tikzpicture}
		\chain{#1,\UA,#2,\UA,\DOTS,\UA,#4,\DLA,#5}
		\abv{A4}{#3}
	\end{tikzpicture}}
	\newcommand{\ZCC}[4]{\begin{tikzpicture}
		\chain{#1,\DRA,#2,\UA,\DOTS,\UA,#3,\DLA,#4}
	\end{tikzpicture}}
	\newcommand{\ZAAeven}[4]{\begin{tikzpicture}
		\chain{#1,\DLA,#2,\UA,\DOTS,\UA,#3,\DLA,#4}
	\end{tikzpicture}}
	\newcommand{\ZFF}[5]{\begin{tikzpicture}
		\chain{#1,\UA,#2,\UA,#3,\DRA,#4,\UA,#5}
	\end{tikzpicture}}
	\newcommand{\ZEEEsix}[5]{\begin{tikzpicture}
		\chain{#1,\UA,#2,\UA,#3,\DLA,#4,\UA,#5}
	\end{tikzpicture}}
	\newcommand{\ZloopsAA}[4]{\begin{tikzpicture}
		\chain{#1,\UA,#2,\UA,\DOTS,\UA,#3,\UA,#4}
		\lp{A2}
		\lp{A10}
	\end{tikzpicture}}
	\newcommand{\ZloopsCC}[4]{\begin{tikzpicture}
		\chain{#1,\DRA,#2,\UA,\DOTS,\UA,#3,\UA,#4}
		\lp{A10}
	\end{tikzpicture}}
	\newcommand{\ZloopsDD}[5]{\begin{tikzpicture}
		\chain{#1,\UA,#2,\UA,\DOTS,\UA,#4,\UA,#5}
		\abv{A4}{#3}
		\lp{A10}
	\end{tikzpicture}}
	\newcommand{\ZloopsDDD}[4]{\begin{tikzpicture}
		\chain{#1,\DLA,#2,\UA,\DOTS,\UA,#3,\UA,#4}
		\lp{A10}
	\end{tikzpicture}}


\begin{figure}
\makebox[\textwidth]{
\scalebox{0.75}{
\begin{tabular}{|c|c|}\hline
    \begin{tabular}{cc}
	\nodeZ{$A_{\ell}$}
	& 
	\ZA{\noode}{\noode}{\noode}{\noode}
\end{tabular}&
\begin{tabular}{cc}
	\nodeZ{$B_{\ell}$}
	& 
	\ZB{\noode}{\noode}{\noode}{\noode}
  \end{tabular}
\\\hline
\begin{tabular}{cc}
	\nodeZ{$C_{\ell}$}
	& 
	\ZC{\noode}{\noode}{\noode}{\noode}
\end{tabular}&
\begin{tabular}{cc}
	\nodeZ{$D_\ell$}
	& 
	\ZD{\noode}{\noode}{\noode}{\noode}{\noode}
\end{tabular}
\\\hline
    \begin{tabular}{cc}
	\nodeZ{$E_{6}$}
	& 
	\ZEsix{\noode}{\noode}{\noode}{\noode}{\noode}{\noode}
\end{tabular}&
\begin{tabular}{cc}
	\nodeZ{$E_{7}$}
	& 
	\ZEseven{\noode}{\noode}{\noode}{\noode}{\noode}{\noode}{\noode}
  \end{tabular}
\\\hline

\begin{tabular}{cc}
	\nodeZ{$E_{8}$}
	& 
	\ZEeight{\noode}{\noode}{\noode}{\noode}{\noode}{\noode}{\noode}{\noode}
  \end{tabular} &
\begin{tabular}{cc}
	\nodeZ{$F_{4}$}
	& 
	\ZF{\noode}{\noode}{\noode}{\noode}
\end{tabular}
\\\hline
\begin{tabular}{cc}
	\nodeZ{$G_{2}$}
	& 
	\ZG{\noode}{\noode}
\end{tabular}
&
\begin{tabular}{cc}
	\nodeZ{$\loops{A_{2\ell}}(\ell\geq 1)$}
	& 
	\ZloopsA{\noode}{\noode}{\noode}{\noode}
\end{tabular}	\\\hline
\end{tabular}
}
}

	\caption{\label{fig:fin} Dynkin diagrams of weak generalized Cartan matrices of finite type. Each diagram whose name contains index $\ell$ has $\ell$ vertices.}
\end{figure}

\begin{figure}
\makebox[\textwidth]{
\scalebox{0.75}{
\begin{tabular}{c}
\begin{tabular}{|c|c|}\hline
    \begin{tabular}{cc}
	\nodeZ{$A_{\ell}^\parr1 (\ell\geq 2)$}
	& 
	\ZAA{1}{1}{1}{1}
\end{tabular}&
\begin{tabular}{cc}
	\nodeZ{$D_{\ell}^\parr1 (\ell\geq 4)$}
	& 
	\ZDD{1}{2}{1}{2}{2}{1}{1}
  \end{tabular}
\\\hline
    \begin{tabular}{cc}
	\nodeZ{$E_{6}^\parr1$}
	& 
	\ZEEsix{1}{2}{3}{2}{1}{2}{1}
\end{tabular}&
\begin{tabular}{cc}
	\nodeZ{$E_{7}^\parr1$}
	& 
	\ZEEseven{1}{2}{3}{4}{2}{3}{2}{1}
  \end{tabular}
\\\hline

    \begin{tabular}{cc} 
	\nodeZ{$D_{\ell+1}^\parr2 (\ell\geq 2)$}
	& 
	\ZDDD{1}{1}{1}{1}
\end{tabular}&
\begin{tabular}{cc}
	\nodeZ{$E_{8}^\parr1$}
	& 
	\ZEEeight{2}{4}{6}{3}{5}{4}{3}{2}{1}
  \end{tabular}
\\\hline

\begin{tabular}{cc|cc}
	\nodeZ{$A_{1}^\parr1$}
	& 
	\ZAAone{1}{1} 
	&
	\nodeZ{$A_{2}^\parr2$}
	& 
	\ZAAtwo{2}{1}
\end{tabular}&
\begin{tabular}{cc|cc}
	\nodeZ{$G_{2}^\parr1$}
	& 
	\ZGG{1}{2}{3}
	&
	\nodeZ{$D_{4}^\parr3$}
	& 
	\ZDDfour{1}{2}{1}
\end{tabular}	\\\hline
\begin{tabular}{cc}
	\nodeZ{$B_{\ell}^\parr1 (\ell\geq 3)$}
	& 
	\ZBB{1}{2}{1}{2}{2}
\end{tabular}&
\begin{tabular}{cc}
	\nodeZ{$A_{2\ell-1}^\parr2 (\ell\geq 3)$}
	& 
	\ZAAodd{1}{2}{1}{2}{1}
	\end{tabular}
\\\hline
\begin{tabular}{cc}
\nodeZ{$C_{\ell}^\parr1 (\ell\geq 2)$}
	& 
	\ZCC{1}{2}{2}{1}
\end{tabular}&
\begin{tabular}{cc}
	\nodeZ{$A_{2\ell}^\parr2 (\ell\geq 2)$}
	& 
	\ZAAeven{2}{2}{2}{1}
\end{tabular}
\\\hline
\begin{tabular}{cc}
	\nodeZ{$F_{4}^\parr1$}
	& 
	\ZFF{1}{2}{3}{4}{2}
\end{tabular}&
\begin{tabular}{cc}
	\nodeZ{$E_{6}^\parr2$}
	& 
	\ZEEEsix{1}{2}{3}{2}{1}
  \end{tabular}\\\hline
  \begin{tabular}{cc}
\nodeZ{$\loops{A_{2\ell+1}^\parr1} (\ell\geq 0)$}
	& 
	\ZloopsAA{1}{1}{1}{1}
\end{tabular}&
\begin{tabular}{cc}
	\nodeZ{$\loops{C_{2\ell+1}^\parr1}(\ell\geq 1)$}
	& 
	\ZloopsCC{1}{2}{2}{2}
\end{tabular}\\\hline 
\begin{tabular}{cc}
	\nodeZ{$\loops{D_{2\ell+1}^\parr1} (\ell\geq 2)$}
	& 
	\ZloopsDD{1}{2}{1}{2}{2}
\end{tabular}&
\begin{tabular}{cc}
	\nodeZ{$\loops{D_{2\ell+3}^\parr2}(\ell\geq 1)$}
	& 
	\ZloopsDDD{1}{1}{1}{1}
\end{tabular}\\\hline 
\end{tabular}
\end{tabular}
}
}
	\caption{\label{fig:aff} Dynkin diagrams of weak generalized Cartan matrices of affine type. Each diagram whose name contains index $\ell$ has $\ell+1$ vertices. The names of the diagrams are taken from~\cite{Kac}.}
\end{figure}

\begin{remark}
	The diagrams in Figure~\ref{fig:finADE} also appear in Figure~\ref{fig:fin}. Similarly, the diagrams in Figure~\ref{fig:affADE} also appear (with slightly different names) in Figure~\ref{fig:aff}. This is done intentionally since we treat diagrams in Figures~\ref{fig:finADE} and~\ref{fig:affADE} as color components of bigraphs (see next section) while we get diagrams in Figures~\ref{fig:fin} and~\ref{fig:aff} as component graphs of bigraphs (see Theorem~\ref{thm:Cartan}).
\end{remark}

\subsection{Bipartite recurrent quivers and bigraphs}\label{subsec:quiver_mutations}

\begin{definition}\label{dfn:mutations}
 Let $Q$ be a quiver. For a vertex $v$ of $Q$ one can define the \emph{quiver mutation $\mu_v$ at $v$} as follows:
\begin{enumerate}
 \item\label{step:trans} for each pair of arrows $u \rightarrow v$ and $v \rightarrow w$, create an arrow $u \rightarrow w$;
 \item\label{step:reverse} reverse the direction of all arrows incident to $v$;
 \item\label{step:remove} if some directed $2$-cycle is present, remove both of its arrows; repeat until there are no more directed $2$-cycles.
\end{enumerate}
It is straightforward to check that the resulting quiver $\mu_v(Q)$ is well defined. See Figure~\ref{figure:mut} for an example of each step.
\end{definition}

\newcommand{\ga}[2]{\ncline[linecolor=orange]{#1}{#2}}
\newcommand{\gaa}[2]{\ncarc[arcangle=30,linecolor=orange]{#1}{#2}}
\newcommand{\arrr}[2]{\ncline[linecolor=orange]{#1}{#2}}

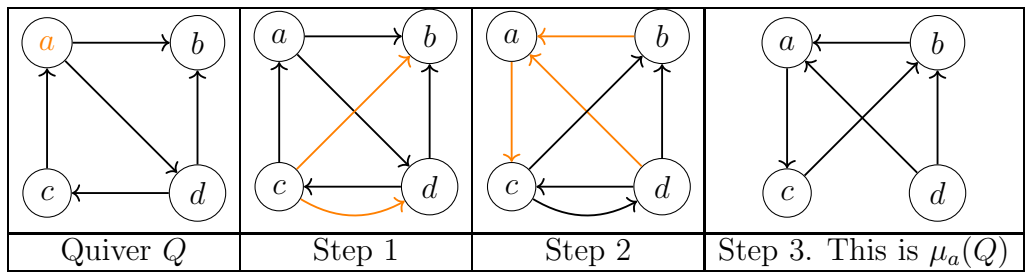
\begin{figure}
 \begin{tabular}{|c|c|c|c|}\hline

\begin{tikzpicture}[scale=2]
	\node[draw,circle] (a) at (0,1) {{\color{orange} $a$}};
	\node[draw,circle] (b) at (1,1) {$b$};
	\node[draw,circle] (d) at (1,0) {$d$};
	\node[draw,circle] (c) at (0,0) {$c$};
	\draw[->,line width=0.25mm] (a) -- (b);
 	\draw[->,line width=0.25mm] (d) -- (b);
 	\draw[->,line width=0.25mm] (d) -- (c);
 	\draw[->,line width=0.25mm] (c) -- (a);
 	\draw[->,line width=0.25mm] (a) -- (d);
\end{tikzpicture}
&
\begin{tikzpicture}[scale=2]
	\node[draw,circle] (a) at (0,1) {$a$};
	\node[draw,circle] (b) at (1,1) {$b$};
	\node[draw,circle] (d) at (1,0) {$d$};
	\node[draw,circle] (c) at (0,0) {$c$};
	\draw[->,line width=0.25mm] (a) -- (b);
 	\draw[->,line width=0.25mm] (d) -- (b);
 	\draw[->,line width=0.25mm] (a) -- (d);
 	\draw[->,line width=0.25mm] (d) -- (c);
 	\draw[->,line width=0.25mm] (c) -- (a);
 	\draw[->,line width=0.25mm,orange] (c) to[bend right] (d);
 	\draw[->,line width=0.25mm,orange] (c) -- (b);
 	
\end{tikzpicture}

&
\begin{tikzpicture}[scale=2]
	\node[draw,circle] (a) at (0,1) {$a$};
	\node[draw,circle] (b) at (1,1) {$b$};
	\node[draw,circle] (d) at (1,0) {$d$};
	\node[draw,circle] (c) at (0,0) {$c$};
	\draw[->,line width=0.25mm,orange] (b) -- (a);
 	\draw[->,line width=0.25mm] (d) -- (b);
 	\draw[->,line width=0.25mm,orange] (d) -- (a);
 	\draw[->,line width=0.25mm] (d) -- (c);
 	\draw[->,line width=0.25mm,orange] (a) -- (c);
 	\draw[->,line width=0.25mm] (c) to[bend right] (d);
 	\draw[->,line width=0.25mm] (c) -- (b);
 	
\end{tikzpicture}

&
\begin{tikzpicture}[scale=2]
	\node[draw,circle] (a) at (0,1) {$a$};
	\node[draw,circle] (b) at (1,1) {$b$};
	\node[draw,circle] (d) at (1,0) {$d$};
	\node[draw,circle] (c) at (0,0) {$c$};
	\draw[->,line width=0.25mm] (b) -- (a);
 	\draw[->,line width=0.25mm] (d) -- (b);
 	\draw[->,line width=0.25mm] (d) -- (a);
 	\draw[->,line width=0.25mm] (a) -- (c);
 	\draw[->,line width=0.25mm] (c) -- (b);
 	
\end{tikzpicture}

\\\hline
Quiver $Q$ & Step~\ref{step:trans} & Step~\ref{step:reverse}& Step~\ref{step:remove}. This is $\mu_a(Q)$\\\hline
 \end{tabular}
\caption{\label{figure:mut}Mutating a quiver $Q$ at vertex $a$. The edges changed at the corresponding step are highlighted in orange.}
\end{figure}

Now, let $Q$ be a bipartite quiver. Recall that $\mu_0$ (resp., $\mu_1$) is the simultaneous mutation at all white (resp., all black) vertices of $Q$, and that $Q$ is \emph{recurrent} if $\mu_0(Q)=\mu_1(Q)=Q^{\op}$, see Definition~\ref{dfn:recurrent}.

\begin{corollary}\label{cor:recurrent_commuting}
 A bipartite quiver $Q$ is recurrent if and only if the associated bipartite bigraph $G(Q)$ has commuting adjacency matrices $A_\Gamma,A_\Delta$. 
\end{corollary}
Equivalently, this means that for any two vertices $u,w\in\VertQ$, the number of directed $2$-paths $u\to v\to w$ in $Q$ equals the number of directed $2$-paths $w\to v\to u$ in $Q$. In other words, the number of \emph{red-blue paths} $(u,v)\in \Gamma, (v,w)\in\Delta$ in $G$ equals the number of \emph{blue-red paths} $(u,v)\in\Delta, (v,w)\in\Gamma$ in $G$.

Let us now recall a few facts from~\cite{GP2}.

\def\v{\mathbf{v}}
\def\u{\mathbf{u}}
\begin{lemma}[{\cite[Lemma~1.1.8]{GP2}}]\label{lemma:eigenvalues}
 Let $G=(\Gamma,\Delta)$ be a connected bigraph and assume that the adjacency matrices $A_\Gamma,A_\Delta$ commute. Then the dominant eigenvalues of all components of $\Gamma$ are equal to the same value $\m_\Gamma>0$, and the dominant eigenvalues of all components of $\Delta$ are equal to the same value $\m_\Delta>0$. Matrices $A_\Gamma$ and $A_\Delta$ have a common dominant eigenvector $\v>0$ such that 
 \[A_\Gamma \v=\m_\Gamma\v;\quad A_\Delta \v=\m_\Delta\v.\] \qed
\end{lemma}

Applying the well known characterization of affine and finite $ADE$ Dynkin diagrams by their eigenvalues (see Proposition~\ref{prop:coxeter_eigenvalues}), we get the following:
\begin{corollary}\label{prop:all_affine_or_all_finite}
	Suppose that a bipartite bigraph $G=(\Gamma,\Delta)$ has commuting adjacency matrices. Then exactly one of the following is true:
	\begin{enumerate}[\normalfont (i)]
		\item\label{item:fin} all components of $\Gamma$ are finite $ADE$ Dynkin diagrams;
		\item\label{item:aff} all components of $\Gamma$ are affine $ADE$ Dynkin diagrams;
		\item every component of $\Gamma$ is neither a finite nor an affine $ADE$ Dynkin diagram.
	\end{enumerate}
	A similar claim holds for the components of $\Delta$.
\end{corollary}

This motivates us to define three families of bipartite bigraphs that will be of the most importance to us.

\begin{definition}[{\cite[Definition~1.1.7]{GP2}}]\label{dfn:affinite}
Let $G=(\Gamma,\Delta)$ be a bipartite bigraph with commuting adjacency matrices. We say that:
\begin{enumerate}
	\item $G$ is a \emph{finite $\boxtimes$ finite $ADE$ bigraph} if both $\Gamma$ and $\Delta$ satisfy~\eqref{item:fin};
	\item $G$ is an \emph{\affinite $ADE$ bigraph} if $\Gamma$ satisfies~\eqref{item:aff} and $\Delta$ satisfies~\eqref{item:fin};
	\item $G$ is an \emph{affine $\boxtimes$ affine $ADE$ bigraph} if both $\Gamma$ and $\Delta$ satisfy~\eqref{item:aff}.
\end{enumerate}
\end{definition}
The finite $\boxtimes$ finite $ADE$ bigraphs have been introduced by Stembridge~\cite{S} under the name \emph{admissible $ADE$ bigraphs}.

Let $G=(\Gamma,\Delta)$ be a bipartite bigraph on vertex set $V$. A {\it {labeling}} of its vertices is a function $\nu: V\rightarrow \mathbb R_{>0}$, which assigns a positive real number $\nu(v)$ to each vertex $v$ of $G$. 
\begin{definition}[{\cite[Definition~1.1.4]{GP2}}]\label{dfn:subadditive}
A labeling $\nu:V\rightarrow\R_{>0}$ is called
\begin{itemize}
 \item \emph{strictly subadditive} if for any vertex $v\in V$, 
\[2 \nu(v) > \sum_{(u,v)\in\Gamma} \nu(u), \; \text{ and }  \; 2 \nu(v) > \sum_{(v,w)\in\Delta} \nu(w).\]
 \item \emph{subadditive} if for any vertex $v\in V$, 
\[2 \nu(v) \geq \sum_{(u,v)\in\Gamma} \nu(u), \; \text{ and }  \; 2 \nu(v) > \sum_{(v,w)\in\Delta} \nu(w).\]
 \item \emph{weakly subadditive} if for any vertex $v\in V$, 
\[2 \nu(v) \geq \sum_{(u,v)\in\Gamma} \nu(u), \; \text{ and }  \; 2 \nu(v) \geq \sum_{(v,w)\in\Delta} \nu(w).\]\item \emph{additive} if for any vertex $v\in V$, 
\[2 \nu(v) = \sum_{(u,v)\in\Gamma} \nu(u), \; \text{ and }  \; 2 \nu(v) = \sum_{(v,w)\in\Delta} \nu(w).\]
\end{itemize}
Examples of each type can be found in Figure~\ref{figure:labelings}.
\end{definition}
Thus any additive labeling is not subadditive but is weakly subadditive. As we will see later, the converse is also true: additive labelings are precisely the weakly subadditive labelings that are not subadditive.

\begin{figure}
\makebox[1.0\textwidth]{
\begin{tabular}{ccc}
\scalebox{0.65}{
\begin{tikzpicture}
\node[draw,circle,fill=black!20!white] (v0x0) at (0.00,0.00) {2};
\node[draw,circle,fill=white] (v0x1) at (0.00,1.50) {3};
\node[draw,circle,fill=black!20!white] (v0x2) at (0.00,3.00) {2};
\node[draw,circle,fill=white] (v1x0) at (1.50,0.00) {3};
\node[draw,circle,fill=black!20!white] (v1x1) at (1.50,1.50) {4};
\node[draw,circle,fill=white] (v1x2) at (1.50,3.00) {3};
\node[draw,circle,fill=black!20!white] (v2x0) at (3.00,0.00) {2};
\node[draw,circle,fill=white] (v2x1) at (3.00,1.50) {3};
\node[draw,circle,fill=black!20!white] (v2x2) at (3.00,3.00) {2};
\draw[color=blue,line width=0.75mm] (v0x1) to[] (v0x0);
\draw[color=blue,line width=0.75mm] (v0x2) to[] (v0x1);
\draw[color=red,line width=0.75mm] (v1x0) to[] (v0x0);
\draw[color=red,line width=0.75mm] (v1x1) to[] (v0x1);
\draw[color=blue,line width=0.75mm] (v1x1) to[] (v1x0);
\draw[color=red,line width=0.75mm] (v1x2) to[] (v0x2);
\draw[color=blue,line width=0.75mm] (v1x2) to[] (v1x1);
\draw[color=red,line width=0.75mm] (v2x0) to[] (v1x0);
\draw[color=red,line width=0.75mm] (v2x1) to[] (v1x1);
\draw[color=blue,line width=0.75mm] (v2x1) to[] (v2x0);
\draw[color=red,line width=0.75mm] (v2x2) to[] (v1x2);
\draw[color=blue,line width=0.75mm] (v2x2) to[] (v2x1);
\end{tikzpicture}}
&
\scalebox{0.65}{
\begin{tikzpicture}
\node[draw,circle,fill=black!20!white] (v0x0) at (3.00,-1.78) {5};
\node[draw,circle,fill=white] (v0x1) at (3.00,1.78) {5};
\node[draw,circle,fill=black!20!white] (v0x2) at (-2.10,1.78) {5};
\node[draw,circle,fill=white] (v0x3) at (-2.10,-1.78) {5};
\node[draw,circle,fill=black!20!white] (v1x0) at (-1.08,-1.08) {3};
\node[draw,circle,fill=black!20!white] (v1x1) at (-1.08,1.08) {3};
\node[draw,circle,fill=white] (v1x2) at (0.00,0.00) {6};
\node[draw,circle,fill=black!20!white] (v1x3) at (1.08,0.00) {6};
\node[draw,circle,fill=white] (v1x4) at (2.16,-1.08) {3};
\node[draw,circle,fill=white] (v1x5) at (2.16,1.08) {3};
\draw[color=red,line width=0.75mm] (v0x1) to[] (v0x0);
\draw[color=red,line width=0.75mm] (v0x2) to[] (v0x1);
\draw[color=red,line width=0.75mm] (v0x3) to[] (v0x0);
\draw[color=red,line width=0.75mm] (v0x3) to[] (v0x2);
\draw[color=blue,line width=0.75mm] (v1x0) to[] (v0x3);
\draw[color=blue,line width=0.75mm] (v1x1) to[] (v0x1);
\draw[color=blue,line width=0.75mm] (v1x2) to[bend right=13] (v0x0);
\draw[color=blue,line width=0.75mm] (v1x2) to[bend left=21] (v0x2);
\draw[color=red,line width=0.75mm] (v1x2) to[] (v1x0);
\draw[color=red,line width=0.75mm] (v1x2) to[] (v1x1);
\draw[color=blue,line width=0.75mm] (v1x3) to[bend right=28] (v0x1);
\draw[color=blue,line width=0.75mm] (v1x3) to[bend left=23] (v0x3);
\draw[color=red,line width=0.75mm] (v1x3) to[] (v1x2);
\draw[color=blue,line width=0.75mm] (v1x4) to[] (v0x0);
\draw[color=red,line width=0.75mm] (v1x4) to[] (v1x3);
\draw[color=blue,line width=0.75mm] (v1x5) to[] (v0x2);
\draw[color=red,line width=0.75mm] (v1x5) to[] (v1x3);
\end{tikzpicture}}
&
\scalebox{0.65}{
\begin{tikzpicture}
\node[draw,circle,fill=black!20!white] (v0x0) at (2.08,1.20) {1};
\node[draw,circle,fill=white] (v0x1) at (1.04,0.60) {2};
\node[draw,circle,fill=black!20!white] (v0x2) at (-2.08,1.20) {1};
\node[draw,circle,fill=white] (v0x3) at (-1.04,0.60) {2};
\node[draw,circle,fill=black!20!white] (v0x4) at (0.00,0.00) {3};
\node[draw,circle,fill=white] (v0x5) at (0.00,-1.20) {2};
\node[draw,circle,fill=black!20!white] (v0x6) at (0.00,-2.40) {1};
\node[draw,circle,fill=white] (v1x0) at (0.00,-3.60) {2};
\node[draw,circle,fill=black!20!white] (v1x1) at (3.12,-1.80) {2};
\node[draw,circle,fill=white] (v1x2) at (3.12,1.80) {2};
\node[draw,circle,fill=black!20!white] (v1x3) at (0.00,3.60) {2};
\node[draw,circle,fill=white] (v1x4) at (-3.12,1.80) {2};
\node[draw,circle,fill=black!20!white] (v1x5) at (-3.12,-1.80) {2};
\draw[color=red,line width=0.75mm] (v0x1) to[] (v0x0);
\draw[color=red,line width=0.75mm] (v0x3) to[] (v0x2);
\draw[color=red,line width=0.75mm] (v0x4) to[] (v0x1);
\draw[color=red,line width=0.75mm] (v0x4) to[] (v0x3);
\draw[color=red,line width=0.75mm] (v0x5) to[] (v0x4);
\draw[color=red,line width=0.75mm] (v0x6) to[] (v0x5);
\draw[color=blue,line width=0.75mm] (v1x0) to[bend right=30] (v0x4);
\draw[color=blue,line width=0.75mm] (v1x0) to[] (v0x6);
\draw[color=blue,line width=0.75mm] (v1x1) to[] (v0x1);
\draw[color=blue,line width=0.75mm] (v1x1) to[] (v0x5);
\draw[color=red,line width=0.75mm] (v1x1) to[] (v1x0);
\draw[color=blue,line width=0.75mm] (v1x2) to[] (v0x0);
\draw[color=blue,line width=0.75mm] (v1x2) to[bend right=30] (v0x4);
\draw[color=red,line width=0.75mm] (v1x2) to[] (v1x1);
\draw[color=blue,line width=0.75mm] (v1x3) to[] (v0x1);
\draw[color=blue,line width=0.75mm] (v1x3) to[] (v0x3);
\draw[color=red,line width=0.75mm] (v1x3) to[] (v1x2);
\draw[color=blue,line width=0.75mm] (v1x4) to[] (v0x2);
\draw[color=blue,line width=0.75mm] (v1x4) to[bend right=30] (v0x4);
\draw[color=red,line width=0.75mm] (v1x4) to[] (v1x3);
\draw[color=blue,line width=0.75mm] (v1x5) to[] (v0x3);
\draw[color=blue,line width=0.75mm] (v1x5) to[] (v0x5);
\draw[color=red,line width=0.75mm] (v1x5) to[] (v1x0);
\draw[color=red,line width=0.75mm] (v1x5) to[] (v1x4);
\end{tikzpicture}}
\\

\end{tabular}
}
\caption{\label{figure:labelings} A strictly subadditive labeling (left). A subadditive labeling (middle). A weakly subadditive labeling which is also an additive labeling (right).}
\end{figure}
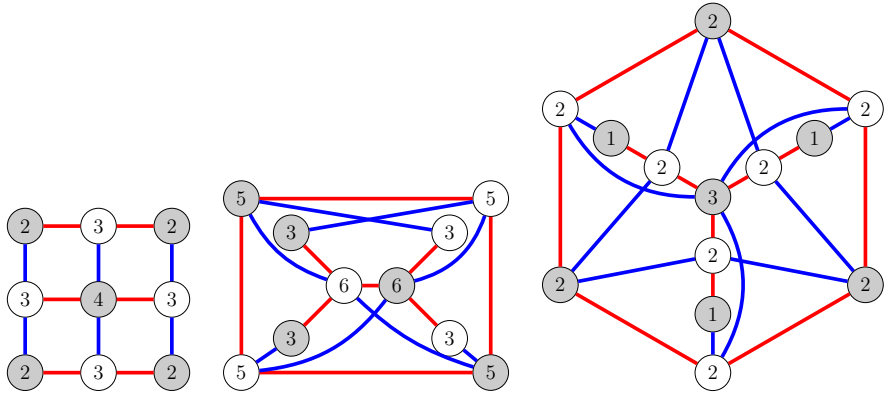
Strictly subadditive, subadditive and weakly subadditive labelings of quivers have been introduced by the second author in~\cite{P}.

The connection between Definitions~\ref{dfn:affinite} and~\ref{dfn:subadditive} is as follows.
\begin{proposition}[{\cite[Proposition~1.1.10]{GP2}}]\label{prop:affinite_subadditive}
 Let $Q$ be a bipartite recurrent quiver $Q$ and $G(Q)=(\Gamma,\Delta)$ be the corresponding bipartite bigraph. Then 
 \begin{enumerate}
  \item \label{item:admissible_ADE} $Q$ admits a strictly subadditive labeling if and only if $G(Q)$ is a finite $\boxtimes$ finite $ADE$ bigraph;
  \item \label{item:affinite} $Q$ admits a subadditive labeling which is not strictly subadditive if and only if $G(Q)$ is an \affinite $ADE$ bigraph;
  \item \label{item:affaff} $Q$ admits a weakly subadditive labeling which is not subadditive if and only if $G(Q)$ is an \affaff $ADE$ bigraph, in which case $Q$ admits an additive labeling.
 \qed
 \end{enumerate}
\end{proposition}

\subsection{Tropical $T$-systems}\label{sec:tropical}
In this section, we recall how to \emph{tropicalize} the $T$-system. We again follow the exposition of~\cite{GP2}.

Given a bipartite recurrent quiver $Q$, we call the associated $T$-system from Definition~\ref{dfn:T_system} \emph{the birational $T$-system} associated with $Q$ in order to distinguish it from another discrete dynamical system which we introduce in this section. 

\begin{definition}
 Let $Q$ be a bipartite recurrent quiver, and let $\l:\VertQ\to\R$ be any assignment of real numbers to the vertices of $Q$. Then the \emph{tropical $T$-system} associated with $Q$ and $\l$ is a family of real numbers $\Ttr_v(t)\in\R$ defined for every $v\in\VertQ,t\in\Z$ with $t\equiv\e_v\pmod2$ satisfying the following relations:
 \begin{equation}
\label{eq:tropical}
\begin{split}\Ttr_v(t+1)+\Ttr_v(t-1)&=\max\left(\sum_{u\to v}\Ttr_u(t),\sum_{v\to w}\Ttr_w(t)\right);\\
  \Ttr_v(\e_v)&=\l(v).
\end{split}
 \end{equation}
\end{definition}

The defining recurrence~\eqref{eq:tropical} can be translated into the language of bigraphs in a similar way: if $G=(\Gamma,\Delta)$ is a bipartite recurrent bigraph then the relation becomes
\[\Ttr_v(t+1)+\Ttr_v(t-1)=\max\left(\sum_{(u,v)\in\Gamma}\Ttr_u(t),\sum_{(v,w)\in\Delta}\Ttr_w(t)\right).\]

Let $P(\x)\in\Z[\x^{\pm1}]$ be a multivariate Laurent polynomial in variables $\x=(x_v)_{v\in\VertQ}$. Define $P\mid_{\x=q^\l}\in\Z[q^{\pm1}]$ to be the (univariate) Laurent polynomial in $q$ obtained from $P$ by substituting $x_v=q^{\l(v)}$ for all $v\in\VertQ$. Further, define $\maxdeg(q,P\mid_{\x=q^\l})$ to be the maximal degree of $q$ in $P\mid_{\x=q^\l}$. The following claim gives a connection between the birational and tropical $T$-systems:

\begin{proposition}[{\cite[Lemma~6.3]{GP1}}]\label{prop:maxdeg}
 For every $v\in\VertQ,t\in\Z$ with $t+\e_v$ even and any $\l:\VertQ\to\R$, we have 
 \[\Ttr_v(t)=\maxdeg\left(q,T_v(t)\mid_{\x=q^\l}\right).\]
\end{proposition}

Thus the fact that $Q$ has algebraic entropy zero can be deduced from the limiting behavior of the tropical $T$-system associated with $Q$.

\section{Algebraic entropy}\label{sect:entropy}

In this section, we prove Theorem~\ref{thm:entropy} that motivates us to classify \affaff $ADE$ bigraphs.
\begin{proof}[Proof of Theorem~\ref{thm:entropy}]
	Let $G=(\Gamma(Q),\Delta(Q))$ be the bigraph associated to $Q$. By Lemma~\ref{lemma:eigenvalues}, there exist positive real numbers $\m_\Gamma,\m_\Delta$ and a map $\l:\VertQ\to \R_{>0}$ given by $\l(v)=\v_v$ for any vertex $v$ of $Q$ such that for any vertex $v\in \VertQ$ we have
	\begin{equation}\label{eq:eigen}
	\sum_{(u,v)\in\Gamma} \l(u)=\m_\Gamma \l(v),\quad \sum_{(v,w)\in\Delta} \l(w)=\m_\Delta \l(v).		
	\end{equation}
	By symmetry we may assume that $\m_\Gamma\geq \m_\Delta$. We claim that $\m_\Gamma>2$. Suppose that this is not the case: $2\geq \m_\Gamma\geq \m_\Delta>0$. Then $\l$ is a weakly subadditive function for $Q$ which contradicts the assumption of the theorem. Thus we have $\m_\Gamma>2$. Now consider the tropical $T$-system $\Ttr$. Combining~\eqref{eq:tropical} with~\eqref{eq:eigen} yields
	\[\Ttr_v(t+1)+\Ttr_v(t-1)=\max\left(\m_\Gamma \Ttr_v(t-1),\m_\Delta\Ttr_v(t-1)\right)=\m_\Gamma\Ttr_v(t-1).\]
	Here we are using the fact that $\Ttr_v(t-1)>0$ which easily follows by induction as well as the fact that 
	\[\sum_{(u,v)\in\Gamma} \Ttr_u(t)=\m_\Gamma \Ttr_v(t-1),\quad \sum_{(v,w)\in\Delta} \Ttr_w(t)=\m_\Delta \Ttr_v(t-1).\]
	Therefore the values of the tropical $T$-system $\Ttr_v(t)$ for this special choice of $\l$ are given by 
	\[\Ttr_v(\e_v+2t)=\l(v)(\m_\Gamma-1)^{t}.\]
	Since $\m_\Gamma>2$, it follows that 
	\[\lim_{t\to \infty} \frac{\log\left(\Ttr_v(\e_v+2t)\right)}{t}=\log(\mu_\Gamma-1)>0.\]
	Now it remains to note that any point $\pbf=(p_u)_{u\in\VertQ}$ of the \emph{Newton polytope} (see~\cite[Section~6.1]{GP1}) of $T_v(\e_v+2t)$ satisfies $p_u\leq \maxdeg(x_u,T_v(\e_v+2t))$. Since $\Ttr_v(\e_v+2t)$ is just the maximum of the dot product $\<\pbf,\l\>$ over all points $\pbf$ of the Newton polytope and since $\l(u)>0$ for all $u\in\VertQ$, it follows by the Cauchy-Schwartz inequality that 
	\[\Ttr_v(\e_v+2t)\leq |\l| \cdot \sqrt{\sum_{u\in\VertQ} (\maxdeg(x_u,T_v(\e_v+2t)))^2},\]
	where $|\l|=\sqrt{\sum_{u\in\VertQ} \l(u)^2}$ does not depend on $t$.
	In particular,
	\[\Ttr_v(\e_v+2t)\leq |\l| \sqrt{|\VertQ|} \max_{u\in\VertQ}|\maxdeg(x_u,T_v(\e_v+2t))|.\]
	Taking the logarithm of both sides and dividing by $t$ yields that for at least one $u\in\VertQ$,~\eqref{eq:entropy} must fail. We are done with the proof of the theorem.
\end{proof}

\subsection{Algebraic entropy for quivers of type $\affA\otimes \affA$}\label{sect:AA}
It turns out that the statement of Conjecture~\ref{conj:master} applied to the quivers of type $\affA\otimes \affA$ can be easily proven using Speyer's formula~\cite{Sp} for the \emph{octahedron recurrence}. Let us first give a quick background before actually stating the result.

\begin{definition}
	The \emph{octahedron recurrence} is a family $T_{i,j,k}$ of rational functions in some set of variables $\x$ indexed by all triples $(i,j,k)\in\Z^3$ of integers. The values $T_{i,j,k}$ are required to satisfy 
	\[T_{i,j,k+1}T_{i,j,k-1}=T_{i,j+1,k}T_{i,j-1,k}+T_{i+1,j,k}T_{i-1,j,k}.\]
\end{definition}
Again, the parity of $i+j+k$ in each term is the same so the system splits into two independent parts. One imposes various initial conditions that define the values $T_{i,j,k}$ uniquely for all triples $(i,j,k)$, and we will be interested in assigning $T_{i,j,0}=T_{i,j,1}=x_{i,j}$ for some family $\x=(x_{i,j})_{i,j\in\Z}$ of variables. 

\begin{theorem}[{\cite{Sp}}]\label{thm:Speyer}
	For $k>1$, the value $T_{i,j,k}$ is a Laurent polynomial in $\x$. More specifically, it is a sum of $2^{k\choose 2}$ monomials, and the power of every variable $x_{i',j'}$ in every monomial belongs to the set $\{-1,0,1\}$.
\end{theorem}

Consider an infinite quiver $Q_\infty$ with vertex set $\Z^2$ such that for each vertex $(i,j)$ with $i+j$ even, we have arrows 
\[(i-1,j)\to (i,j),\quad (i+1,j)\to (i,j), \quad (i,j)\to (i,j-1),\quad (i,j)\to (i,j+1).\]
For each vertex with $(i,j)$ odd we therefore have the reverses of the above arrows in $Q_\infty$.

Fix two linearly independent vectors $A=(a_1,a_2)$ and $B=(b_1,b_2)$ in $\Z^2$ such that $a_1+a_2$ and $b_1+b_2$ are even. Suppose that the variables $x_{i,j}$ satisfy 
\begin{equation}\label{eq:x_periodic}	
x_{i,j}=x_{i+a_1,j+a_2}=x_{i+b_1,i+b_2}
\end{equation}
One can take a factor of $Q_\infty$ by the lattice generated by $A$ and $B$, we denote the resulting quiver by $Q_{A,B}$. Thus the vertices of $Q_{A,B}$ correspond to equivalence classes $(i,j)+\Z A+\Z B$ and there is an arrow from one such class to another in $Q_{A,B}$ if there is an arrow from a vertex of the first class to a vertex of the second class in $Q_\infty$. It is clear that $Q_{A,B}$ is a bipartite recurrent \affaff quiver since all components of both $\Gamma(Q)$ and $\Delta(Q)$ are of type $\affA$. In particular, if $A=(2n,0)$ and $B=(0,2m)$ then $Q_{A,B}$ has type $\affA_{2n-1}\otimes \affA_{2m-1}$. We will consider these quivers more closely in Section~\ref{subsub:toric_affA}. 

One easily observes that if the initial conditions of the octahedron recurrence satisfy~\eqref{eq:x_periodic} then the values of the octahedron recurrence coincide with the values of the $T$-system associated with the quiver $Q_{A,B}$ that we have just constructed. Substituting the values into Theorem~\ref{thm:Speyer} yields the following:
\begin{corollary}
	The $T$-system associated with $Q_{A,B}$ grows quadratic exponentially.\footnote{As it was pointed out to us by Andrew Hone, the quadratic growth for the octahedron recurrence has been shown recently by Mase~\cite[Theorem~6.8]{Mase}.}
\end{corollary}
\begin{proof}
	Indeed, the number of terms grows quadratic exponentially, and since we have substituted periodic variables into Theorem~\ref{thm:Speyer}, the degree of a variable in a monomial now grows quadratically as well. 
\end{proof}

\newcommand{\selfb}[1]{\Scal_{#1}}
\newcommand{\self}[2]{\Scal_{#1,#2}}
\def\Cartan{A}
\section{The general structure of $ADE$ bigraphs}\label{sect:general_structure}
In this section, we prove some general properties of \affaff and \affinite $ADE$ bigraphs. The main result of this section will be the construction of a weak generalized Cartan matrix $\Cartan(G)$ associated to $G$ which will later help us with the classification. We assume that the red components of $G$ are affine $ADE$ Dynkin diagrams while the blue components of $G$ are either affine or finite $ADE$ Dynkin diagrams. If $G$ has one red connected component then we say that $G$ is a \emph{self binding}. If $G$ has two red connected components then  we say that $G$ is a \emph{double binding}.

\def\wind{ \operatorname{mult}}
\subsection{The structure of self bindings}
\begin{proposition}\label{prop:self_scf}
	Let $G$ be a self binding. If $G$ is an \affinite $ADE$ bigraph, set $\wind(G):=1$, and if $G$ is an \affaff $ADE$ bigraph, set $\wind(G)=2$. Let $\v$ be the additive function for the unique red connected component of $G$.Then we have 
	\begin{equation}\label{eq:points_0}
	\wind(G)\v(v)=\sum_{(u,v)\in\Delta} \v(u).
	\end{equation}
\end{proposition}
\begin{proof}
	It follows that $\v$ is the common eigenvector for $A_\Gamma$ and $A_\Delta$ from Lemma~\ref{lemma:eigenvalues}. Thus for any $v$ we have 
	\[\sum_{(u,v)\in\Delta} \v(u)=\mu_\Delta\v(v),\]
	which shows that $\mu_\Delta$ is an integer and therefore is either equal to $1$ (in which case all components of $\Delta$ are of type $A_2$) or to $2$ (in which case all components of $\Delta$ are affine $ADE$ Dynkin diagrams).
\end{proof}

\subsection{Double bindings: scaling factor}\label{sect:double_bindings_structure}
Throughout this section, we assume that $G=(\Gamma,\Delta)$ is a double binding, and that $\Vert(G)=X\sqcup Y$, where $X$ and $Y$ are the two connected components of $\Gamma$, and recall that they are affine $ADE$ Dynkin diagrams. We also assume that every edge of $\Delta$ connects a vertex of $X$ to a vertex of $Y$. Again, we do not assume here that $h(\Gamma)=\infty$.

A \emph{parallel binding} is a bigraph of type $\affL\otimes A_2$. We let $\v$ be the common eigenvector for $A_\Gamma$ and $A_\Delta$ from Lemma~\ref{lemma:eigenvalues}, and we denote by $\v_X$ and $\v_Y$ the additive functions for $\Gamma(X)$ and $\Gamma(Y)$ from Figure~\ref{fig:affADE}. 

\begin{proposition}\label{prop:scf_types}
 There exist two integers $\scf_X(G)$ and $\scf_Y(G)$ such that 
 \begin{eqnarray}
 \label{eq:points_1}\sum_{(v,w)\in\Delta} \v_Y(w)&=&\scf_X(G)\v_X(v),\quad \forall\,v\in X;\\
 \label{eq:points_2}\sum_{(v,w)\in\Delta}\v_X(v)&=&\scf_Y(G) \v_Y(w),\quad \forall\,w\in Y. 
 \end{eqnarray}
 The pair $(\scf_X(G),\scf_Y(G))$ is denoted $\scf(G)$ and is called the \emph{scaling factor} of $G$. Exactly one of the following holds:
 
 \begin{itemize}
  \item $\scf(G)=(1,1)$ and connected components of $\Delta$ are of type $A_2$;
  \item $\scf(G)=(2,1)$ or $\scf(G)=(1,2)$ and connected components of $\Delta$ are of type $A_3$;
  \item $\scf(G)=(3,1)$ or $\scf(G)=(1,3)$ and connected components of $\Delta$ are either of type $A_5$ or of type $D_4$;
  \item $\scf(G)=(4,1)$ or $\scf(G)=(1,4)$ and connected components of $\Delta$ are affine $ADE$ Dynkin diagrams;
  \item $\scf(G)=(2,2)$ and connected components of $\Delta$ are affine $ADE$ Dynkin diagrams.
 \end{itemize}
\end{proposition}
\begin{proof}
 We copy the proof of~\cite[Proposition~2.1.4]{GP2} with slight modifications. We view maps $\tau:\Vert(G)\to\R$ as pairs $\col{\tau_X}{\tau_Y}$ where $\tau_X:\Vert(X)\to\R$ and $\tau_Y:\Vert(Y)\to\R$ are restrictions of $\tau$ to the corresponding subsets. Let $\tau=\col{\tau_X}{\tau_Y}$ be the common dominant eigenvector for $A_\Gamma$ and $A_\Delta$ from Lemma~\ref{lemma:eigenvalues}, thus $\tau(v)=\v(v)$ for all $v\in\Vert(G)$.\footnote{Recall that $\v_X$ denotes the additive function for $\Gamma(X)$ from Figure~\ref{fig:affADE}. On the other hand $\tau_X$ denote the restriction of $\tau=\v$ to $X$.} We may rescale it so that $\tau_X=\alpha\v_X$ and $\tau_Y=\v_Y$ for some $\alpha\in\R$. Since the entries of the dominant eigenvector are positive, we may assume $\alpha>0$. Now, let $\mu_\Delta:=2\cos(\pi/h(\Delta))$ be the dominant eigenvalue for $A_\Delta$, including the case $h(\Delta)=\infty$. Since $A_\Delta\tau=\mu_\Delta\tau$, we have
 \begin{eqnarray*}
 \sum_{(v,w)\in\Delta} \v_Y(w)&=&\mu_\Delta \alpha \v_X(v),\quad \forall\,v\in X;\\
 \sum_{(v,w)\in\Delta}\alpha \v_X(v)&=&\mu_\Delta \v_Y(w),\quad \forall\,w\in Y. 
 \end{eqnarray*}
 In particular, the second equation can be rewritten as
 \begin{eqnarray*}
 \sum_{(v,w)\in\Delta} \v_X(v)&=&\frac{\mu_\Delta}{\alpha} \v_Y(w),\quad \forall\,w\in Y. 
 \end{eqnarray*}
If we substitute $v\in X$ such that $\v_X(v)=1$ in the first equation, we will get that $\mu_\Delta\alpha\in\Z_{>0}$. Similarly, if we substitute $w\in X$ such that $\v_Y(w)=1$ in the second equation, we will get that $\mu_\Delta/\alpha\in\Z_{>0}$. Therefore their product $\mu_\Delta^2$ belongs to $\Z_{>0}$ as well. This proves the first part of the proposition: we have $\scf_X(G)=\mu_\Delta\alpha$ and $\scf_Y(G)=\mu_\Delta/\alpha$. 

In particular, their product $\scf_X(G)\scf_Y(G)=\mu_\Delta^2$ is an integer which can only happen when $h(\Delta)=3,4,6,$ or $\infty$ corresponding to $\mu_\Delta^2=1,2,3,$ or $4$. This makes the second part of the proposition obvious.
\end{proof}

\begin{definition}
 When $X$ is an affine $ADE$ Dynkin diagram of type $\affL$ and $Y$ is an affine $ADE$ Dynkin diagram of type $\affL'$ then we say that $G$ is a double binding \emph{of type $\affL\ast\affL'$}.
\end{definition}

Note that Proposition~\ref{prop:scf_types} is not symmetric in $X$ and $Y$, so if $G$ is a double binding of type $\affL\ast\affL'$ then necessarily $X$ has type $\affL$, $Y$ has type $\affL'$ and \eqref{eq:points_1} and \eqref{eq:points_2} hold. In other words, we treat double bindings of types $\affL\ast\affL'$ and $\affL'\ast\affL$ differently.

We now show how the McKay number defined in Section~\ref{sec:ADEDynkin} comes into play.
\begin{proposition}\label{prop:affH}
	Suppose $G$ is a double binding of type $\affL\ast\affL'$ and scaling factor $(a,b)$. Then we have
	\begin{equation}\label{eq:affH}
	\frac a b=\frac{\affH(\affL')}{\affH(\affL)}.		
	\end{equation}
\end{proposition}
\begin{proof}
	Recall that $a=\scf_X(G)$ and $b=\scf_Y(G)$. By Equations~\eqref{eq:points_1} and~\eqref{eq:points_2}, we have
	\begin{align*}
	\scf_X(G)\affH(\affL)&=\sum_{v\in X}\scf_X(G)\v_X(v)^2
	=\sum_{v\in X} \v_X(v) \sum_{(v,w)\in\Delta} \v_Y(w)\\
	&=\sum_{(v,w)\in\Delta} \v_X(v)\v_Y(w)
	=\sum_{w\in Y} \v_Y(w) \sum_{(v,w)\in\Delta} \v_X(v)\\
	&=\sum_{w\in Y}\scf_Y(G)\v_Y(w)^2
	=\scf_Y(G)\affH(\affL').
	\end{align*}
\end{proof}

This simple double counting argument dramatically reduces the number of options one needs to consider in the proof of the classification in Section~\ref{sect:classif_proof}.

\subsection{The weak generalized Cartan matrix of an $ADE$ bigraph}
We are now ready to define the matrix $\Cartan(G)$ for an arbitrary \affaff or \affinite $ADE$ bigraph $G$. Given a subset $C$ of vertices of $G$, denote by $G(C)$ the restriction (induced subgraph) of $G$ to $C$. Denote by $G^\circ$ the bigraph obtained from $G$ by removing all blue edges that connect two vertices from the same red connected component (thus $G^\circ$ is obtained from $G$ by removing all self bindings). It is easy to see that if $G$ is an \affaff or \affinite $ADE$ bigraph then the same is true for $G^\circ$.
\begin{definition}
	Let $G$ be an \affaff or \affinite $ADE$ bigraph, and let $C_1,C_2,\dots,C_n$ be its red connected components. Define an $n\times n$ matrix $\Cartan(G)=(a_{ij})$ as follows.
	\begin{itemize}
		\item For $i\in[n]$, set $a_{ii}=2-\wind(G(C_i))$.
		\item For $i\neq j\in [n]$, set $a_{ij}=0$ if there is no blue edge in $G$ connecting a vertex of $C_i$ to a vertex of $C_j$.
		\item For all other pairs of $i\neq j\in [n]$, let $\scf(G^\circ(C_i\cup C_j))=(p,q)$ and we set $a_{ij}=-p$, $a_{ji}=-q$.
	\end{itemize}
\end{definition}

Thus, given two connected components $C_i$ and $C_j$ that form a double binding with scaling factor $(1,1)$, $(1,2)$, $(1,3)$, $(1,4)$, or $(2,2)$, we connect $i$ and $j$ in $S(\Cartan(G))$ by an edge of the form $i\UA j$, $i\DRA j$, $i\TRA j$, $i\QRA j$, or $i\DLRA j$ respectively. 

Let $\v$ be the common eigenvector for $A_\Gamma$ and $A_\Delta$ from Lemma~\ref{lemma:eigenvalues}, and let $\v_{C_i}$ be the additive function for $\Gamma(C_i)$ from Figure~\ref{fig:affADE}. 
\begin{lemma}\label{lemma:delta}
	For each $i\in[n]$, there exists a positive real number $\delta_i$ such that for any $v\in C_i$ we have
	\[\delta_i=\frac{\v(v)}{\v_{C_i}(v)}.\]
\end{lemma}
\begin{proof}
	Since $\v$ is the common eigenvector for $A_\Gamma$ and $A_\Delta$, it must be proportional to $\v_{C_i}$ on $C_i$.
\end{proof}

\begin{theorem}\label{thm:Cartan}
	For an \affinite (resp., affine $\boxtimes$ affine) $ADE$ bigraph $G$, the matrix $\Cartan=\Cartan(G)$ is a weak generalized Cartan matrix of finite (resp., affine) type. The vector $\delta=(\delta_1,\dots,\delta_n)$ from Lemma~\ref{lemma:delta} satisfies $\Cartan\delta>0$ (resp., $\Cartan\delta=0$).
\end{theorem}
\begin{proof}
Since the matrix $\Cartan$ is clearly indecomposable, by Theorems~\ref{thm:fin_aff_ind} and~\ref{thm:Kac}, we only need to show that $\Cartan\delta>0$ (resp., $\Cartan\delta=0$). 

	Recall that $\v$ is an eigenvector for $A_\Delta$ with eigenvalue $\mu_\Delta$ which is either less than $2$ (if $h(\Delta)<\infty$) or equal to $2$ (if $h(\Delta)=\infty$). Now let $v\in C_i$ be a vertex. Using~\eqref{eq:points_0},~\eqref{eq:points_1}, and~\eqref{eq:points_2}, we get 
	\[\mu_\Delta\v(v)=\sum_{(u,v)\in\Delta} \v(u)=-\sum_{j\neq i} a_{ij}\delta_j\v_{C_i}(v)+(2-a_{ii})\delta_i\v_{C_i}(v).\]
	Using the fact that $\mu_\Delta<2$ (resp., $\mu_\Delta=2$) and $\v(v)=\delta_i\v_{C_i}(v)$, we get $\Cartan\delta>0$ (resp., $\Cartan\delta=0$), as desired.
\end{proof}

We let $S(G):=S(\Cartan(G))$ be the Dynkin diagram of $\Cartan(G)$ from Figure~\ref{fig:aff}.

\def\type{ \operatorname{type}}
\def\descr{ \operatorname{descr}}
Let us now introduce a convenient way to encode $G$ that often determines $G$ uniquely.
\begin{definition}\label{dfn:descr}
	Let $G$ be an \affinite or an \affaff $ADE$ bigraph. The \emph{description} $\descr(G)$ of $G$ is the Dynkin diagram $S(G)$ of $\Cartan(G)$ with each vertex $i\in[n]$ labeled by $\type(\Gamma(C_i))$. Here $\type(\Gamma(C_i))$ is the \emph{type} of $\Gamma(C_i)$ as an affine $ADE$ Dynkin diagram, in other words, $\type(\Gamma(C_i))$ belongs to the set 
	\[\{\affA_{2m-1}, \affD_m, \affE_6, \affE_7, \affE_8\}.\]
\end{definition}
For example, for the bigraph $G=\affD_5\otimes \affA_2$, $\descr(G)$ is equal to $\affD_5\DLRA\affD_5$.

\def\opp{\ast}

\subsection{Affine $\boxtimes$ finite self and double bindings}\label{sect:affinite_review} 
In this section, we review some of our results from~\cite{GP2}.

\begin{figure}
\scalebox{0.4}{
\begin{tikzpicture}

\coordinate (v0x0) at (0.00,-6.00);
\coordinate (v0x1) at (3.53,-4.85);
\coordinate (v0x2) at (5.71,-1.85);
\coordinate (v0x3) at (5.71,1.85);
\coordinate (v0x4) at (3.53,4.85);
\coordinate (v0x5) at (0.00,6.00);
\coordinate (v0x6) at (-3.53,4.85);
\coordinate (v0x7) at (-5.71,1.85);
\coordinate (v0x8) at (-5.71,-1.85);
\coordinate (v0x9) at (-3.53,-4.85);
\draw[color=red,line width=0.75mm] (v0x1) to[] (v0x0);
\draw[color=red,line width=0.75mm] (v0x2) to[] (v0x1);
\draw[color=red,line width=0.75mm] (v0x3) to[] (v0x2);
\draw[color=red,line width=0.75mm] (v0x4) to[] (v0x3);
\draw[color=blue,line width=0.75mm] (v0x5) to[] (v0x0);
\draw[color=red,line width=0.75mm] (v0x5) to[] (v0x4);
\draw[color=blue,line width=0.75mm] (v0x6) to[] (v0x1);
\draw[color=red,line width=0.75mm] (v0x6) to[] (v0x5);
\draw[color=blue,line width=0.75mm] (v0x7) to[] (v0x2);
\draw[color=red,line width=0.75mm] (v0x7) to[] (v0x6);
\draw[color=blue,line width=0.75mm] (v0x8) to[] (v0x3);
\draw[color=red,line width=0.75mm] (v0x8) to[] (v0x7);
\draw[color=red,line width=0.75mm] (v0x9) to[] (v0x0);
\draw[color=blue,line width=0.75mm] (v0x9) to[] (v0x4);
\draw[color=red,line width=0.75mm] (v0x9) to[] (v0x8);
\draw[fill=black!20!white] (v0x0.center) circle (0.2);
\draw[fill=white] (v0x1.center) circle (0.2);
\draw[fill=black!20!white] (v0x2.center) circle (0.2);
\draw[fill=white] (v0x3.center) circle (0.2);
\draw[fill=black!20!white] (v0x4.center) circle (0.2);
\draw[fill=white] (v0x5.center) circle (0.2);
\draw[fill=black!20!white] (v0x6.center) circle (0.2);
\draw[fill=white] (v0x7.center) circle (0.2);
\draw[fill=black!20!white] (v0x8.center) circle (0.2);
\draw[fill=white] (v0x9.center) circle (0.2);
\end{tikzpicture}}
\caption{\label{fig:selfb} An \affinite self-binding $\selfb{4n+1}$ for $n=2$.}
\end{figure}

For $n\geq1$, the bigraph $\selfb{4n+1}$ is defined as follows. Its unique red connected component is a $(4n+2)$-gon and the blue edges of $\selfb{4n+1}$ connect pairs of opposite vertices of this $(4n+2)$-gon. See Figure~\ref{fig:selfb}.

\begin{theorem}[{\cite{GP2}}]\label{thm:affinite_class}
\leavevmode
\begin{itemize}

 \item The only possible \affinite self bindings are $\selfb{4n+1}$ for $n\geq 1$. We have
 \[\nodeZ{$\descr(\selfb{4n+1})=$}\Zloop{\affA_{4n+1}}\nodeZ{.}\]
 \item all the double bindings with scaling factor $(1,2)$ are listed in Figure~\ref{figure:double_bindings_scf_2};
 \item all the double bindings with scaling factor $(1,3)$ are listed in Figure~\ref{figure:double_bindings_scf_3};
 \item the only other \affinite double bindings are parallel bindings $\affL\UA\affL$.
\end{itemize}
\end{theorem}

\begin{figure}

\begin{tabular}{|c|c|c|c|}\hline
\scalebox{0.4}{
\begin{tikzpicture}
\coordinate (v0x0) at (-1.20,-4.44);
\coordinate (v0x1) at (1.20,-3.60);
\coordinate (v0x2) at (0.00,-1.20);
\coordinate (v0x3) at (0.00,1.20);
\coordinate (v0x4) at (0.00,3.60);
\coordinate (v0x5) at (-1.20,6.84);
\coordinate (v0x6) at (1.20,6.00);
\coordinate (v1x0) at (4.20,-4.44);
\coordinate (v1x1) at (5.22,-2.83);
\coordinate (v1x2) at (5.64,1.16);
\coordinate (v1x3) at (5.22,5.14);
\coordinate (v1x4) at (4.20,6.84);
\coordinate (v1x5) at (3.18,5.54);
\coordinate (v1x6) at (2.76,1.55);
\coordinate (v1x7) at (3.18,-2.44);
\draw[color=red,line width=0.75mm] (v0x2) to[] (v0x0);
\draw[color=red,line width=0.75mm] (v0x2) to[] (v0x1);
\draw[color=red,line width=0.75mm] (v0x3) to[] (v0x2);
\draw[color=red,line width=0.75mm] (v0x4) to[] (v0x3);
\draw[color=red,line width=0.75mm] (v0x5) to[] (v0x4);
\draw[color=red,line width=0.75mm] (v0x6) to[] (v0x4);
\draw[color=blue,line width=0.75mm] (v1x0) to[] (v0x0);
\draw[color=blue,line width=0.75mm] (v1x0) to[] (v0x1);
\draw[color=blue,line width=0.75mm] (v1x1) to[bend right=10] (v0x2);
\draw[color=red,line width=0.75mm] (v1x1) to[] (v1x0);
\draw[color=blue,line width=0.75mm] (v1x2) to[bend left=3] (v0x3);
\draw[color=red,line width=0.75mm] (v1x2) to[] (v1x1);
\draw[color=blue,line width=0.75mm] (v1x3) to[] (v0x4);
\draw[color=red,line width=0.75mm] (v1x3) to[] (v1x2);
\draw[color=blue,line width=0.75mm] (v1x4) to[] (v0x5);
\draw[color=blue,line width=0.75mm] (v1x4) to[] (v0x6);
\draw[color=red,line width=0.75mm] (v1x4) to[] (v1x3);
\draw[color=blue,line width=0.75mm] (v1x5) to[] (v0x4);
\draw[color=red,line width=0.75mm] (v1x5) to[] (v1x4);
\draw[color=blue,line width=0.75mm] (v1x6) to[] (v0x3);
\draw[color=red,line width=0.75mm] (v1x6) to[] (v1x5);
\draw[color=blue,line width=0.75mm] (v1x7) to[] (v0x2);
\draw[color=red,line width=0.75mm] (v1x7) to[] (v1x0);
\draw[color=red,line width=0.75mm] (v1x7) to[] (v1x6);
\draw[fill=white] (v0x0.center) circle (0.2);
\draw[fill=white] (v0x1.center) circle (0.2);
\draw[fill=black!20!white] (v0x2.center) circle (0.2);
\draw[fill=white] (v0x3.center) circle (0.2);
\draw[fill=black!20!white] (v0x4.center) circle (0.2);
\draw[fill=white] (v0x5.center) circle (0.2);
\draw[fill=white] (v0x6.center) circle (0.2);
\draw[fill=black!20!white] (v1x0.center) circle (0.2);
\draw[fill=white] (v1x1.center) circle (0.2);
\draw[fill=black!20!white] (v1x2.center) circle (0.2);
\draw[fill=white] (v1x3.center) circle (0.2);
\draw[fill=black!20!white] (v1x4.center) circle (0.2);
\draw[fill=white] (v1x5.center) circle (0.2);
\draw[fill=black!20!white] (v1x6.center) circle (0.2);
\draw[fill=white] (v1x7.center) circle (0.2);
\end{tikzpicture}}
&
\scalebox{0.4}{
\begin{tikzpicture}
\coordinate (v0x0) at (-1.20,7.95);
\coordinate (v0x1) at (-1.65,4.56);
\coordinate (v0x10) at (-0.75,2.34);
\coordinate (v0x11) at (-0.75,5.34);
\coordinate (v0x2) at (-1.65,1.56);
\coordinate (v0x3) at (-0.45,-0.39);
\coordinate (v0x4) at (0.75,3.66);
\coordinate (v0x5) at (0.75,6.66);
\coordinate (v0x6) at (1.20,10.44);
\coordinate (v0x7) at (1.65,7.44);
\coordinate (v0x8) at (1.65,4.44);
\coordinate (v0x9) at (0.45,0.39);
\coordinate (v1x0) at (4.65,-0.39);
\coordinate (v1x1) at (5.59,2.24);
\coordinate (v1x2) at (5.59,7.66);
\coordinate (v1x3) at (4.65,10.44);
\coordinate (v1x4) at (3.71,8.18);
\coordinate (v1x5) at (3.71,2.77);
\draw[color=red,line width=0.75mm] (v0x1) to[] (v0x0);
\draw[color=red,line width=0.75mm] (v0x11) to[] (v0x0);
\draw[color=red,line width=0.75mm] (v0x11) to[] (v0x10);
\draw[color=red,line width=0.75mm] (v0x2) to[] (v0x1);
\draw[color=red,line width=0.75mm] (v0x3) to[] (v0x2);
\draw[color=red,line width=0.75mm] (v0x4) to[] (v0x3);
\draw[color=red,line width=0.75mm] (v0x5) to[] (v0x4);
\draw[color=red,line width=0.75mm] (v0x6) to[] (v0x5);
\draw[color=red,line width=0.75mm] (v0x7) to[] (v0x6);
\draw[color=red,line width=0.75mm] (v0x8) to[] (v0x7);
\draw[color=red,line width=0.75mm] (v0x9) to[] (v0x10);
\draw[color=red,line width=0.75mm] (v0x9) to[] (v0x8);
\draw[color=blue,line width=0.75mm] (v1x0) to[] (v0x3);
\draw[color=blue,line width=0.75mm] (v1x0) to[] (v0x9);
\draw[color=blue,line width=0.75mm] (v1x1) to[] (v0x2);
\draw[color=blue,line width=0.75mm] (v1x1) to[] (v0x8);
\draw[color=red,line width=0.75mm] (v1x1) to[] (v1x0);
\draw[color=blue,line width=0.75mm] (v1x2) to[] (v0x1);
\draw[color=blue,line width=0.75mm] (v1x2) to[] (v0x7);
\draw[color=red,line width=0.75mm] (v1x2) to[] (v1x1);
\draw[color=blue,line width=0.75mm] (v1x3) to[] (v0x0);
\draw[color=blue,line width=0.75mm] (v1x3) to[] (v0x6);
\draw[color=red,line width=0.75mm] (v1x3) to[] (v1x2);
\draw[color=blue,line width=0.75mm] (v1x4) to[bend left=5] (v0x11);
\draw[color=blue,line width=0.75mm] (v1x4) to[bend left=2] (v0x5);
\draw[color=red,line width=0.75mm] (v1x4) to[] (v1x3);
\draw[color=blue,line width=0.75mm] (v1x5) to[] (v0x10);
\draw[color=blue,line width=0.75mm] (v1x5) to[] (v0x4);
\draw[color=red,line width=0.75mm] (v1x5) to[] (v1x0);
\draw[color=red,line width=0.75mm] (v1x5) to[] (v1x4);
\draw[fill=black!20!white] (v0x0.center) circle (0.2);
\draw[fill=white] (v0x1.center) circle (0.2);
\draw[fill=black!20!white] (v0x10.center) circle (0.2);
\draw[fill=white] (v0x11.center) circle (0.2);
\draw[fill=black!20!white] (v0x2.center) circle (0.2);
\draw[fill=white] (v0x3.center) circle (0.2);
\draw[fill=black!20!white] (v0x4.center) circle (0.2);
\draw[fill=white] (v0x5.center) circle (0.2);
\draw[fill=black!20!white] (v0x6.center) circle (0.2);
\draw[fill=white] (v0x7.center) circle (0.2);
\draw[fill=black!20!white] (v0x8.center) circle (0.2);
\draw[fill=white] (v0x9.center) circle (0.2);
\draw[fill=black!20!white] (v1x0.center) circle (0.2);
\draw[fill=white] (v1x1.center) circle (0.2);
\draw[fill=black!20!white] (v1x2.center) circle (0.2);
\draw[fill=white] (v1x3.center) circle (0.2);
\draw[fill=black!20!white] (v1x4.center) circle (0.2);
\draw[fill=white] (v1x5.center) circle (0.2);
\end{tikzpicture}}
&
\scalebox{0.4}{
\begin{tikzpicture}
\coordinate (v0x0) at (-3.00,6.00);
\coordinate (v0x1) at (-0.60,5.16);
\coordinate (v0x10) at (3.00,6.00);
\coordinate (v0x2) at (-1.80,2.76);
\coordinate (v0x3) at (-1.80,0.36);
\coordinate (v0x4) at (-1.80,-2.04);
\coordinate (v0x5) at (0.00,-4.44);
\coordinate (v0x6) at (1.80,-1.20);
\coordinate (v0x7) at (1.80,1.20);
\coordinate (v0x8) at (1.80,3.60);
\coordinate (v0x9) at (0.60,6.84);
\coordinate (v1x0) at (4.80,-4.44);
\coordinate (v1x1) at (7.20,-3.60);
\coordinate (v1x2) at (6.00,-1.20);
\coordinate (v1x3) at (6.00,1.20);
\coordinate (v1x4) at (6.00,3.60);
\coordinate (v1x5) at (4.80,6.84);
\coordinate (v1x6) at (7.20,6.00);
\draw[color=red,line width=0.75mm] (v0x2) to[] (v0x0);
\draw[color=red,line width=0.75mm] (v0x2) to[] (v0x1);
\draw[color=red,line width=0.75mm] (v0x3) to[] (v0x2);
\draw[color=red,line width=0.75mm] (v0x4) to[] (v0x3);
\draw[color=red,line width=0.75mm] (v0x5) to[] (v0x4);
\draw[color=red,line width=0.75mm] (v0x6) to[] (v0x5);
\draw[color=red,line width=0.75mm] (v0x7) to[] (v0x6);
\draw[color=red,line width=0.75mm] (v0x8) to[] (v0x10);
\draw[color=red,line width=0.75mm] (v0x8) to[] (v0x7);
\draw[color=red,line width=0.75mm] (v0x9) to[] (v0x8);
\draw[color=blue,line width=0.75mm] (v1x0) to[] (v0x5);
\draw[color=blue,line width=0.75mm] (v1x1) to[] (v0x5);
\draw[color=blue,line width=0.75mm] (v1x2) to[bend left=2] (v0x4);
\draw[color=blue,line width=0.75mm] (v1x2) to[] (v0x6);
\draw[color=red,line width=0.75mm] (v1x2) to[] (v1x0);
\draw[color=red,line width=0.75mm] (v1x2) to[] (v1x1);
\draw[color=blue,line width=0.75mm] (v1x3) to[bend left=2] (v0x3);
\draw[color=blue,line width=0.75mm] (v1x3) to[] (v0x7);
\draw[color=red,line width=0.75mm] (v1x3) to[] (v1x2);
\draw[color=blue,line width=0.75mm] (v1x4) to[bend left=2] (v0x2);
\draw[color=blue,line width=0.75mm] (v1x4) to[] (v0x8);
\draw[color=red,line width=0.75mm] (v1x4) to[] (v1x3);
\draw[color=blue,line width=0.75mm] (v1x5) to[bend left=2] (v0x0);
\draw[color=blue,line width=0.75mm] (v1x5) to[] (v0x10);
\draw[color=red,line width=0.75mm] (v1x5) to[] (v1x4);
\draw[color=blue,line width=0.75mm] (v1x6) to[bend left=2] (v0x1);
\draw[color=blue,line width=0.75mm] (v1x6) to[] (v0x9);
\draw[color=red,line width=0.75mm] (v1x6) to[] (v1x4);
\draw[fill=white] (v0x0.center) circle (0.2);
\draw[fill=white] (v0x1.center) circle (0.2);
\draw[fill=white] (v0x10.center) circle (0.2);
\draw[fill=black!20!white] (v0x2.center) circle (0.2);
\draw[fill=white] (v0x3.center) circle (0.2);
\draw[fill=black!20!white] (v0x4.center) circle (0.2);
\draw[fill=white] (v0x5.center) circle (0.2);
\draw[fill=black!20!white] (v0x6.center) circle (0.2);
\draw[fill=white] (v0x7.center) circle (0.2);
\draw[fill=black!20!white] (v0x8.center) circle (0.2);
\draw[fill=white] (v0x9.center) circle (0.2);
\draw[fill=black!20!white] (v1x0.center) circle (0.2);
\draw[fill=black!20!white] (v1x1.center) circle (0.2);
\draw[fill=white] (v1x2.center) circle (0.2);
\draw[fill=black!20!white] (v1x3.center) circle (0.2);
\draw[fill=white] (v1x4.center) circle (0.2);
\draw[fill=black!20!white] (v1x5.center) circle (0.2);
\draw[fill=black!20!white] (v1x6.center) circle (0.2);
\end{tikzpicture}}
&
\scalebox{0.4}{
\begin{tikzpicture}
\coordinate (v0x0) at (0.00,-0.72);
\coordinate (v0x1) at (-0.90,8.16);
\coordinate (v0x2) at (-0.90,5.76);
\coordinate (v0x3) at (-0.90,3.36);
\coordinate (v0x4) at (0.00,1.68);
\coordinate (v0x5) at (0.90,4.80);
\coordinate (v0x6) at (0.90,7.20);
\coordinate (v0x7) at (0.90,9.60);
\coordinate (v1x0) at (3.90,9.60);
\coordinate (v1x1) at (3.90,7.20);
\coordinate (v1x2) at (3.00,0.72);
\coordinate (v1x3) at (3.00,3.12);
\coordinate (v1x4) at (3.90,4.80);
\coordinate (v1x5) at (4.80,1.68);
\coordinate (v1x6) at (4.80,-0.72);
\draw[color=red,line width=0.75mm] (v0x2) to[] (v0x1);
\draw[color=red,line width=0.75mm] (v0x3) to[] (v0x2);
\draw[color=red,line width=0.75mm] (v0x4) to[] (v0x0);
\draw[color=red,line width=0.75mm] (v0x4) to[] (v0x3);
\draw[color=red,line width=0.75mm] (v0x5) to[] (v0x4);
\draw[color=red,line width=0.75mm] (v0x6) to[] (v0x5);
\draw[color=red,line width=0.75mm] (v0x7) to[] (v0x6);
\draw[color=blue,line width=0.75mm] (v1x0) to[] (v0x1);
\draw[color=blue,line width=0.75mm] (v1x0) to[] (v0x7);
\draw[color=blue,line width=0.75mm] (v1x1) to[] (v0x2);
\draw[color=blue,line width=0.75mm] (v1x1) to[] (v0x6);
\draw[color=red,line width=0.75mm] (v1x1) to[] (v1x0);
\draw[color=blue,line width=0.75mm] (v1x2) to[] (v0x0);
\draw[color=blue,line width=0.75mm] (v1x3) to[] (v0x4);
\draw[color=red,line width=0.75mm] (v1x3) to[] (v1x2);
\draw[color=blue,line width=0.75mm] (v1x4) to[] (v0x3);
\draw[color=blue,line width=0.75mm] (v1x4) to[] (v0x5);
\draw[color=red,line width=0.75mm] (v1x4) to[] (v1x1);
\draw[color=red,line width=0.75mm] (v1x4) to[] (v1x3);
\draw[color=blue,line width=0.75mm] (v1x5) to[] (v0x4);
\draw[color=red,line width=0.75mm] (v1x5) to[] (v1x4);
\draw[color=blue,line width=0.75mm] (v1x6) to[] (v0x0);
\draw[color=red,line width=0.75mm] (v1x6) to[] (v1x5);
\draw[fill=white] (v0x0.center) circle (0.2);
\draw[fill=white] (v0x1.center) circle (0.2);
\draw[fill=black!20!white] (v0x2.center) circle (0.2);
\draw[fill=white] (v0x3.center) circle (0.2);
\draw[fill=black!20!white] (v0x4.center) circle (0.2);
\draw[fill=white] (v0x5.center) circle (0.2);
\draw[fill=black!20!white] (v0x6.center) circle (0.2);
\draw[fill=white] (v0x7.center) circle (0.2);
\draw[fill=black!20!white] (v1x0.center) circle (0.2);
\draw[fill=white] (v1x1.center) circle (0.2);
\draw[fill=black!20!white] (v1x2.center) circle (0.2);
\draw[fill=white] (v1x3.center) circle (0.2);
\draw[fill=black!20!white] (v1x4.center) circle (0.2);
\draw[fill=white] (v1x5.center) circle (0.2);
\draw[fill=black!20!white] (v1x6.center) circle (0.2);
\end{tikzpicture}}
\\

$\affD_{n+2}\DRA \affA_{2n-1}$
&
$\affA_{4n-1}\DRA \affA_{2n-1}$
&
$\affD_{2n+2}\DRA \affD_{n+2}$
&
$\affE_{7}\DRA \affE_{6}$
\\

for $n=4$
&
for $n=3$
&
for $n=4$
&

\\\hline

\end{tabular}
\caption{\label{figure:double_bindings_scf_2}Three infinite and one exceptional family of double bindings with scaling factor $(1,2)$. All blue components have type $A_3$.}
\end{figure}
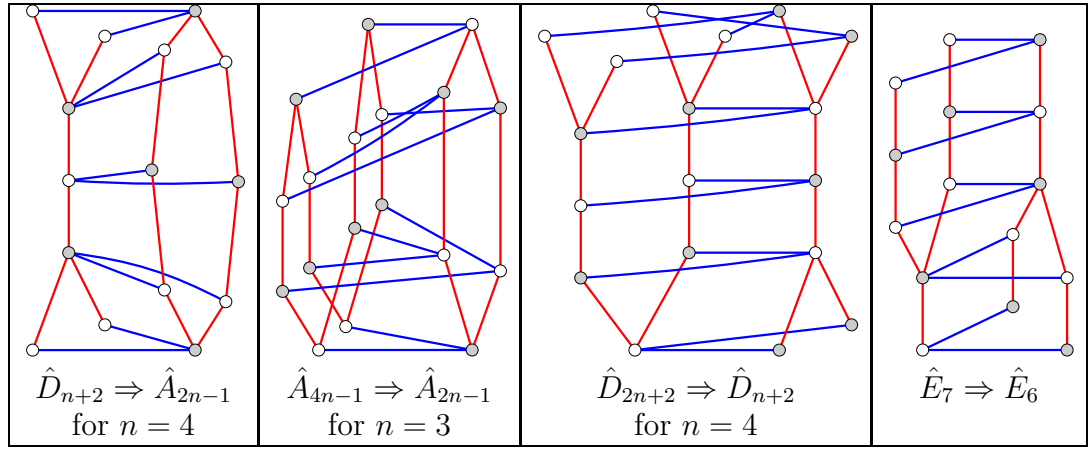

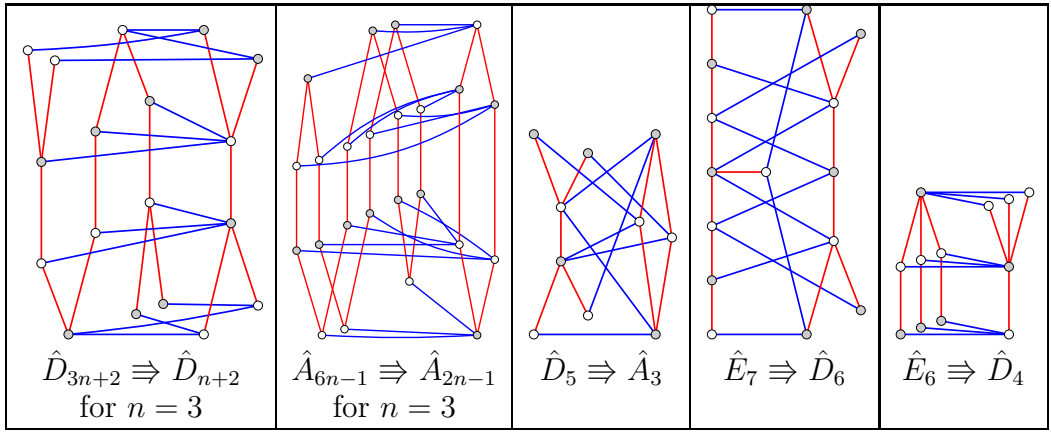
\begin{figure}
\makebox[1.0\textwidth]{
\begin{tabular}{|c|c|c|c|c|}\hline
\scalebox{0.3}{
\begin{tikzpicture}
\coordinate (v0x0) at (-3.00,12.60);
\coordinate (v0x1) at (-1.80,12.15);
\coordinate (v0x10) at (1.80,0.90);
\coordinate (v0x11) at (3.00,1.35);
\coordinate (v0x2) at (-2.40,7.65);
\coordinate (v0x3) at (-2.40,3.15);
\coordinate (v0x4) at (-1.20,0.00);
\coordinate (v0x5) at (0.00,4.50);
\coordinate (v0x6) at (0.00,9.00);
\coordinate (v0x7) at (1.20,13.50);
\coordinate (v0x8) at (2.40,10.35);
\coordinate (v0x9) at (2.40,5.85);
\coordinate (v1x0) at (4.80,0.00);
\coordinate (v1x1) at (7.20,1.28);
\coordinate (v1x2) at (6.00,4.93);
\coordinate (v1x3) at (6.00,8.57);
\coordinate (v1x4) at (4.80,13.50);
\coordinate (v1x5) at (7.20,12.22);
\draw[color=red,line width=0.75mm] (v0x2) to[] (v0x0);
\draw[color=red,line width=0.75mm] (v0x2) to[] (v0x1);
\draw[color=red,line width=0.75mm] (v0x3) to[] (v0x2);
\draw[color=red,line width=0.75mm] (v0x4) to[] (v0x3);
\draw[color=red,line width=0.75mm] (v0x5) to[] (v0x4);
\draw[color=red,line width=0.75mm] (v0x6) to[] (v0x5);
\draw[color=red,line width=0.75mm] (v0x7) to[] (v0x6);
\draw[color=red,line width=0.75mm] (v0x8) to[] (v0x7);
\draw[color=red,line width=0.75mm] (v0x9) to[] (v0x10);
\draw[color=red,line width=0.75mm] (v0x9) to[] (v0x11);
\draw[color=red,line width=0.75mm] (v0x9) to[] (v0x8);
\draw[color=blue,line width=0.75mm] (v1x0) to[] (v0x10);
\draw[color=blue,line width=0.75mm] (v1x0) to[] (v0x4);
\draw[color=blue,line width=0.75mm] (v1x1) to[] (v0x11);
\draw[color=blue,line width=0.75mm] (v1x1) to[bend left=3] (v0x4);
\draw[color=blue,line width=0.75mm] (v1x2) to[] (v0x3);
\draw[color=blue,line width=0.75mm] (v1x2) to[] (v0x5);
\draw[color=blue,line width=0.75mm] (v1x2) to[] (v0x9);
\draw[color=red,line width=0.75mm] (v1x2) to[] (v1x0);
\draw[color=red,line width=0.75mm] (v1x2) to[] (v1x1);
\draw[color=blue,line width=0.75mm] (v1x3) to[] (v0x2);
\draw[color=blue,line width=0.75mm] (v1x3) to[] (v0x6);
\draw[color=blue,line width=0.75mm] (v1x3) to[] (v0x8);
\draw[color=red,line width=0.75mm] (v1x3) to[] (v1x2);
\draw[color=blue,line width=0.75mm] (v1x4) to[bend left=4] (v0x0);
\draw[color=blue,line width=0.75mm] (v1x4) to[] (v0x7);
\draw[color=red,line width=0.75mm] (v1x4) to[] (v1x3);
\draw[color=blue,line width=0.75mm] (v1x5) to[] (v0x1);
\draw[color=blue,line width=0.75mm] (v1x5) to[] (v0x7);
\draw[color=red,line width=0.75mm] (v1x5) to[] (v1x3);
\draw[fill=white] (v0x0.center) circle (0.2);
\draw[fill=white] (v0x1.center) circle (0.2);
\draw[fill=black!20!white] (v0x10.center) circle (0.2);
\draw[fill=black!20!white] (v0x11.center) circle (0.2);
\draw[fill=black!20!white] (v0x2.center) circle (0.2);
\draw[fill=white] (v0x3.center) circle (0.2);
\draw[fill=black!20!white] (v0x4.center) circle (0.2);
\draw[fill=white] (v0x5.center) circle (0.2);
\draw[fill=black!20!white] (v0x6.center) circle (0.2);
\draw[fill=white] (v0x7.center) circle (0.2);
\draw[fill=black!20!white] (v0x8.center) circle (0.2);
\draw[fill=white] (v0x9.center) circle (0.2);
\draw[fill=white] (v1x0.center) circle (0.2);
\draw[fill=white] (v1x1.center) circle (0.2);
\draw[fill=black!20!white] (v1x2.center) circle (0.2);
\draw[fill=white] (v1x3.center) circle (0.2);
\draw[fill=black!20!white] (v1x4.center) circle (0.2);
\draw[fill=black!20!white] (v1x5.center) circle (0.2);
\end{tikzpicture}}
&
\scalebox{0.25}{
\begin{tikzpicture}
\coordinate (v0x0) at (-2.70,12.15);
\coordinate (v0x1) at (-3.30,7.49);
\coordinate (v0x10) at (3.30,6.01);
\coordinate (v0x11) at (3.30,10.51);
\coordinate (v0x12) at (1.95,15.01);
\coordinate (v0x13) at (0.60,9.16);
\coordinate (v0x14) at (0.60,4.97);
\coordinate (v0x15) at (-0.75,-1.19);
\coordinate (v0x16) at (-2.10,3.31);
\coordinate (v0x17) at (-2.10,7.81);
\coordinate (v0x2) at (-3.30,2.99);
\coordinate (v0x3) at (-1.95,-1.51);
\coordinate (v0x4) at (-0.60,4.34);
\coordinate (v0x5) at (-0.60,8.53);
\coordinate (v0x6) at (0.75,14.69);
\coordinate (v0x7) at (2.10,10.19);
\coordinate (v0x8) at (2.10,5.69);
\coordinate (v0x9) at (2.70,1.35);
\coordinate (v1x0) at (6.30,-1.51);
\coordinate (v1x1) at (7.24,2.51);
\coordinate (v1x2) at (7.24,10.76);
\coordinate (v1x3) at (6.30,15.01);
\coordinate (v1x4) at (5.36,11.57);
\coordinate (v1x5) at (5.36,3.31);
\draw[color=red,line width=0.75mm] (v0x1) to[] (v0x0);
\draw[color=red,line width=0.75mm] (v0x11) to[] (v0x10);
\draw[color=red,line width=0.75mm] (v0x12) to[] (v0x11);
\draw[color=red,line width=0.75mm] (v0x13) to[] (v0x12);
\draw[color=red,line width=0.75mm] (v0x14) to[] (v0x13);
\draw[color=red,line width=0.75mm] (v0x15) to[] (v0x14);
\draw[color=red,line width=0.75mm] (v0x16) to[] (v0x15);
\draw[color=red,line width=0.75mm] (v0x17) to[] (v0x0);
\draw[color=red,line width=0.75mm] (v0x17) to[] (v0x16);
\draw[color=red,line width=0.75mm] (v0x2) to[] (v0x1);
\draw[color=red,line width=0.75mm] (v0x3) to[] (v0x2);
\draw[color=red,line width=0.75mm] (v0x4) to[] (v0x3);
\draw[color=red,line width=0.75mm] (v0x5) to[] (v0x4);
\draw[color=red,line width=0.75mm] (v0x6) to[] (v0x5);
\draw[color=red,line width=0.75mm] (v0x7) to[] (v0x6);
\draw[color=red,line width=0.75mm] (v0x8) to[] (v0x7);
\draw[color=red,line width=0.75mm] (v0x9) to[] (v0x10);
\draw[color=red,line width=0.75mm] (v0x9) to[] (v0x8);
\draw[color=blue,line width=0.75mm] (v1x0) to[] (v0x15);
\draw[color=blue,line width=0.75mm] (v1x0) to[bend left=3] (v0x3);
\draw[color=blue,line width=0.75mm] (v1x0) to[] (v0x9);
\draw[color=blue,line width=0.75mm] (v1x1) to[bend left=14] (v0x14);
\draw[color=blue,line width=0.75mm] (v1x1) to[] (v0x2);
\draw[color=blue,line width=0.75mm] (v1x1) to[bend right=6] (v0x8);
\draw[color=red,line width=0.75mm] (v1x1) to[] (v1x0);
\draw[color=blue,line width=0.75mm] (v1x2) to[bend left=15] (v0x1);
\draw[color=blue,line width=0.75mm] (v1x2) to[] (v0x13);
\draw[color=blue,line width=0.75mm] (v1x2) to[bend left=11] (v0x7);
\draw[color=red,line width=0.75mm] (v1x2) to[] (v1x1);
\draw[color=blue,line width=0.75mm] (v1x3) to[] (v0x0);
\draw[color=blue,line width=0.75mm] (v1x3) to[] (v0x12);
\draw[color=blue,line width=0.75mm] (v1x3) to[bend left=7] (v0x6);
\draw[color=red,line width=0.75mm] (v1x3) to[] (v1x2);
\draw[color=blue,line width=0.75mm] (v1x4) to[] (v0x11);
\draw[color=blue,line width=0.75mm] (v1x4) to[bend right=15] (v0x17);
\draw[color=blue,line width=0.75mm] (v1x4) to[bend right=15] (v0x5);
\draw[color=red,line width=0.75mm] (v1x4) to[] (v1x3);
\draw[color=blue,line width=0.75mm] (v1x5) to[] (v0x10);
\draw[color=blue,line width=0.75mm] (v1x5) to[] (v0x16);
\draw[color=blue,line width=0.75mm] (v1x5) to[] (v0x4);
\draw[color=red,line width=0.75mm] (v1x5) to[] (v1x0);
\draw[color=red,line width=0.75mm] (v1x5) to[] (v1x4);
\draw[fill=black!20!white] (v0x0.center) circle (0.2);
\draw[fill=white] (v0x1.center) circle (0.2);
\draw[fill=black!20!white] (v0x10.center) circle (0.2);
\draw[fill=white] (v0x11.center) circle (0.2);
\draw[fill=black!20!white] (v0x12.center) circle (0.2);
\draw[fill=white] (v0x13.center) circle (0.2);
\draw[fill=black!20!white] (v0x14.center) circle (0.2);
\draw[fill=white] (v0x15.center) circle (0.2);
\draw[fill=black!20!white] (v0x16.center) circle (0.2);
\draw[fill=white] (v0x17.center) circle (0.2);
\draw[fill=black!20!white] (v0x2.center) circle (0.2);
\draw[fill=white] (v0x3.center) circle (0.2);
\draw[fill=black!20!white] (v0x4.center) circle (0.2);
\draw[fill=white] (v0x5.center) circle (0.2);
\draw[fill=black!20!white] (v0x6.center) circle (0.2);
\draw[fill=white] (v0x7.center) circle (0.2);
\draw[fill=black!20!white] (v0x8.center) circle (0.2);
\draw[fill=white] (v0x9.center) circle (0.2);
\draw[fill=black!20!white] (v1x0.center) circle (0.2);
\draw[fill=white] (v1x1.center) circle (0.2);
\draw[fill=black!20!white] (v1x2.center) circle (0.2);
\draw[fill=white] (v1x3.center) circle (0.2);
\draw[fill=black!20!white] (v1x4.center) circle (0.2);
\draw[fill=white] (v1x5.center) circle (0.2);
\end{tikzpicture}}
&
\scalebox{0.3}{
\begin{tikzpicture}
\coordinate (v0x0) at (-1.20,-3.24);
\coordinate (v0x1) at (1.20,-2.40);
\coordinate (v0x2) at (0.00,0.00);
\coordinate (v0x3) at (0.00,2.40);
\coordinate (v0x4) at (-1.20,5.64);
\coordinate (v0x5) at (1.20,4.80);
\coordinate (v1x0) at (4.20,-3.24);
\coordinate (v1x1) at (4.92,1.06);
\coordinate (v1x2) at (4.20,5.64);
\coordinate (v1x3) at (3.48,1.76);
\draw[color=red,line width=0.75mm] (v0x2) to[] (v0x0);
\draw[color=red,line width=0.75mm] (v0x2) to[] (v0x1);
\draw[color=red,line width=0.75mm] (v0x3) to[] (v0x2);
\draw[color=red,line width=0.75mm] (v0x4) to[] (v0x3);
\draw[color=red,line width=0.75mm] (v0x5) to[] (v0x3);
\draw[color=blue,line width=0.75mm] (v1x0) to[] (v0x0);
\draw[color=blue,line width=0.75mm] (v1x0) to[] (v0x3);
\draw[color=blue,line width=0.75mm] (v1x1) to[] (v0x2);
\draw[color=blue,line width=0.75mm] (v1x1) to[] (v0x5);
\draw[color=red,line width=0.75mm] (v1x1) to[] (v1x0);
\draw[color=blue,line width=0.75mm] (v1x2) to[] (v0x1);
\draw[color=blue,line width=0.75mm] (v1x2) to[] (v0x3);
\draw[color=red,line width=0.75mm] (v1x2) to[] (v1x1);
\draw[color=blue,line width=0.75mm] (v1x3) to[] (v0x2);
\draw[color=blue,line width=0.75mm] (v1x3) to[] (v0x4);
\draw[color=red,line width=0.75mm] (v1x3) to[] (v1x0);
\draw[color=red,line width=0.75mm] (v1x3) to[] (v1x2);
\draw[fill=white] (v0x0.center) circle (0.2);
\draw[fill=white] (v0x1.center) circle (0.2);
\draw[fill=black!20!white] (v0x2.center) circle (0.2);
\draw[fill=white] (v0x3.center) circle (0.2);
\draw[fill=black!20!white] (v0x4.center) circle (0.2);
\draw[fill=black!20!white] (v0x5.center) circle (0.2);
\draw[fill=black!20!white] (v1x0.center) circle (0.2);
\draw[fill=white] (v1x1.center) circle (0.2);
\draw[fill=black!20!white] (v1x2.center) circle (0.2);
\draw[fill=white] (v1x3.center) circle (0.2);
\end{tikzpicture}}
&
\scalebox{0.3}{
\begin{tikzpicture}
\coordinate (v0x0) at (2.40,7.20);
\coordinate (v0x1) at (0.00,0.00);
\coordinate (v0x2) at (0.00,2.40);
\coordinate (v0x3) at (0.00,4.80);
\coordinate (v0x4) at (0.00,7.20);
\coordinate (v0x5) at (0.00,9.60);
\coordinate (v0x6) at (0.00,12.00);
\coordinate (v0x7) at (0.00,14.40);
\coordinate (v1x0) at (4.20,0.00);
\coordinate (v1x1) at (6.60,1.07);
\coordinate (v1x2) at (5.40,4.14);
\coordinate (v1x3) at (5.40,7.20);
\coordinate (v1x4) at (5.40,10.26);
\coordinate (v1x5) at (4.20,14.40);
\coordinate (v1x6) at (6.60,13.33);
\draw[color=red,line width=0.75mm] (v0x2) to[] (v0x1);
\draw[color=red,line width=0.75mm] (v0x3) to[] (v0x2);
\draw[color=red,line width=0.75mm] (v0x4) to[] (v0x0);
\draw[color=red,line width=0.75mm] (v0x4) to[] (v0x3);
\draw[color=red,line width=0.75mm] (v0x5) to[] (v0x4);
\draw[color=red,line width=0.75mm] (v0x6) to[] (v0x5);
\draw[color=red,line width=0.75mm] (v0x7) to[] (v0x6);
\draw[color=blue,line width=0.75mm] (v1x0) to[] (v0x0);
\draw[color=blue,line width=0.75mm] (v1x0) to[] (v0x1);
\draw[color=blue,line width=0.75mm] (v1x1) to[] (v0x3);
\draw[color=blue,line width=0.75mm] (v1x2) to[] (v0x2);
\draw[color=blue,line width=0.75mm] (v1x2) to[] (v0x4);
\draw[color=red,line width=0.75mm] (v1x2) to[] (v1x0);
\draw[color=red,line width=0.75mm] (v1x2) to[] (v1x1);
\draw[color=blue,line width=0.75mm] (v1x3) to[] (v0x3);
\draw[color=blue,line width=0.75mm] (v1x3) to[] (v0x5);
\draw[color=red,line width=0.75mm] (v1x3) to[] (v1x2);
\draw[color=blue,line width=0.75mm] (v1x4) to[] (v0x4);
\draw[color=blue,line width=0.75mm] (v1x4) to[] (v0x6);
\draw[color=red,line width=0.75mm] (v1x4) to[] (v1x3);
\draw[color=blue,line width=0.75mm] (v1x5) to[] (v0x0);
\draw[color=blue,line width=0.75mm] (v1x5) to[] (v0x7);
\draw[color=red,line width=0.75mm] (v1x5) to[] (v1x4);
\draw[color=blue,line width=0.75mm] (v1x6) to[] (v0x5);
\draw[color=red,line width=0.75mm] (v1x6) to[] (v1x4);
\draw[fill=white] (v0x0.center) circle (0.2);
\draw[fill=white] (v0x1.center) circle (0.2);
\draw[fill=black!20!white] (v0x2.center) circle (0.2);
\draw[fill=white] (v0x3.center) circle (0.2);
\draw[fill=black!20!white] (v0x4.center) circle (0.2);
\draw[fill=white] (v0x5.center) circle (0.2);
\draw[fill=black!20!white] (v0x6.center) circle (0.2);
\draw[fill=white] (v0x7.center) circle (0.2);
\draw[fill=black!20!white] (v1x0.center) circle (0.2);
\draw[fill=black!20!white] (v1x1.center) circle (0.2);
\draw[fill=white] (v1x2.center) circle (0.2);
\draw[fill=black!20!white] (v1x3.center) circle (0.2);
\draw[fill=white] (v1x4.center) circle (0.2);
\draw[fill=black!20!white] (v1x5.center) circle (0.2);
\draw[fill=black!20!white] (v1x6.center) circle (0.2);
\end{tikzpicture}}
&
\scalebox{0.3}{
\begin{tikzpicture}
\coordinate (v0x0) at (-0.90,0.00);
\coordinate (v0x1) at (-0.90,3.00);
\coordinate (v0x2) at (0.00,0.30);
\coordinate (v0x3) at (0.00,3.30);
\coordinate (v0x4) at (0.00,6.30);
\coordinate (v0x5) at (0.90,3.60);
\coordinate (v0x6) at (0.90,0.60);
\coordinate (v1x0) at (3.90,0.00);
\coordinate (v1x1) at (3.00,5.70);
\coordinate (v1x2) at (3.90,3.00);
\coordinate (v1x3) at (3.90,6.00);
\coordinate (v1x4) at (4.80,6.30);
\draw[color=red,line width=0.75mm] (v0x1) to[] (v0x0);
\draw[color=red,line width=0.75mm] (v0x3) to[] (v0x2);
\draw[color=red,line width=0.75mm] (v0x4) to[] (v0x1);
\draw[color=red,line width=0.75mm] (v0x4) to[] (v0x3);
\draw[color=red,line width=0.75mm] (v0x5) to[] (v0x4);
\draw[color=red,line width=0.75mm] (v0x6) to[] (v0x5);
\draw[color=blue,line width=0.75mm] (v1x0) to[] (v0x0);
\draw[color=blue,line width=0.75mm] (v1x0) to[] (v0x2);
\draw[color=blue,line width=0.75mm] (v1x0) to[] (v0x6);
\draw[color=blue,line width=0.75mm] (v1x1) to[] (v0x4);
\draw[color=blue,line width=0.75mm] (v1x2) to[] (v0x1);
\draw[color=blue,line width=0.75mm] (v1x2) to[] (v0x3);
\draw[color=blue,line width=0.75mm] (v1x2) to[] (v0x5);
\draw[color=red,line width=0.75mm] (v1x2) to[] (v1x0);
\draw[color=red,line width=0.75mm] (v1x2) to[] (v1x1);
\draw[color=blue,line width=0.75mm] (v1x3) to[] (v0x4);
\draw[color=red,line width=0.75mm] (v1x3) to[] (v1x2);
\draw[color=blue,line width=0.75mm] (v1x4) to[] (v0x4);
\draw[color=red,line width=0.75mm] (v1x4) to[] (v1x2);
\draw[fill=black!20!white] (v0x0.center) circle (0.2);
\draw[fill=white] (v0x1.center) circle (0.2);
\draw[fill=black!20!white] (v0x2.center) circle (0.2);
\draw[fill=white] (v0x3.center) circle (0.2);
\draw[fill=black!20!white] (v0x4.center) circle (0.2);
\draw[fill=white] (v0x5.center) circle (0.2);
\draw[fill=black!20!white] (v0x6.center) circle (0.2);
\draw[fill=white] (v1x0.center) circle (0.2);
\draw[fill=white] (v1x1.center) circle (0.2);
\draw[fill=black!20!white] (v1x2.center) circle (0.2);
\draw[fill=white] (v1x3.center) circle (0.2);
\draw[fill=white] (v1x4.center) circle (0.2);
\end{tikzpicture}}
\\

$\affD_{3n+2}\TRA \affD_{n+2}$
&
$\affA_{6n-1}\TRA \affA_{2n-1}$
&
$\affD_{5}\TRA \affA_{3}$
&
$\affE_{7}\TRA \affD_{6}$
&
$\affE_{6}\TRA \affD_{4}$
\\

for $n=3$
&
for $n=3$
&

&

&

\\\hline

\end{tabular}
}
\caption{\label{figure:double_bindings_scf_3}Two infinite and three exceptional families of double bindings with scaling factor $(1,3)$. All blue components have types $A_5$ or $D_4$.}
\end{figure}

\begin{definition}
	We say that a Dynkin diagram of a weak generalized Cartan matrix of affine type (see Figure~\ref{fig:aff}) is \emph{ambiguous} if it either has at most two vertices (with the exception of $\loops{A_{1}^\parr1}$) or it is a path with two double arrows at the ends. Otherwise, we call it \emph{unambiguous}. The set of ambiguous diagrams is equal to
	\[\{A_1^\parr1, A_\ell^\parr1 (\ell\geq 2), A_2^\parr2, \loops{A_{1}^\parr1}\}\bigcup\{D_{\ell+1}^\parr2,C_\ell^\parr1,A_{2\ell}^\parr2\}.\]
\end{definition}

The terminology is motivated by the following proposition which is a variation on~\cite[Remark~2.1]{S}.

\begin{proposition}\label{prop:descr}
	Let $G$ be an \affinite or an \affaff $ADE$ bigraph and suppose that its Dynkin diagram $S(G)$ is unambiguous. Then $G$ is uniquely determined by $\descr(G)$.
\end{proposition}
\begin{proof}
	This is easy to see because if $S(G)$ is not one of the ambiguous Dynkin diagrams then $S(G)$ is a tree (with possibly some loops) and at most one affine $\boxtimes$ finite double binding. As it follows from Theorem~\ref{thm:affinite_class}, each of them is uniquely determined by its description, and the result follows since an automorphism of the double binding always induces an automorphism of the rest of $G$. This is slightly non-trivial to see when $G$ has a loop but in this case the result easily follows from our considerations in Section~\ref{subsub:path_affA}.
\end{proof}

\section{Many \affaff $ADE$ bigraphs}\label{sect:constr}
In this section we give several constructions that produce \affaff $ADE$ bigraphs. As we will see in the next section, they will be sufficient for us to state our classification theorem which is the main result of this paper.

\subsection{Twists}
The following construction is due to Stembridge~\cite{S}.
\begin{definition}\label{dfn:twist}
Given a bipartite undirected graph $H$ with vertex set $V$, we define the \emph{twist} $H\times H=(\Gamma,\Delta)$ to be a bipartite bigraph with vertex set $V'\cup V''$ and edge sets defined as follows.
\begin{itemize}
	\item For any edge $(u,v)$ of $H$, $\Gamma$ contains edges $(u', v')$ and $(u'',v'')$.
	\item For any edge $(u,v)$ of $H$, $\Delta$ contains edges $(u', v'')$ and $(u'',v')$.
\end{itemize}
\end{definition}

\begin{figure}
\scalebox{0.4}{
\begin{tikzpicture}
\coordinate (v0x0) at (-6.84,-1.20);
\coordinate (v0x1) at (-6.00,1.20);
\coordinate (v0x2) at (-3.60,0.00);
\coordinate (v0x3) at (-1.20,0.00);
\coordinate (v0x4) at (1.20,0.00);
\coordinate (v0x5) at (3.60,0.00);
\coordinate (v0x6) at (6.00,0.00);
\coordinate (v0x7) at (9.24,-1.20);
\coordinate (v0x8) at (8.40,1.20);
\coordinate (v1x0) at (-6.84,4.20);
\coordinate (v1x1) at (-6.00,6.60);
\coordinate (v1x2) at (-3.60,5.40);
\coordinate (v1x3) at (-1.20,5.40);
\coordinate (v1x4) at (1.20,5.40);
\coordinate (v1x5) at (3.60,5.40);
\coordinate (v1x6) at (6.00,5.40);
\coordinate (v1x7) at (9.24,4.20);
\coordinate (v1x8) at (8.40,6.60);
\draw[color=red,line width=0.75mm] (v0x2) to[] (v0x0);
\draw[color=red,line width=0.75mm] (v0x2) to[] (v0x1);
\draw[color=red,line width=0.75mm] (v0x3) to[] (v0x2);
\draw[color=red,line width=0.75mm] (v0x4) to[] (v0x3);
\draw[color=red,line width=0.75mm] (v0x5) to[] (v0x4);
\draw[color=red,line width=0.75mm] (v0x6) to[] (v0x5);
\draw[color=red,line width=0.75mm] (v0x7) to[] (v0x6);
\draw[color=red,line width=0.75mm] (v0x8) to[] (v0x6);
\draw[color=blue,line width=0.75mm] (v1x0) to[] (v0x2);
\draw[color=blue,line width=0.75mm] (v1x1) to[] (v0x2);
\draw[color=blue,line width=0.75mm] (v1x2) to[bend left=8] (v0x0);
\draw[color=blue,line width=0.75mm] (v1x2) to[] (v0x1);
\draw[color=blue,line width=0.75mm] (v1x2) to[] (v0x3);
\draw[color=red,line width=0.75mm] (v1x2) to[] (v1x0);
\draw[color=red,line width=0.75mm] (v1x2) to[] (v1x1);
\draw[color=blue,line width=0.75mm] (v1x3) to[] (v0x2);
\draw[color=blue,line width=0.75mm] (v1x3) to[] (v0x4);
\draw[color=red,line width=0.75mm] (v1x3) to[] (v1x2);
\draw[color=blue,line width=0.75mm] (v1x4) to[] (v0x3);
\draw[color=blue,line width=0.75mm] (v1x4) to[] (v0x5);
\draw[color=red,line width=0.75mm] (v1x4) to[] (v1x3);
\draw[color=blue,line width=0.75mm] (v1x5) to[] (v0x4);
\draw[color=blue,line width=0.75mm] (v1x5) to[] (v0x6);
\draw[color=red,line width=0.75mm] (v1x5) to[] (v1x4);
\draw[color=blue,line width=0.75mm] (v1x6) to[] (v0x5);
\draw[color=blue,line width=0.75mm] (v1x6) to[bend right=8] (v0x7);
\draw[color=blue,line width=0.75mm] (v1x6) to[] (v0x8);
\draw[color=red,line width=0.75mm] (v1x6) to[] (v1x5);
\draw[color=blue,line width=0.75mm] (v1x7) to[] (v0x6);
\draw[color=red,line width=0.75mm] (v1x7) to[] (v1x6);
\draw[color=blue,line width=0.75mm] (v1x8) to[] (v0x6);
\draw[color=red,line width=0.75mm] (v1x8) to[] (v1x6);
\draw[fill=white] (v0x0.center) circle (0.2);
\draw[fill=white] (v0x1.center) circle (0.2);
\draw[fill=black!20!white] (v0x2.center) circle (0.2);
\draw[fill=white] (v0x3.center) circle (0.2);
\draw[fill=black!20!white] (v0x4.center) circle (0.2);
\draw[fill=white] (v0x5.center) circle (0.2);
\draw[fill=black!20!white] (v0x6.center) circle (0.2);
\draw[fill=white] (v0x7.center) circle (0.2);
\draw[fill=white] (v0x8.center) circle (0.2);
\draw[fill=white] (v1x0.center) circle (0.2);
\draw[fill=white] (v1x1.center) circle (0.2);
\draw[fill=black!20!white] (v1x2.center) circle (0.2);
\draw[fill=white] (v1x3.center) circle (0.2);
\draw[fill=black!20!white] (v1x4.center) circle (0.2);
\draw[fill=white] (v1x5.center) circle (0.2);
\draw[fill=black!20!white] (v1x6.center) circle (0.2);
\draw[fill=white] (v1x7.center) circle (0.2);
\draw[fill=white] (v1x8.center) circle (0.2);
\end{tikzpicture}}
\caption{\label{fig:twist} A twist $\affD_8\times\affD_8$. Twists are listed as family~\bg{twist} in the classification.}
\end{figure}
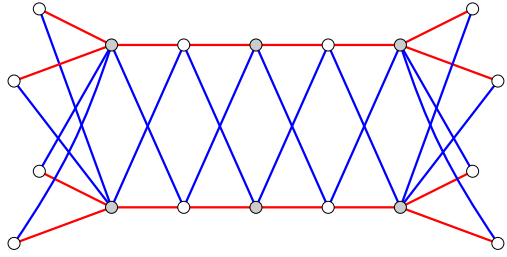

In particular, if $H$ is a bipartite affine $ADE$ Dynkin diagram $\affL$ then $H\times H$ is an \affaff $ADE$ bigraph (see Corollary~\ref{cor:twist_recurrent}) which is called a \emph{twist of type $\affL\times\affL$}. For $G=\affL\times\affL$, we have $\descr(G)=\descr(G^\op)=\affL\DLRA \affL$ thus by Proposition~\ref{prop:descr}, twists may not be uniquely determined by their description. See Figure~\ref{fig:twist} (or Figure~\ref{fig:pstwist}) for an example.

\def\Toric{\Tcal}
\subsection{Toric bigraphs}\label{sect:toric}
Let $\affL$ be a bipartite affine $ADE$ Dynkin diagram and let $\eta$ be its automorphism (not necessarily of order two or color-preserving). For an integer $n\geq 1$, we define a \emph{toric bigraph} $\Toric(\affL,\eta,n)=(\Gamma,\Delta)$ as follows. The red connected components of $\Toric(\affL,\eta,n)$ are $C_1,C_2,\dots,C_n$, and the restriction of $\Gamma$ on each $C_i$ has type $\affL$. In particular, for each $i\in[n]$, let us fix a map $\phi_i:\Vert(\affL)\to C_i$ that induces an isomorphism between $\affL$ and $\Gamma(C_i)$. Now, for every $i=1,2,\dots,n-1$ and every vertex $v$ of $\affL$, $\Delta$ contains an edge $(\phi_i(v),\phi_{i+1}(v))$. Also, for every $v\in\Vert(\affL)$, $\Delta$ contains an edge $(\phi_n(v),\phi_1(\eta(v)))$. Thus if one starts at some vertex $v\in C_1$ and follows the blue path that traverses the components $C_1,C_2,C_3,\dots,C_n,C_1$, one arrives at $\eta(v)$.

\begin{lemma}\label{lemma:parity}
	In the following cases, $\Toric(\affL,\eta,n)$ is an \affaff $ADE$ bigraph:
	\begin{enumerate}
		\item $\eta$ is color-reversing and $n\geq3$ is odd;
		\item $\eta$ is color-preserving and $n\geq2$ is even;
		\item $n=1$, $\eta$ is color-reversing and does not send any vertex to one of its neighbors.
	\end{enumerate}
\end{lemma}
\begin{proof}
	The fact that $\Toric(\affL,\eta,n)$ is always recurrent is trivial to check, thus we only need to make sure that it is bipartite and that $\Gamma$ and $\Delta$ do not share edges. This is easy to see and the result follows since the components of $\Gamma$ and $\Delta$ are affine $ADE$ Dynkin diagrams by construction.
\end{proof}

If $G=\Toric(\affL,\eta,n)$ is an \affaff $ADE$ bigraph and $n>1$ then $S(G)=A_{n-1}^\parr1$ and $\descr(G)$ equals
\[\ZAA{\affL}{\affL}{\affL}{\affL}\nodeZ{,}\]
where the number of components is $n$. If $n=1$ then $S(G)=\loops{A_1^\parr1}$ and $\descr(G)$ equals  
\[\Zlooploop{\affL}\nodeZ{.}\]


Even though the main purpose of this section is to produce many \affaff $ADE$ bigraphs and not worry about which of them are isomorphic, we give a simple criterion for when two toric bigraphs are isomorphic.

\begin{proposition}\label{prop:conjugacy}
	Let $\affL$ be an affine $ADE$ Dynkin diagram and let $\Aut(\affL)$ be the automorphism group of $\affL$. Let $\eta,\eta'\in \Aut(\affL)$ be two automorphisms of $\affL$. Then the bigraphs $G=\Toric(\affL,\eta,n)$ and $G=\Toric(\affL,\eta',n)$ are isomorphic if and only if $\eta$ is conjugate to either $\eta'$ or its inverse in $\Aut(\affL)$.
\end{proposition}
\begin{proof}
	Suppose that there exists $g\in\Aut(\affL)$ such that $\eta'=g\eta g^{-1}$. Consider a map from $G$ to $G'$ that applies $g$ to each red connected component of $G$. It is clear that this map provides an isomorphism between $G$ and $G'$.
	
	Suppose now that $\eta$ and $\eta'$ are inverses of each other. Then consider a map from $G$ to $G'$ that reverses the order of the red connected components. Again, this map clearly gives an isomorphism between $G$ and $G'$.

	Conversely, suppose that there is an isomorphism $\psi$ between $G$ and $G'$. Then it must either preserve or reverse the cyclic ordering of the red connected components $C_1,C_2,\dots,C_n$. But note that reversing their ordering just corresponds to replacing $\eta$ with its inverse, and cyclically permuting the components does not change $\eta$, so we may assume that $\psi$ sends $C_i$ to $C_i'$ for all $i\in[n]$. It follows that $(\phi_i')^{-1}\circ\psi\circ\phi_i$ is the same element of $\Aut(\affL)$ for all $i\in[n]$ where $\phi_i$ and $\phi_i'$ are the maps used in the definitions of $\Toric(\affL,\eta,n)$ and $\Toric(\affL,\eta',n)$. The result follows since conjugating by this element takes $\eta$ to $\eta'$.
\end{proof}

Let us say that a \emph{weak conjugacy class} of an element $g$ in a group $H$ is the union of the conjugacy class of $g$ with the conjugacy class of $g^{-1}$. 

Thus in order to classify \affaff toric $ADE$ bigraphs it suffices to list representatives of weak conjugacy classes in $\Aut(\affL)$ for each affine $ADE$ Dynkin diagram $\affL$. Since for any even $n$ we have $\Toric(\affL,\id,n)=\affL\otimes\affA_{n-1}$, we only list \emph{non-identity weak conjugacy classes} of $\Aut(\affL)$ in each case. Thus \emph{we do not consider the case $\eta=\id$ to be a toric bigraph in what follows}.

\begin{remark}
	Whenever $\affL$ is a diagram from Figure~\ref{fig:fin} and $\eta\in\Aut(\affL)$, there is another diagram in Figure~\ref{fig:fin} which we denote $\affL/\eta$. It is obtained from $\affL$ by \emph{folding via $\eta$}. This notion is standard but we do not define it rigorously here. Note that if $G=\Toric(\affL,\eta,n)$ then $S(G^\opp)$ is exactly $\affL/\eta$, and the label of a vertex $v$ of $\affL/\eta$ in $\descr(G^\opp)$ is $A_{rn-1}$ where $r$ is the size of the preimage of $v$ in $\affL$ (i.e. $v$ corresponds to an orbit of some vertex of $\affL$ under $\eta$ and $r$ is the size of this orbit). 
\end{remark}

\def\Dih{ \operatorname{Dih}}
\subsubsection{The case $\affL=\affA_{2m-1}$.}\label{subsub:toric_affA} 

\begin{figure}
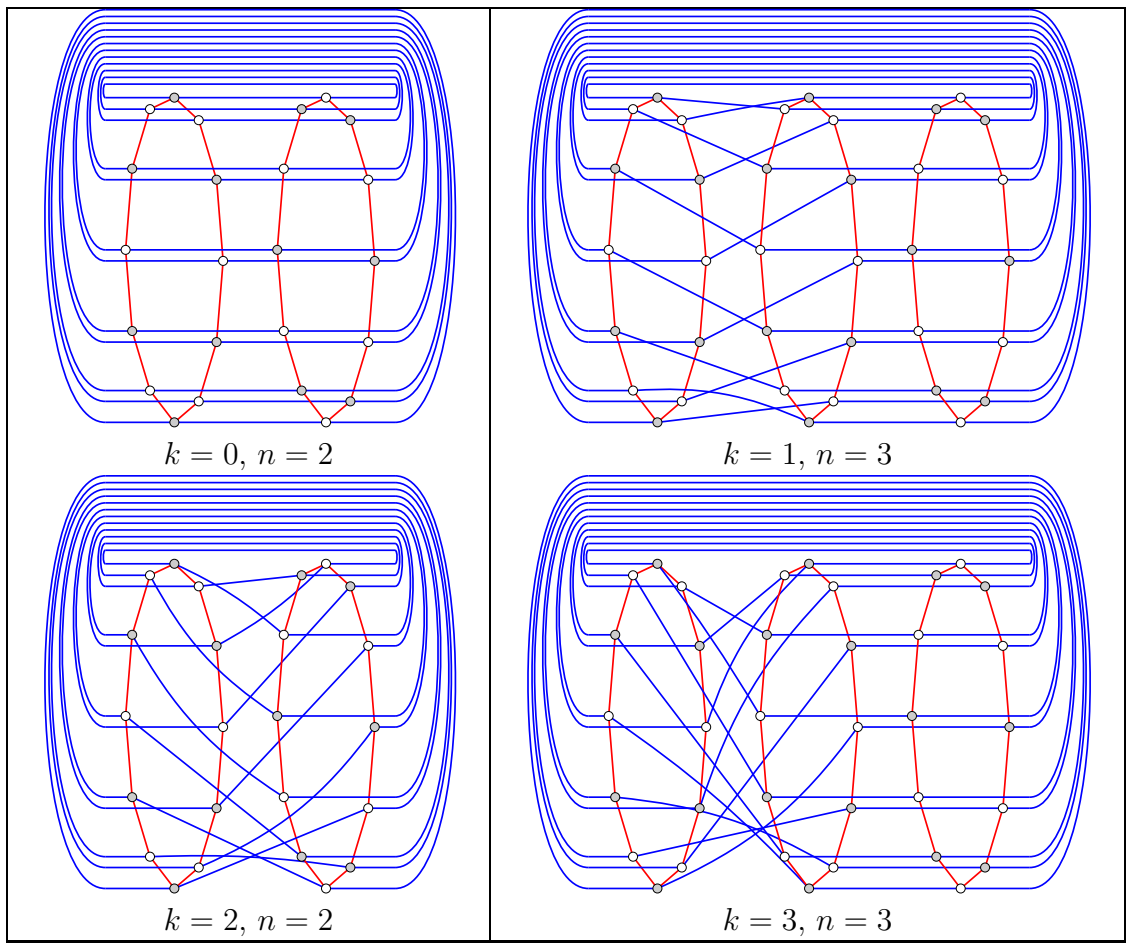

\makebox[1.0\textwidth]{

}
\caption{\label{fig:toric-A-rotn} The family $\Toric(\affA_{2m-1},\exp(\pi i k/m),n)$ for $m=6$, $k=0,1,2,3$, and $n=2$ or $n=3$ depending on the parity of~$k$. In the classification in Section~\ref{sect:classif}, this is family~\bg{toric-A-rotn}.}
\end{figure}

\begin{figure}
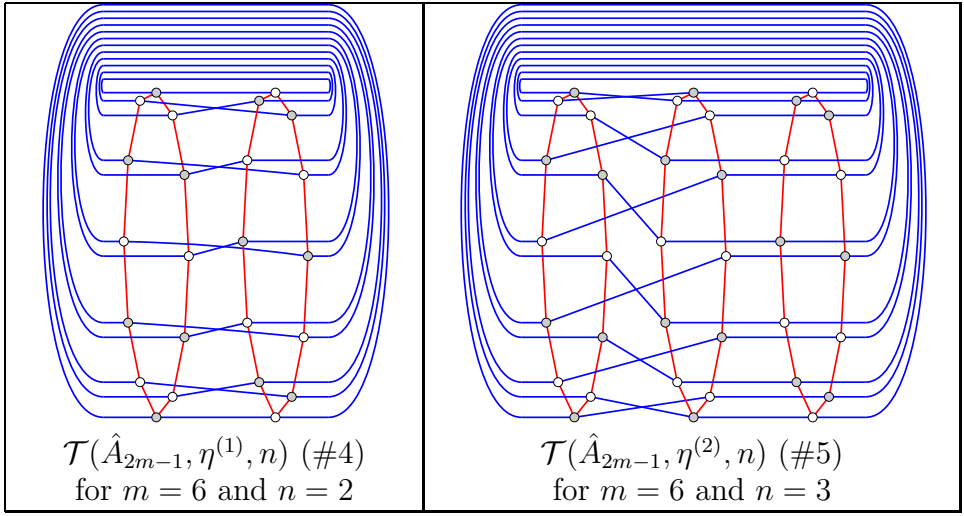

\makebox[1.0\textwidth]{
%
}
\caption{\label{fig:toric-A} The other two families of toric bigraphs of type $\affA$.}
\end{figure}

In this case, $\Aut(\affL)$ is isomorphic to $\Dih(4m)$, the dihedral group of the $2m$-gon with $4m$ elements. It contains a subgroup $\Z/2m\Z$ whose elements we represent by $\exp(\pi i k/m)$ for every residue $k$ modulo $2m$ (that is, for every $k=0,1,\dots,2m-1$). There are $m+2$ non-identity weak conjugacy classes in $\Aut(\affL)$. The representatives of the first $m$ of them are $\exp(\pi i k/m)$ for $k=1,2,\dots,m$, see Figure~\ref{fig:toric-A-rotn}. The other two are a \emph{reflection about a diagonal} (denoted $\eta^\parr{1}$) and a \emph{reflection about a line joining the midpoints of two opposite edges} (denoted $\eta^\parr{2}$), see Figure~\ref{fig:toric-A}. For $k\in[m]$, $\exp(\pi ik/m)$ is color-preserving if and only if $k$ is even. Additionally, $\eta^\parr{1}$ is color-preserving while $\eta^\parr{2}$ is color-reversing.

\subsubsection{The case $\affL=\affD_{m+2}$, $m\geq 3$.}\label{subsub:toric_affD}

\begin{figure}
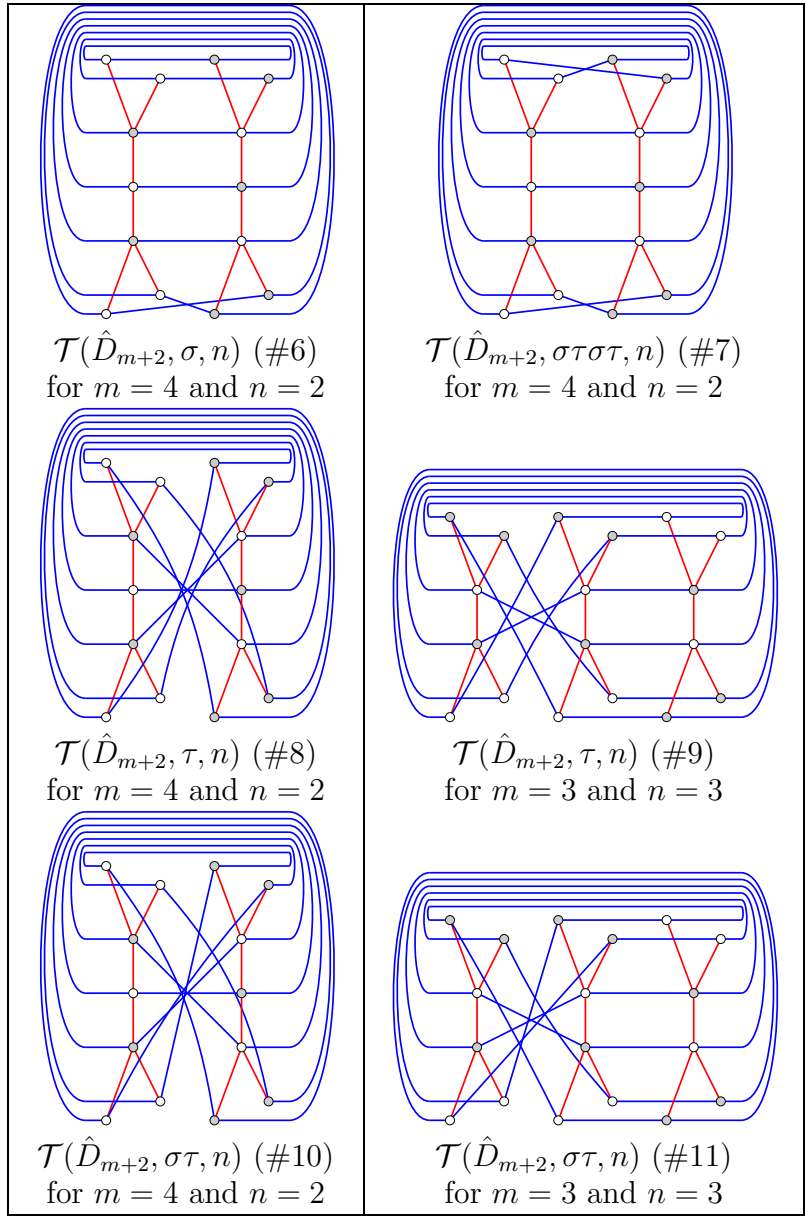

\makebox[1.0\textwidth]{

}
\caption{\label{fig:toric-D} Toric bigraphs of type $\affD_{m+2}$.}
\end{figure}

Let us describe the group $\Aut(\affL)$ in this case. Let $u^+,u^-$ be two leaves of $\affL$ that have a common neighbor, and let $v^+,v^-$ be the other two leaves of $\affL$. Let $\sigma$ be the automorphism of $\affL$ that switches $u^+$ and $u^-$ and fixes the rest of $\affL$. Similarly, let $\tau$ be the automorphism of $\affL$ that switches $u^+$ with $v^+$ and $u^-$ with $v^-$. It is non-trivial to see that $\Aut(\affL)$ is isomorphic to the group $\Dih(8)$ of symmetries of the square.\footnote{The vertices of the square are $u^+,v^+,u^-,v^-$ in this cyclic order.} It is clear that $\Aut(\affL)$ is generated by $\sigma$ and $\tau$. There are four non-identity weak conjugacy classes in $\Aut(\affL)$, and their representatives are 
\[\sigma,\quad\sigma\tau\sigma\tau, \quad\tau, \quad\sigma\tau.\]
The first three elements have order $2$ and the last element has order $4$. The first two elements are always color-preserving while the last two elements are color-preserving if and only if $m$ is even. If $m$ is odd, the last two elements send some vertex to its neighbor.  See Figure~\ref{fig:toric-D} for some examples.

\subsubsection{The case $\affL=\affD_{4}$.}\label{subsub:toric_affD4}

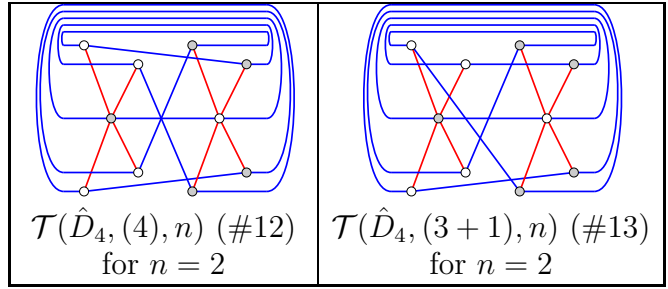
\begin{figure}
\makebox[1.0\textwidth]{
\begin{tabular}{|c|c|}\hline
\scalebox{0.3}{
\begin{tikzpicture}
\coordinate (v0x0) at (-1.20,-2.04);
\coordinate (v0x1) at (1.20,-1.20);
\coordinate (v0x2) at (0.00,1.20);
\coordinate (v0x3) at (-1.20,4.44);
\coordinate (v0x4) at (1.20,3.60);
\coordinate (v1x0) at (3.60,-2.04);
\coordinate (v1x1) at (6.00,-1.20);
\coordinate (v1x2) at (4.80,1.20);
\coordinate (v1x3) at (3.60,4.44);
\coordinate (v1x4) at (6.00,3.60);
\draw[color=red,line width=0.75mm] (v0x2) to[] (v0x0);
\draw[color=red,line width=0.75mm] (v0x2) to[] (v0x1);
\draw[color=red,line width=0.75mm] (v0x3) to[] (v0x2);
\draw[color=red,line width=0.75mm] (v0x4) to[] (v0x2);
\draw[color=blue,line width=0.75mm] (v1x0) to[] (v0x4);
\draw[color=blue,line width=0.75mm] (v1x1) to[] (v0x0);
\draw[color=blue,line width=0.75mm] (v1x2) to[] (v0x2);
\draw[color=red,line width=0.75mm] (v1x2) to[] (v1x0);
\draw[color=red,line width=0.75mm] (v1x2) to[] (v1x1);
\draw[color=blue,line width=0.75mm] (v1x3) to[] (v0x1);
\draw[color=red,line width=0.75mm] (v1x3) to[] (v1x2);
\draw[color=blue,line width=0.75mm] (v1x4) to[] (v0x3);
\draw[color=red,line width=0.75mm] (v1x4) to[] (v1x2);
\draw[color=blue,line width=0.75mm] (-1.20,4.44) -- (-2.10,4.44);
\draw[color=blue,line width=0.75mm] (-2.10,4.44) to[out=180,in=180, looseness=0.5] (-2.10,5.04);
\draw[color=blue,line width=0.75mm] (-2.10,5.04) -- (6.90,5.04);
\draw[color=blue,line width=0.75mm] (6.90,5.04) to[out=0,in=0, looseness=0.5] (6.90,4.44);
\draw[color=blue,line width=0.75mm] (6.90,4.44) -- (3.60,4.44);
\draw[color=blue,line width=0.75mm] (1.20,3.60) -- (-2.10,3.60);
\draw[color=blue,line width=0.75mm] (-2.10,3.60) to[out=180,in=180, looseness=0.5] (-2.10,5.34);
\draw[color=blue,line width=0.75mm] (-2.10,5.34) -- (6.90,5.34);
\draw[color=blue,line width=0.75mm] (6.90,5.34) to[out=0,in=0, looseness=0.5] (6.90,3.60);
\draw[color=blue,line width=0.75mm] (6.90,3.60) -- (6.00,3.60);
\draw[color=blue,line width=0.75mm] (0.00,1.20) -- (-2.10,1.20);
\draw[color=blue,line width=0.75mm] (-2.10,1.20) to[out=180,in=180, looseness=0.5] (-2.10,5.64);
\draw[color=blue,line width=0.75mm] (-2.10,5.64) -- (6.90,5.64);
\draw[color=blue,line width=0.75mm] (6.90,5.64) to[out=0,in=0, looseness=0.5] (6.90,1.20);
\draw[color=blue,line width=0.75mm] (6.90,1.20) -- (4.80,1.20);
\draw[color=blue,line width=0.75mm] (1.20,-1.20) -- (-2.10,-1.20);
\draw[color=blue,line width=0.75mm] (-2.10,-1.20) to[out=180,in=180, looseness=0.5] (-2.10,5.94);
\draw[color=blue,line width=0.75mm] (-2.10,5.94) -- (6.90,5.94);
\draw[color=blue,line width=0.75mm] (6.90,5.94) to[out=0,in=0, looseness=0.5] (6.90,-1.20);
\draw[color=blue,line width=0.75mm] (6.90,-1.20) -- (6.00,-1.20);
\draw[color=blue,line width=0.75mm] (-1.20,-2.04) -- (-2.10,-2.04);
\draw[color=blue,line width=0.75mm] (-2.10,-2.04) to[out=180,in=180, looseness=0.5] (-2.10,6.24);
\draw[color=blue,line width=0.75mm] (-2.10,6.24) -- (6.90,6.24);
\draw[color=blue,line width=0.75mm] (6.90,6.24) to[out=0,in=0, looseness=0.5] (6.90,-2.04);
\draw[color=blue,line width=0.75mm] (6.90,-2.04) -- (3.60,-2.04);
\draw[fill=white] (v0x0.center) circle (0.2);
\draw[fill=white] (v0x1.center) circle (0.2);
\draw[fill=black!20!white] (v0x2.center) circle (0.2);
\draw[fill=white] (v0x3.center) circle (0.2);
\draw[fill=white] (v0x4.center) circle (0.2);
\draw[fill=black!20!white] (v1x0.center) circle (0.2);
\draw[fill=black!20!white] (v1x1.center) circle (0.2);
\draw[fill=white] (v1x2.center) circle (0.2);
\draw[fill=black!20!white] (v1x3.center) circle (0.2);
\draw[fill=black!20!white] (v1x4.center) circle (0.2);
\end{tikzpicture}}
&
\scalebox{0.3}{
\begin{tikzpicture}
\coordinate (v0x0) at (-1.20,-2.04);
\coordinate (v0x1) at (1.20,-1.20);
\coordinate (v0x2) at (0.00,1.20);
\coordinate (v0x3) at (-1.20,4.44);
\coordinate (v0x4) at (1.20,3.60);
\coordinate (v1x0) at (3.60,-2.04);
\coordinate (v1x1) at (6.00,-1.20);
\coordinate (v1x2) at (4.80,1.20);
\coordinate (v1x3) at (3.60,4.44);
\coordinate (v1x4) at (6.00,3.60);
\draw[color=red,line width=0.75mm] (v0x2) to[] (v0x0);
\draw[color=red,line width=0.75mm] (v0x2) to[] (v0x1);
\draw[color=red,line width=0.75mm] (v0x3) to[] (v0x2);
\draw[color=red,line width=0.75mm] (v0x4) to[] (v0x2);
\draw[color=blue,line width=0.75mm] (v1x0) to[] (v0x3);
\draw[color=blue,line width=0.75mm] (v1x1) to[] (v0x0);
\draw[color=blue,line width=0.75mm] (v1x2) to[] (v0x2);
\draw[color=red,line width=0.75mm] (v1x2) to[] (v1x0);
\draw[color=red,line width=0.75mm] (v1x2) to[] (v1x1);
\draw[color=blue,line width=0.75mm] (v1x3) to[] (v0x1);
\draw[color=red,line width=0.75mm] (v1x3) to[] (v1x2);
\draw[color=blue,line width=0.75mm] (v1x4) to[] (v0x4);
\draw[color=red,line width=0.75mm] (v1x4) to[] (v1x2);
\draw[color=blue,line width=0.75mm] (-1.20,4.44) -- (-2.10,4.44);
\draw[color=blue,line width=0.75mm] (-2.10,4.44) to[out=180,in=180, looseness=0.5] (-2.10,5.04);
\draw[color=blue,line width=0.75mm] (-2.10,5.04) -- (6.90,5.04);
\draw[color=blue,line width=0.75mm] (6.90,5.04) to[out=0,in=0, looseness=0.5] (6.90,4.44);
\draw[color=blue,line width=0.75mm] (6.90,4.44) -- (3.60,4.44);
\draw[color=blue,line width=0.75mm] (1.20,3.60) -- (-2.10,3.60);
\draw[color=blue,line width=0.75mm] (-2.10,3.60) to[out=180,in=180, looseness=0.5] (-2.10,5.34);
\draw[color=blue,line width=0.75mm] (-2.10,5.34) -- (6.90,5.34);
\draw[color=blue,line width=0.75mm] (6.90,5.34) to[out=0,in=0, looseness=0.5] (6.90,3.60);
\draw[color=blue,line width=0.75mm] (6.90,3.60) -- (6.00,3.60);
\draw[color=blue,line width=0.75mm] (0.00,1.20) -- (-2.10,1.20);
\draw[color=blue,line width=0.75mm] (-2.10,1.20) to[out=180,in=180, looseness=0.5] (-2.10,5.64);
\draw[color=blue,line width=0.75mm] (-2.10,5.64) -- (6.90,5.64);
\draw[color=blue,line width=0.75mm] (6.90,5.64) to[out=0,in=0, looseness=0.5] (6.90,1.20);
\draw[color=blue,line width=0.75mm] (6.90,1.20) -- (4.80,1.20);
\draw[color=blue,line width=0.75mm] (1.20,-1.20) -- (-2.10,-1.20);
\draw[color=blue,line width=0.75mm] (-2.10,-1.20) to[out=180,in=180, looseness=0.5] (-2.10,5.94);
\draw[color=blue,line width=0.75mm] (-2.10,5.94) -- (6.90,5.94);
\draw[color=blue,line width=0.75mm] (6.90,5.94) to[out=0,in=0, looseness=0.5] (6.90,-1.20);
\draw[color=blue,line width=0.75mm] (6.90,-1.20) -- (6.00,-1.20);
\draw[color=blue,line width=0.75mm] (-1.20,-2.04) -- (-2.10,-2.04);
\draw[color=blue,line width=0.75mm] (-2.10,-2.04) to[out=180,in=180, looseness=0.5] (-2.10,6.24);
\draw[color=blue,line width=0.75mm] (-2.10,6.24) -- (6.90,6.24);
\draw[color=blue,line width=0.75mm] (6.90,6.24) to[out=0,in=0, looseness=0.5] (6.90,-2.04);
\draw[color=blue,line width=0.75mm] (6.90,-2.04) -- (3.60,-2.04);
\draw[fill=white] (v0x0.center) circle (0.2);
\draw[fill=white] (v0x1.center) circle (0.2);
\draw[fill=black!20!white] (v0x2.center) circle (0.2);
\draw[fill=white] (v0x3.center) circle (0.2);
\draw[fill=white] (v0x4.center) circle (0.2);
\draw[fill=black!20!white] (v1x0.center) circle (0.2);
\draw[fill=black!20!white] (v1x1.center) circle (0.2);
\draw[fill=white] (v1x2.center) circle (0.2);
\draw[fill=black!20!white] (v1x3.center) circle (0.2);
\draw[fill=black!20!white] (v1x4.center) circle (0.2);
\end{tikzpicture}}
\\

$\Toric(\affD_{4},(4),n)$ (\bg{toric-D4-4})
&
$\Toric(\affD_{4},(3+1),n)$ (\bg{toric-D4-31})
\\

for $n=2$
&
for $n=2$
\\\hline

\end{tabular}
}
\caption{\label{fig:toric-D4} Additional toric bigraphs of type $\affD_4$.}
\end{figure}

In this case $\Aut(\affL)=\Sfr_4$, the symmetric group on four elements. Two permutations belong to the same (weak) conjugacy class if and only if they have the same cycle type, thus the weak conjugacy classes are in bijection with partitions $\l$ of $4$ which we denote by
\[(1+1+1+1),\quad (2+1+1),\quad (2+2),\quad (4),\quad (3+1).\]
Since $(1+1+1+1)$ corresponds to the identity permutation, we only need to consider the other four cases. The only case that we have not considered in the previous section is $(3+1)$, however, note that the permutations $\sigma\tau\sigma\tau$ and $\tau$ were not conjugate for $m\geq 3$ but are conjugate for $m=2$ since they both have cycle type $(2+2)$. In the classification, we treat $\sigma$ as a representative of $(2+1+1)$ and $\sigma\tau\sigma\tau$ as a representative of $(2+2)$ which naturally includes $\affD_4$ as a special $m=2$ case of $\affD_{m+2}$. The cases $(4)$ and $(3+1)$ are listed in the classification as separate items, see also Figure~\ref{fig:toric-D4}.

\subsubsection{The case $\affL=\affE_{6}$.}\label{subsub:toric_affE6}

\begin{figure}
\makebox[1.0\textwidth]{
\begin{tabular}{|c|}\hline

\begin{tabular}{cc}
\scalebox{0.3}{
\begin{tikzpicture}
\coordinate (v0x0) at (0.00,9.60);
\coordinate (v0x1) at (0.00,7.20);
\coordinate (v0x2) at (-0.90,0.72);
\coordinate (v0x3) at (-0.90,3.12);
\coordinate (v0x4) at (0.00,4.80);
\coordinate (v0x5) at (0.90,1.68);
\coordinate (v0x6) at (0.90,-0.72);
\coordinate (v1x0) at (4.20,9.60);
\coordinate (v1x1) at (4.20,7.20);
\coordinate (v1x2) at (3.30,0.72);
\coordinate (v1x3) at (3.30,3.12);
\coordinate (v1x4) at (4.20,4.80);
\coordinate (v1x5) at (5.10,1.68);
\coordinate (v1x6) at (5.10,-0.72);
\draw[color=red,line width=0.75mm] (v0x1) to[] (v0x0);
\draw[color=red,line width=0.75mm] (v0x3) to[] (v0x2);
\draw[color=red,line width=0.75mm] (v0x4) to[] (v0x1);
\draw[color=red,line width=0.75mm] (v0x4) to[] (v0x3);
\draw[color=red,line width=0.75mm] (v0x5) to[] (v0x4);
\draw[color=red,line width=0.75mm] (v0x6) to[] (v0x5);
\draw[color=blue,line width=0.75mm] (v1x0) to[] (v0x0);
\draw[color=blue,line width=0.75mm] (v1x1) to[] (v0x1);
\draw[color=red,line width=0.75mm] (v1x1) to[] (v1x0);
\draw[color=blue,line width=0.75mm] (v1x2) to[] (v0x6);
\draw[color=blue,line width=0.75mm] (v1x3) to[] (v0x5);
\draw[color=red,line width=0.75mm] (v1x3) to[] (v1x2);
\draw[color=blue,line width=0.75mm] (v1x4) to[] (v0x4);
\draw[color=red,line width=0.75mm] (v1x4) to[] (v1x1);
\draw[color=red,line width=0.75mm] (v1x4) to[] (v1x3);
\draw[color=blue,line width=0.75mm] (v1x5) to[] (v0x3);
\draw[color=red,line width=0.75mm] (v1x5) to[] (v1x4);
\draw[color=blue,line width=0.75mm] (v1x6) to[] (v0x2);
\draw[color=red,line width=0.75mm] (v1x6) to[] (v1x5);
\draw[color=blue,line width=0.75mm] (0.00,9.60) -- (-1.80,9.60);
\draw[color=blue,line width=0.75mm] (-1.80,9.60) to[out=180,in=180, looseness=0.5] (-1.80,10.20);
\draw[color=blue,line width=0.75mm] (-1.80,10.20) -- (6.00,10.20);
\draw[color=blue,line width=0.75mm] (6.00,10.20) to[out=0,in=0, looseness=0.5] (6.00,9.60);
\draw[color=blue,line width=0.75mm] (6.00,9.60) -- (4.20,9.60);
\draw[color=blue,line width=0.75mm] (0.00,7.20) -- (-1.80,7.20);
\draw[color=blue,line width=0.75mm] (-1.80,7.20) to[out=180,in=180, looseness=0.5] (-1.80,10.50);
\draw[color=blue,line width=0.75mm] (-1.80,10.50) -- (6.00,10.50);
\draw[color=blue,line width=0.75mm] (6.00,10.50) to[out=0,in=0, looseness=0.5] (6.00,7.20);
\draw[color=blue,line width=0.75mm] (6.00,7.20) -- (4.20,7.20);
\draw[color=blue,line width=0.75mm] (0.00,4.80) -- (-1.80,4.80);
\draw[color=blue,line width=0.75mm] (-1.80,4.80) to[out=180,in=180, looseness=0.5] (-1.80,10.80);
\draw[color=blue,line width=0.75mm] (-1.80,10.80) -- (6.00,10.80);
\draw[color=blue,line width=0.75mm] (6.00,10.80) to[out=0,in=0, looseness=0.5] (6.00,4.80);
\draw[color=blue,line width=0.75mm] (6.00,4.80) -- (4.20,4.80);
\draw[color=blue,line width=0.75mm] (-0.90,3.12) -- (-1.80,3.12);
\draw[color=blue,line width=0.75mm] (-1.80,3.12) to[out=180,in=180, looseness=0.5] (-1.80,11.10);
\draw[color=blue,line width=0.75mm] (-1.80,11.10) -- (6.00,11.10);
\draw[color=blue,line width=0.75mm] (6.00,11.10) to[out=0,in=0, looseness=0.5] (6.00,3.12);
\draw[color=blue,line width=0.75mm] (6.00,3.12) -- (3.30,3.12);
\draw[color=blue,line width=0.75mm] (0.90,1.68) -- (-1.80,1.68);
\draw[color=blue,line width=0.75mm] (-1.80,1.68) to[out=180,in=180, looseness=0.5] (-1.80,11.40);
\draw[color=blue,line width=0.75mm] (-1.80,11.40) -- (6.00,11.40);
\draw[color=blue,line width=0.75mm] (6.00,11.40) to[out=0,in=0, looseness=0.5] (6.00,1.68);
\draw[color=blue,line width=0.75mm] (6.00,1.68) -- (5.10,1.68);
\draw[color=blue,line width=0.75mm] (-0.90,0.72) -- (-1.80,0.72);
\draw[color=blue,line width=0.75mm] (-1.80,0.72) to[out=180,in=180, looseness=0.5] (-1.80,11.70);
\draw[color=blue,line width=0.75mm] (-1.80,11.70) -- (6.00,11.70);
\draw[color=blue,line width=0.75mm] (6.00,11.70) to[out=0,in=0, looseness=0.5] (6.00,0.72);
\draw[color=blue,line width=0.75mm] (6.00,0.72) -- (3.30,0.72);
\draw[color=blue,line width=0.75mm] (0.90,-0.72) -- (-1.80,-0.72);
\draw[color=blue,line width=0.75mm] (-1.80,-0.72) to[out=180,in=180, looseness=0.5] (-1.80,12.00);
\draw[color=blue,line width=0.75mm] (-1.80,12.00) -- (6.00,12.00);
\draw[color=blue,line width=0.75mm] (6.00,12.00) to[out=0,in=0, looseness=0.5] (6.00,-0.72);
\draw[color=blue,line width=0.75mm] (6.00,-0.72) -- (5.10,-0.72);
\draw[fill=black!20!white] (v0x0.center) circle (0.2);
\draw[fill=white] (v0x1.center) circle (0.2);
\draw[fill=black!20!white] (v0x2.center) circle (0.2);
\draw[fill=white] (v0x3.center) circle (0.2);
\draw[fill=black!20!white] (v0x4.center) circle (0.2);
\draw[fill=white] (v0x5.center) circle (0.2);
\draw[fill=black!20!white] (v0x6.center) circle (0.2);
\draw[fill=white] (v1x0.center) circle (0.2);
\draw[fill=black!20!white] (v1x1.center) circle (0.2);
\draw[fill=white] (v1x2.center) circle (0.2);
\draw[fill=black!20!white] (v1x3.center) circle (0.2);
\draw[fill=white] (v1x4.center) circle (0.2);
\draw[fill=black!20!white] (v1x5.center) circle (0.2);
\draw[fill=white] (v1x6.center) circle (0.2);
\end{tikzpicture}}
&
\scalebox{0.3}{
\begin{tikzpicture}
\coordinate (v0x0) at (0.00,9.60);
\coordinate (v0x1) at (0.00,7.20);
\coordinate (v0x2) at (-0.90,0.72);
\coordinate (v0x3) at (-0.90,3.12);
\coordinate (v0x4) at (0.00,4.80);
\coordinate (v0x5) at (0.90,1.68);
\coordinate (v0x6) at (0.90,-0.72);
\coordinate (v1x0) at (4.20,9.60);
\coordinate (v1x1) at (4.20,7.20);
\coordinate (v1x2) at (3.30,0.72);
\coordinate (v1x3) at (3.30,3.12);
\coordinate (v1x4) at (4.20,4.80);
\coordinate (v1x5) at (5.10,1.68);
\coordinate (v1x6) at (5.10,-0.72);
\draw[color=red,line width=0.75mm] (v0x1) to[] (v0x0);
\draw[color=red,line width=0.75mm] (v0x3) to[] (v0x2);
\draw[color=red,line width=0.75mm] (v0x4) to[] (v0x1);
\draw[color=red,line width=0.75mm] (v0x4) to[] (v0x3);
\draw[color=red,line width=0.75mm] (v0x5) to[] (v0x4);
\draw[color=red,line width=0.75mm] (v0x6) to[] (v0x5);
\draw[color=blue,line width=0.75mm] (v1x0) to[] (v0x6);
\draw[color=blue,line width=0.75mm] (v1x1) to[] (v0x5);
\draw[color=red,line width=0.75mm] (v1x1) to[] (v1x0);
\draw[color=blue,line width=0.75mm] (v1x2) to[] (v0x0);
\draw[color=blue,line width=0.75mm] (v1x3) to[] (v0x1);
\draw[color=red,line width=0.75mm] (v1x3) to[] (v1x2);
\draw[color=blue,line width=0.75mm] (v1x4) to[] (v0x4);
\draw[color=red,line width=0.75mm] (v1x4) to[] (v1x1);
\draw[color=red,line width=0.75mm] (v1x4) to[] (v1x3);
\draw[color=blue,line width=0.75mm] (v1x5) to[] (v0x3);
\draw[color=red,line width=0.75mm] (v1x5) to[] (v1x4);
\draw[color=blue,line width=0.75mm] (v1x6) to[] (v0x2);
\draw[color=red,line width=0.75mm] (v1x6) to[] (v1x5);
\draw[color=blue,line width=0.75mm] (0.00,9.60) -- (-1.80,9.60);
\draw[color=blue,line width=0.75mm] (-1.80,9.60) to[out=180,in=180, looseness=0.5] (-1.80,10.20);
\draw[color=blue,line width=0.75mm] (-1.80,10.20) -- (6.00,10.20);
\draw[color=blue,line width=0.75mm] (6.00,10.20) to[out=0,in=0, looseness=0.5] (6.00,9.60);
\draw[color=blue,line width=0.75mm] (6.00,9.60) -- (4.20,9.60);
\draw[color=blue,line width=0.75mm] (0.00,7.20) -- (-1.80,7.20);
\draw[color=blue,line width=0.75mm] (-1.80,7.20) to[out=180,in=180, looseness=0.5] (-1.80,10.50);
\draw[color=blue,line width=0.75mm] (-1.80,10.50) -- (6.00,10.50);
\draw[color=blue,line width=0.75mm] (6.00,10.50) to[out=0,in=0, looseness=0.5] (6.00,7.20);
\draw[color=blue,line width=0.75mm] (6.00,7.20) -- (4.20,7.20);
\draw[color=blue,line width=0.75mm] (0.00,4.80) -- (-1.80,4.80);
\draw[color=blue,line width=0.75mm] (-1.80,4.80) to[out=180,in=180, looseness=0.5] (-1.80,10.80);
\draw[color=blue,line width=0.75mm] (-1.80,10.80) -- (6.00,10.80);
\draw[color=blue,line width=0.75mm] (6.00,10.80) to[out=0,in=0, looseness=0.5] (6.00,4.80);
\draw[color=blue,line width=0.75mm] (6.00,4.80) -- (4.20,4.80);
\draw[color=blue,line width=0.75mm] (-0.90,3.12) -- (-1.80,3.12);
\draw[color=blue,line width=0.75mm] (-1.80,3.12) to[out=180,in=180, looseness=0.5] (-1.80,11.10);
\draw[color=blue,line width=0.75mm] (-1.80,11.10) -- (6.00,11.10);
\draw[color=blue,line width=0.75mm] (6.00,11.10) to[out=0,in=0, looseness=0.5] (6.00,3.12);
\draw[color=blue,line width=0.75mm] (6.00,3.12) -- (3.30,3.12);
\draw[color=blue,line width=0.75mm] (0.90,1.68) -- (-1.80,1.68);
\draw[color=blue,line width=0.75mm] (-1.80,1.68) to[out=180,in=180, looseness=0.5] (-1.80,11.40);
\draw[color=blue,line width=0.75mm] (-1.80,11.40) -- (6.00,11.40);
\draw[color=blue,line width=0.75mm] (6.00,11.40) to[out=0,in=0, looseness=0.5] (6.00,1.68);
\draw[color=blue,line width=0.75mm] (6.00,1.68) -- (5.10,1.68);
\draw[color=blue,line width=0.75mm] (-0.90,0.72) -- (-1.80,0.72);
\draw[color=blue,line width=0.75mm] (-1.80,0.72) to[out=180,in=180, looseness=0.5] (-1.80,11.70);
\draw[color=blue,line width=0.75mm] (-1.80,11.70) -- (6.00,11.70);
\draw[color=blue,line width=0.75mm] (6.00,11.70) to[out=0,in=0, looseness=0.5] (6.00,0.72);
\draw[color=blue,line width=0.75mm] (6.00,0.72) -- (3.30,0.72);
\draw[color=blue,line width=0.75mm] (0.90,-0.72) -- (-1.80,-0.72);
\draw[color=blue,line width=0.75mm] (-1.80,-0.72) to[out=180,in=180, looseness=0.5] (-1.80,12.00);
\draw[color=blue,line width=0.75mm] (-1.80,12.00) -- (6.00,12.00);
\draw[color=blue,line width=0.75mm] (6.00,12.00) to[out=0,in=0, looseness=0.5] (6.00,-0.72);
\draw[color=blue,line width=0.75mm] (6.00,-0.72) -- (5.10,-0.72);
\draw[fill=black!20!white] (v0x0.center) circle (0.2);
\draw[fill=white] (v0x1.center) circle (0.2);
\draw[fill=black!20!white] (v0x2.center) circle (0.2);
\draw[fill=white] (v0x3.center) circle (0.2);
\draw[fill=black!20!white] (v0x4.center) circle (0.2);
\draw[fill=white] (v0x5.center) circle (0.2);
\draw[fill=black!20!white] (v0x6.center) circle (0.2);
\draw[fill=white] (v1x0.center) circle (0.2);
\draw[fill=black!20!white] (v1x1.center) circle (0.2);
\draw[fill=white] (v1x2.center) circle (0.2);
\draw[fill=black!20!white] (v1x3.center) circle (0.2);
\draw[fill=white] (v1x4.center) circle (0.2);
\draw[fill=black!20!white] (v1x5.center) circle (0.2);
\draw[fill=white] (v1x6.center) circle (0.2);
\end{tikzpicture}}
\\

$\Toric(\affE_{6},(2+1),n)$ (\bg{toric-E6-21})
&
$\Toric(\affE_{6},(3),n)$ (\bg{toric-E6-3})
\\

for $n=2$
&
for $n=2$
\\

\end{tabular}

\\

\scalebox{0.3}{
\begin{tikzpicture}
\coordinate (v0x0) at (0.00,0.00);
\coordinate (v0x1) at (-0.90,8.88);
\coordinate (v0x2) at (-0.90,6.48);
\coordinate (v0x3) at (-0.90,4.08);
\coordinate (v0x4) at (0.00,2.40);
\coordinate (v0x5) at (0.90,5.52);
\coordinate (v0x6) at (0.90,7.92);
\coordinate (v0x7) at (0.90,10.32);
\coordinate (v1x0) at (4.20,0.00);
\coordinate (v1x1) at (3.30,8.88);
\coordinate (v1x2) at (3.30,6.48);
\coordinate (v1x3) at (3.30,4.08);
\coordinate (v1x4) at (4.20,2.40);
\coordinate (v1x5) at (5.10,5.52);
\coordinate (v1x6) at (5.10,7.92);
\coordinate (v1x7) at (5.10,10.32);
\coordinate (v2x0) at (8.40,0.00);
\coordinate (v2x1) at (7.50,8.88);
\coordinate (v2x2) at (7.50,6.48);
\coordinate (v2x3) at (7.50,4.08);
\coordinate (v2x4) at (8.40,2.40);
\coordinate (v2x5) at (9.30,5.52);
\coordinate (v2x6) at (9.30,7.92);
\coordinate (v2x7) at (9.30,10.32);
\coordinate (v3x0) at (12.60,0.00);
\coordinate (v3x1) at (11.70,8.88);
\coordinate (v3x2) at (11.70,6.48);
\coordinate (v3x3) at (11.70,4.08);
\coordinate (v3x4) at (12.60,2.40);
\coordinate (v3x5) at (13.50,5.52);
\coordinate (v3x6) at (13.50,7.92);
\coordinate (v3x7) at (13.50,10.32);
\draw[color=red,line width=0.75mm] (v0x2) to[] (v0x1);
\draw[color=red,line width=0.75mm] (v0x3) to[] (v0x2);
\draw[color=red,line width=0.75mm] (v0x4) to[] (v0x0);
\draw[color=red,line width=0.75mm] (v0x4) to[] (v0x3);
\draw[color=red,line width=0.75mm] (v0x5) to[] (v0x4);
\draw[color=red,line width=0.75mm] (v0x6) to[] (v0x5);
\draw[color=red,line width=0.75mm] (v0x7) to[] (v0x6);
\draw[color=blue,line width=0.75mm] (v1x0) to[] (v0x0);
\draw[color=blue,line width=0.75mm] (v1x1) to[] (v0x1);
\draw[color=blue,line width=0.75mm] (v1x2) to[] (v0x2);
\draw[color=red,line width=0.75mm] (v1x2) to[] (v1x1);
\draw[color=blue,line width=0.75mm] (v1x3) to[] (v0x3);
\draw[color=red,line width=0.75mm] (v1x3) to[] (v1x2);
\draw[color=blue,line width=0.75mm] (v1x4) to[] (v0x4);
\draw[color=red,line width=0.75mm] (v1x4) to[] (v1x0);
\draw[color=red,line width=0.75mm] (v1x4) to[] (v1x3);
\draw[color=blue,line width=0.75mm] (v1x5) to[] (v0x5);
\draw[color=red,line width=0.75mm] (v1x5) to[] (v1x4);
\draw[color=blue,line width=0.75mm] (v1x6) to[] (v0x6);
\draw[color=red,line width=0.75mm] (v1x6) to[] (v1x5);
\draw[color=blue,line width=0.75mm] (v1x7) to[] (v0x7);
\draw[color=red,line width=0.75mm] (v1x7) to[] (v1x6);
\draw[color=blue,line width=0.75mm] (v2x0) to[] (v1x0);
\draw[color=blue,line width=0.75mm] (v2x1) to[] (v1x7);
\draw[color=blue,line width=0.75mm] (v2x2) to[] (v1x6);
\draw[color=red,line width=0.75mm] (v2x2) to[] (v2x1);
\draw[color=blue,line width=0.75mm] (v2x3) to[] (v1x5);
\draw[color=red,line width=0.75mm] (v2x3) to[] (v2x2);
\draw[color=blue,line width=0.75mm] (v2x4) to[] (v1x4);
\draw[color=red,line width=0.75mm] (v2x4) to[] (v2x0);
\draw[color=red,line width=0.75mm] (v2x4) to[] (v2x3);
\draw[color=blue,line width=0.75mm] (v2x5) to[] (v1x3);
\draw[color=red,line width=0.75mm] (v2x5) to[] (v2x4);
\draw[color=blue,line width=0.75mm] (v2x6) to[] (v1x2);
\draw[color=red,line width=0.75mm] (v2x6) to[] (v2x5);
\draw[color=blue,line width=0.75mm] (v2x7) to[] (v1x1);
\draw[color=red,line width=0.75mm] (v2x7) to[] (v2x6);
\draw[color=blue,line width=0.75mm] (v3x0) to[] (v2x0);
\draw[color=blue,line width=0.75mm] (v3x1) to[] (v2x1);
\draw[color=blue,line width=0.75mm] (v3x2) to[] (v2x2);
\draw[color=red,line width=0.75mm] (v3x2) to[] (v3x1);
\draw[color=blue,line width=0.75mm] (v3x3) to[] (v2x3);
\draw[color=red,line width=0.75mm] (v3x3) to[] (v3x2);
\draw[color=blue,line width=0.75mm] (v3x4) to[] (v2x4);
\draw[color=red,line width=0.75mm] (v3x4) to[] (v3x0);
\draw[color=red,line width=0.75mm] (v3x4) to[] (v3x3);
\draw[color=blue,line width=0.75mm] (v3x5) to[] (v2x5);
\draw[color=red,line width=0.75mm] (v3x5) to[] (v3x4);
\draw[color=blue,line width=0.75mm] (v3x6) to[] (v2x6);
\draw[color=red,line width=0.75mm] (v3x6) to[] (v3x5);
\draw[color=blue,line width=0.75mm] (v3x7) to[] (v2x7);
\draw[color=red,line width=0.75mm] (v3x7) to[] (v3x6);
\draw[color=blue,line width=0.75mm] (0.90,10.32) -- (-1.80,10.32);
\draw[color=blue,line width=0.75mm] (-1.80,10.32) to[out=180,in=180, looseness=0.5] (-1.80,10.92);
\draw[color=blue,line width=0.75mm] (-1.80,10.92) -- (14.40,10.92);
\draw[color=blue,line width=0.75mm] (14.40,10.92) to[out=0,in=0, looseness=0.5] (14.40,10.32);
\draw[color=blue,line width=0.75mm] (14.40,10.32) -- (13.50,10.32);
\draw[color=blue,line width=0.75mm] (-0.90,8.88) -- (-1.80,8.88);
\draw[color=blue,line width=0.75mm] (-1.80,8.88) to[out=180,in=180, looseness=0.5] (-1.80,11.22);
\draw[color=blue,line width=0.75mm] (-1.80,11.22) -- (14.40,11.22);
\draw[color=blue,line width=0.75mm] (14.40,11.22) to[out=0,in=0, looseness=0.5] (14.40,8.88);
\draw[color=blue,line width=0.75mm] (14.40,8.88) -- (11.70,8.88);
\draw[color=blue,line width=0.75mm] (0.90,7.92) -- (-1.80,7.92);
\draw[color=blue,line width=0.75mm] (-1.80,7.92) to[out=180,in=180, looseness=0.5] (-1.80,11.52);
\draw[color=blue,line width=0.75mm] (-1.80,11.52) -- (14.40,11.52);
\draw[color=blue,line width=0.75mm] (14.40,11.52) to[out=0,in=0, looseness=0.5] (14.40,7.92);
\draw[color=blue,line width=0.75mm] (14.40,7.92) -- (13.50,7.92);
\draw[color=blue,line width=0.75mm] (-0.90,6.48) -- (-1.80,6.48);
\draw[color=blue,line width=0.75mm] (-1.80,6.48) to[out=180,in=180, looseness=0.5] (-1.80,11.82);
\draw[color=blue,line width=0.75mm] (-1.80,11.82) -- (14.40,11.82);
\draw[color=blue,line width=0.75mm] (14.40,11.82) to[out=0,in=0, looseness=0.5] (14.40,6.48);
\draw[color=blue,line width=0.75mm] (14.40,6.48) -- (11.70,6.48);
\draw[color=blue,line width=0.75mm] (0.90,5.52) -- (-1.80,5.52);
\draw[color=blue,line width=0.75mm] (-1.80,5.52) to[out=180,in=180, looseness=0.5] (-1.80,12.12);
\draw[color=blue,line width=0.75mm] (-1.80,12.12) -- (14.40,12.12);
\draw[color=blue,line width=0.75mm] (14.40,12.12) to[out=0,in=0, looseness=0.5] (14.40,5.52);
\draw[color=blue,line width=0.75mm] (14.40,5.52) -- (13.50,5.52);
\draw[color=blue,line width=0.75mm] (-0.90,4.08) -- (-1.80,4.08);
\draw[color=blue,line width=0.75mm] (-1.80,4.08) to[out=180,in=180, looseness=0.5] (-1.80,12.42);
\draw[color=blue,line width=0.75mm] (-1.80,12.42) -- (14.40,12.42);
\draw[color=blue,line width=0.75mm] (14.40,12.42) to[out=0,in=0, looseness=0.5] (14.40,4.08);
\draw[color=blue,line width=0.75mm] (14.40,4.08) -- (11.70,4.08);
\draw[color=blue,line width=0.75mm] (0.00,2.40) -- (-1.80,2.40);
\draw[color=blue,line width=0.75mm] (-1.80,2.40) to[out=180,in=180, looseness=0.5] (-1.80,12.72);
\draw[color=blue,line width=0.75mm] (-1.80,12.72) -- (14.40,12.72);
\draw[color=blue,line width=0.75mm] (14.40,12.72) to[out=0,in=0, looseness=0.5] (14.40,2.40);
\draw[color=blue,line width=0.75mm] (14.40,2.40) -- (12.60,2.40);
\draw[color=blue,line width=0.75mm] (0.00,0.00) -- (-1.80,0.00);
\draw[color=blue,line width=0.75mm] (-1.80,0.00) to[out=180,in=180, looseness=0.5] (-1.80,13.02);
\draw[color=blue,line width=0.75mm] (-1.80,13.02) -- (14.40,13.02);
\draw[color=blue,line width=0.75mm] (14.40,13.02) to[out=0,in=0, looseness=0.5] (14.40,0.00);
\draw[color=blue,line width=0.75mm] (14.40,0.00) -- (12.60,0.00);
\draw[fill=white] (v0x0.center) circle (0.2);
\draw[fill=white] (v0x1.center) circle (0.2);
\draw[fill=black!20!white] (v0x2.center) circle (0.2);
\draw[fill=white] (v0x3.center) circle (0.2);
\draw[fill=black!20!white] (v0x4.center) circle (0.2);
\draw[fill=white] (v0x5.center) circle (0.2);
\draw[fill=black!20!white] (v0x6.center) circle (0.2);
\draw[fill=white] (v0x7.center) circle (0.2);
\draw[fill=black!20!white] (v1x0.center) circle (0.2);
\draw[fill=black!20!white] (v1x1.center) circle (0.2);
\draw[fill=white] (v1x2.center) circle (0.2);
\draw[fill=black!20!white] (v1x3.center) circle (0.2);
\draw[fill=white] (v1x4.center) circle (0.2);
\draw[fill=black!20!white] (v1x5.center) circle (0.2);
\draw[fill=white] (v1x6.center) circle (0.2);
\draw[fill=black!20!white] (v1x7.center) circle (0.2);
\draw[fill=white] (v2x0.center) circle (0.2);
\draw[fill=white] (v2x1.center) circle (0.2);
\draw[fill=black!20!white] (v2x2.center) circle (0.2);
\draw[fill=white] (v2x3.center) circle (0.2);
\draw[fill=black!20!white] (v2x4.center) circle (0.2);
\draw[fill=white] (v2x5.center) circle (0.2);
\draw[fill=black!20!white] (v2x6.center) circle (0.2);
\draw[fill=white] (v2x7.center) circle (0.2);
\draw[fill=black!20!white] (v3x0.center) circle (0.2);
\draw[fill=black!20!white] (v3x1.center) circle (0.2);
\draw[fill=white] (v3x2.center) circle (0.2);
\draw[fill=black!20!white] (v3x3.center) circle (0.2);
\draw[fill=white] (v3x4.center) circle (0.2);
\draw[fill=black!20!white] (v3x5.center) circle (0.2);
\draw[fill=white] (v3x6.center) circle (0.2);
\draw[fill=black!20!white] (v3x7.center) circle (0.2);
\end{tikzpicture}}
\\

$\Toric(\affE_{7},\theta,n)$ (\bg{toric-E7})
\\

for $n=4$
\\\hline

\end{tabular}
}
\caption{\label{fig:toric-E} Toric bigraphs of type $\affE$.}
\end{figure}
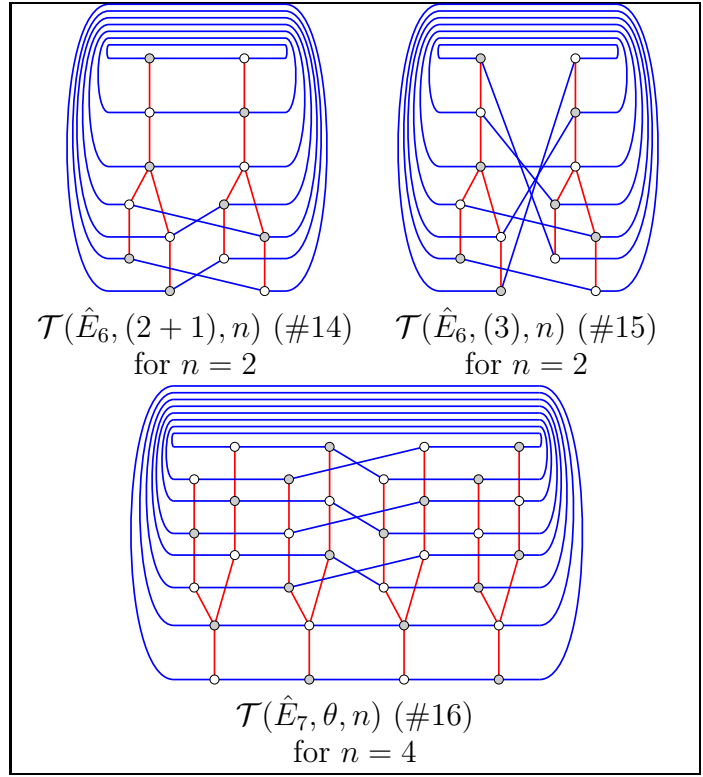

In this case $\Aut(\affL)=\Sfr_3$ so the non-identity weak conjugacy classes correspond precisely to partitions $(2+1)$ and $(3)$, see Figure~\ref{fig:toric-E} (top).

\subsubsection{The case $\affL=\affE_{7}$.}\label{subsub:toric_affE7}
The only non-trivial automorphism $\theta$ of $\affL$ has order $2$ and thus the only \affaff toric $ADE$ bigraphs of the form $\Toric(\affE_7,\eta,n)$ have $n$ even and coincide with $G=\Toric(\affE_7,\theta,n)$, see Figure~\ref{fig:toric-E} (bottom).

\subsubsection{The case $\affL=\affE_{8}$.}\label{subsub:toric_affE8}
In this case $\Aut(\affL)=\{\id\}$ so there are no non-identity conjugacy classes. 

\def\descrbar{\overline{\descr}}
\def\Path{\Pcal}

\subsection{Path bigraphs}\label{sect:path}
In this section, we would like to give a list of bigraphs $G$ with $S(G)$ being a path with two double arrows at the ends, that is, 
\[S(G)\in\{D_{\ell+1}^\parr2,C_\ell^\parr1,A_{2\ell}^\parr2\}.\]

Let us revisit the classification of \affinite double bindings with $\scf=(2,1)$ classified in Figure~\ref{figure:double_bindings_scf_2}. Consider such a double binding $G$ with red components $X$ of type $\affL$ and $Y$ of type $\affL'$ so that $\descr(G)$ equals $\affL\DRA\affL'$. Thus every blue component of $G$ has type $A_3$, in particular, every vertex $v$ of $X$ either has blue degree $2$ (in which case we set $v':=v$) or there exists a unique other vertex $v'\in X$ that belongs to the same blue connected component of $v$. One easily checks that this construction defines a color-preserving involution $\alpha\in\Aut(\affL)$ via $\alpha(v)=v'$. Similarly, we define a color-preserving involution $\beta\in\Aut(\affL')$. 

One easily checks that for every affine $ADE$ Dynkin diagram $\affL$ and every color-preserving involution $\alpha\in\Aut(\affL)$ (except for the identity in types $\affD$ and $\affE$), the pair $(\affL,\alpha)$ appears in Figure~\ref{figure:double_bindings_scf_2} exactly once in this way. Here we consider $\alpha$ up to conjugation. In particular, let us rewrite for each double binding in Figure~\ref{figure:double_bindings_scf_2} the corresponding pairs $(\affL,\alpha)$ that it involves (we use the description of conjugacy classes from the previous section):
\begin{equation}\label{eq:involutions}
\begin{split}
(\affD_{m+2},\sigma\tau\sigma\tau)\DRA (\affA_{2m-1},\eta^\parr1);&\quad
(\affD_{2m+2},\tau)\DRA (\affD_{m+2},\sigma);\\
(\affA_{4m-1},\exp(\pi i))\DRA (\affA_{2m-1},\id);&\quad (\affE_7,\theta)\DRA(\affE_6,(2+1)).
\end{split}
\end{equation}

\begin{definition}
	Let $\affL$ be an affine $ADE$ Dynkin diagram and consider two color-preserving non-identity involutions $\alpha,\beta\in\Aut(\affL)$. Given a positive integer $n\geq 2$, The \emph{path bigraph} $\Path(\affL,\alpha,\beta,n)$ is an \affaff $ADE$ bigraph $G$ obtained from the tensor product $\affL\otimes A_{n-1}$ by attaching on the left a double binding from~\eqref{eq:involutions} involving $(\affL,\alpha)$ and attaching on the right a double binding from~\eqref{eq:involutions} involving $(\affL,\beta)$. 
\end{definition}
See Figures~\ref{fig:path-A-refl-refl}--\ref{fig:path-E} for examples. It is clear that any path bigraph is always an \affaff $ADE$ bigraph since all the blue components are of types $\affA$ or $\affD$. Just as in the previous section, we give a simple criterion for when two path bigraphs are isomorphic.

\begin{proposition}\label{prop:conjugacy_path}
	Two path bigraphs $G=\Path(\affL,\alpha,\beta,n)$ and $G'=\Path(\affL',\alpha',\beta',n')$ are isomorphic if and only if $\affL=\affL'$, $n=n'$, and there exists an element $g\in\Aut(\affL)$ such that at least one of the following holds:
	\begin{itemize}
		\item $g\alpha g^{-1}=\alpha'$ and $g\beta g^{-1}=\beta'$, or
		\item $g\alpha g^{-1}=\beta'$ and $g\beta g^{-1}=\alpha'$.
	\end{itemize}
\end{proposition}
\begin{proof}
	Let $C_1,\dots,C_{n+1}$ be the red components of $G$ and $C_1',\dots,C_{n+1}'$ be the red components of $G'$. Since the red component graph of $G$ is a path, an isomorphism $\phi:\Vert(G)\to\Vert(G')$ either flips it or preserves it. In the former case, $\phi$ sends $C_i$ to $C_i'$ and in the former case $\phi$ sends $C_i$ to $C_{n+2-i}'$ isomorphically in both cases. Suppose that $\phi$ sends $C_i$ to $C_i'$ for all $i\in[n+1]$. Consider a vertex $v\in C_2$. If $\alpha(v)=v$ then $v$ has blue degree $2$ in $G(C_1,C_2)$  so $\phi(v)$ must have blue degree $2$ in $G(C_1',C_2')$ and thus $\alpha'(\phi(v))=\phi(v)$. Similarly, if there is a blue path of length $2$ from $v$ to $u=\alpha(v)\in C_2$ through $C_1$ then there must be a blue path of length $2$ from $\phi(v)$ to $\phi(u)$ in $C_2'$ through $C_1'$. The conclusion is that $\phi\circ\alpha=\alpha'\circ\phi$. Similarly we get that $\phi\circ\beta=\beta'\circ\phi$ so we can just put $g$ to be the restriction of $\phi$ to $C_2$ or $C_m$ (they have to coincide). The case when $\phi$ sends $C_i$ to $C_{n+2-i}$ is completely analogous. This shows one direction of the proposition. The converse direction is shown in a way similar to the proof of Proposition~\ref{prop:conjugacy} and we leave it as an exercise for the reader.
\end{proof}

Thus isomorphism classes of path bigraphs correspond to (unordered) pairs of color-preserving involutions modulo simultaneous conjugation. Let us say that two pairs $(\alpha,\beta)$ and $(\alpha',\beta')$ related by these transformations are \emph{equivalent}. Note that we can always conjugate $\alpha$ to be a specific fixed representative of a conjugacy class from the previous section, and after that $\beta$ will be determined up to conjugation by an element from the \emph{centralizer} of $\alpha$, that is, by an element $g\in\Aut(\affL)$ (not necessarily color-preserving) that commutes with $\alpha$. We are ready to list all pairs of involutions from~\eqref{eq:involutions} that give non-isomorphic path bigraphs.

\begin{remark}
	We list $\descr(G)$ and $\descr(G^\opp)$ for path bigraphs in Section~\ref{sect:classif}. We give the following informal explanation on how to quickly compute $\descr(G^\opp)$ when $G=\Path(\affL,\alpha,\beta,n)$. Since $\alpha$ and $\beta$ are involutions, they define a matching on $\Vert(\affL)$. Superimposing these matchings yields several cycles and paths. Each path with $r$ vertices corresponds to a node labeled $\affD_{rn+2}$ in $\descr(G^\opp)$. Each cycle with $r$ vertices corresponds to a node labeled $\affA_{rn-1}$ in $\descr(G^\opp)$. The Dynkin diagram $S(G^\opp)$ is obtained from $\affL$ by \emph{folding} via the subgroup generated by $\alpha$ and $\beta$. 
\end{remark}

\subsubsection{The case $\affL=\affA_{2m-1}$.}\label{subsub:path_affA}

\begin{figure}
\makebox[1.0\textwidth]{

}
\caption{\label{fig:path-A-refl-refl} Path bigraphs of the form $\Path(\affA_{2m-1},\eta^\parr1,\eta_p,n)$ for $m=6$, $n=3$, and $p=0,1,2,3$. The case $p=1$ belongs to family~\bg{path-A-refl-refl-coprime}, the remaining cases $p=0,2,3$ belong to family~\bg{path-A-refl-refl}.}
\end{figure}

\begin{figure}
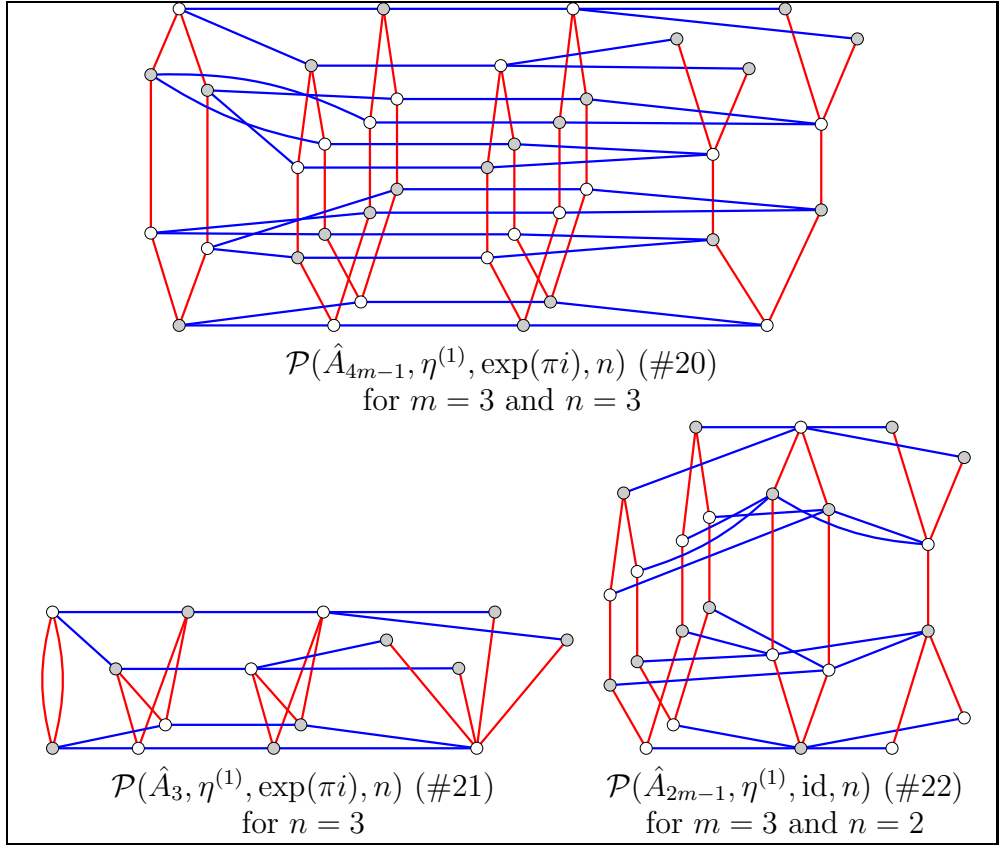

\makebox[1.0\textwidth]{
%

\\\hline

\end{tabular}
}
\caption{\label{fig:path-A-refl-rotn} The remaining (non-toric) path bigraphs of type $\affA$.}
\end{figure}

There are three color-preserving involutions in~\eqref{eq:involutions} for type $\affA_{2m-1}$: 
\begin{itemize}
	\item a reflection about a diagonal $\eta^\parr1$ for $m\geq2$,
	\item a $180^\circ$ rotation $\exp(\pi i)$ for even $m$, and 
	\item the identity $\id$ for $m\geq1$.
\end{itemize}
  Note that for the last two cases, the size of the conjugacy class is equal to $1$. For the case $\eta^\parr1$, the conjugacy class consists of $m$ reflections which we denote $\eta_0=\eta^\parr1,\eta_1,\dots,\eta_{m-1}$ in the cyclic order. The elements that commute with $\eta^\parr1$ are $\id,\eta^\parr1,\exp(\pi i),$ and a reflection $\eta_0^\perp$ that switches the two fixed points of $\eta^\parr1$. Conjugating $\eta_p$ for $p\in[m-1]$ by $\eta^\parr1$ or by $\eta_0^\perp$ produces the reflection $\eta_{m-p}$. Thus the list of all the non-equivalent pairs in this case is:
\begin{itemize}
	\item $(\eta^\parr1,\eta_p)$ for $1\leq p\leq m/2$;
	\item $(\eta^\parr1,\exp(\pi i))$ when $m$ is even;
	\item $(\eta^\parr1,\id)$;
	\item $(\exp(\pi i),\exp(\pi i))$ when $m$ is even;
	\item $(\exp(\pi i),\id)$ when $m$ is even;
	\item $(\id,\id)$.
\end{itemize}
All the cases are possible for $m\geq2$ except for the last case which is possible for $m\geq 1$. For the last three cases, $G^\opp$ is a toric bigraph and thus $G$ will not be listed as a path bigraph in the classification. The first case is shown in Figure~\ref{fig:path-A-refl-refl}, and the second and third cases are shown in Figure~\ref{fig:path-A-refl-rotn}.

\begin{figure}
\makebox[1.0\textwidth]{

}
\caption{\label{fig:path-D} Path bigraphs of type $\affD$.}
\end{figure}

\subsubsection{The case $\affL=\affD_{m+2}$, $m\geq 3$.}\label{subsub:path_affD}
The three involutions from~\eqref{eq:involutions} in this case are $\sigma$, $\sigma\tau\sigma\tau$, and $\tau$ when $m$ is even. We see that $\sigma$ has conjugacy class $\{\sigma,\sigma^\perp\}$ of size $2$, $\sigma\tau\sigma\tau$ has conjugacy class of size $1$, and $\tau$ has conjugacy class $\tau, \tau^\perp$ of size $2$. Conjugating $\tau$ by $\sigma$ gives $\tau^\perp$. Thus in this case the list of all non-equivalent pairs is as follows:
\begin{itemize}
	\item $(\sigma,\sigma)$;
	\item $(\sigma,\sigma^\perp)$;
	\item $(\sigma,\sigma\tau\sigma\tau)$;
	\item $(\sigma,\tau)$ when $m$ is even;
	\item $(\sigma\tau\sigma\tau,\sigma\tau\sigma\tau)$;
	\item $(\sigma\tau\sigma\tau,\tau)$ when $m$ is even;
	\item $(\tau,\tau)$ when $m$ is even.
	\item $(\tau,\tau^\perp)$ when $m$ is even.
\end{itemize}
The last case will not appear in the classification as the bigraph $\Path(\affD_{2m+2},\tau,\tau^\perp,n)$ is dual to $\Path(\affA_{2n-1},\eta^\parr1,\id,m)$. The rest of the cases are shown in Figure~\ref{fig:path-D}.

\subsubsection{The case $\affL=\affD_{4}$.}\label{subsub:path_affD4}
In this case, $\Aut(\affD_4)$ is the symmetric group $\Sfr_4$ so we write elements in the cycle notation. For example, the permutation $(12)$ is the transposition of $1$ and $2$. The involutions from~\eqref{eq:involutions} now split into two conjugacy classes: $(2+1+1)$ of size $6$ and $(2+2)$ of size $3$ with respective representatives $(12)$ and $(12)(34)$. The centralizer of each of the permutations is generated by the transpositions $(12)$ and $(34)$. Therefore we get the following list of non-equivalent pairs:

\begin{itemize}
	\item $((12),(12))$;
	\item $((12),(34))$;
	\item $((12),(13))$;
	\item $((12),(12)(34))$;
	\item $((12),(13)(24))$;
	\item $((12)(34),(12)(34))$;
	\item $((12)(34),(13)(24))$.
\end{itemize}

Just as for toric bigraphs, the case of $\affD_4$ naturally becomes a special case of $\affD_{m+2}$ when all the permutations involved belong to the set 
\[\{\sigma=(12),\sigma^\perp=(34), \sigma\tau\sigma\tau=(12)(34)\}.\]
See Figure~\ref{fig:path-D} for examples.

\subsubsection{The case $\affL=\affE_{6}$.}\label{subsub:path_affE6}

\begin{figure}
\makebox[1.0\textwidth]{
\begin{tabular}{|c|c|c|}\hline
\scalebox{0.3}{
\begin{tikzpicture}
\coordinate (v0x0) at (0.00,-0.72);
\coordinate (v0x1) at (-0.90,8.16);
\coordinate (v0x2) at (-0.90,5.76);
\coordinate (v0x3) at (-0.90,3.36);
\coordinate (v0x4) at (0.00,1.68);
\coordinate (v0x5) at (0.90,4.80);
\coordinate (v0x6) at (0.90,7.20);
\coordinate (v0x7) at (0.90,9.60);
\coordinate (v1x0) at (3.90,9.60);
\coordinate (v1x1) at (3.90,7.20);
\coordinate (v1x2) at (3.00,0.72);
\coordinate (v1x3) at (3.00,3.12);
\coordinate (v1x4) at (3.90,4.80);
\coordinate (v1x5) at (4.80,1.68);
\coordinate (v1x6) at (4.80,-0.72);
\coordinate (v2x0) at (7.80,9.60);
\coordinate (v2x1) at (7.80,7.20);
\coordinate (v2x2) at (6.90,0.72);
\coordinate (v2x3) at (6.90,3.12);
\coordinate (v2x4) at (7.80,4.80);
\coordinate (v2x5) at (8.70,1.68);
\coordinate (v2x6) at (8.70,-0.72);
\coordinate (v3x0) at (11.70,-0.72);
\coordinate (v3x1) at (10.80,8.16);
\coordinate (v3x2) at (10.80,5.76);
\coordinate (v3x3) at (10.80,3.36);
\coordinate (v3x4) at (11.70,1.68);
\coordinate (v3x5) at (12.60,4.80);
\coordinate (v3x6) at (12.60,7.20);
\coordinate (v3x7) at (12.60,9.60);
\draw[color=red,line width=0.75mm] (v0x2) to[] (v0x1);
\draw[color=red,line width=0.75mm] (v0x3) to[] (v0x2);
\draw[color=red,line width=0.75mm] (v0x4) to[] (v0x0);
\draw[color=red,line width=0.75mm] (v0x4) to[] (v0x3);
\draw[color=red,line width=0.75mm] (v0x5) to[] (v0x4);
\draw[color=red,line width=0.75mm] (v0x6) to[] (v0x5);
\draw[color=red,line width=0.75mm] (v0x7) to[] (v0x6);
\draw[color=blue,line width=0.75mm] (v1x0) to[] (v0x1);
\draw[color=blue,line width=0.75mm] (v1x0) to[] (v0x7);
\draw[color=blue,line width=0.75mm] (v1x1) to[] (v0x2);
\draw[color=blue,line width=0.75mm] (v1x1) to[] (v0x6);
\draw[color=red,line width=0.75mm] (v1x1) to[] (v1x0);
\draw[color=blue,line width=0.75mm] (v1x2) to[] (v0x0);
\draw[color=blue,line width=0.75mm] (v1x3) to[] (v0x4);
\draw[color=red,line width=0.75mm] (v1x3) to[] (v1x2);
\draw[color=blue,line width=0.75mm] (v1x4) to[] (v0x3);
\draw[color=blue,line width=0.75mm] (v1x4) to[] (v0x5);
\draw[color=red,line width=0.75mm] (v1x4) to[] (v1x1);
\draw[color=red,line width=0.75mm] (v1x4) to[] (v1x3);
\draw[color=blue,line width=0.75mm] (v1x5) to[] (v0x4);
\draw[color=red,line width=0.75mm] (v1x5) to[] (v1x4);
\draw[color=blue,line width=0.75mm] (v1x6) to[] (v0x0);
\draw[color=red,line width=0.75mm] (v1x6) to[] (v1x5);
\draw[color=blue,line width=0.75mm] (v2x0) to[] (v1x0);
\draw[color=blue,line width=0.75mm] (v2x1) to[] (v1x1);
\draw[color=red,line width=0.75mm] (v2x1) to[] (v2x0);
\draw[color=blue,line width=0.75mm] (v2x2) to[] (v1x2);
\draw[color=blue,line width=0.75mm] (v2x3) to[] (v1x3);
\draw[color=red,line width=0.75mm] (v2x3) to[] (v2x2);
\draw[color=blue,line width=0.75mm] (v2x4) to[] (v1x4);
\draw[color=red,line width=0.75mm] (v2x4) to[] (v2x1);
\draw[color=red,line width=0.75mm] (v2x4) to[] (v2x3);
\draw[color=blue,line width=0.75mm] (v2x5) to[] (v1x5);
\draw[color=red,line width=0.75mm] (v2x5) to[] (v2x4);
\draw[color=blue,line width=0.75mm] (v2x6) to[] (v1x6);
\draw[color=red,line width=0.75mm] (v2x6) to[] (v2x5);
\draw[color=blue,line width=0.75mm] (v3x0) to[] (v2x2);
\draw[color=blue,line width=0.75mm] (v3x0) to[] (v2x6);
\draw[color=blue,line width=0.75mm] (v3x1) to[] (v2x0);
\draw[color=blue,line width=0.75mm] (v3x2) to[] (v2x1);
\draw[color=red,line width=0.75mm] (v3x2) to[] (v3x1);
\draw[color=blue,line width=0.75mm] (v3x3) to[] (v2x4);
\draw[color=red,line width=0.75mm] (v3x3) to[] (v3x2);
\draw[color=blue,line width=0.75mm] (v3x4) to[] (v2x3);
\draw[color=blue,line width=0.75mm] (v3x4) to[] (v2x5);
\draw[color=red,line width=0.75mm] (v3x4) to[] (v3x0);
\draw[color=red,line width=0.75mm] (v3x4) to[] (v3x3);
\draw[color=blue,line width=0.75mm] (v3x5) to[] (v2x4);
\draw[color=red,line width=0.75mm] (v3x5) to[] (v3x4);
\draw[color=blue,line width=0.75mm] (v3x6) to[] (v2x1);
\draw[color=red,line width=0.75mm] (v3x6) to[] (v3x5);
\draw[color=blue,line width=0.75mm] (v3x7) to[] (v2x0);
\draw[color=red,line width=0.75mm] (v3x7) to[] (v3x6);
\draw[fill=white] (v0x0.center) circle (0.2);
\draw[fill=white] (v0x1.center) circle (0.2);
\draw[fill=black!20!white] (v0x2.center) circle (0.2);
\draw[fill=white] (v0x3.center) circle (0.2);
\draw[fill=black!20!white] (v0x4.center) circle (0.2);
\draw[fill=white] (v0x5.center) circle (0.2);
\draw[fill=black!20!white] (v0x6.center) circle (0.2);
\draw[fill=white] (v0x7.center) circle (0.2);
\draw[fill=black!20!white] (v1x0.center) circle (0.2);
\draw[fill=white] (v1x1.center) circle (0.2);
\draw[fill=black!20!white] (v1x2.center) circle (0.2);
\draw[fill=white] (v1x3.center) circle (0.2);
\draw[fill=black!20!white] (v1x4.center) circle (0.2);
\draw[fill=white] (v1x5.center) circle (0.2);
\draw[fill=black!20!white] (v1x6.center) circle (0.2);
\draw[fill=white] (v2x0.center) circle (0.2);
\draw[fill=black!20!white] (v2x1.center) circle (0.2);
\draw[fill=white] (v2x2.center) circle (0.2);
\draw[fill=black!20!white] (v2x3.center) circle (0.2);
\draw[fill=white] (v2x4.center) circle (0.2);
\draw[fill=black!20!white] (v2x5.center) circle (0.2);
\draw[fill=white] (v2x6.center) circle (0.2);
\draw[fill=black!20!white] (v3x0.center) circle (0.2);
\draw[fill=black!20!white] (v3x1.center) circle (0.2);
\draw[fill=white] (v3x2.center) circle (0.2);
\draw[fill=black!20!white] (v3x3.center) circle (0.2);
\draw[fill=white] (v3x4.center) circle (0.2);
\draw[fill=black!20!white] (v3x5.center) circle (0.2);
\draw[fill=white] (v3x6.center) circle (0.2);
\draw[fill=black!20!white] (v3x7.center) circle (0.2);
\end{tikzpicture}}
&
\scalebox{0.3}{
\begin{tikzpicture}
\coordinate (v0x0) at (0.00,-0.72);
\coordinate (v0x1) at (-0.90,8.16);
\coordinate (v0x2) at (-0.90,5.76);
\coordinate (v0x3) at (-0.90,3.36);
\coordinate (v0x4) at (0.00,1.68);
\coordinate (v0x5) at (0.90,4.80);
\coordinate (v0x6) at (0.90,7.20);
\coordinate (v0x7) at (0.90,9.60);
\coordinate (v1x0) at (3.90,9.60);
\coordinate (v1x1) at (3.90,7.20);
\coordinate (v1x2) at (3.00,0.72);
\coordinate (v1x3) at (3.00,3.12);
\coordinate (v1x4) at (3.90,4.80);
\coordinate (v1x5) at (4.80,1.68);
\coordinate (v1x6) at (4.80,-0.72);
\coordinate (v2x0) at (7.80,-0.72);
\coordinate (v2x1) at (7.80,1.68);
\coordinate (v2x2) at (6.90,8.16);
\coordinate (v2x3) at (6.90,5.76);
\coordinate (v2x4) at (7.80,4.08);
\coordinate (v2x5) at (8.70,7.20);
\coordinate (v2x6) at (8.70,9.60);
\coordinate (v3x0) at (11.70,9.60);
\coordinate (v3x1) at (10.80,0.72);
\coordinate (v3x2) at (10.80,3.12);
\coordinate (v3x3) at (10.80,5.52);
\coordinate (v3x4) at (11.70,7.20);
\coordinate (v3x5) at (12.60,4.08);
\coordinate (v3x6) at (12.60,1.68);
\coordinate (v3x7) at (12.60,-0.72);
\draw[color=red,line width=0.75mm] (v0x2) to[] (v0x1);
\draw[color=red,line width=0.75mm] (v0x3) to[] (v0x2);
\draw[color=red,line width=0.75mm] (v0x4) to[] (v0x0);
\draw[color=red,line width=0.75mm] (v0x4) to[] (v0x3);
\draw[color=red,line width=0.75mm] (v0x5) to[] (v0x4);
\draw[color=red,line width=0.75mm] (v0x6) to[] (v0x5);
\draw[color=red,line width=0.75mm] (v0x7) to[] (v0x6);
\draw[color=blue,line width=0.75mm] (v1x0) to[] (v0x1);
\draw[color=blue,line width=0.75mm] (v1x0) to[] (v0x7);
\draw[color=blue,line width=0.75mm] (v1x1) to[] (v0x2);
\draw[color=blue,line width=0.75mm] (v1x1) to[] (v0x6);
\draw[color=red,line width=0.75mm] (v1x1) to[] (v1x0);
\draw[color=blue,line width=0.75mm] (v1x2) to[] (v0x0);
\draw[color=blue,line width=0.75mm] (v1x3) to[] (v0x4);
\draw[color=red,line width=0.75mm] (v1x3) to[] (v1x2);
\draw[color=blue,line width=0.75mm] (v1x4) to[] (v0x3);
\draw[color=blue,line width=0.75mm] (v1x4) to[] (v0x5);
\draw[color=red,line width=0.75mm] (v1x4) to[] (v1x1);
\draw[color=red,line width=0.75mm] (v1x4) to[] (v1x3);
\draw[color=blue,line width=0.75mm] (v1x5) to[] (v0x4);
\draw[color=red,line width=0.75mm] (v1x5) to[] (v1x4);
\draw[color=blue,line width=0.75mm] (v1x6) to[] (v0x0);
\draw[color=red,line width=0.75mm] (v1x6) to[] (v1x5);
\draw[color=blue,line width=0.75mm] (v2x0) to[] (v1x2);
\draw[color=blue,line width=0.75mm] (v2x1) to[] (v1x3);
\draw[color=red,line width=0.75mm] (v2x1) to[] (v2x0);
\draw[color=blue,line width=0.75mm] (v2x2) to[] (v1x6);
\draw[color=blue,line width=0.75mm] (v2x3) to[] (v1x5);
\draw[color=red,line width=0.75mm] (v2x3) to[] (v2x2);
\draw[color=blue,line width=0.75mm] (v2x4) to[] (v1x4);
\draw[color=red,line width=0.75mm] (v2x4) to[] (v2x1);
\draw[color=red,line width=0.75mm] (v2x4) to[] (v2x3);
\draw[color=blue,line width=0.75mm] (v2x5) to[] (v1x1);
\draw[color=red,line width=0.75mm] (v2x5) to[] (v2x4);
\draw[color=blue,line width=0.75mm] (v2x6) to[] (v1x0);
\draw[color=red,line width=0.75mm] (v2x6) to[] (v2x5);
\draw[color=blue,line width=0.75mm] (v3x0) to[] (v2x2);
\draw[color=blue,line width=0.75mm] (v3x0) to[] (v2x6);
\draw[color=blue,line width=0.75mm] (v3x1) to[] (v2x0);
\draw[color=blue,line width=0.75mm] (v3x2) to[] (v2x1);
\draw[color=red,line width=0.75mm] (v3x2) to[] (v3x1);
\draw[color=blue,line width=0.75mm] (v3x3) to[] (v2x4);
\draw[color=red,line width=0.75mm] (v3x3) to[] (v3x2);
\draw[color=blue,line width=0.75mm] (v3x4) to[] (v2x3);
\draw[color=blue,line width=0.75mm] (v3x4) to[] (v2x5);
\draw[color=red,line width=0.75mm] (v3x4) to[] (v3x0);
\draw[color=red,line width=0.75mm] (v3x4) to[] (v3x3);
\draw[color=blue,line width=0.75mm] (v3x5) to[] (v2x4);
\draw[color=red,line width=0.75mm] (v3x5) to[] (v3x4);
\draw[color=blue,line width=0.75mm] (v3x6) to[] (v2x1);
\draw[color=red,line width=0.75mm] (v3x6) to[] (v3x5);
\draw[color=blue,line width=0.75mm] (v3x7) to[] (v2x0);
\draw[color=red,line width=0.75mm] (v3x7) to[] (v3x6);
\draw[fill=white] (v0x0.center) circle (0.2);
\draw[fill=white] (v0x1.center) circle (0.2);
\draw[fill=black!20!white] (v0x2.center) circle (0.2);
\draw[fill=white] (v0x3.center) circle (0.2);
\draw[fill=black!20!white] (v0x4.center) circle (0.2);
\draw[fill=white] (v0x5.center) circle (0.2);
\draw[fill=black!20!white] (v0x6.center) circle (0.2);
\draw[fill=white] (v0x7.center) circle (0.2);
\draw[fill=black!20!white] (v1x0.center) circle (0.2);
\draw[fill=white] (v1x1.center) circle (0.2);
\draw[fill=black!20!white] (v1x2.center) circle (0.2);
\draw[fill=white] (v1x3.center) circle (0.2);
\draw[fill=black!20!white] (v1x4.center) circle (0.2);
\draw[fill=white] (v1x5.center) circle (0.2);
\draw[fill=black!20!white] (v1x6.center) circle (0.2);
\draw[fill=white] (v2x0.center) circle (0.2);
\draw[fill=black!20!white] (v2x1.center) circle (0.2);
\draw[fill=white] (v2x2.center) circle (0.2);
\draw[fill=black!20!white] (v2x3.center) circle (0.2);
\draw[fill=white] (v2x4.center) circle (0.2);
\draw[fill=black!20!white] (v2x5.center) circle (0.2);
\draw[fill=white] (v2x6.center) circle (0.2);
\draw[fill=black!20!white] (v3x0.center) circle (0.2);
\draw[fill=black!20!white] (v3x1.center) circle (0.2);
\draw[fill=white] (v3x2.center) circle (0.2);
\draw[fill=black!20!white] (v3x3.center) circle (0.2);
\draw[fill=white] (v3x4.center) circle (0.2);
\draw[fill=black!20!white] (v3x5.center) circle (0.2);
\draw[fill=white] (v3x6.center) circle (0.2);
\draw[fill=black!20!white] (v3x7.center) circle (0.2);
\end{tikzpicture}}
&
\scalebox{0.3}{
\begin{tikzpicture}
\coordinate (v0x0) at (0.00,9.60);
\coordinate (v0x1) at (0.00,7.20);
\coordinate (v0x2) at (-0.90,0.72);
\coordinate (v0x3) at (-0.90,3.12);
\coordinate (v0x4) at (0.00,4.80);
\coordinate (v0x5) at (0.90,1.68);
\coordinate (v0x6) at (0.90,-0.72);
\coordinate (v1x0) at (3.90,-0.72);
\coordinate (v1x1) at (3.00,8.16);
\coordinate (v1x2) at (3.00,5.76);
\coordinate (v1x3) at (3.00,3.36);
\coordinate (v1x4) at (3.90,1.68);
\coordinate (v1x5) at (4.80,4.80);
\coordinate (v1x6) at (4.80,7.20);
\coordinate (v1x7) at (4.80,9.60);
\coordinate (v2x0) at (7.80,9.60);
\coordinate (v2x1) at (7.80,7.20);
\coordinate (v2x2) at (6.90,0.72);
\coordinate (v2x3) at (6.90,3.12);
\coordinate (v2x4) at (7.80,4.80);
\coordinate (v2x5) at (8.70,1.68);
\coordinate (v2x6) at (8.70,-0.72);
\draw[color=red,line width=0.75mm] (v0x1) to[] (v0x0);
\draw[color=red,line width=0.75mm] (v0x3) to[] (v0x2);
\draw[color=red,line width=0.75mm] (v0x4) to[] (v0x1);
\draw[color=red,line width=0.75mm] (v0x4) to[] (v0x3);
\draw[color=red,line width=0.75mm] (v0x5) to[] (v0x4);
\draw[color=red,line width=0.75mm] (v0x6) to[] (v0x5);
\draw[color=blue,line width=0.75mm] (v1x0) to[] (v0x2);
\draw[color=blue,line width=0.75mm] (v1x0) to[] (v0x6);
\draw[color=blue,line width=0.75mm] (v1x1) to[] (v0x0);
\draw[color=blue,line width=0.75mm] (v1x2) to[] (v0x1);
\draw[color=red,line width=0.75mm] (v1x2) to[] (v1x1);
\draw[color=blue,line width=0.75mm] (v1x3) to[] (v0x4);
\draw[color=red,line width=0.75mm] (v1x3) to[] (v1x2);
\draw[color=blue,line width=0.75mm] (v1x4) to[] (v0x3);
\draw[color=blue,line width=0.75mm] (v1x4) to[] (v0x5);
\draw[color=red,line width=0.75mm] (v1x4) to[] (v1x0);
\draw[color=red,line width=0.75mm] (v1x4) to[] (v1x3);
\draw[color=blue,line width=0.75mm] (v1x5) to[] (v0x4);
\draw[color=red,line width=0.75mm] (v1x5) to[] (v1x4);
\draw[color=blue,line width=0.75mm] (v1x6) to[] (v0x1);
\draw[color=red,line width=0.75mm] (v1x6) to[] (v1x5);
\draw[color=blue,line width=0.75mm] (v1x7) to[] (v0x0);
\draw[color=red,line width=0.75mm] (v1x7) to[] (v1x6);
\draw[color=blue,line width=0.75mm] (v2x0) to[] (v1x1);
\draw[color=blue,line width=0.75mm] (v2x0) to[] (v1x7);
\draw[color=blue,line width=0.75mm] (v2x1) to[] (v1x2);
\draw[color=blue,line width=0.75mm] (v2x1) to[] (v1x6);
\draw[color=red,line width=0.75mm] (v2x1) to[] (v2x0);
\draw[color=blue,line width=0.75mm] (v2x2) to[] (v1x0);
\draw[color=blue,line width=0.75mm] (v2x3) to[] (v1x4);
\draw[color=red,line width=0.75mm] (v2x3) to[] (v2x2);
\draw[color=blue,line width=0.75mm] (v2x4) to[] (v1x3);
\draw[color=blue,line width=0.75mm] (v2x4) to[] (v1x5);
\draw[color=red,line width=0.75mm] (v2x4) to[] (v2x1);
\draw[color=red,line width=0.75mm] (v2x4) to[] (v2x3);
\draw[color=blue,line width=0.75mm] (v2x5) to[] (v1x4);
\draw[color=red,line width=0.75mm] (v2x5) to[] (v2x4);
\draw[color=blue,line width=0.75mm] (v2x6) to[] (v1x0);
\draw[color=red,line width=0.75mm] (v2x6) to[] (v2x5);
\draw[fill=black!20!white] (v0x0.center) circle (0.2);
\draw[fill=white] (v0x1.center) circle (0.2);
\draw[fill=black!20!white] (v0x2.center) circle (0.2);
\draw[fill=white] (v0x3.center) circle (0.2);
\draw[fill=black!20!white] (v0x4.center) circle (0.2);
\draw[fill=white] (v0x5.center) circle (0.2);
\draw[fill=black!20!white] (v0x6.center) circle (0.2);
\draw[fill=white] (v1x0.center) circle (0.2);
\draw[fill=white] (v1x1.center) circle (0.2);
\draw[fill=black!20!white] (v1x2.center) circle (0.2);
\draw[fill=white] (v1x3.center) circle (0.2);
\draw[fill=black!20!white] (v1x4.center) circle (0.2);
\draw[fill=white] (v1x5.center) circle (0.2);
\draw[fill=black!20!white] (v1x6.center) circle (0.2);
\draw[fill=white] (v1x7.center) circle (0.2);
\draw[fill=black!20!white] (v2x0.center) circle (0.2);
\draw[fill=white] (v2x1.center) circle (0.2);
\draw[fill=black!20!white] (v2x2.center) circle (0.2);
\draw[fill=white] (v2x3.center) circle (0.2);
\draw[fill=black!20!white] (v2x4.center) circle (0.2);
\draw[fill=white] (v2x5.center) circle (0.2);
\draw[fill=black!20!white] (v2x6.center) circle (0.2);
\end{tikzpicture}}
\\

$\Path(\affE_{6},(12),(12),n)$  (\bg{path-E6-12-12})
&
$\Path(\affE_{6},(12),(13),n)$  (\bg{path-E6-12-13})
&
$\Path(\affE_{7},\theta,\theta,n)$  (\bg{path-E7-theta-theta})
\\

for $n=3$
&
for $n=3$
&
for $n=2$
\\\hline

\end{tabular}
}
\caption{\label{fig:path-E} Path bigraphs of type $\affE$.}
\end{figure}

In this case $\Aut(\affL)=\Sfr_3$ so there is just one conjugacy class $(2+1)$ from~\eqref{eq:involutions} with three permutations $(12)$, $(13)$, and $(23)$. The centralizer of $(12)$ is just $\{\id,(12)\}$ and conjugating $(13)$ by $(12)$ produces $(23)$. Thus we get only two non-equivalent pairs:
\begin{itemize}
	\item $((12),(12))$;
	\item $((12),(13))$.
\end{itemize}
See Figure~\ref{fig:path-E}.

\subsubsection{The case $\affL=\affE_{7}$.}\label{subsub:path_affE7}
The only involution in~\eqref{eq:involutions} is $\theta$ so the only pair that we can have here is $(\theta,\theta)$, see Figure~\ref{fig:path-E}.

\newcommand{\pstwist}[1]{\rtimes_{#1}}
\subsection{Pseudo twists of type $\affD_{m+2}\pstwist p\affD_{m+2}$}\label{sect:pstwist}

\begin{figure}
\makebox[1.0\textwidth]{
\begin{tabular}{|cccc|}\hline
\scalebox{0.5}{
\begin{tikzpicture}
\coordinate (v0x0) at (0.30,-4.32);
\coordinate (v0x1) at (-0.30,-4.20);
\coordinate (v0x10) at (-0.30,5.40);
\coordinate (v0x2) at (0.00,-3.00);
\coordinate (v0x3) at (0.00,-1.80);
\coordinate (v0x4) at (0.00,-0.60);
\coordinate (v0x5) at (0.00,0.60);
\coordinate (v0x6) at (0.00,1.80);
\coordinate (v0x7) at (0.00,3.00);
\coordinate (v0x8) at (0.00,4.20);
\coordinate (v0x9) at (0.30,5.52);
\coordinate (v1x0) at (2.70,-4.32);
\coordinate (v1x1) at (3.30,-4.20);
\coordinate (v1x10) at (3.30,5.40);
\coordinate (v1x2) at (3.00,-3.00);
\coordinate (v1x3) at (3.00,-1.80);
\coordinate (v1x4) at (3.00,-0.60);
\coordinate (v1x5) at (3.00,0.60);
\coordinate (v1x6) at (3.00,1.80);
\coordinate (v1x7) at (3.00,3.00);
\coordinate (v1x8) at (3.00,4.20);
\coordinate (v1x9) at (2.70,5.52);
\draw[color=red,line width=0.75mm] (v0x2) to[] (v0x0);
\draw[color=red,line width=0.75mm] (v0x2) to[] (v0x1);
\draw[color=red,line width=0.75mm] (v0x3) to[] (v0x2);
\draw[color=red,line width=0.75mm] (v0x4) to[] (v0x3);
\draw[color=red,line width=0.75mm] (v0x5) to[] (v0x4);
\draw[color=red,line width=0.75mm] (v0x6) to[] (v0x5);
\draw[color=red,line width=0.75mm] (v0x7) to[] (v0x6);
\draw[color=red,line width=0.75mm] (v0x8) to[] (v0x10);
\draw[color=red,line width=0.75mm] (v0x8) to[] (v0x7);
\draw[color=red,line width=0.75mm] (v0x9) to[] (v0x8);
\draw[color=red,line width=0.75mm] (v1x2) to[] (v1x0);
\draw[color=red,line width=0.75mm] (v1x2) to[] (v1x1);
\draw[color=red,line width=0.75mm] (v1x3) to[] (v1x2);
\draw[color=red,line width=0.75mm] (v1x4) to[] (v1x3);
\draw[color=red,line width=0.75mm] (v1x5) to[] (v1x4);
\draw[color=red,line width=0.75mm] (v1x6) to[] (v1x5);
\draw[color=red,line width=0.75mm] (v1x7) to[] (v1x6);
\draw[color=red,line width=0.75mm] (v1x8) to[] (v1x10);
\draw[color=red,line width=0.75mm] (v1x8) to[] (v1x7);
\draw[color=red,line width=0.75mm] (v1x9) to[] (v1x8);
\draw[color=blue,line width=0.75mm] (v0x0) -- (v1x2);
\draw[color=blue,line width=0.75mm] (v1x0) -- (v0x2);
\draw[color=blue,line width=0.75mm] (v0x1) -- (v1x2);
\draw[color=blue,line width=0.75mm] (v1x1) -- (v0x2);
\draw[color=blue,line width=0.75mm] (v0x2) -- (v1x3);
\draw[color=blue,line width=0.75mm] (v1x2) -- (v0x3);
\draw[color=blue,line width=0.75mm] (v0x3) -- (v1x4);
\draw[color=blue,line width=0.75mm] (v1x3) -- (v0x4);
\draw[color=blue,line width=0.75mm] (v0x4) -- (v1x5);
\draw[color=blue,line width=0.75mm] (v1x4) -- (v0x5);
\draw[color=blue,line width=0.75mm] (v0x5) -- (v1x6);
\draw[color=blue,line width=0.75mm] (v1x5) -- (v0x6);
\draw[color=blue,line width=0.75mm] (v0x6) -- (v1x7);
\draw[color=blue,line width=0.75mm] (v1x6) -- (v0x7);
\draw[color=blue,line width=0.75mm] (v0x7) -- (v1x8);
\draw[color=blue,line width=0.75mm] (v1x7) -- (v0x8);
\draw[color=blue,line width=0.75mm] (v0x8) -- (v1x9);
\draw[color=blue,line width=0.75mm] (v1x8) -- (v0x9);
\draw[color=blue,line width=0.75mm] (v0x8) -- (v1x10);
\draw[color=blue,line width=0.75mm] (v1x8) -- (v0x10);
\draw[fill=white] (v0x0.center) circle (0.2);
\draw[fill=white] (v0x1.center) circle (0.2);
\draw[fill=white] (v0x10.center) circle (0.2);
\draw[fill=black!20!white] (v0x2.center) circle (0.2);
\draw[fill=white] (v0x3.center) circle (0.2);
\draw[fill=black!20!white] (v0x4.center) circle (0.2);
\draw[fill=white] (v0x5.center) circle (0.2);
\draw[fill=black!20!white] (v0x6.center) circle (0.2);
\draw[fill=white] (v0x7.center) circle (0.2);
\draw[fill=black!20!white] (v0x8.center) circle (0.2);
\draw[fill=white] (v0x9.center) circle (0.2);
\draw[fill=black!20!white] (v1x0.center) circle (0.2);
\draw[fill=black!20!white] (v1x1.center) circle (0.2);
\draw[fill=black!20!white] (v1x10.center) circle (0.2);
\draw[fill=white] (v1x2.center) circle (0.2);
\draw[fill=black!20!white] (v1x3.center) circle (0.2);
\draw[fill=white] (v1x4.center) circle (0.2);
\draw[fill=black!20!white] (v1x5.center) circle (0.2);
\draw[fill=white] (v1x6.center) circle (0.2);
\draw[fill=black!20!white] (v1x7.center) circle (0.2);
\draw[fill=white] (v1x8.center) circle (0.2);
\draw[fill=black!20!white] (v1x9.center) circle (0.2);
\end{tikzpicture}}
&
\scalebox{0.5}{
\begin{tikzpicture}
\coordinate (v0x0) at (0.30,-4.32);
\coordinate (v0x1) at (-0.30,-4.20);
\coordinate (v0x10) at (-0.30,5.40);
\coordinate (v0x2) at (0.00,-3.00);
\coordinate (v0x3) at (0.00,-1.80);
\coordinate (v0x4) at (0.00,-0.60);
\coordinate (v0x5) at (0.00,0.60);
\coordinate (v0x6) at (0.00,1.80);
\coordinate (v0x7) at (0.00,3.00);
\coordinate (v0x8) at (0.00,4.20);
\coordinate (v0x9) at (0.30,5.52);
\coordinate (v1x0) at (2.70,-4.32);
\coordinate (v1x1) at (3.30,-4.20);
\coordinate (v1x10) at (3.30,5.40);
\coordinate (v1x2) at (3.00,-3.00);
\coordinate (v1x3) at (3.00,-1.80);
\coordinate (v1x4) at (3.00,-0.60);
\coordinate (v1x5) at (3.00,0.60);
\coordinate (v1x6) at (3.00,1.80);
\coordinate (v1x7) at (3.00,3.00);
\coordinate (v1x8) at (3.00,4.20);
\coordinate (v1x9) at (2.70,5.52);
\draw[color=red,line width=0.75mm] (v0x2) to[] (v0x0);
\draw[color=red,line width=0.75mm] (v0x2) to[] (v0x1);
\draw[color=red,line width=0.75mm] (v0x3) to[] (v0x2);
\draw[color=red,line width=0.75mm] (v0x4) to[] (v0x3);
\draw[color=red,line width=0.75mm] (v0x5) to[] (v0x4);
\draw[color=red,line width=0.75mm] (v0x6) to[] (v0x5);
\draw[color=red,line width=0.75mm] (v0x7) to[] (v0x6);
\draw[color=red,line width=0.75mm] (v0x8) to[] (v0x10);
\draw[color=red,line width=0.75mm] (v0x8) to[] (v0x7);
\draw[color=red,line width=0.75mm] (v0x9) to[] (v0x8);
\draw[color=red,line width=0.75mm] (v1x2) to[] (v1x0);
\draw[color=red,line width=0.75mm] (v1x2) to[] (v1x1);
\draw[color=red,line width=0.75mm] (v1x3) to[] (v1x2);
\draw[color=red,line width=0.75mm] (v1x4) to[] (v1x3);
\draw[color=red,line width=0.75mm] (v1x5) to[] (v1x4);
\draw[color=red,line width=0.75mm] (v1x6) to[] (v1x5);
\draw[color=red,line width=0.75mm] (v1x7) to[] (v1x6);
\draw[color=red,line width=0.75mm] (v1x8) to[] (v1x10);
\draw[color=red,line width=0.75mm] (v1x8) to[] (v1x7);
\draw[color=red,line width=0.75mm] (v1x9) to[] (v1x8);
\coordinate (up1) at (1.50,5.52);
\coordinate (down1) at (1.50,-4.32);
\draw[color=blue,line width=0.75mm] (v0x0) -- (v1x3);
\draw[color=blue,line width=0.75mm] (v1x0) -- (v0x3);
\draw[color=blue,line width=0.75mm] (v0x1) -- (v1x3);
\draw[color=blue,line width=0.75mm] (v1x1) -- (v0x3);
\draw[color=blue,line width=0.75mm] (v0x2) -- (v1x4);
\draw[color=blue,line width=0.75mm] (v1x2) -- (v0x4);
\draw[color=blue,line width=0.75mm] (v0x3) -- (v1x5);
\draw[color=blue,line width=0.75mm] (v1x3) -- (v0x5);
\draw[color=blue,line width=0.75mm] (v0x4) -- (v1x6);
\draw[color=blue,line width=0.75mm] (v1x4) -- (v0x6);
\draw[color=blue,line width=0.75mm] (v0x5) -- (v1x7);
\draw[color=blue,line width=0.75mm] (v1x5) -- (v0x7);
\draw[color=blue,line width=0.75mm] (v0x6) -- (v1x8);
\draw[color=blue,line width=0.75mm] (v1x6) -- (v0x8);
\draw[color=blue,line width=0.75mm] (v0x7) -- (v1x9);
\draw[color=blue,line width=0.75mm] (v1x7) -- (v0x9);
\draw[color=blue,line width=0.75mm] (v0x7) -- (v1x10);
\draw[color=blue,line width=0.75mm] (v1x7) -- (v0x10);
\draw[color=blue,line width=0.75mm, rounded corners=10] (v0x8)--(up1)--(v1x8);
\draw[color=blue,line width=0.75mm, rounded corners=10] (v0x2)--(down1)--(v1x2);
\draw[fill=white] (v0x0.center) circle (0.2);
\draw[fill=white] (v0x1.center) circle (0.2);
\draw[fill=white] (v0x10.center) circle (0.2);
\draw[fill=black!20!white] (v0x2.center) circle (0.2);
\draw[fill=white] (v0x3.center) circle (0.2);
\draw[fill=black!20!white] (v0x4.center) circle (0.2);
\draw[fill=white] (v0x5.center) circle (0.2);
\draw[fill=black!20!white] (v0x6.center) circle (0.2);
\draw[fill=white] (v0x7.center) circle (0.2);
\draw[fill=black!20!white] (v0x8.center) circle (0.2);
\draw[fill=white] (v0x9.center) circle (0.2);
\draw[fill=black!20!white] (v1x0.center) circle (0.2);
\draw[fill=black!20!white] (v1x1.center) circle (0.2);
\draw[fill=black!20!white] (v1x10.center) circle (0.2);
\draw[fill=white] (v1x2.center) circle (0.2);
\draw[fill=black!20!white] (v1x3.center) circle (0.2);
\draw[fill=white] (v1x4.center) circle (0.2);
\draw[fill=black!20!white] (v1x5.center) circle (0.2);
\draw[fill=white] (v1x6.center) circle (0.2);
\draw[fill=black!20!white] (v1x7.center) circle (0.2);
\draw[fill=white] (v1x8.center) circle (0.2);
\draw[fill=black!20!white] (v1x9.center) circle (0.2);
\end{tikzpicture}}
&
\scalebox{0.5}{
\begin{tikzpicture}
\coordinate (v0x0) at (0.30,-4.32);
\coordinate (v0x1) at (-0.30,-4.20);
\coordinate (v0x10) at (-0.30,5.40);
\coordinate (v0x2) at (0.00,-3.00);
\coordinate (v0x3) at (0.00,-1.80);
\coordinate (v0x4) at (0.00,-0.60);
\coordinate (v0x5) at (0.00,0.60);
\coordinate (v0x6) at (0.00,1.80);
\coordinate (v0x7) at (0.00,3.00);
\coordinate (v0x8) at (0.00,4.20);
\coordinate (v0x9) at (0.30,5.52);
\coordinate (v1x0) at (2.70,-4.32);
\coordinate (v1x1) at (3.30,-4.20);
\coordinate (v1x10) at (3.30,5.40);
\coordinate (v1x2) at (3.00,-3.00);
\coordinate (v1x3) at (3.00,-1.80);
\coordinate (v1x4) at (3.00,-0.60);
\coordinate (v1x5) at (3.00,0.60);
\coordinate (v1x6) at (3.00,1.80);
\coordinate (v1x7) at (3.00,3.00);
\coordinate (v1x8) at (3.00,4.20);
\coordinate (v1x9) at (2.70,5.52);
\draw[color=red,line width=0.75mm] (v0x2) to[] (v0x0);
\draw[color=red,line width=0.75mm] (v0x2) to[] (v0x1);
\draw[color=red,line width=0.75mm] (v0x3) to[] (v0x2);
\draw[color=red,line width=0.75mm] (v0x4) to[] (v0x3);
\draw[color=red,line width=0.75mm] (v0x5) to[] (v0x4);
\draw[color=red,line width=0.75mm] (v0x6) to[] (v0x5);
\draw[color=red,line width=0.75mm] (v0x7) to[] (v0x6);
\draw[color=red,line width=0.75mm] (v0x8) to[] (v0x10);
\draw[color=red,line width=0.75mm] (v0x8) to[] (v0x7);
\draw[color=red,line width=0.75mm] (v0x9) to[] (v0x8);
\draw[color=red,line width=0.75mm] (v1x2) to[] (v1x0);
\draw[color=red,line width=0.75mm] (v1x2) to[] (v1x1);
\draw[color=red,line width=0.75mm] (v1x3) to[] (v1x2);
\draw[color=red,line width=0.75mm] (v1x4) to[] (v1x3);
\draw[color=red,line width=0.75mm] (v1x5) to[] (v1x4);
\draw[color=red,line width=0.75mm] (v1x6) to[] (v1x5);
\draw[color=red,line width=0.75mm] (v1x7) to[] (v1x6);
\draw[color=red,line width=0.75mm] (v1x8) to[] (v1x10);
\draw[color=red,line width=0.75mm] (v1x8) to[] (v1x7);
\draw[color=red,line width=0.75mm] (v1x9) to[] (v1x8);
\coordinate (up1) at (1.00,5.52);
\coordinate (down1) at (1.00,-4.32);
\coordinate (up2) at (2.00,5.52);
\coordinate (down2) at (2.00,-4.32);
\draw[color=blue,line width=0.75mm] (v0x0) -- (v1x4);
\draw[color=blue,line width=0.75mm] (v1x0) -- (v0x4);
\draw[color=blue,line width=0.75mm] (v0x1) -- (v1x4);
\draw[color=blue,line width=0.75mm] (v1x1) -- (v0x4);
\draw[color=blue,line width=0.75mm] (v0x2) -- (v1x5);
\draw[color=blue,line width=0.75mm] (v1x2) -- (v0x5);
\draw[color=blue,line width=0.75mm] (v0x3) -- (v1x6);
\draw[color=blue,line width=0.75mm] (v1x3) -- (v0x6);
\draw[color=blue,line width=0.75mm] (v0x4) -- (v1x7);
\draw[color=blue,line width=0.75mm] (v1x4) -- (v0x7);
\draw[color=blue,line width=0.75mm] (v0x5) -- (v1x8);
\draw[color=blue,line width=0.75mm] (v1x5) -- (v0x8);
\draw[color=blue,line width=0.75mm] (v0x6) -- (v1x9);
\draw[color=blue,line width=0.75mm] (v1x6) -- (v0x9);
\draw[color=blue,line width=0.75mm] (v0x6) -- (v1x10);
\draw[color=blue,line width=0.75mm] (v1x6) -- (v0x10);
\draw[color=blue,line width=0.75mm, rounded corners=10] (v0x7)--(up2)--(v1x8);
\draw[color=blue,line width=0.75mm, rounded corners=10] (v0x8)--(up1)--(v1x7);
\draw[color=blue,line width=0.75mm, rounded corners=10] (v0x2)--(down1)--(v1x3);
\draw[color=blue,line width=0.75mm, rounded corners=10] (v0x3)--(down2)--(v1x2);
\draw[fill=white] (v0x0.center) circle (0.2);
\draw[fill=white] (v0x1.center) circle (0.2);
\draw[fill=white] (v0x10.center) circle (0.2);
\draw[fill=black!20!white] (v0x2.center) circle (0.2);
\draw[fill=white] (v0x3.center) circle (0.2);
\draw[fill=black!20!white] (v0x4.center) circle (0.2);
\draw[fill=white] (v0x5.center) circle (0.2);
\draw[fill=black!20!white] (v0x6.center) circle (0.2);
\draw[fill=white] (v0x7.center) circle (0.2);
\draw[fill=black!20!white] (v0x8.center) circle (0.2);
\draw[fill=white] (v0x9.center) circle (0.2);
\draw[fill=black!20!white] (v1x0.center) circle (0.2);
\draw[fill=black!20!white] (v1x1.center) circle (0.2);
\draw[fill=black!20!white] (v1x10.center) circle (0.2);
\draw[fill=white] (v1x2.center) circle (0.2);
\draw[fill=black!20!white] (v1x3.center) circle (0.2);
\draw[fill=white] (v1x4.center) circle (0.2);
\draw[fill=black!20!white] (v1x5.center) circle (0.2);
\draw[fill=white] (v1x6.center) circle (0.2);
\draw[fill=black!20!white] (v1x7.center) circle (0.2);
\draw[fill=white] (v1x8.center) circle (0.2);
\draw[fill=black!20!white] (v1x9.center) circle (0.2);
\end{tikzpicture}}
&
\scalebox{0.5}{
\begin{tikzpicture}
\coordinate (v0x0) at (0.30,-4.32);
\coordinate (v0x1) at (-0.30,-4.20);
\coordinate (v0x10) at (-0.30,5.40);
\coordinate (v0x2) at (0.00,-3.00);
\coordinate (v0x3) at (0.00,-1.80);
\coordinate (v0x4) at (0.00,-0.60);
\coordinate (v0x5) at (0.00,0.60);
\coordinate (v0x6) at (0.00,1.80);
\coordinate (v0x7) at (0.00,3.00);
\coordinate (v0x8) at (0.00,4.20);
\coordinate (v0x9) at (0.30,5.52);
\coordinate (v1x0) at (2.70,-4.32);
\coordinate (v1x1) at (3.30,-4.20);
\coordinate (v1x10) at (3.30,5.40);
\coordinate (v1x2) at (3.00,-3.00);
\coordinate (v1x3) at (3.00,-1.80);
\coordinate (v1x4) at (3.00,-0.60);
\coordinate (v1x5) at (3.00,0.60);
\coordinate (v1x6) at (3.00,1.80);
\coordinate (v1x7) at (3.00,3.00);
\coordinate (v1x8) at (3.00,4.20);
\coordinate (v1x9) at (2.70,5.52);
\draw[color=red,line width=0.75mm] (v0x2) to[] (v0x0);
\draw[color=red,line width=0.75mm] (v0x2) to[] (v0x1);
\draw[color=red,line width=0.75mm] (v0x3) to[] (v0x2);
\draw[color=red,line width=0.75mm] (v0x4) to[] (v0x3);
\draw[color=red,line width=0.75mm] (v0x5) to[] (v0x4);
\draw[color=red,line width=0.75mm] (v0x6) to[] (v0x5);
\draw[color=red,line width=0.75mm] (v0x7) to[] (v0x6);
\draw[color=red,line width=0.75mm] (v0x8) to[] (v0x10);
\draw[color=red,line width=0.75mm] (v0x8) to[] (v0x7);
\draw[color=red,line width=0.75mm] (v0x9) to[] (v0x8);
\draw[color=red,line width=0.75mm] (v1x2) to[] (v1x0);
\draw[color=red,line width=0.75mm] (v1x2) to[] (v1x1);
\draw[color=red,line width=0.75mm] (v1x3) to[] (v1x2);
\draw[color=red,line width=0.75mm] (v1x4) to[] (v1x3);
\draw[color=red,line width=0.75mm] (v1x5) to[] (v1x4);
\draw[color=red,line width=0.75mm] (v1x6) to[] (v1x5);
\draw[color=red,line width=0.75mm] (v1x7) to[] (v1x6);
\draw[color=red,line width=0.75mm] (v1x8) to[] (v1x10);
\draw[color=red,line width=0.75mm] (v1x8) to[] (v1x7);
\draw[color=red,line width=0.75mm] (v1x9) to[] (v1x8);
\coordinate (up1) at (0.75,5.52);
\coordinate (down1) at (0.75,-4.32);
\coordinate (up2) at (1.50,5.52);
\coordinate (down2) at (1.50,-4.32);
\coordinate (up3) at (2.25,5.52);
\coordinate (down3) at (2.25,-4.32);
\draw[color=blue,line width=0.75mm] (v0x0) -- (v1x5);
\draw[color=blue,line width=0.75mm] (v1x0) -- (v0x5);
\draw[color=blue,line width=0.75mm] (v0x1) -- (v1x5);
\draw[color=blue,line width=0.75mm] (v1x1) -- (v0x5);
\draw[color=blue,line width=0.75mm] (v0x2) -- (v1x6);
\draw[color=blue,line width=0.75mm] (v1x2) -- (v0x6);
\draw[color=blue,line width=0.75mm] (v0x3) -- (v1x7);
\draw[color=blue,line width=0.75mm] (v1x3) -- (v0x7);
\draw[color=blue,line width=0.75mm] (v0x4) -- (v1x8);
\draw[color=blue,line width=0.75mm] (v1x4) -- (v0x8);
\draw[color=blue,line width=0.75mm] (v0x5) -- (v1x9);
\draw[color=blue,line width=0.75mm] (v1x5) -- (v0x9);
\draw[color=blue,line width=0.75mm] (v0x5) -- (v1x10);
\draw[color=blue,line width=0.75mm] (v1x5) -- (v0x10);
\draw[color=blue,line width=0.75mm, rounded corners=10] (v0x6)--(up3)--(v1x8);
\draw[color=blue,line width=0.75mm, rounded corners=10] (v0x7)--(up2)--(v1x7);
\draw[color=blue,line width=0.75mm, rounded corners=10] (v0x8)--(up1)--(v1x6);
\draw[color=blue,line width=0.75mm, rounded corners=10] (v0x2)--(down1)--(v1x4);
\draw[color=blue,line width=0.75mm, rounded corners=10] (v0x3)--(down2)--(v1x3);
\draw[color=blue,line width=0.75mm, rounded corners=10] (v0x4)--(down3)--(v1x2);
\draw[fill=white] (v0x0.center) circle (0.2);
\draw[fill=white] (v0x1.center) circle (0.2);
\draw[fill=white] (v0x10.center) circle (0.2);
\draw[fill=black!20!white] (v0x2.center) circle (0.2);
\draw[fill=white] (v0x3.center) circle (0.2);
\draw[fill=black!20!white] (v0x4.center) circle (0.2);
\draw[fill=white] (v0x5.center) circle (0.2);
\draw[fill=black!20!white] (v0x6.center) circle (0.2);
\draw[fill=white] (v0x7.center) circle (0.2);
\draw[fill=black!20!white] (v0x8.center) circle (0.2);
\draw[fill=white] (v0x9.center) circle (0.2);
\draw[fill=black!20!white] (v1x0.center) circle (0.2);
\draw[fill=black!20!white] (v1x1.center) circle (0.2);
\draw[fill=black!20!white] (v1x10.center) circle (0.2);
\draw[fill=white] (v1x2.center) circle (0.2);
\draw[fill=black!20!white] (v1x3.center) circle (0.2);
\draw[fill=white] (v1x4.center) circle (0.2);
\draw[fill=black!20!white] (v1x5.center) circle (0.2);
\draw[fill=white] (v1x6.center) circle (0.2);
\draw[fill=black!20!white] (v1x7.center) circle (0.2);
\draw[fill=white] (v1x8.center) circle (0.2);
\draw[fill=black!20!white] (v1x9.center) circle (0.2);
\end{tikzpicture}}
\\

$p=1$
&
$p=2$
&
$p=3$
&
$p=4$
\\\hline

\end{tabular}
}
\caption{\label{fig:pstwist} Pseudo twists $\affD_{m+2}\pstwist{p}\affD_{m+2}$ for $m=8$ and $p=1,2,3,4$. The cases $p=1,3$ belong to family~\bg{pstwist} while the cases $p=2,4$ belong to family~\bg{path-A-refl-refl-coprime}.}
\end{figure}

Let $X$ and $Y$ be two red components of type $\affD_{m+2}$. Label the vertices of $X$ by 
\[u_0^+,u_0^-,u_1,u_2,\dots,u_{m-1}, u_{m}^+, u_{m}^-\]
so that the leaves $u_0^+$ and $u_0^-$ are connected to $u_1$ and the leaves $u_{m}^+$ and $u_{m}^-$ are connected to $u_{m-1}$. Let $p\in[m-1]$ be an integer. We denote by $X_0,X_1,X_2,\dots,X_{m+p}$ a sequence of subsets of $X$ defined as follows:
\begin{multline*}
X_0=\{u_0^+,u_0^-\}, X_1=\{u_1\},\dots,X_{m-1}=\{u_{m-1}\}, X_{m}=\{u_{m}^+,u_{m}^-\},\\
X_{m+1}=X_{m-1}, X_{m+2}=X_{m-2},\dots,X_{m+p}=X_{m-p}.	
\end{multline*}
We similarly label the vertices of $Y$ by 
\[v_0^+,v_0^-,v_1,v_2,\dots,v_{m-1}, v_{m}^+, v_{m}^-\]
and introduce the subsets $Y_0,Y_1,Y_2,\dots,Y_{m+p}$. 
\begin{definition}
	The bigraph $\affD_{m+2}\pstwist p\affD_{m+2}=(\Gamma,\Delta)$ is a double binding with two red components $X$ and $Y$ as above and blue edges as follows: for every $i\in\{0,1,\dots,m\}$, connect every vertex in $X_i$ with every vertex in $Y_{i+p}$ and every vertex in $Y_i$ with every vertex in $X_{i+p}$ by an edge of $\Delta$.
\end{definition}

An example is given in Figure~\ref{fig:pstwist}. It is easy to see that for every $p\in[m-1]$, $G=\affD_{m+2}\pstwist p\affD_{m+2}$ is an \affaff $ADE$ bigraph with $S(G)=A_1^\parr1$ and 
\[\nodeZ{$\descr(G)=$}\ZAAone{\affD_{m+2}}{\affD_{m+2}}\nodeZ{.}\]

Note that the case $p=1$ recovers the twist $\affD_{m+2}\times \affD_{m+2}$.

\def\Five{Five }
\def\five{five }

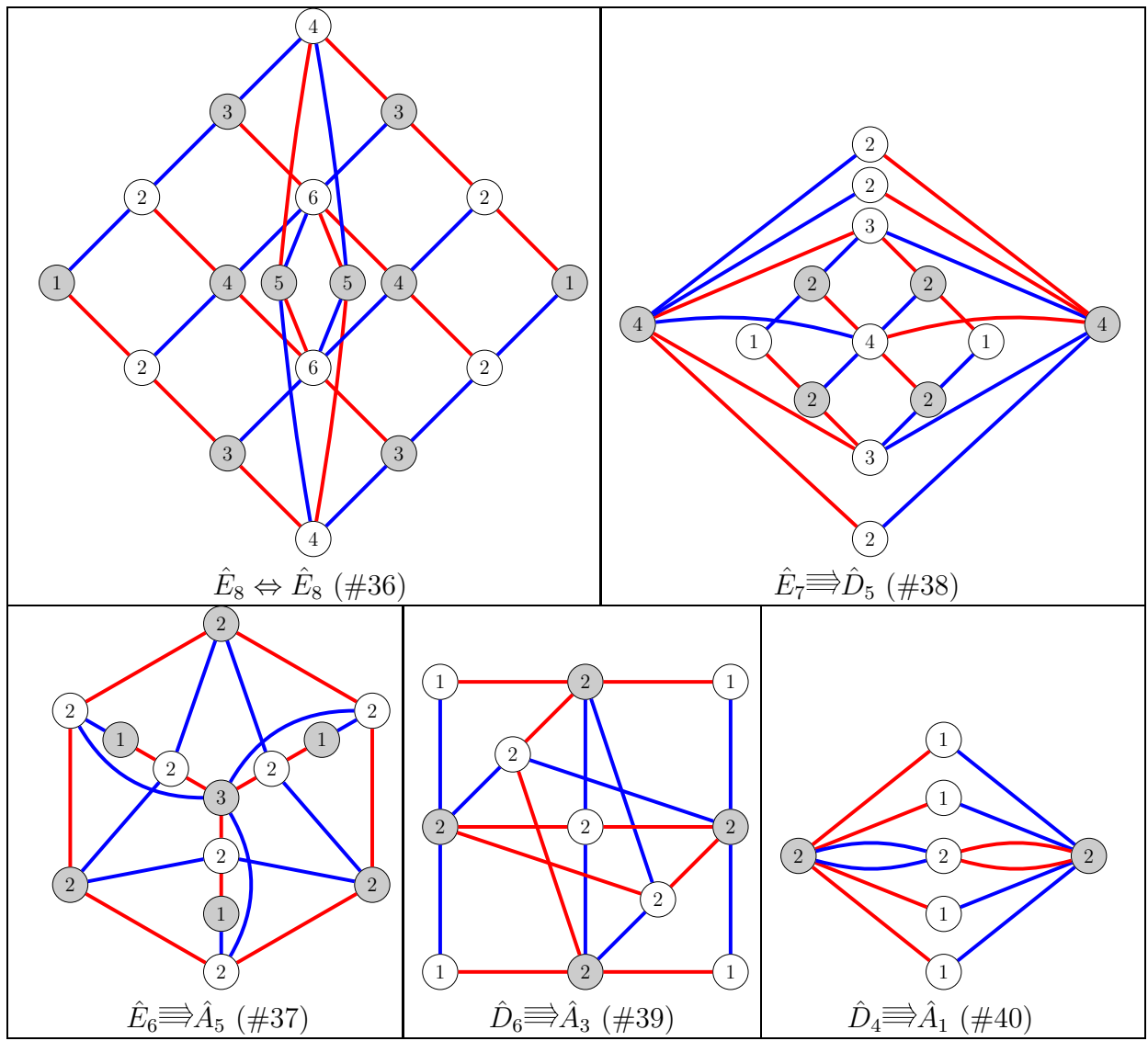
\begin{figure}
\makebox[1.0\textwidth]{
\begin{tabular}{|c|}\hline

\begin{tabular}{c|c}
\scalebox{0.7}{
\begin{tikzpicture}
\node[draw,circle,fill=black!20!white] (v0x0) at (-1.77,8.84) {3};
\node[draw,circle,fill=white] (v0x1) at (3.54,3.54) {2};
\node[draw,circle,fill=black!20!white] (v0x2) at (1.77,5.30) {4};
\node[draw,circle,fill=white] (v0x3) at (0.00,7.07) {6};
\node[draw,circle,fill=black!20!white] (v0x4) at (0.71,5.30) {5};
\node[draw,circle,fill=white] (v0x5) at (0.00,0.00) {4};
\node[draw,circle,fill=black!20!white] (v0x6) at (-1.77,1.77) {3};
\node[draw,circle,fill=white] (v0x7) at (-3.54,3.54) {2};
\node[draw,circle,fill=black!20!white] (v0x8) at (-5.30,5.30) {1};
\node[draw,circle,fill=black!20!white] (v1x0) at (1.77,1.77) {3};
\node[draw,circle,fill=white] (v1x1) at (-3.54,7.07) {2};
\node[draw,circle,fill=black!20!white] (v1x2) at (-1.77,5.30) {4};
\node[draw,circle,fill=white] (v1x3) at (0.00,3.54) {6};
\node[draw,circle,fill=black!20!white] (v1x4) at (-0.71,5.30) {5};
\node[draw,circle,fill=white] (v1x5) at (0.00,10.61) {4};
\node[draw,circle,fill=black!20!white] (v1x6) at (1.77,8.84) {3};
\node[draw,circle,fill=white] (v1x7) at (3.54,7.07) {2};
\node[draw,circle,fill=black!20!white] (v1x8) at (5.30,5.30) {1};
\draw[color=red,line width=0.75mm] (v0x2) to[] (v0x1);
\draw[color=red,line width=0.75mm] (v0x3) to[] (v0x0);
\draw[color=red,line width=0.75mm] (v0x3) to[] (v0x2);
\draw[color=red,line width=0.75mm] (v0x4) to[] (v0x3);
\draw[color=red,line width=0.75mm] (v0x5) to[bend right=2] (v0x4);
\draw[color=red,line width=0.75mm] (v0x6) to[] (v0x5);
\draw[color=red,line width=0.75mm] (v0x7) to[] (v0x6);
\draw[color=red,line width=0.75mm] (v0x8) to[] (v0x7);
\draw[color=blue,line width=0.75mm] (v1x0) to[] (v0x1);
\draw[color=blue,line width=0.75mm] (v1x0) to[] (v0x5);
\draw[color=blue,line width=0.75mm] (v1x1) to[] (v0x0);
\draw[color=blue,line width=0.75mm] (v1x1) to[] (v0x8);
\draw[color=blue,line width=0.75mm] (v1x2) to[] (v0x3);
\draw[color=blue,line width=0.75mm] (v1x2) to[] (v0x7);
\draw[color=red,line width=0.75mm] (v1x2) to[] (v1x1);
\draw[color=blue,line width=0.75mm] (v1x3) to[] (v0x2);
\draw[color=blue,line width=0.75mm] (v1x3) to[] (v0x4);
\draw[color=blue,line width=0.75mm] (v1x3) to[] (v0x6);
\draw[color=red,line width=0.75mm] (v1x3) to[] (v1x0);
\draw[color=red,line width=0.75mm] (v1x3) to[] (v1x2);
\draw[color=blue,line width=0.75mm] (v1x4) to[] (v0x3);
\draw[color=blue,line width=0.75mm] (v1x4) to[bend right=2] (v0x5);
\draw[color=red,line width=0.75mm] (v1x4) to[] (v1x3);
\draw[color=blue,line width=0.75mm] (v1x5) to[] (v0x0);
\draw[color=blue,line width=0.75mm] (v1x5) to[bend left=2] (v0x4);
\draw[color=red,line width=0.75mm] (v1x5) to[bend right=2] (v1x4);
\draw[color=blue,line width=0.75mm] (v1x6) to[] (v0x3);
\draw[color=red,line width=0.75mm] (v1x6) to[] (v1x5);
\draw[color=blue,line width=0.75mm] (v1x7) to[] (v0x2);
\draw[color=red,line width=0.75mm] (v1x7) to[] (v1x6);
\draw[color=blue,line width=0.75mm] (v1x8) to[] (v0x1);
\draw[color=red,line width=0.75mm] (v1x8) to[] (v1x7);
\end{tikzpicture}}
&
\scalebox{0.7}{
\begin{tikzpicture}
\node[draw,circle,fill=white] (v0x0) at (0.00,-4.08) {2};
\node[draw,circle,fill=white] (v0x1) at (-2.40,0.00) {1};
\node[draw,circle,fill=black!20!white] (v0x2) at (-1.20,-1.20) {2};
\node[draw,circle,fill=white] (v0x3) at (0.00,-2.40) {3};
\node[draw,circle,fill=black!20!white] (v0x4) at (-4.80,0.36) {4};
\node[draw,circle,fill=white] (v0x5) at (0.00,2.40) {3};
\node[draw,circle,fill=black!20!white] (v0x6) at (1.20,1.20) {2};
\node[draw,circle,fill=white] (v0x7) at (2.40,0.00) {1};
\node[draw,circle,fill=white] (v1x0) at (0.00,4.08) {2};
\node[draw,circle,fill=white] (v1x1) at (0.00,3.24) {2};
\node[draw,circle,fill=black!20!white] (v1x2) at (4.80,0.36) {4};
\node[draw,circle,fill=white] (v1x3) at (0.00,0.00) {4};
\node[draw,circle,fill=black!20!white] (v1x4) at (-1.20,1.20) {2};
\node[draw,circle,fill=black!20!white] (v1x5) at (1.20,-1.20) {2};
\draw[color=red,line width=0.75mm] (v0x2) to[] (v0x1);
\draw[color=red,line width=0.75mm] (v0x3) to[] (v0x2);
\draw[color=red,line width=0.75mm] (v0x4) to[] (v0x0);
\draw[color=red,line width=0.75mm] (v0x4) to[] (v0x3);
\draw[color=red,line width=0.75mm] (v0x5) to[] (v0x4);
\draw[color=red,line width=0.75mm] (v0x6) to[] (v0x5);
\draw[color=red,line width=0.75mm] (v0x7) to[] (v0x6);
\draw[color=blue,line width=0.75mm] (v1x0) to[] (v0x4);
\draw[color=blue,line width=0.75mm] (v1x1) to[] (v0x4);
\draw[color=blue,line width=0.75mm] (v1x2) to[] (v0x0);
\draw[color=blue,line width=0.75mm] (v1x2) to[] (v0x3);
\draw[color=blue,line width=0.75mm] (v1x2) to[] (v0x5);
\draw[color=red,line width=0.75mm] (v1x2) to[] (v1x0);
\draw[color=red,line width=0.75mm] (v1x2) to[] (v1x1);
\draw[color=blue,line width=0.75mm] (v1x3) to[] (v0x2);
\draw[color=blue,line width=0.75mm] (v1x3) to[bend right=11] (v0x4);
\draw[color=blue,line width=0.75mm] (v1x3) to[] (v0x6);
\draw[color=red,line width=0.75mm] (v1x3) to[bend left=11] (v1x2);
\draw[color=blue,line width=0.75mm] (v1x4) to[] (v0x1);
\draw[color=blue,line width=0.75mm] (v1x4) to[] (v0x5);
\draw[color=red,line width=0.75mm] (v1x4) to[] (v1x3);
\draw[color=blue,line width=0.75mm] (v1x5) to[] (v0x3);
\draw[color=blue,line width=0.75mm] (v1x5) to[] (v0x7);
\draw[color=red,line width=0.75mm] (v1x5) to[] (v1x3);
\end{tikzpicture}}
\\

$\affE_{8}\DLRA \affE_{8}$ (\bg{E8E8})
&
$\affE_{7}\QRA \affD_{5}$ (\bg{D5E7})
\\

\end{tabular}

\\\hline

\begin{tabular}{c|c|c}
\scalebox{0.7}{
\begin{tikzpicture}
\node[draw,circle,fill=black!20!white] (v0x0) at (2.08,1.20) {1};
\node[draw,circle,fill=white] (v0x1) at (1.04,0.60) {2};
\node[draw,circle,fill=black!20!white] (v0x2) at (-2.08,1.20) {1};
\node[draw,circle,fill=white] (v0x3) at (-1.04,0.60) {2};
\node[draw,circle,fill=black!20!white] (v0x4) at (0.00,0.00) {3};
\node[draw,circle,fill=white] (v0x5) at (0.00,-1.20) {2};
\node[draw,circle,fill=black!20!white] (v0x6) at (0.00,-2.40) {1};
\node[draw,circle,fill=white] (v1x0) at (0.00,-3.60) {2};
\node[draw,circle,fill=black!20!white] (v1x1) at (3.12,-1.80) {2};
\node[draw,circle,fill=white] (v1x2) at (3.12,1.80) {2};
\node[draw,circle,fill=black!20!white] (v1x3) at (0.00,3.60) {2};
\node[draw,circle,fill=white] (v1x4) at (-3.12,1.80) {2};
\node[draw,circle,fill=black!20!white] (v1x5) at (-3.12,-1.80) {2};
\draw[color=red,line width=0.75mm] (v0x1) to[] (v0x0);
\draw[color=red,line width=0.75mm] (v0x3) to[] (v0x2);
\draw[color=red,line width=0.75mm] (v0x4) to[] (v0x1);
\draw[color=red,line width=0.75mm] (v0x4) to[] (v0x3);
\draw[color=red,line width=0.75mm] (v0x5) to[] (v0x4);
\draw[color=red,line width=0.75mm] (v0x6) to[] (v0x5);
\draw[color=blue,line width=0.75mm] (v1x0) to[bend right=30] (v0x4);
\draw[color=blue,line width=0.75mm] (v1x0) to[] (v0x6);
\draw[color=blue,line width=0.75mm] (v1x1) to[] (v0x1);
\draw[color=blue,line width=0.75mm] (v1x1) to[] (v0x5);
\draw[color=red,line width=0.75mm] (v1x1) to[] (v1x0);
\draw[color=blue,line width=0.75mm] (v1x2) to[] (v0x0);
\draw[color=blue,line width=0.75mm] (v1x2) to[bend right=30] (v0x4);
\draw[color=red,line width=0.75mm] (v1x2) to[] (v1x1);
\draw[color=blue,line width=0.75mm] (v1x3) to[] (v0x1);
\draw[color=blue,line width=0.75mm] (v1x3) to[] (v0x3);
\draw[color=red,line width=0.75mm] (v1x3) to[] (v1x2);
\draw[color=blue,line width=0.75mm] (v1x4) to[] (v0x2);
\draw[color=blue,line width=0.75mm] (v1x4) to[bend right=30] (v0x4);
\draw[color=red,line width=0.75mm] (v1x4) to[] (v1x3);
\draw[color=blue,line width=0.75mm] (v1x5) to[] (v0x3);
\draw[color=blue,line width=0.75mm] (v1x5) to[] (v0x5);
\draw[color=red,line width=0.75mm] (v1x5) to[] (v1x0);
\draw[color=red,line width=0.75mm] (v1x5) to[] (v1x4);
\end{tikzpicture}}
&
\scalebox{0.7}{
\begin{tikzpicture}
\node[draw,circle,fill=white] (v0x0) at (0.00,0.00) {1};
\node[draw,circle,fill=white] (v0x1) at (6.00,0.00) {1};
\node[draw,circle,fill=black!20!white] (v0x2) at (3.00,0.00) {2};
\node[draw,circle,fill=white] (v0x3) at (1.50,4.50) {2};
\node[draw,circle,fill=black!20!white] (v0x4) at (3.00,6.00) {2};
\node[draw,circle,fill=white] (v0x5) at (0.00,6.00) {1};
\node[draw,circle,fill=white] (v0x6) at (6.00,6.00) {1};
\node[draw,circle,fill=white] (v1x0) at (4.50,1.50) {2};
\node[draw,circle,fill=black!20!white] (v1x1) at (6.00,3.00) {2};
\node[draw,circle,fill=white] (v1x2) at (3.00,3.00) {2};
\node[draw,circle,fill=black!20!white] (v1x3) at (0.00,3.00) {2};
\draw[color=red,line width=0.75mm] (v0x2) to[] (v0x0);
\draw[color=red,line width=0.75mm] (v0x2) to[] (v0x1);
\draw[color=red,line width=0.75mm] (v0x3) to[] (v0x2);
\draw[color=red,line width=0.75mm] (v0x4) to[] (v0x3);
\draw[color=red,line width=0.75mm] (v0x5) to[] (v0x4);
\draw[color=red,line width=0.75mm] (v0x6) to[] (v0x4);
\draw[color=blue,line width=0.75mm] (v1x0) to[] (v0x2);
\draw[color=blue,line width=0.75mm] (v1x0) to[] (v0x4);
\draw[color=blue,line width=0.75mm] (v1x1) to[] (v0x1);
\draw[color=blue,line width=0.75mm] (v1x1) to[] (v0x3);
\draw[color=blue,line width=0.75mm] (v1x1) to[] (v0x6);
\draw[color=red,line width=0.75mm] (v1x1) to[] (v1x0);
\draw[color=blue,line width=0.75mm] (v1x2) to[] (v0x2);
\draw[color=blue,line width=0.75mm] (v1x2) to[] (v0x4);
\draw[color=red,line width=0.75mm] (v1x2) to[] (v1x1);
\draw[color=blue,line width=0.75mm] (v1x3) to[] (v0x0);
\draw[color=blue,line width=0.75mm] (v1x3) to[] (v0x3);
\draw[color=blue,line width=0.75mm] (v1x3) to[] (v0x5);
\draw[color=red,line width=0.75mm] (v1x3) to[] (v1x0);
\draw[color=red,line width=0.75mm] (v1x3) to[] (v1x2);
\end{tikzpicture}}
&
\scalebox{0.7}{
\begin{tikzpicture}
\node[draw,circle,fill=white] (v0x0) at (0.00,-2.40) {1};
\node[draw,circle,fill=white] (v0x1) at (0.00,-1.20) {1};
\node[draw,circle,fill=black!20!white] (v0x2) at (-3.00,0.00) {2};
\node[draw,circle,fill=white] (v0x3) at (0.00,1.20) {1};
\node[draw,circle,fill=white] (v0x4) at (0.00,2.40) {1};
\node[draw,circle,fill=black!20!white] (v1x0) at (3.00,0.00) {2};
\node[draw,circle,fill=white] (v1x1) at (0.00,0.00) {2};
\draw[color=red,line width=0.75mm] (v0x2) to[] (v0x0);
\draw[color=red,line width=0.75mm] (v0x2) to[] (v0x1);
\draw[color=red,line width=0.75mm] (v0x3) to[] (v0x2);
\draw[color=red,line width=0.75mm] (v0x4) to[] (v0x2);
\draw[color=blue,line width=0.75mm] (v1x0) to[] (v0x0);
\draw[color=blue,line width=0.75mm] (v1x0) to[] (v0x1);
\draw[color=blue,line width=0.75mm] (v1x0) to[] (v0x3);
\draw[color=blue,line width=0.75mm] (v1x0) to[] (v0x4);
\draw[color=blue,line width=0.75mm] (v1x1) to[bend right=15] (v0x2);
\draw[color=blue,line width=0.75mm] (v1x1) to[bend left=15] (v0x2);
\draw[color=red,line width=0.75mm] (v1x1) to[bend right=15] (v1x0);
\draw[color=red,line width=0.75mm] (v1x1) to[bend left=15] (v1x0);
\end{tikzpicture}}
\\

$\affE_{6}\QRA \affA_{5}$ (\bg{A5E6})
&
$\affD_{6}\QRA \affA_{3}$ (\bg{A3D6})
&
$\affD_{4}\QRA \affA_{1}$ (\bg{A1D4})
\\

\end{tabular}

\\\hline

\end{tabular}
}
\caption{\label{fig:exc_double} \Five exceptional \affaff double bindings given together with their subadditive functions. Each double binding $G$ is \emph{sefl-dual}, i.e., is isomorphic to~$G^\opp$.}
\end{figure}

\subsection{\Five exceptional \affaff double bindings}
In this section we list \five \affaff $ADE$ bigraphs $G$ such that both $S(G)$ and $S(G^\opp)$ are equal to either $A_1^\parr1$ or $A_2^\parr2$. This is equivalent to saying that both $G$ and $G^\opp$ are \affaff double bindings. As we will see later in Section~\ref{sect:classif_proof}, these \five bigraphs are the only \affaff $ADE$ bigraphs with this property that do not belong to any of the infinite families that we have already constructed in the previous sections. The \five bigraphs are shown in Figure~\ref{fig:exc_double}.

\section{The classification of \affaff $ADE$ bigraphs}\label{sect:classif}

\begin{theorem}\label{thm:classification}
	Let $G$ be an \affaff $ADE$ bigraph. Then either of the following is true. 
	\begin{itemize}
		\item Both $G$ and $G^\opp$ appear exactly once in the below list. They are members of the same self-dual family.
		\item The below list contains a unique bigraph that is isomorphic to either $G$ or $G^\opp$.
	\end{itemize}
\end{theorem}

\def\sdf{[SD] }
\newcommand{\qlabel}[1]{\label{#1} }
\newcommand{\qitem}[1]{\item \qlabel{#1} }
\newcommand{\all}{\bg{tensor}--\bg{D2A4n} }
Here we say that a \emph{self-dual family of bigraphs} is a collection of bigraphs that is closed under taking duals. In the below list, such families are marked with \sdf.

For each \affaff $ADE$ bigraph $G$, define the \emph{Kac quadruple} of $G$ to be one of the following tables:
\newcommand{\KACSC}[5]{
\begin{center}
\begin{tabular}{|c|c|c|c|}\hline
	#1 & \scalebox{#5}{#2} & #3 & \scalebox{#5}{#4}\\\hline
\end{tabular}
\end{center}}
\newcommand{\KACV}[5]{
\begin{center}
\begin{tabular}{|c|c|}\hline
	\nodeZ{$#1$} & \scalebox{#5}{#2} \\ \hline
	\nodeZ{$#3$} & \scalebox{#5}{#4} \\\hline
\end{tabular}
\end{center}}
\newcommand{\KAC}[4]{\KACSC{\nodeZ{$#1$}}{$#2$}{\nodeZ{$#3$}}{$#4$}{1.0}}
\newcommand{\KACH}[5]{\KACSC{\nodeZ{$#1$}}{#2}{\nodeZ{$#3$}}{#4}{#5}}

\begin{center}
\begin{tabular}{cccc}
\begin{tabular}{|c|c|}\hline
	{$S(G)$} & $\descr(G)$ \\ \hline
	{$S(G^\opp)$} &$\descr(G^\opp)$ \\\hline
\end{tabular}
& OR & 
\begin{tabular}{|c|c|c|c|}\hline
	{$S(G)$} & $\descr(G)$ & {$S(G^\opp)$} &$\descr(G^\opp)$ \\\hline
\end{tabular}
\end{tabular}
\end{center}

We list each family of \affaff $ADE$ bigraphs together with its parameters and the corresponding Kac quadruples, except that we do not list the Kac quadruples for tensor products. For exceptional families, we just give Kac quadruples (KQ for short) and omit the name and parameters. 

\def\KQ{{\bf KQ:}}
\newcommand{\fgr}[1]{(Fig.~\ref{fig:#1})}
\def\atoricbigraph{}
\begin{enumerate}[\#1.]
	\qitem{tensor}\sdf \fgr{tensor_product}   {\bf Name:} a tensor product $\affL\otimes\affL'$. {\bf Parameters:} $\affL$, $\affL'$ --- two bipartite affine $ADE$ Dynkin diagrams.

	\qitem{twist}  \sdf  \fgr{twist} {\bf Name:} a twist $\affL\times\affL$. {\bf Parameters:} a bipartite affine $ADE$ Dynkin diagram $\affL$. \KQ
	\KACH{A_{1}^\parr1}{\ZAAone{\affL}{\affL}}{A_{1}^\parr1}{\ZAAone{\affL}{\affL}}{1.0}

	\qitem{toric-A-rotn} \sdf  \fgr{toric-A-rotn} {\bf Name:} \atoricbigraph $\Toric(\affA_{rd-1},\exp(2\pi i p/r),n)$ . {\bf Parameters:}  $r\geq 1$ and $1\leq p\leq r/2$ coprime with $r$; $n,d\geq 1$ such that either $n,d$ are both even or $n,d,p$ are odd and $r$ is even. The forbidden cases are $n=d=p=1$, $r$ even (in which case $\Gamma$ and $\Delta$ share edges) and $n=d=2,p=1,r\geq1$ (in which case $G$ is a twist $\affA_{2r-1}\times\affA_{2r-1}$). \KQ
	    \KACV{A_{n-1}^\parr1}{\ZAA{\affA_{rd-1}}{\affA_{rd-1}}{\affA_{rd-1}}{\affA_{rd-1}}}{A_{d-1}^\parr1}{\ZAA{\affA_{rn-1}}{\affA_{rn-1}}{\affA_{rn-1}}{\affA_{rn-1}}}{1.0}
	The dual bigraph is $\Toric(\affA_{rn-1},\exp(2\pi i q/r),d)$, where $1\leq q\leq r/2$ is the unique integer satisfying $pq\equiv \pm1\pmod r$. For the case $n=1$ (resp., $d=1$), we have
	\[\nodeZ{$\descr(G)=$\strut}\Zlooploop{\affA_{rd-1}}\nodeZ{,\strut}\quad\nodeZ{resp.,\strut}\quad \nodeZ{$\descr(G^\opp)=$}\Zlooploop{\affA_{rn-1}}\nodeZ{.\strut}\]

	\qitem{toric-A-refl-1} \fgr{toric-A} {\bf Name:} \atoricbigraph $\Toric(\affA_{2m-1},\eta^\parr1,n)$. {\bf Parameters:}  $m\geq 2$; $n\geq2$ even. \KQ
	    \KACV{A_{n-1}^\parr1}{\ZAA{\affA_{2m-1}}{\affA_{2m-1}}{\affA_{2m-1}}{\affA_{2m-1}}}{D_{m+1}^\parr2}{\ZDDD{\affA_{n-1}}{\affA_{2n-1}}{\affA_{2n-1}}{\affA_{n-1}}}{1.0}
	The dual bigraph is $\Path(\affA_{2n-1},\exp(\pi i), \exp(\pi i),m)$.
%
	\qitem{toric-A-refl-2} \fgr{toric-A} {\bf Name:} \atoricbigraph $\Toric(\affA_{2m-1},\eta^\parr2,n)$. {\bf Parameters:}  $m\geq 2$; $n\geq3$ odd. \KQ
	    \KACV{A_{n-1}^\parr1}{\ZAA{\affA_{2m-1}}{\affA_{2m-1}}{\affA_{2m-1}}{\affA_{2m-1}}}{\loops{A_{2m-1}^\parr1}}{\ZloopsAA{\affA_{2n-1}}{\affA_{2n-1}}{\affA_{2n-1}}{\affA_{2n-1}}}{1.0}

	\qitem{toric-D-sigma} \fgr{toric-D} {\bf Name:} \atoricbigraph $\Toric(\affD_{m+2},\sigma,n)$. {\bf Parameters:} $m\geq 2$; $n\geq2$ even. \KQ
	\KACV{A_{n-1}^\parr1}{\ZAA{\affD_{m+2}}{\affD_{m+2}}{\affD_{m+2}}{\affD_{m+2}}}{A_{2m+1}^\parr2}{\ZAAodd{\affA_{n-1}}{\affA_{n-1}}{\affA_{n-1}}{\affA_{n-1}}{\affA_{2n-1}}}{1.0}

	\qitem{toric-D-stst} \fgr{toric-D} {\bf Name:} \atoricbigraph $\Toric(\affD_{m+2},\sigma\tau\sigma\tau,n)$. {\bf Parameters:} $m\geq 2$; $n\geq2$ even. \KQ
	\KACV{A_{n-1}^\parr1}{\ZAA{\affD_{m+2}}{\affD_{m+2}}{\affD_{m+2}}{\affD_{m+2}}}{C_{m}^\parr1}{\ZCC{\affA_{2n-1}}{\affA_{n-1}}{\affA_{n-1}}{\affA_{2n-1}}}{1.0}
	The dual bigraph is $\Path(\affA_{n-1},\id,\id,m)$.

	\qitem{toric-D-tau-even} \fgr{toric-D} {\bf Name:} \atoricbigraph $\Toric(\affD_{2m+2},\tau,n)$. {\bf Parameters:} $m\geq 2$; $n\geq2$ even. \KQ
	\KACV{A_{n-1}^\parr1}{\ZAA{\affD_{2m+2}}{\affD_{2m+2}}{\affD_{2m+2}}{\affD_{2m+2}}}{B_{m+1}^\parr1}{\ZBB{\affA_{2n-1}}{\affA_{2n-1}}{\affA_{2n-1}}{\affA_{2n-1}}{\affA_{n-1}}}{1.0}

	\qitem{toric-D-tau-odd} \fgr{toric-D} {\bf Name:} \atoricbigraph $\Toric(\affD_{2m+1},\tau,n)$. {\bf Parameters:} $m\geq 2$; $n\geq3$ odd. \KQ
	\KACV{A_{n-1}^\parr1}{\ZAA{\affD_{2m+1}}{\affD_{2m+1}}{\affD_{2m+1}}{\affD_{2m+1}}}{\loops{D_{2m+1}^\parr1}}{\ZloopsDD{\affA_{2n-1}}{\affA_{2n-1}}{\affA_{2n-1}}{\affA_{2n-1}}{\affA_{2n-1}}}{1.0}

	\qitem{toric-D-sigma-tau-even} \fgr{toric-D} {\bf Name:} \atoricbigraph $\Toric(\affD_{2m+2},\sigma\tau,n)$. {\bf Parameters:} $m\geq 2$; $n\geq2$ even. \KQ
	\KACV{A_{n-1}^\parr1}{\ZAA{\affD_{2m+2}}{\affD_{2m+2}}{\affD_{2m+2}}{\affD_{2m+2}}}{A_{2m}^\parr2}{\ZAAeven{\affA_{n-1}}{\affA_{2n-1}}{\affA_{2n-1}}{\affA_{4n-1}}}{1.0}
	The dual bigraph is $\Path(\affA_{2n-1},\exp(\pi i),\id,m)$.

	\qitem{toric-D-sigma-tau-odd} \fgr{toric-D}  {\bf Name:} \atoricbigraph $\Toric(\affD_{2m+3},\sigma\tau,n)$. {\bf Parameters:} $m\geq 1$; $n\geq3$ odd. \KQ
	\KACV{A_{n-1}^\parr1}{\ZAA{\affD_{2m+3}}{\affD_{2m+3}}{\affD_{2m+3}}{\affD_{2m+3}}}{\loops{C_{2m+1}^\parr1}}{\ZloopsCC{\affA_{4n-1}}{\affA_{2n-1}}{\affA_{2n-1}}{\affA_{2n-1}}}{1.0}

%
%

	\qitem{toric-D4-4} \fgr{toric-D4} {\bf Name:} \atoricbigraph $\Toric(\affD_4,(4),n)$. {\bf Parameters:} $n\geq2$ even. \KQ
	\KACV{A_{n-1}^\parr1}{\ZAA{\affD_{4}}{\affD_{4}}{\affD_{4}}{\affD_{4}}}{A_{2}^\parr2}{\ZAAtwo{\affA_{n-1}}{\affA_{4n-1}}}{1.0}

	\qitem{toric-D4-31} \fgr{toric-D4} {\bf Name:} \atoricbigraph $\Toric(\affD_4,(3+1),n)$. {\bf Parameters:} $n\geq2$ even. \KQ
	\KACV{A_{n-1}^\parr1}{\ZAA{\affD_{4}}{\affD_{4}}{\affD_{4}}{\affD_{4}}}{D_{4}^\parr3}{\ZDDfour{\affA_{n-1}}{\affA_{n-1}}{\affA_{3n-1}}}{1.0}

	\qitem{toric-E6-21} \fgr{toric-E} {\bf Name:} \atoricbigraph $\Toric(\affE_6,(2+1),n)$. {\bf Parameters:} $n\geq2$ even. \KQ
	\KACV{A_{n-1}^\parr1}{\ZAA{\affE_{6}}{\affE_{6}}{\affE_{6}}{\affE_{6}}}{E_{6}^\parr2}{\ZEEEsix{\affA_{n-1}}{\affA_{n-1}}{\affA_{n-1}}{\affA_{2n-1}}{\affA_{2n-1}}}{1.0}

	\qitem{toric-E6-3} \fgr{toric-E} {\bf Name:} \atoricbigraph $\Toric(\affE_6,(3),n)$. {\bf Parameters:} $n\geq2$ even. \KQ
	\KACV{A_{n-1}^\parr1}{\ZAA{\affE_{6}}{\affE_{6}}{\affE_{6}}{\affE_{6}}}{G_{2}^\parr1}{\ZGG{\affA_{3n-1}}{\affA_{3n-1}}{\affA_{n-1}}}{1.0}

	\qitem{toric-E7} \fgr{toric-E} {\bf Name:} \atoricbigraph $\Toric(\affE_7,\theta,n)$. {\bf Parameters:} $n\geq2$ even. \KQ
	\KACV{A_{n-1}^\parr1}{\ZAA{\affE_{7}}{\affE_{7}}{\affE_{7}}{\affE_{7}}}{F_4^\parr1}{\ZFF{\affA_{2n-1}}{\affA_{2n-1}}{\affA_{2n-1}}{\affA_{n-1}}{\affA_{n-1}}}{1.0}


	\qitem{path-A-refl-refl} \sdf \fgr{path-A-refl-refl} {\bf Name:} $\Path(\affA_{2rd-1},\eta^\parr1,\eta_{pd},n)$. {\bf Parameters:} $n,d\geq2$; $r\geq1$; $1\leq p\leq r/2$ coprime with $r\geq 2$ or $p=0$ if $r=1$. \KQ
	\KACV{C_{n}^\parr1}{\ZCC{\affD_{rd+2}}{\affA_{2rd-1}}{\affA_{2rd-1}}{\affD_{rd+2}}}{C_{d}^\parr1}{\ZCC{\affD_{rn+2}}{\affA_{2rn-1}}{\affA_{2rn-1}}{\affD_{rn+2}}}{1.0}
	The dual bigraph is $\Path(\affA_{2rn-1},\eta^\parr1,\eta_{qn},d)$, where $1\leq q\leq r/2$ is defined by $pq\equiv\pm1\pmod r$.

	\qitem{path-A-refl-refl-coprime} \fgr{path-A-refl-refl} {\bf Name:} $\Path(\affA_{2r-1},\eta^\parr1,\eta_{p},n)$. {\bf Parameters:} $n\geq2$; $r\geq1$; $1\leq p\leq r/2$ coprime with $r$. \KQ
	\KACV{C_{n}^\parr1}{\ZCC{\affD_{r+2}}{\affA_{2r-1}}{\affA_{2r-1}}{\affD_{r+2}}}{A_1^\parr1}{\ZAAone{\affD_{rn+2}}{\affD_{rn+2}}}{1.0}
	The dual bigraph is $\affD_{rn+2}\pstwist{qn}\affD_{rn+2}$, where $1\leq q\leq r/2$ is defined by $pq\equiv\pm1\pmod r$.
	
	\qitem{pstwist} \sdf \fgr{pstwist} {\bf Name:} $\affD_{m+2}\pstwist p\affD_{m+2}$. {\bf Parameters:} $m\geq2$; $2\leq p\leq m/2$ coprime with $m$. \KQ
	\KACH{A_{1}^\parr1}{\ZAAone{\affD_{m+2}}{\affD_{m+2}}}{A_{1}^\parr1}{\ZAAone{\affD_{m+2}}{\affD_{m+2}}}{1.0}
	The dual bigraph is $\affD_{m+2}\pstwist q\affD_{m+2}$ where $2\leq q\leq m/2$ is the unique integer such that $pq\equiv\pm1\pmod {m}$. The case $p=q=1$ is not included here as it corresponds to the twist $\affD_{m+2}\times\affD_{m+2}$.

	\qitem{path-A-refl-rotn} \sdf \fgr{path-A-refl-rotn} {\bf Name:} $\Path(\affA_{4m-1}, \eta^\parr1,\exp(\pi i),n)$. {\bf Parameters:} $n,m\geq2$. \KQ
	\KACV{A_{2n}^\parr2}{\ZAAeven{\affA_{2m-1}}{\affA_{4m-1}}{\affA_{4m-1}}{\affD_{2m+2}}}{A_{2m}^\parr2}{\ZAAeven{\affA_{2n-1}}{\affA_{4n-1}}{\affA_{4n-1}}{\affD_{2n+2}}}{1.0}
	The dual bigraph is $\Path(\affA_{4n-1}, \eta^\parr1,\exp(\pi i),m)$.
	
	\qitem{path-A3-refl-rotn} \fgr{path-A-refl-rotn} {\bf Name:} $\Path(\affA_{3}, \eta^\parr1,\exp(\pi i),n)$. {\bf Parameters:} $n\geq2$. \KQ
	\KACV{A_{2n}^\parr2}{\ZAAeven{\affA_{1}}{\affA_{3}}{\affA_{3}}{\affD_{4}}}{A_{2}^\parr2}{\ZAAtwo{\affA_{2n-1}}{\affD_{2n+2}}}{1.0}

	\qitem{path-A-refl-id} \fgr{path-A-refl-rotn} {\bf Name:} $\Path(\affA_{2m-1}, \eta^\parr1,\id,n)$. {\bf Parameters:} $n,m\geq 2$. \KQ
	\KACV{C_{n}^\parr1}{\ZCC{\affA_{4m-1}}{\affA_{2m-1}}{\affA_{2m-1}}{\affD_{m+2}}}{D_{m+1}^\parr2}{\ZDDD{\affD_{n+2}}{\affD_{2n+2}}{\affD_{2n+2}}{\affD_{n+2}}}{1.0}
	The dual bigraph is $\Path(\affD_{2n+2},\tau,\tau^\perp,m)$.

	\qitem{path-D-sigma-sigma} \fgr{path-D} {\bf Name:} $\Path(\affD_{m+2},\sigma,\sigma,n)$. {\bf Parameters:} $n,m\geq2$. \KQ
	\KACV{C_{n}^\parr1}{\ZCC{\affD_{2m+2}}{\affD_{m+2}}{\affD_{m+2}}{\affD_{2m+2}}}{B_{m+1}^\parr1}{\ZBB{\affD_{n+2}}{\affD_{n+2}}{\affD_{n+2}}{\affD_{n+2}}{\affA_{2n-1}}}{1.0}

	\qitem{path-D-sigma-sigma-perp} \sdf \fgr{path-D} {\bf Name:} $\Path(\affD_{m+2},\sigma,\sigma^\perp,n)$. {\bf Parameters:} $n,m\geq2$. \KQ
	\KACV{C_{n}^\parr1}{\ZCC{\affD_{2m+2}}{\affD_{m+2}}{\affD_{m+2}}{\affD_{2m+2}}}{C_{m}^\parr1}{\ZCC{\affD_{2n+2}}{\affD_{n+2}}{\affD_{n+2}}{\affD_{2n+2}}}{1.0}
	The dual bigraph is $\Path(\affD_{n+2},\sigma,\sigma^\perp,m)$.

	\qitem{path-D-sigma-stst} \sdf \fgr{path-D} {\bf Name:} $\Path(\affD_{m+2},\sigma,\sigma\tau\sigma\tau,n)$. {\bf Parameters:} $n,m\geq2$. \KQ
	\KACV{A_{2n}^\parr2}{\ZAAeven{\affA_{2m-1}}{\affD_{m+2}}{\affD_{m+2}}{\affD_{2m+2}}}{A_{2m}^\parr2}{\ZAAeven{\affA_{2n-1}}{\affD_{n+2}}{\affD_{n+2}}{\affD_{2n+2}}}{1.0}
	The dual bigraph is $\Path(\affD_{n+2},\sigma,\sigma\tau\sigma\tau,m)$.

	\qitem{path-D-sigma-tau} \sdf \fgr{path-D} {\bf Name:} $\Path(\affD_{2m+2},\sigma,\tau,n)$. {\bf Parameters:} $n,m\geq2$. \KQ
	\KACV{A_{2n}^\parr2}{\ZAAeven{\affD_{m+2}}{\affD_{2m+2}}{\affD_{2m+2}}{\affD_{4m+2}}}{A_{2m}^\parr2}{\ZAAeven{\affD_{n+2}}{\affD_{2n+2}}{\affD_{2n+2}}{\affD_{4n+2}}}{1.0}
	The dual bigraph is $\Path(\affD_{2n+2},\sigma,\tau,m)$.

	\qitem{path-D-stst-stst} \sdf \fgr{path-D} {\bf Name:} $\Path(\affD_{m+2},\sigma\tau\sigma\tau,\sigma\tau\sigma\tau,n)$. {\bf Parameters:} $n,m\geq2$. \KQ
	\KACV{D_{n+1}^\parr2}{\ZDDD{\affA_{2m-1}}{\affD_{m+2}}{\affD_{m+2}}{\affA_{2m-1}}}{D_{m+1}^\parr2}{\ZDDD{\affA_{2n-1}}{\affD_{n+2}}{\affD_{n+2}}{\affA_{2n-1}}}{1.0}
	The dual bigraph is $\Path(\affD_{n+2},\sigma\tau\sigma\tau,\sigma\tau\sigma\tau,m)$.

	\qitem{path-D-stst-tau} \sdf \fgr{path-D} {\bf Name:} $\Path(\affD_{2m+2},\sigma\tau\sigma\tau,\tau,n)$. {\bf Parameters:} $n,m\geq2$. \KQ
	\KACV{D_{n+1}^\parr2}{\ZDDD{\affA_{4m-1}}{\affD_{2m+2}}{\affD_{2m+2}}{\affD_{m+2}}}{D_{m+1}^\parr2}{\ZDDD{\affA_{4n-1}}{\affD_{2n+2}}{\affD_{2n+2}}{\affD_{n+2}}}{1.0}
	The dual bigraph is $\Path(\affD_{2n+2},\sigma\tau\sigma\tau,\tau,m)$.

	\qitem{path-D-tau-tau} \fgr{path-D} {\bf Name:} $\Path(\affD_{2m+2},\tau,\tau,n)$.  {\bf Parameters:} $n,m\geq2$. \KQ
	\KACV{D_{n+1}^\parr2}{\ZDDD{\affD_{m+2}}{\affD_{2m+2}}{\affD_{2m+2}}{\affD_{m+2}}}{A_{2m+1}^\parr2}{\ZAAodd{\affA_{2n-1}}{\affA_{2n-1}}{\affA_{2n-1}}{\affA_{2n-1}}{\affD_{n+2}}}{1.0}

	\qitem{path-D4-12-13} \fgr{path-D} {\bf Name:} $\Path(\affD_{4},(12),(13),n)$.  {\bf Parameters:} $n\geq2$. \KQ
	\KACV{C_{n}^\parr1}{\ZCC{\affD_{6}}{\affD_{4}}{\affD_{4}}{\affD_{6}}}{D_{4}^\parr3}{\ZDDfour{\affD_{n+2}}{\affD_{n+2}}{\affD_{3n+2}}}{1.0}


	\qitem{path-D4-12-1324} \fgr{path-D} {\bf Name:} $\Path(\affD_{4},(12),(13)(24),n)$.  {\bf Parameters:} $n\geq2$. \KQ
	\KACV{A_{2n}^\parr2}{\ZAAeven{\affA_{3}}{\affD_{4}}{\affD_{4}}{\affD_{6}}}{A_2^\parr2}{\ZAAtwo{\affD_{n+2}}{\affD_{4n+2}}}{1.0}

	\qitem{path-D4-1234-1324} \fgr{path-D} {\bf Name:} $\Path(\affD_{4},(12)(34),(13)(24),n)$.  {\bf Parameters:} $n\geq2$. \KQ
	\KACV{D_{n+1}^\parr2}{\ZDDD{\affA_{3}}{\affD_{4}}{\affD_{4}}{\affA_{3}}}{A_1^\parr1}{\ZAAone{A_{4n-1}}{D_{n+2}}}{1.0}

	\qitem{path-E6-12-12} \fgr{path-E} {\bf Name:} $\Path(\affE_{6},(12),(12),n)$.  {\bf Parameters:} $n\geq2$. \KQ
	\KACV{C_{n}^\parr1}{\ZCC{\affE_7}{\affE_6}{\affE_6}{\affE_7}}{F_4^\parr1}{\ZFF{\affD_{n+2}}{\affD_{n+2}}{\affD_{n+2}}{\affA_{2n-1}}{\affA_{2n-1}}}{1.0}

	\qitem{path-E6-12-13} \fgr{path-E} {\bf Name:} $\Path(\affE_{6},(12),(13),n)$.  {\bf Parameters:} $n\geq2$. \KQ
	\KACV{C_{n}^\parr1}{\ZCC{\affE_7}{\affE_6}{\affE_6}{\affE_7}}{G_{2}^\parr1}{\ZGG{\affD_{3n+2}}{\affD_{3n+2}}{\affD_{n+2}}}{1.0}


	\qitem{path-E7-theta-theta}\fgr{path-E} {\bf Name:} $\Path(\affE_{7},\theta,\theta,n)$.  {\bf Parameters:} $n\geq2$. \KQ
	\KACV{D_{n+1}^\parr2}{\ZDDD{\affE_6}{\affE_7}{\affE_7}{\affE_6}}{E_6^\parr2}{\ZEEEsix{\affA_{2n-1}}{\affA_{2n-1}}{\affA_{2n-1}}{\affD_{n+2}}{\affD_{n+2}}}{1.0}


	\qitem{E8E8} \sdf \fgr{exc_double} 
	\KQ
	\KACH{A_{1}^\parr1}{\ZAAone{\affE_{8}}{\affE_{8}}}{A_{1}^\parr1}{\ZAAone{\affE_{8}}{\affE_{8}}}{1.0}

	\qitem{A5E6} \sdf \fgr{exc_double} 
	\KQ
	\KACH{A_{2}^\parr2}{\ZAAtwo{\affA_{5}}{\affE_{6}}}{A_{2}^\parr2}{\ZAAtwo{\affA_{5}}{\affE_{6}}}{1.0}

	\qitem{D5E7} \sdf \fgr{exc_double} 
	\KQ
	\KACH{A_{2}^\parr2}{\ZAAtwo{\affD_5}{\affE_7}}{A_{2}^\parr2}{\ZAAtwo{\affD_5}{\affE_7}}{1.0}

	\qitem{A3D6} \sdf \fgr{exc_double} 
	\KQ
	\KACH{A_{2}^\parr2}{\ZAAtwo{\affA_3}{\affD_6}}{A_{2}^\parr2}{\ZAAtwo{\affA_3}{\affD_6}}{1.0}

	\qitem{A1D4} \sdf \fgr{exc_double} 
	\KQ
	\KACH{A_{2}^\parr2}{\ZAAtwo{\affA_1}{\affD_4}}{A_{2}^\parr2}{\ZAAtwo{\affA_1}{\affD_4}}{1.0}


	\qitem{A3A3D5} \fgr{triple_arrow} \KQ
	\KACH{D_{4}^\parr3}{\ZDDfour{\affA_{3}}{\affA_{3}}{\affD_{5}}}{A_{1}^\parr1}{\ZAAone{\affD_{6}}{\affD_{6}}}{1.0}

	\qitem{A3D5D5} \fgr{triple_arrow} \KQ
	\KACH{G_{2}^\parr1}{\ZGG{\affD_{5}}{\affD_{5}}{\affA_3}}{A_1^\parr1}{\ZAAone{\affE_{7}}{\affE_{7}}}{1.0}

	\qitem{D6D6E7} \sdf\fgr{triple_arrow} \KQ
	\KACH{D_{4}^\parr3}{\ZDDfour{\affD_{6}}{\affD_{6}}{\affE_{7}}}{D_{4}^\parr3}{\ZDDfour{\affD_{6}}{\affD_{6}}{\affE_{7}}}{1.0}

	\qitem{D6E7E7} \sdf\fgr{triple_arrow} \KQ
	\KACH{G_{2}^\parr1}{\ZGG{\affE_{7}}{\affE_{7}}{\affD_6}}{G_{2}^\parr1}{\ZGG{\affE_{7}}{\affE_{7}}{\affD_6}}{1.0}

	\qitem{D4D4E6} \sdf\fgr{triple_arrow} \KQ
	\KACH{D_{4}^\parr3}{\ZDDfour{\affD_{4}}{\affD_{4}}{\affE_{6}}}{D_{4}^\parr3}{\ZDDfour{\affD_{4}}{\affD_{4}}{\affE_{6}}}{1.0}

	\qitem{D4E6E6} \sdf\fgr{triple_arrow} \KQ
	\KACH{G_{2}^\parr1}{\ZGG{\affE_{6}}{\affE_{6}}{\affD_4}}{G_{2}^\parr1}{\ZGG{\affE_{6}}{\affE_{6}}{\affD_4}}{1.0}

	\qitem{D2nD1} \sdf \fgr{type-D} {\bf Name:} $(\affD_{2m+2})^{n+1}\affD_{m+2}$. {\bf Parameters:} $n,m\geq2$. \KQ
	\KACV{B_{n+1}^\parr1}{\ZBB{\affD_{2m+2}}{\affD_{2m+2}}{\affD_{2m+2}}{\affD_{2m+2}}{\affD_{m+2}}}{B_{m+1}^\parr1}{\ZBB{\affD_{2n+2}}{\affD_{2n+2}}{\affD_{2n+2}}{\affD_{2n+2}}{\affD_{n+2}}}{1.0}
	The dual bigraph is $(\affD_{2n+2})^{m+1}\affD_{n+2}$.

	\qitem{D1nD2} \sdf \fgr{type-D} {\bf Name:} $(\affD_{m+2})^{n+1}\affD_{2m+2}$. {\bf Parameters:} $n,m\geq2$. \KQ
	\KACV{A_{2n+1}^\parr2}{\ZAAodd{\affD_{m+2}}{\affD_{m+2}}{\affD_{m+2}}{\affD_{m+2}}{\affD_{2m+2}}}{A_{2m+1}^\parr2}{\ZAAodd{\affD_{n+2}}{\affD_{n+2}}{\affD_{n+2}}{\affD_{n+2}}{\affD_{2n+2}}}{1.0}
	The dual bigraph is $(\affD_{n+2})^{m+1}\affD_{2n+2}$.

	\qitem{E7nE6} \fgr{type-D} {\bf Name:} $(\affE_7)^{n+1}\affE_6$. {\bf Parameters:} $n\geq2$. \KQ
	\KACV{B_{n+1}^\parr1}{\ZBB{\affE_7}{\affE_7}{\affE_7}{\affE_7}{\affE_6}}{F_4^\parr1}{\ZFF{\affD_{2n+2}}{\affD_{2n+2}}{\affD_{2n+2}}{\affD_{n+2}}{\affD_{n+2}}}{1.0}

	\qitem{E6nE7} \fgr{type-D} {\bf Name:} $(\affE_6)^{n+1}\affE_7$. {\bf Parameters:} $n\geq2$. \KQ
	\KACV{A_{2n+1}^\parr2}{\ZAAodd{\affE_6}{\affE_6}{\affE_6}{\affE_6}{\affE_7}}{E_6^\parr2}{\ZEEEsix{\affD_{n+2}}{\affD_{n+2}}{\affD_{n+2}}{\affD_{2n+2}}{\affD_{2n+2}}}{1.0}


	\qitem{E6E6E7E7E7} \sdf\fgr{full-house} \KQ
	\KACV{F_4^\parr1}{\ZFF{\affE_7}{\affE_7}{\affE_7}{\affE_6}{\affE_6}}{F_4^\parr1}{\ZFF{\affE_7}{\affE_7}{\affE_7}{\affE_6}{\affE_6}}{1.0}

	\qitem{E6E6E6E7E7} \sdf\fgr{full-house} \KQ
	\KACV{E_6^\parr2}{\ZEEEsix{\affE_6}{\affE_6}{\affE_6}{\affE_7}{\affE_7}}{E_6^\parr2}{\ZEEEsix{\affE_6}{\affE_6}{\affE_6}{\affE_7}{\affE_7}}{1.0}


	\qitem{D2A4n}  \sdf \fgr{loop} {\bf Name:} $\affD_{2m+3}(\affA_{4m+1})^{n}$. {\bf Parameters:} $n,m\geq1$. \KQ
	\KACV{\loops{C_{2n+1}^\parr1}}{\ZloopsCC{\affD_{2m+3}}{\affA_{4m+1}}{\affA_{4m+1}}{\affA_{4m+1}}}{\loops{C_{2m+1}^\parr1}}{\ZloopsCC{\affD_{2n+3}}{\affA_{4n+1}}{\affA_{4n+1}}{\affA_{4n+1}}}{1.0}
	The dual bigraph is $\affD_{2n+3}(\affA_{4n+1})^{m}$.

\end{enumerate}

Note that for families~\bg{tensor}--\bg{A1D4}, the illustrations have been given in Section~\ref{sect:constr} (and in Figure~\ref{fig:tensor_product}). The rest of the families are shown in Figures~\ref{fig:triple_arrow}--\ref{fig:loop}. The exceptional bigraphs are~\bg{E8E8}--\bg{D4E6E6}, \bg{E6E6E7E7E7}, and  \bg{E6E6E6E7E7}. Thus there are $13$ of them, and the rest $40$ items in the classification are infinite families (including two $3$-parameter families~\bg{toric-A-rotn} and~\bg{path-A-refl-refl}).

\begin{figure}
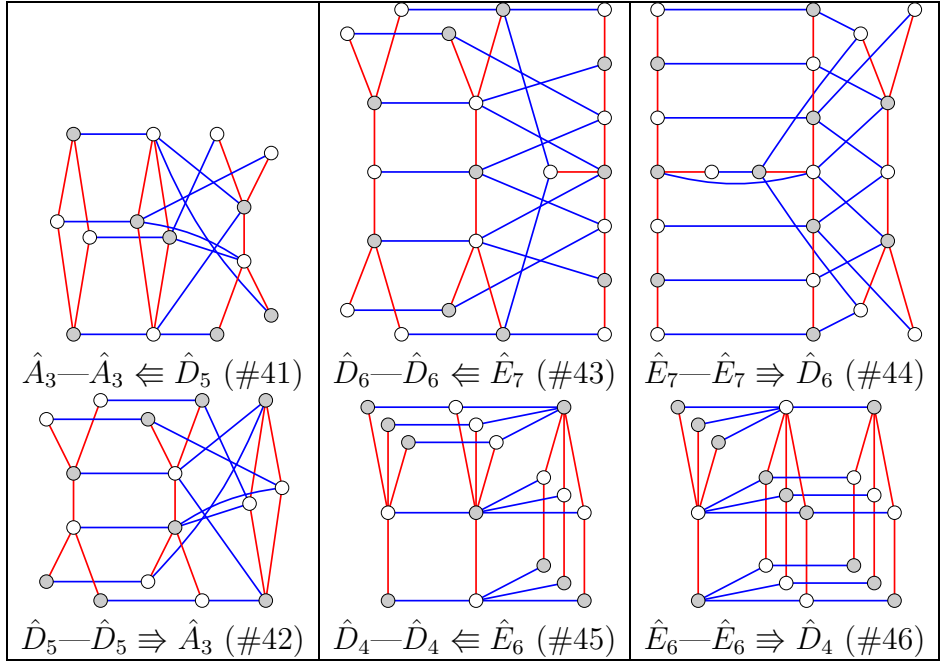

\makebox[1.0\textwidth]{

}
\caption{\label{fig:triple_arrow} Bigraphs $G$ such that $S(G)$ contains a triple arrow.}
\end{figure}

\begin{figure}
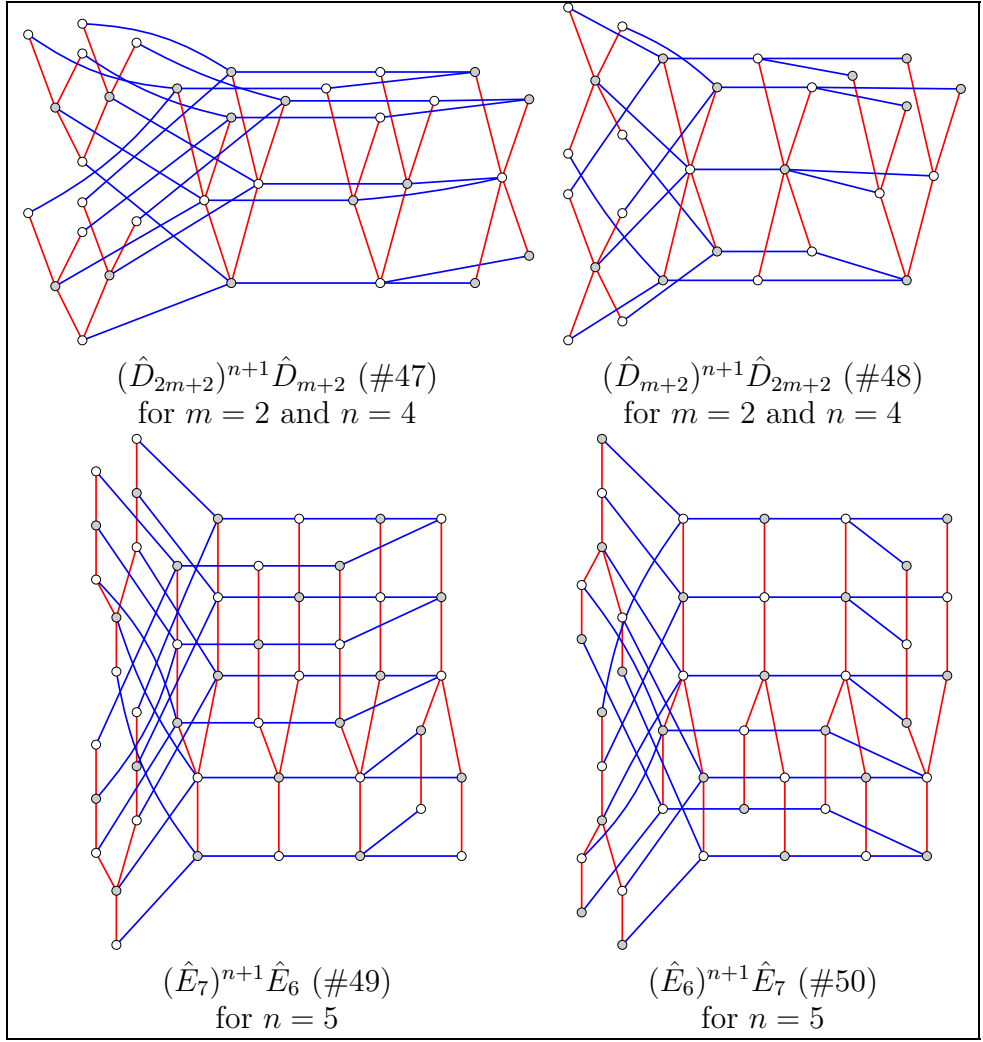

\makebox[1.0\textwidth]{
%
}
\caption{\label{fig:type-D} Bigraphs $G$ such that $S(G)$ is of type $D$ with a double arrow at the end.}
\end{figure}

\begin{figure}
\makebox[1.0\textwidth]{
%
}
\caption{\label{fig:full-house} Bigraphs~\bg{E6E6E7E7E7} and~\bg{E6E6E6E7E7}.}
\end{figure}

\begin{figure}
\scalebox{0.3}{
%
}
\caption{\label{fig:loop} The family $\affD_{2m+3}(\affA_{4m+1})^{n}$  (\bg{D2A4n}) for $m=2$ and $n=4$.}
\end{figure}


\section{Proof of the classification}\label{sect:classif_proof}
First, it is straightforward to check that each of~\all is an \affaff $ADE$ bigraph. One helpful result~\cite[Lemma~2.4]{S} of Stembridge states that $G$ has commuting adjacency matrices if and only if all of the self and double bindings involved have commuting adjacency matrices. From this it follows almost immediately that each of~\all has commuting adjacency matrices; one needs to check this fact separately for \affaff and \affinite self and double bindings. The fact that all red and blue components are affine $ADE$ Dynkin diagrams is clear from looking at $\descr(G)$ and $\descr(G^\opp)$.

Second, it is easy to verify that no two of the bigraphs~\all are isomorphic. Indeed, the only cases where we can have both $\descr(G_1)=\descr(G_2)$ and $\descr(G_1^\opp)=\descr(G_2^\opp)$ without $G_1$ and $G_2$ being isomorphic arise when both $G_1$ and $G_2$ belong to one of the following families: \bg{twist}, \bg{toric-A-rotn}, \bg{path-A-refl-refl}, \bg{path-A-refl-refl-coprime}, \bg{pstwist}, and in each case the fact that they are not isomorphic follows from the results of Section~\ref{sect:constr}.  

Thus it remains to prove that we listed \emph{all} possible \affaff $ADE$ bigraphs.

Suppose that $G$ is an \affaff $ADE$ bigraph and consider the diagram $S(G)$ which by Theorem~\ref{thm:Cartan} is a diagram from Figure~\ref{fig:aff}. According to whether $S(G)$ is \emph{ambiguous} (see Proposition~\ref{prop:descr}), we will consider the following disjoint cases:
\begin{enumerate}[\normalfont(i)]
	\item\label{item:toric} $S(G)=A_\ell^\parr1$ for $\ell\geq 2$;
	\item\label{item:path} $S(G)$ is either one of $D_{\ell+1}^\parr2$, $C_\ell^\parr1$, or $A_{2\ell}^\parr2$ for $\ell\geq2$;
	\item\label{item:small} $S(G)$ is either one of $A_1^\parr1$, $A_2^\parr2$, or $\loops{A_{1}^\parr1}$;
	\item\label{item:other} $S(G)$ is none of the above, i.e. is unambiguous.
\end{enumerate}

For the case~\eqref{item:toric}, we showed in Section~\ref{sect:toric} that such graphs are classified by weak conjugacy classes of automorphisms of diagrams in Figure~\ref{fig:aff}. One can verify directly that these graphs are exactly the ones listed in families~\bg{toric-A-rotn}--\bg{toric-E7} together with tensor products $\affL\otimes\affA$ from~\bg{tensor}. 

Similarly, for the case~\eqref{item:path}, we showed in Section~\ref{sect:path} that such graphs are classified by pairs of color-preserving involutions of affine $ADE$ Dynkin diagrams up to simultaneous conjugation. Since we have listed all such pairs in Section~\ref{sect:path}, it is straightforward to check that these graphs are exactly the ones listed in families~\bg{path-A-refl-refl}, \bg{path-A-refl-refl-coprime}, \bg{path-A-refl-rotn}--\bg{path-E7-theta-theta}, together with the duals of~\bg{toric-A-refl-1}, \bg{toric-D-stst}, \bg{toric-D-sigma-tau-even}, and~\bg{path-A-refl-id}.

Suppose now that~\eqref{item:other} holds. By Proposition~\ref{prop:descr} then $G$ is uniquely determined by $\descr(G)$. It remains to go through all the possible unambiguous diagrams $S$ in Figure~\ref{fig:aff} and for each of them list all possible assignments of double bindings with $\scf=(2,1)$, double bindings with $\scf=(3,1)$, and self bindings to double arrows, triple arrows, and loops in $S(G)$ respectively that yield affine $ADE$ Dynkin diagrams as blue components. This can be done in a straightforward way producing the families~\bg{A3A3D5}--\bg{D2A4n} and their duals, together with the duals of some toric and path bigraphs. Note also that family~\bg{tensor} of tensor products falls into this category as well.

Similarly, if $G^\opp$ falls into categories~\eqref{item:toric}, \eqref{item:path}, or~\eqref{item:other} then $G^\opp$ appears in the list. 

It remains to consider the case when both $G$ and $G^\opp$ satisfy~\eqref{item:small}. 

\subsection{Affine $\boxtimes$ affine self bindings}
In this section, we consider the case $S(G)=\loops{A_{1}^\parr1}$ which means that $G$ is an \affaff self binding.

 Let $\v:\Vert(G)\to\R$ be the common eigenvector for $A_\Gamma$ and $A_\Delta$ from Lemma~\ref{lemma:eigenvalues}. Thus $A_\Gamma\v=2\v$ and $A_\Delta\v=2\v$. Since $\Gamma$ has just one connected component, we may rescale $\v$ so that it is equal to $\v_\Gamma$. By Proposition~\ref{prop:self_scf}, we have that for every $v\in\Vert(G)$, 
 \begin{equation}\label{eq:self}
 \sum_{(v,w)\in\Delta} \v(w)=\sum_{(v,w)\in\Gamma} \v(w)=2\v(v). 
 \end{equation}

\begin{theorem}
 
 The only possible \affaff self bindings are 
 \[\Toric(\affA_{2n-1},\exp(\pi i (2m-1)/n),1),\quad\text{for $n\geq 2$ and $1<2m-1\leq n$,}\] 
  and any two such bigraphs are non-isomorphic.
\end{theorem}
\begin{proof}
 Let $G$ be an \affaff self binding, thus $\Gamma$ is an affine $ADE$ Dynkin diagram. We are first going to eliminate the cases when $\Gamma$ has type $\affE_6,\affE_7,$ or $\affE_8$. Suppose $\Gamma$ is of one of these exceptional types. Consider the vertex $u\in\Vert(G)$ with the maximum value of $\v_u$. Thus $\v_u=3$ for $\affE_6$, $\v_u=4$ for $\affE_7$, and $\v_u=6$ for $\affE_8$, see Figure~\ref{fig:affADE}. Note that since $G$ is a bigraph, $\Gamma$ and $\Delta$ do not share edges. In particular, the neighbors of $u$ in $\Gamma$ cannot be the neighbors of $u$ in $\Delta$. It remains to note that the sum of $\v_w$ over $w\in\Vert(G)$ with $\e_w\neq \e_u$ and $(u,v)\not\in\Gamma$ is less than $2\v_u$ so \eqref{eq:self} cannot hold even if $u$ is connected to all available vertices of $G$. 
 
 Let us now assume that $\Gamma$ is of type $\affD_{n+2}$ for $n\geq 2$. Thus $G$ has $n+3$ vertices which we denote $v_0^+,v_0^-,v_1,\dots,v_{n-1},v_{n}^+,v_{n}^-$. Here $\Gamma$ consists of edges $(v_i,v_{i+1})$ for $i\in [n-2]$ together with four edges $(v_0^\pm,v_1),(v_{n-1},v_{n}^\pm)$.
 
 Suppose $\Delta$ contains an edge $(v_k,v_{k+m})$ for $k,k+m\in[n-1]$. Among all such edges, choose the one with the minimal value of $m$. Since $v_k$ and $v_{k+m}$ are not neighbors in $\Gamma$, we have $m\geq 2$. There is a red-blue path from $v_{k+1}$ to $v_{k+m}$ so there must be a blue-red path as well by Corollary~\ref{cor:recurrent_commuting}. Since $v_{k+1}$ cannot be connected to $v_{k+m-1}$ by a blue edge, we have either $k+m=n-1$ or $(v_{k+1},v_{k+m+1})\in\Delta$. Similarly, we have either $k=1$ or $(v_{k-1},v_{k+m-1})\in\Delta$. Thus for every $i=1,2,\dots, n-m-1$, we have an edge $(v_i,v_{i+m})\in\Delta$. Consider the blue edge $(v_1,v_{1+m})$. There is a red-blue path from $v_m$ to $v_1$ so without loss of generality we may assume that $(v_0^+,v_m)\in\Delta$. But then there is a blue-red path from $v_0$ to $v_{m-1}$ so there must be a red-blue path which necessarily passes through $v_1$, the only red neighbor of $v_0^+$. Thus $(v_1,v_{m-1})\in\Delta$, a contradiction. This shows that $\Delta$ has no edges of the form $(v_k,v_{k+m})$ for $k,k+m\in[n-1]$. Now, consider any vertex $v_k$ for $k\in [n-1]$. It can only be connected by a blue edge to $v_0^\pm$ and $v_{n}^\pm$, and by~\eqref{eq:self}, it is connected to all these four vertices. This holds for any $k\in[n-1]$ but by~\eqref{eq:self} applied to $v_0^+$, there can be only one such vertex. It follows that $n=2$ in which case the edge $(v_0^+,v_1)$ belongs to both $\Gamma$ and $\Delta$ which is impossible. Thus there are no self bindings with $\Gamma$ being of type $\affD_{n+2}$.
 
 Finally, suppose $\Gamma$ is of type $\affA_{2n-1}$ Let $v_1,\dots,v_{2n}$ be the vertices of $\Gamma$, and we label them cyclically so that $v_{2n+1}=v_1$, etc. Let $(v_k,v_{k+2m-1})\in\Delta$ be an edge with the minimal positive value of $m$, where again $m\geq 1$ (in fact, $m\geq 2$ because $(v_k,v_{k+1})$ is already an edge of $\Gamma$). Then by the above reasoning, we have $(v_i,v_{i+2m-1})\in\Delta$ for every $i\in [2n]$. By~\eqref{eq:self}, there are no other edges in $\Delta$. We get precisely the bigraph $\Toric(\affA_{2n-1},\exp(\pi i (2m-1)/n),1)$. The fact that any two of these bigraphs are not isomorphic follows from Proposition~\ref{prop:conjugacy}.
\end{proof}

Thus every \affaff self binding appears as a special case for $n=1$ in family~\bg{toric-A-rotn}. 

\subsection{Affine $\boxtimes$ affine double bindings: preliminaries}
We are left with the case~\eqref{item:small} where both $S(G)$ and $S(G^\opp)$ belong to the set $\{A_1^\parr1, A_2^\parr2\}$. This implies that each of $G$ and $G^\opp$ is an \affaff double binding. Proving Theorem~\ref{thm:classification} reduces to showing the following.

\begin{theorem}
	Suppose that both $S(G)$ and $S(G^\opp)$ belong to the set $\{A_1^\parr1, A_2^\parr2\}$. Then either $G$ or $G^\opp$ belongs to one of the families~\bg{tensor}, \bg{twist}, \bg{toric-A-rotn}, \bg{toric-D4-4}, \bg{pstwist}, or \bg{E8E8}--\bg{A1D4}.
\end{theorem}
\begin{proof}
We use the notation of Section~\ref{sect:double_bindings_structure}: let $X$ and $Y$ be the two red connected components of $G$, every edge of $\Delta$ connects a vertex of $X$ to a vertex of $Y$. We let $\v$ be the common eigenvector for $A_\Gamma$ and $A_\Delta$ from Lemma~\ref{lemma:eigenvalues}, and we denote by $\v_X$ and $\v_Y$ the additive functions for $\Gamma(X)$ and $\Gamma(Y)$ from Figure~\ref{fig:affADE}.

Recall that $G$ is a double binding \emph{of type $\affL\ast\affL'$} if $X$ has type $\affL$ and $Y$ has type $\affL'$.

\begin{definition}
	We say that $G$ is a double binding \emph{of type $\affL\DLRA \affL'$} if $X$ has type $\affL$ and $Y$ has type $\affL'$ and $\scf(G)=(2,2)$. We similarly introduce double bindings \emph{of type $\affL\QLA \affL'$} and \emph{of type $\affL\QRA \affL'$} for the cases $\scf(G)=(1,4)$ and $\scf(G)=(4,1)$ respectively.
\end{definition}

As it follows from~\eqref{eq:affH}, the values of $\v$, $\v_X$, and $\v_Y$ are related as follows. For any $u\in X$ and $v\in Y$, we have
\begin{itemize}
	\item $\v(u)=\v_X(u), \v(v)=\v_Y(v)$ if $G$ is of type $\affL\DLRA \affL'$;
	\item $\v(u)=2\v_X(u), \v(v)=\v_Y(v)$ if $G$ is of type $\affL\QLA \affL'$;
	\item $\v(u)=\v_X(u), \v(v)=2\v_Y(v)$ if $G$ is of type $\affL\QRA \affL'$.
\end{itemize}

Thus if we know $\descr(G)$ then we know $\v$. 

\def\values{M}
\begin{definition}
	Given an affine $ADE$ Dynkin diagram $\affL$, denote by $\values(\affL)$ the \emph{multiset of values} of $\v_\affL$. We also define $\values_0(\affL)$ and $\values_1(\affL)$ to be the multisets of values of $\v_\affL$ restricted to the set of black (resp., white) vertices of $\affL$. For example, $\values(\affE_6)=\{1,1,1,2,2,2,3\}$ which splits into $\values_0(\affE_6)=\{1,1,1,3\}$ and $\values_1(\affE_6)=\{2,2,2\}$.
\end{definition}

We denote $X=X_0\sqcup X_1$ and $Y=Y_0\sqcup Y_1$ the partitions of $X$ and $Y$ into sets of vertices that have the same color. Thus either of the following is true:
\begin{itemize}
	\item every edge of $\Delta$ connects a vertex of $X_i$ to a vertex of $Y_{1-i}$ for $i=0,1$;
	\item every edge of $\Delta$ connects a vertex of $X_i$ to a vertex of $Y_i$ for $i=0,1$.
\end{itemize}

We denote by $\values(X_0)$ the multiset of values of $\v$ restricted to $X_0$. We similarly define $\values(X_1),\values(Y_0),\values(Y_1)$. We identify two multisets if one of them is obtained from another one via rescaling every element by a positive real number. The following lemma is immediate.

\begin{lemma}
	Suppose that $G$ and $G^\opp$ are both double bindings and suppose that $G^\opp$ has type $\affL\ast\affL'$. Then after a possible swapping of $\affL$ with $\affL'$, one of the following holds:
	\begin{itemize}
		\item $\values(X_0)=\values_i(\affL)$ and $\values(Y_0)=\values_{1-i}(\affL)$ for some $i\in\{0,1\}$;
		\item $\values(X_0)=\values_i(\affL)$ and $\values(Y_1)=\values_{1-i}(\affL)$ for some $i\in\{0,1\}$.
	\end{itemize}
\end{lemma}

Let us sum up these observations. Suppose that both $G$ and $G^\opp$ are double bindings and let $G^\opp$ have type $\affL\ast\affL'$. If we know $\descr(G)$  then we know $\v$ and this gives us the multisets $\values(X_0),\values(X_1),\values(Y_0),\values(Y_1)$. There are two options for the components $\affL$ and $\affL'$ of $\Delta$ described in the above lemma. Since an affine $ADE$ Dynkin diagram can be recovered from its additive function, we get that knowing $\descr(G)$ gives two possibilities for the types of its blue components. In each case, we recover $\descr(G^\opp)$.

We now finish the proof with a simple case analysis according to the types of $X$ and $Y$.


\subsection{Affine $\boxtimes$ affine double bindings involving type $\affE$}\label{sect:aff_E}

\subsubsection{The case $\affE\ast\affE$}
Applying~\eqref{eq:affH}, we get that there are the following possibilities for $\descr(G)$:
\begin{enumerate}[(a)]
	\item\label{item:E6E6} $\affE_6\DLRA\affE_6$;
	\item\label{item:E7E7} $\affE_7\DLRA\affE_7$;
	\item\label{item:E8E8} $\affE_8\DLRA\affE_8$.
\end{enumerate}
Indeed, all the other cases are impossible since the ratio of any pair of numbers from $\{24,48,120\}$ is not equal to $4$.

From looking at the multisets of values, it follows that for each case we have $\descr(G)=\descr(G^\opp)$. This actually implies that in the first two cases, $G$ is a twist. For example, in case~\eqref{item:E6E6}, take a vertex $v_1$ of $X$ with $\v_X(v_1)=1$. It must be connected to a vertex $u_2$ of $Y$ with $\v_Y(u_2)=2$ which in turn must be connected to a vertex $v_3$ of $X$ with $\v_X(v_3)=3$. Thus $u_2$ is not connected to anything else in $\Delta$. We can now consider another leaf $v_1'\in X$ with $\v_X(v_1')=1$. It must be connected to a vertex $u_2'\in Y$ with $\v_Y(u_2')=2$ and by the above observation, $u_2'\neq u_2$. Continuing in this fashion, we get that $G$ is a twist $\affE_6\times\affE_6$. 

Let us now give a similar but a bit more complicated argument for $\affE_7$. Let $v_4$ and $u_4$ be the vertices of $X$ and $Y$ respectively with $\v_X(v_4)=\v_Y(u_4)=4$. Then they must be of the same color since $v_4$ must be connected to the neighbors $u_3$ and $u_3'$ of $u_4$ with $\v_Y(u_3)=\v_Y(u_3')=3$. This implies that $u_3$ is connected to a vertex $v_2$ of $X$ with $\v_X(v_2)=2$ that has the same color as $v_4$ and now we this argument is finished in the same way as our proof for $\affE_6$.

For the third case, we actually get something besides twists, namely, the bigraph~\bg{E8E8}. Label the vertices of the two copies of $\affE_8$ with $v_1$ through $v_9$ and with $u_1$ through $u_9$ as shown in Figure~\ref{fig:aff12}. 

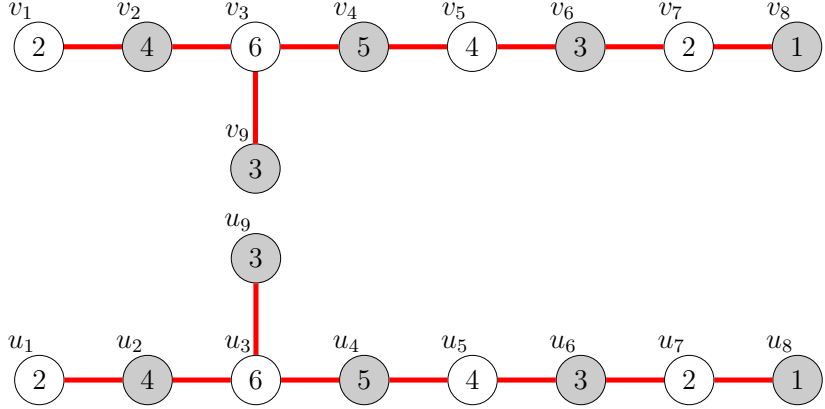
\begin{figure}[ht]

\begin{tabular}{c}
\scalebox{0.9}{
\begin{tikzpicture}
\node[draw,circle,fill=black!20!white] (v0) at (3.20,-1.80) {3};
\node[anchor=south] (label0) at (v0.135) {$v_9$};
\node[draw,circle,fill=white] (v1) at (0.00,0.00) {2};
\node[anchor=south] (label1) at (v1.135) {$v_1$};
\node[draw,circle,fill=black!20!white] (v2) at (1.60,0.00) {4};
\node[anchor=south] (label2) at (v2.135) {$v_2$};
\node[draw,circle,fill=white] (v3) at (3.20,0.00) {6};
\node[anchor=south] (label3) at (v3.135) {$v_3$};
\node[draw,circle,fill=black!20!white] (v4) at (4.80,0.00) {5};
\node[anchor=south] (label4) at (v4.135) {$v_4$};
\node[draw,circle,fill=white] (v5) at (6.40,0.00) {4};
\node[anchor=south] (label5) at (v5.135) {$v_5$};
\node[draw,circle,fill=black!20!white] (v6) at (8.00,0.00) {3};
\node[anchor=south] (label6) at (v6.135) {$v_6$};
\node[draw,circle,fill=white] (v7) at (9.60,0.00) {2};
\node[anchor=south] (label7) at (v7.135) {$v_7$};
\node[draw,circle,fill=black!20!white] (v8) at (11.20,0.00) {1};
\node[anchor=south] (label8) at (v8.135) {$v_8$};
\draw[color=red,line width=0.75mm] (v2) to[] (v1);
\draw[color=red,line width=0.75mm] (v3) to[] (v0);
\draw[color=red,line width=0.75mm] (v3) to[] (v2);
\draw[color=red,line width=0.75mm] (v4) to[] (v3);
\draw[color=red,line width=0.75mm] (v5) to[] (v4);
\draw[color=red,line width=0.75mm] (v6) to[] (v5);
\draw[color=red,line width=0.75mm] (v7) to[] (v6);
\draw[color=red,line width=0.75mm] (v8) to[] (v7);
\end{tikzpicture}}
\\

\scalebox{0.9}{
\begin{tikzpicture}
\node[draw,circle,fill=black!20!white] (v0) at (3.20,1.80) {3};
\node[anchor=south] (label0) at (v0.135) {$u_9$};
\node[draw,circle,fill=white] (v1) at (0.00,-0.00) {2};
\node[anchor=south] (label1) at (v1.135) {$u_1$};
\node[draw,circle,fill=black!20!white] (v2) at (1.60,-0.00) {4};
\node[anchor=south] (label2) at (v2.135) {$u_2$};
\node[draw,circle,fill=white] (v3) at (3.20,-0.00) {6};
\node[anchor=south] (label3) at (v3.135) {$u_3$};
\node[draw,circle,fill=black!20!white] (v4) at (4.80,-0.00) {5};
\node[anchor=south] (label4) at (v4.135) {$u_4$};
\node[draw,circle,fill=white] (v5) at (6.40,-0.00) {4};
\node[anchor=south] (label5) at (v5.135) {$u_5$};
\node[draw,circle,fill=black!20!white] (v6) at (8.00,-0.00) {3};
\node[anchor=south] (label6) at (v6.135) {$u_6$};
\node[draw,circle,fill=white] (v7) at (9.60,-0.00) {2};
\node[anchor=south] (label7) at (v7.135) {$u_7$};
\node[draw,circle,fill=black!20!white] (v8) at (11.20,-0.00) {1};
\node[anchor=south] (label8) at (v8.135) {$u_8$};
\draw[color=red,line width=0.75mm] (v2) to[] (v1);
\draw[color=red,line width=0.75mm] (v3) to[] (v0);
\draw[color=red,line width=0.75mm] (v3) to[] (v2);
\draw[color=red,line width=0.75mm] (v4) to[] (v3);
\draw[color=red,line width=0.75mm] (v5) to[] (v4);
\draw[color=red,line width=0.75mm] (v6) to[] (v5);
\draw[color=red,line width=0.75mm] (v7) to[] (v6);
\draw[color=red,line width=0.75mm] (v8) to[] (v7);
\end{tikzpicture}}
\\

\end{tabular}

\caption{\label{fig:aff12} Labeling the nodes of $X$ and $Y$.}
\end{figure}
From looking at the multisets of values, we get that the vertices $v_3$ and $u_3$ have the same color. We see that there are {\red {two cases}} to consider: $v_3$ is connected to either $u_2,u_4,u_9$ or to $u_2,u_4,u_6$.

{\red {Assume}} $v_3$ is connected to $u_2,u_4,u_6$. The same consideration as for $v_3$ can be applied to $u_3$, with the only possible choice now of $u_3$ being connected to $v_2,v_4,v_6$. Since $u_2$ is connected to $v_3$, its only other blue neighbor has to be a vertex $v$ with $\v_X(v)=2$ so either $v=v_1$ or $v=v_7$. Counting the red-blue and blue-red paths\footnote{Throughout the text, the phrase \emph{counting the red-blue and blue-red paths} refers to Corollary~\ref{cor:recurrent_commuting} and the discussion after it.} between $u_2$ and $v_6$, we conclude that $u_2$ is connected to $v=v_7$. Since $u_8$ is connected to either $v_7$ or $v_1$, it follows that $u_8$ is connected to $v_1$. Since the vertex $u_6$ is already connected to $v_3$, it cannot be connected to $v_1$, and now it follows that $\Delta$ contains a path with vertices 
\[u_8,v_1,u_9,v_5,u_4,v_3,u_2,v_7\] 
which determines the graph uniquely and we see that it is exactly the bigraph~\bg{E8E8}.

{\red {Assume}} $v_3$ is connected to $u_2,u_4,u_9$. The same consideration as for $v_3$ can be applied to $u_3$, with the only possible choice now of $u_3$ being connected to $v_2,v_4,v_9$. We know that $v_2$ is connected to either $u_1$ or $u_7$ and counting the red-blue and blue-red paths between $u_1$ and $v_3$ we conclude $v_2$ is connected to $u_1$. We now have enough information to conclude that $\Delta$ contains a path with vertices 
\[u_1,v_2,u_3,v_4,u_5,v_6,u_7,v_8\]
which implies that $G$ is a twist $\affE_8\times\affE_8$.

\subsubsection{The case $\affE_6*\affA$}

The component $X$ of type $\affE_6$ has a vertex $v_3$ with $\v_X(v_3)=3$ and therefore it cannot be connected to anything by an edge of $\Delta$ unless $\scf(G)=(1,4)$. Thus we only need to find all double bindings of type $\affE_6\QRA\affA_5$. 

It follows from looking at the multisets of values that $\descr(G^\opp)=\descr(G)$ in this case, and because of the symmetries of $X$ and $Y$, there is essentially one way to connect the vertices $v_2,v_2',v_2''$ of $X$ with $\v_X(v_2)=\v_X(v_2')=\v_X(v_2'')=2$ to the vertices $u_2,u_4,u_6$ of $Y$ to form a $6$-cycle in $\Delta$. The rest of the edges are reconstructed uniquely and we get the bigraph~\bg{A5E6}.

\subsubsection{The case $\affE_7*\affA$}

The component $X$ of type $\affE_7$ has a vertex $v_4$ with $\v_X(v_4)=4$ and therefore it cannot be connected to anything by an edge of $\Delta$ unless $\scf(G)=(1,4)$. Thus we only need to find all double bindings of type $\affE_7\QRA\affA_{11}$. We get an immediate contradiction from looking at the multisets of values since there is no affine $ADE$ Dynkin diagram $\affL$ with $\values_0(\affL)=\{2,2,2,2,2,2\}$ and $\values_1(\affL)=\{1,1,2,3,3\}$.

\subsubsection{The case $\affE_8*\affA$}

The component $X$ of type $\affE_8$ has a vertex $v_6$ with $\v_X(v_6)=6$ and therefore it cannot be connected to anything by an edge of $\Delta$, regardless of $\scf(G)$.

\subsubsection{The case $\affE_6*\affD$}
Let $X$ be the component of type $\affE_6$ and let $Y$ the the component of type $\affD$. The cases $\scf(G)=(4,1)$ or $\scf(G)=(1,4)$ are impossible. Indeed, if $Y$ gets a scaling factor of $1$, none of its vertices can be connected to the vertex $v_3$ of $X$ with $\v_X(v_3)=3$.
If on the other hand $Y$ gets scaling factor of $4$, there is just not enough vertices in $X$ of the same color to collect $\v_Y(u) \times 4 = 8$ for a vertex $u \in Y$ with $\v_Y(u)=2$. 

Thus $\scf(G)=(2,2)$ and $\descr(G)=\descr(G^\opp)=\affE_6\DLRA\affD_8$. Label vertices of $X$ with $v_1$ through $v_7$, and label vertices of $Y$ with $u_1$ through $u_9$ as shown in Figure~\ref{fig:aff23}. 

\begin{figure}

\begin{tabular}{c}
\scalebox{0.9}{
\begin{tikzpicture}
\node[draw,circle,fill=white] (v0) at (6.40,0.00) {1};
\node[anchor=south] (label0) at (v0.135) {$v_1$};
\node[draw,circle,fill=black!20!white] (v1) at (4.80,0.00) {2};
\node[anchor=south] (label1) at (v1.135) {$v_2$};
\node[draw,circle,fill=white] (v2) at (0.48,-0.60) {1};
\node[anchor=south] (label2) at (v2.135) {$v_5$};
\node[draw,circle,fill=black!20!white] (v3) at (2.08,-0.60) {2};
\node[anchor=south] (label3) at (v3.135) {$v_4$};
\node[draw,circle,fill=white] (v4) at (3.20,0.00) {3};
\node[anchor=south] (label4) at (v4.135) {$v_3$};
\node[draw,circle,fill=black!20!white] (v5) at (1.12,0.60) {2};
\node[anchor=south] (label5) at (v5.135) {$v_6$};
\node[draw,circle,fill=white] (v6) at (-0.48,0.60) {1};
\node[anchor=south] (label6) at (v6.135) {$v_7$};
\draw[color=red,line width=0.75mm] (v1) to[] (v0);
\draw[color=red,line width=0.75mm] (v3) to[] (v2);
\draw[color=red,line width=0.75mm] (v4) to[] (v1);
\draw[color=red,line width=0.75mm] (v4) to[] (v3);
\draw[color=red,line width=0.75mm] (v5) to[] (v4);
\draw[color=red,line width=0.75mm] (v6) to[] (v5);
\end{tikzpicture}}
\\

\scalebox{0.9}{
\begin{tikzpicture}
\node[draw,circle,fill=white] (v0) at (-4.56,-0.80) {1};
\node[anchor=south] (label0) at (v0.135) {$u_1$};
\node[draw,circle,fill=white] (v1) at (-4.00,0.80) {1};
\node[anchor=south] (label1) at (v1.135) {$u_2$};
\node[draw,circle,fill=black!20!white] (v2) at (-2.40,0.00) {2};
\node[anchor=south] (label2) at (v2.135) {$u_3$};
\node[draw,circle,fill=white] (v3) at (-0.80,0.00) {2};
\node[anchor=south] (label3) at (v3.135) {$u_4$};
\node[draw,circle,fill=black!20!white] (v4) at (0.80,0.00) {2};
\node[anchor=south] (label4) at (v4.135) {$u_5$};
\node[draw,circle,fill=white] (v5) at (2.40,0.00) {2};
\node[anchor=south] (label5) at (v5.135) {$u_6$};
\node[draw,circle,fill=black!20!white] (v6) at (4.00,0.00) {2};
\node[anchor=south] (label6) at (v6.135) {$u_7$};
\node[draw,circle,fill=white] (v7) at (6.16,-0.80) {1};
\node[anchor=south] (label7) at (v7.135) {$u_8$};
\node[draw,circle,fill=white] (v8) at (5.60,0.80) {1};
\node[anchor=south] (label8) at (v8.135) {$u_9$};
\draw[color=red,line width=0.75mm] (v2) to[] (v0);
\draw[color=red,line width=0.75mm] (v2) to[] (v1);
\draw[color=red,line width=0.75mm] (v3) to[] (v2);
\draw[color=red,line width=0.75mm] (v4) to[] (v3);
\draw[color=red,line width=0.75mm] (v5) to[] (v4);
\draw[color=red,line width=0.75mm] (v6) to[] (v5);
\draw[color=red,line width=0.75mm] (v7) to[] (v6);
\draw[color=red,line width=0.75mm] (v8) to[] (v6);
\end{tikzpicture}}
\\

\end{tabular}

    \caption{\label{fig:aff23}  Labeling the nodes of $X$ and $Y$.}
\end{figure}
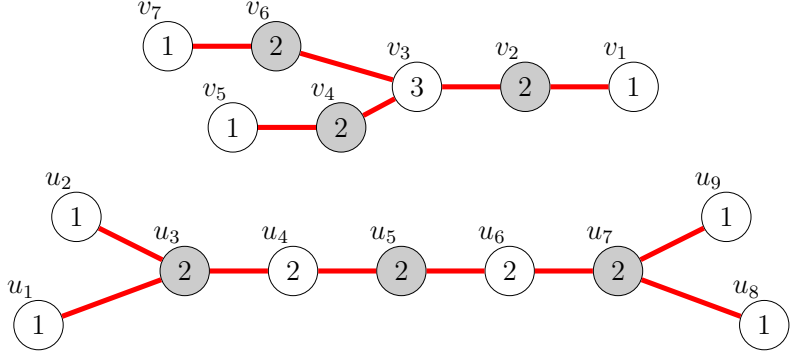

Assume $v_1$ is white, and then it follows that $u_1$ is also white from looking at the multisets of values. By the same reason, $v_3$ is connected to $u_3$, $u_5$ and $u_7$, which must be connected to $v_1,v_5,v_7$. Without loss of generality we can assume that we have edges $(u_3,v_1)$, $(u_5,v_7)$, $(u_7,v_5)$. This determines the position of the blue component of type $\affE_6$ from which we can uniquely reconstruct the edges of the blue component of type $\affD_8$ by counting the corresponding red-blue paths: we get that $v_2$ is connected to $u_1,u_2$, $v_4$ is connected to $u_8,u_9$, and thus $v_6$ is connected to $u_4,u_6$ which are connected to $v_2$ and $v_4$ or vice versa, and we see that in any case we get a contradiction.

\subsubsection{The case $\affE_7*\affD$}
Let $X$ be the component of type $\affE_7$ and let $Y$ the the component of type $\affD$. Our three possibilities by~\eqref{eq:affH} are: 
\begin{itemize}
	\item $\affE_7\DLRA\affD_{14}$;
	\item $\affE_7\QLA\affD_{50}$;
	\item $\affE_7\QRA\affD_{5}$;
\end{itemize}

Since $\values_0(X)=\{1,1,3,3\}$ and since there is no Dynkin diagram $\affL$ for which the maximum of $\v(\affL)$ is greater than $2$ and is achieved more than once, it follows that only the third case is possible. From looking at the multisets of values, we get that $\descr(G)=\descr(G^\opp)=\affE_7\QRA\affD_{5}$, and in fact all the edges of $\Delta$ can be immediately reconstructed from knowing the additive function for each component of $\Delta$, yielding the bigraph~\bg{D5E7}.

\subsubsection{The case $\affE_8*\affD$}
 The cases $\affE_8\DLRA\affD_{m+2}$ and $\affE_8\QLA\affD_{m+2}$ are impossible because the vertex in $\affE_8$ with $\v=6$ cannot be connected to any vertex in $\affD_{m+2}$. The only possibility is the case $\affE_8\QRA\affD_m$ which is also impossible because by~\eqref{eq:affH}, $m$ must satisfy $4\times 4m=120$, but $120$ is not divisible by $16$.

\subsection{Affine $\boxtimes$ affine double bindings involving type $\affA$}

\subsubsection{The case $\affA*\affA$}
By~\eqref{eq:affH}, there are two possibilities:
\begin{itemize}
	\item $\affA_{2n-1}\DLRA\affA_{2n-1}$, $n\geq 1$;
	\item $\affA_{8n-1}\QRA\affA_{2n-1}$, $n\geq 1$.
\end{itemize}

In the second case, the multisets of values tell us that the red components of $G^\opp$ are $2n$ copies of $\affD_4$, and since $G^\opp$ is also required to be a double binding, we get that $n=1$, in which case there is only one possible double binding which already is listed in family~\bg{toric-D4-4}. 

In the first case, the multisets of values tell us that $\descr(G)=\descr(G^\opp)=\affA_{2n-1}\DLRA\affA_{2n-1}$. Label the vertices of $X$ with $v_1$ through $v_{2n}$, and label the vertices of $Y$ with $u_1$ through $u_{2n}$ (we will be taking the indices modulo $2n$). Since the case $n=1$ is trivial, we assume that $\Delta$ has no double edges. 

\begin{lemma} \label{lem:rot}
 Assume edges $(v_i,u_j)$ and $(v_{i+1},u_{j+1})$ are present. Then so are the edges $(v_{i+2},u_{j+2})$ and $(v_{i-1},u_{j-1})$. Similarly, assume the edges $(v_i,u_j)$ and $(v_{i+1},u_{j-1})$ are present. Then so are edges $(v_{i+2},u_{j-2})$ and $(v_{i-1},u_{j+1})$.
\end{lemma}

\begin{proof}
 Due to symmetry, it is enough to prove the first claim. Again due to symmetry, it is enough to argue $(v_{i+2},u_{j+2})$ exists. Assume not. Counting blue-red and red-blue paths between $v_{i+1}$ and $u_{j+2}$ we see that the edge $(v_i,u_{j+2})$ must exist. Similarly the edge 
 $(u_j,v_{i+2})$ exists. Counting blue-red and red-blue paths between $u_{j+1}$ and $v_i$ we see that the edge $(v_{i-1},u_{j+1})$ exists. A similar logic gives us the edge $(v_{i+1},u_{j-1})$. Continuing this way we get edges between $v_{i-k}$ and $v_{j+k}$ and $v_{j+k+2}$ for $k=1,2,\ldots$. For $k=2n-2$
 we see that $v_{i+2}$ is connected to $u_{j+2}$ after all -- a contradiction to our assumption. 
\end{proof}

Now it is easy to see that all edges of $\Delta$ consist of two such families as in Lemma \ref{lem:rot}: each family consisting of all $v_iu_{i \pm k}$ for fixed $k$. Furthermore, if one family has a plus sign and the other has a minus sign, we would have a double edge. Therefore both families have the same
sign, which without loss of generality we can assume to be plus. Thus, there are two fixed choices $k,k'$ of residues modulo $2n$ such that the edges of $\Delta$ are $v_iu_{i+k}$ and $v_iu_{i+k'}$ for all $i$. Such a bigraph is therefore listed in family~\bg{toric-A-rotn}.

\subsubsection{The case $\affA*\affD$}
Since the component $Y$ of type $\affD_{n+2}$ has a vertex $v$ with $\v_Y(v)=2$, we cannot have a bigraph of type $\affA_{2n-1}\QRA \affD_n$. Thus the only two possibilities that we have are:
\begin{itemize}
	\item $\affA_{4n-1}\DLRA\affD_{n+2}$, $n\geq 2$;
	\item $\affA_{2n-1}\QLA\affD_{2n+2}$, $n\geq 1$.
\end{itemize}

In the first case, the multisets of values tell us the following. If $n=2k+2$ is even for $k\geq 0$, one blue component will have $2n=4k+4$ ones and $k$ twos and the other component will have $2n+4$ ones and $k+1$ twos. Thus the second component must be of type $\affD$ which implies that $n=0$, a contradiction. Now if $n=2k+1$ is odd for $k\geq 1$, each blue component has $2n+2=4k+4$ ones and $k$ twos which is also impossible. 

For the case $\affA_{2n-1}\QLA\affD_{2n+2}$, one blue component has $2n$ twos while the other blue component has $4$ ones and $2n-1$ twos. Thus the first component must be of type $\affA_{2n-1}$ and the second component must be of type $\affD_{2n+2}$ so we get $\descr(G)=\descr(G^\opp)$. 

Let $v_1$ and $v_2$ be two vertices with $\v_X(v_1)=\v_X(v_2)=1$ adjacent to a vertex $v_3$ in $X$ with $\v_X(v_3)=2$. Label the vertices of $Y$ as $u_1$ through $u_{2n}$. Let $v_2$ be connected to $u_1\in Y$. Without loss of generality we can assume $v_3$ is connected to $u_2$. 
Counting red-blue and blue-red paths between $v_1$ and $u_2$ we see that there are {\red {two options}}: $v_1$ is either connected to $u_1$, or to $u_3$.

{\red {Assume}} $v_1$ is connected to $u_3$. Counting the red-blue and blue-red paths between $v_1$ and $u_4$ we see that $v_3$ has to be connected to $u_4$. Counting the red-blue and blue-red paths between $v_2$ and $u_{2n}$ we see that $v_3$ has to be connected to $u_{2n}$.
By~\eqref{eq:points_1} applied to $v_3$ we see that this is impossible unless $u_{2n} = u_4$, that is, $n=2$. In this case $v_4$ has to be connected to $u_1$ and $u_3$ and now we uniquely recover another exceptional bigraph listed as~\bg{A3D6} in our classification.

{\red {Assume}} now $v_1$ is connected to $u_1$. Counting red-blue and blue-red paths between $v_3$ and $u_1$ we see that there are {\blue {two options}} to consider: either $v_3$ is connected by a double edge to $u_2$, or it is connected by a single edge to $u_2$ and a single edge to $u_{2n}$. 

In the {\blue {former}} case, we get a component of type $\affA_1$ so $n=1$ and we get the bigraph~\bg{A1D4}.

Consider now the {\blue {latter}} case. Counting red-blue and blue-red paths between $v_3$ and $u_3$ and $u_{2n-1}$ we see that again there are {\blue {two options}}: either $2n-1=3$ and $v_4$ is connected to it by a double edge, or not, in which case $v_4$ is connected to both $u_3$ and $u_{2n-1}$.
Continuing in this manner we arrive to the moment when $v_{n+2}$ is connected to $u_{n+1}$ by a double edge. Thus, the {\blue {option}} of a double edge does get realized sooner or later, with the only choice of how soon it comes to be. But for $n>1$ this contradicts the assumption that $G^\opp$ is a double binding.

\subsection{Affine $\boxtimes$ affine double bindings involving only type $\affD$}\label{sect:double_D}
By~\eqref{eq:affH}, we have the following two possibilities:

\begin{itemize}
	\item $\affD_{n+2}\DLRA\affD_{n+2}$, $n\geq 2$;
	\item $\affD_{n+2}\QLA\affD_{4n+2}$, $n\geq 2$.
\end{itemize}

Let us start with the {\blue {second case}}. {\red {Assume}} that $n=2k$ is even for some $k\geq 1$. Without loss of generality we may assume that 
\[\values(X_0)=\{2,2,2,2,\underbrace{4,\dots,4}_{k-1}\},\quad \values(X_1)=\{\underbrace{4,\dots,4}_{k}\};\]
\[\values(Y_0)=\{1,1,1,1,\underbrace{2,\dots,2}_{4k-1}\},\quad \values(Y_1)=\{\underbrace{2,\dots,2}_{4k}\}.\]
We see that if $k>1$ then there is no way to combine $\values(X_0)$ with either $\values(Y_0)$ or $\values(Y_1)$ to get $\values(\affL)$ in the union for any affine $ADE$ Dynkin diagram $\affL$. For $k=1$, the blue components are necessarily $X_0\cup Y_0$ of type $\affD_{10}$ and $X_1\cup Y_1$ of type $\affD_4$ so we get $\descr(G)=\descr(G^\opp)=\affD_{4}\QLA\affD_{10}$. We would like to show that such a double binding does not exist. Let $v$ be the unique vertex of $X$ of degree $4$ and suppose that it is black. Then it is connected to all four white vertices of $Y$, and each of them satisfies $\v_Y(u)=2$. By repeatedly counting red-blue and blue-red paths, we recover the dual of~\bg{path-D4-12-1324} with $\descr(G^\opp)=\affD_6\DRA\affD_4\DRA\affA_3$. In particular, $G^\opp$ has three red components so is not a double binding.

{\red {Assume}} now that $n=2k+1$ is odd for some $k\geq 1$. Then without loss of generality we get that 
\[\values(X_0)=\{2,2,\underbrace{4,\dots,4}_{k}\},\quad \values(X_1)=\{2,2,\underbrace{4,\dots,4}_{k}\};\]
\[\values(Y_0)=\{1,1,1,1,\underbrace{2,\dots,2}_{4k+1}\},\quad \values(Y_1)=\{\underbrace{2,\dots,2}_{4k+2}\}.\]
We see that there is no way to combine $\values(Y_0)$ with either $\values(X_0)$ or $\values(X_1)$ to get $\values(\affL)$ because if $1,4\in\values(\affL)$ then we must also have $3\in\values(\affL)$.

For the {\blue {first case}}, both blue components have to also have type $\affD_{n+2}$ so we get \[\descr(G)=\descr(G^\opp)=\affD_{n+2}\DLRA\affD_{n+2}.\]
First, {\red{assume}} $n=2$. Since the blue components are two copies of $\affD_4$, the unique way to get them is to connect the vertex of $X$ of red degree $4$ to all the leaves of $Y$ and vice versa. We obtain a twist $\affD_4\times\affD_4$.

{\red {Suppose}} now that $n\geq3$. We would like to show that in this case we get family~\bg{pstwist} or~\bg{twist}, i.e., $G=\affD_{n+2}\pstwist{p}\affD_{n+2}$ for some $p\in [n-1]$. Recall that the bigraphs $\affD_{n+2}\pstwist{p}\affD_{n+2}$ and $\affD_{n+2}\pstwist{n-p}\affD_{n+2}$ are isomorphic and $\affD_{n+2}\pstwist{1}\affD_{n+2}$ is isomorphic to the twist $\affD_{n+2}\times\affD_{n+2}$.

Let us label the vertices of $X$ and $Y$ as in Section~\ref{sect:pstwist}. Thus the vertices of $X$ are labeled by 
\[u_0^+,u_0^-,u_1,u_2,\dots,u_{n-1}, u_{n}^+, u_{n}^-\]
so that the leaves $u_0^+$ and $u_0^-$ are connected to $u_1$ and the leaves $u_{n}^+$ and $u_{n}^-$ are connected to $u_{n-1}$. Similarly, the vertices of $Y$ are labeled in a similar way by
\[v_0^+,v_0^-,v_1,v_2,\dots,v_{n-1}, v_{n}^+, v_{n}^-.\]
Since we know that the blue components have type $\affD_{n+2}$, we get that the leaves of $X$ are not connected to the leaves of $Y$ by blue edges. Without loss of generality we can assume that $u_0^+$ is connected to some $v_p$, where $p\in[n-1]$. Counting red-blue and blue-red paths we see that $u_1$ is connected to a neighbor $v'$
of $v_p$. Counting red-blue and blue-red paths between $u_0^-$ and $v'$, we get {\color{purple} cases}: $u_0^-$ has to be connected to either $v_p$ or to another neighbor $v''$ of $v'$. 

{\color{purple} Consider} the case when $u_0^-$ is connected to $v''$. Since $u_0^-$ is a leaf of $X$, $v''$ cannot be a leaf of $Y$. Moreover, $v'$ has two different neighbors $v_p$ and $v''$ so it also cannot be a leaf of $Y$. Without loss of generality we can therefore assume that $p\leq n-3$, $v'=v_{p+1}$ and $v''=v_{p+2}$. Hence there exists a path of length $5$ in $Y$ of the form $(v_{p-1}^\parr+,v_p,v_{p+1},v_{p+2},v_{p+3}^\parr+)$, where $v_i^\parr+$ is equal to $v_i$ if $i\neq 0,n$ and to $v_i^+$ if $i=0$ or $i=n$. Counting red-blue and blue-red paths between $v_{p-1}^\parr+$ and $u_0^+$ we get that $v_{p-1}^\parr+$ is connected to $u_1$. Counting red-blue and blue-red paths between $v_{p+3}^\parr+$ and $u_0^-$ we get that $v_{p+3}^\parr+$ is connected to $u_1$ as well. This contradicts~\eqref{eq:points_2} for $u_1$.

{\color{purple} Thus} $u_0^-$ is connected to $v_p$ as well as $u_0^+$. Since $v_p$ is connected to a leaf $u_0^+$ by a blue edge, every red neighbor of $v_p$ must be connected to $u_1$. Thus $u_1$ is connected to $v_{p-1}^\parr+$. {\color{orange} Then either} $p-1=0$ or $p-1>0$. {\color{orange} Similarly either} $p+1=n$ or $p+1<n$. {\color{orange} Assume} that $1<p<n-1$. Note that $v_p$ cannot be connected to $u_2$ since counting red-blue and blue-red paths between $v_p$ and $u_1$ we would arrive at a contradiction. Then counting red-blue and blue-red paths between $u_2$ and $v_{p-1},v_{p+1}$ we conclude that $v_{p-2}^\parr+$ and $v_{p+2}^\parr+$ exist and are connected to $u_2$. Again, {\color{orange} either}  $p-2=0$ or $p-2>0$, and either $p+2=n$ or $p+2<n$. Also, again, $v_{p \pm 1}$ are not connected to $u_3$, since otherwise we get a contradiction by counting red-blue and blue-red paths between $v_{p \pm 1}$ and $u_2$, etc. Continuing this way we get $u_{i}$ to be connected to $v_{p \pm i}$ for $i=0,1,\dots$. 

{\color{orange} Eventually} we arrive at a situation when without loss of generality $p-i=0$, that is, $i=p$. Thus, $v_0^+$ is connected to $u_p$, and so is $v_{2p}$. 
Counting red-blue and blue-red paths between $v_0^+$ and $u_{p+1}$ we conclude that $v_1$ is connected to $u_{p+1}$. Counting red-blue and blue-red paths between $v_0^-$ and $u_{p+1}$ we conclude that $v_0^-$ is connected to 
either $u_{p+2}$ or to $u_p$. The former case as before leads to a contradiction. In the latter case we proceed as before, with $u$-s and $v$-s swapped, concluding the existence of edges between $v_{j}$ and $u_{p \pm j}$ for 
$1 \leq j \leq p$. 

Counting blue-red and red-blue paths between $u_{p+1}$ and $v_{2p}$, and taking into account that the edge $(v_{2p-1},u_{p+1})$ does not exist in $\Delta$ since $(v_{2p-2},u_{p})$ does not, we conclude that $u_{p+1}$ is connected to $v_{2p+1}$, etc. Continuing in this manner we get edges connecting $u_{p+j}$ to $v_{2p+j}^\parr+$, and also edges connecting $v_{p+j}$ to $u_{2p+j}^\parr+$ for $1 \leq j \leq n-2p$. Finally, in a symmetric way to the previous argument we obtain edges connecting $u_{n-2p+j}$ to $v_{n-j}$, and also edges connecting $v_{n-2p+j}$ to $u_{n-j}$ for $1 \leq j \leq p$.

As a result we obtain precisely the pseudo twist $\affD_{n+2}\pstwist p \affD_{n+2}$. 
\end{proof}

\section{Twists}\label{sect:twists}
In this section, we concentrate on the case when the bigraph $G=(\Gamma,\Delta)$ is a \emph{twist} $\affL\times\affL$ for some affine $ADE$ Dynkin diagram $\affL$. First, we introduce a certain game one can play on any undirected graph that very much resembles the Kostant's \emph{find the highest root game} which is due to Allen Knutson. We deduce a positivity result for this game from the theory of Kac-Moody algebras~\cite{Kac}. We then give a general construction of a twist for any quiver in Section~\ref{sect:twists:general}. We prove a factorization theorem for any such twist in Section~\ref{sect:twists:any_sequence} thus directly generalizing~\cite[Proposition~2.4]{Gregg} where this was done for the \emph{del Pezzo $3$ quiver} which can be seen as a twist of a triangle with itself as we explain in Section~\ref{sect:twists:general}. Finally, in Section~\ref{sect:twists:entropy} we apply these results to the case $\affL\times\affL$ where $\affL$ is an affine $ADE$ Dynkin diagram and deduce Conjecture~\ref{conj:master} for such twists as a special case.

\def\b{{\mathfrak{b}}}
\def\r{s}
\def\rb{s^\parr\b}
\def\W{V}
\subsection{Reflections on undirected graphs}\label{sect:twists:game}
Let $G=(I,E)$ be a connected undirected graph with possibly multiple edges but no loops. We denote by $\W=\{h:I\to \R\}$ the vector space of all functions from $I$ to $\R$ and for each $i\in I$, denote by $\alpha_i:I\to\R$ the $i$-th basis vector defined by $\alpha_i(j)=\delta_{ij}$. 

Suppose that $G$ has one distinguished vertex $\b\in V$. For every vertex $i\in I$, we define two \emph{reflections} $\r_i,\rb_i:\W\to\W$ as follows. Given a vector $h\in\W$, put
\[(\r_ih)(j)=\begin{cases}
            	\displaystyle h(j),&\text{if $i\neq j$;}\\
            	\displaystyle-h(j)+{\sum_{(j,k)\in E} h(k)},&\text{if $i=j$};\\
            \end{cases}\quad 
  \rb_ih=\begin{cases}
            	\r_ih,&\text{if $i\neq \b$;}\\
            	\r_ih+\alpha_\b,&\text{if $i=\b$.}\\
            \end{cases}\]
\begin{example}
	Let $I=\{1,2,3,4\}$ and let $G$ be an undirected path with edges $E=\{(1,2),(2,3),(3,4)\}$. Thus $\W$ can be identified with $\R^4$. Let $\b=2$ be the distinguished vertex. Suppose that $h=(a,b,c,d)\in\W$, thus for example $h(2)=b$. The values of $\r_2(h),\r_3\r_2(h),\r_2\r_3\r_2(h),$ as well as of $\rb_2(h),\rb_3\rb_2(h)$, and $\rb_2\rb_3\rb_2(h)$ are given in Table~\ref{tab:rb_i}. 
	
	\begin{table}
		\begin{tabular}{ccc}
			\begin{tabular}{|c|cccc|}\hline
			$h$ & $ a$ & $b$ & $c$ & $d$\\\hline
			$\r_2(h)$ & $a$ & $ a+c-b$ & $c$ & $d$\\\hline
			$\r_3\r_2(h)$ & $a$ & $a+c-b$ & $a+d-b$ & $d$\\\hline
			$\r_2\r_3\r_2(h)$ & $a$ & $a+d-c$ & $a+d-b$ & $d$\\\hline
			\end{tabular}\\
			\begin{tabular}{|c|cccc|}\hline
			$h$ & $a$ & $b$ & $c$ & $d$\\\hline
			$\rb_2(h)$ & $a$ & $ a+c-b+1$ & $c$ & $d$\\\hline
			$\rb_3\rb_2(h)$ & $a$ & $a+c-b+1$ & $a+d-b+1$ & $d$\\\hline
			$\rb_2\rb_3\rb_2(h)$ & $a$ & $a+d-c+1$ & $a+d-b+1$ & $d$\\\hline			
			\end{tabular}
		\end{tabular}
		\caption{\label{tab:rb_i}An example of applying the operators $\r_i$ and $\rb_i$.}
	\end{table}

\end{example}

\def\id{ \operatorname{id}}

For an element $h\in\W$, we write $h\geq 0$ if for any $i\in I$, $h(i)\geq 0$. The rest of this section will be concentrated on showing the following result:
\begin{proposition}\label{prop:positivity}
	For any sequence $\i=(i_1,i_2,\dots,i_p)$ of vertices of $G$, we have
	\begin{equation}\label{eq:positivity}
	\rb_{i_p}\rb_{i_{p-1}}\cdots\rb_{i_1} (0)\geq 0.
	\end{equation}
\end{proposition}
\begin{proof}
	First, one easily checks (for example, using Table~\ref{tab:rb_i}) that if the vertices $i$ and $j$ are connected by exactly one edge (resp., zero edges) in $G$, we have $(\rb_i\rb_j)^{m_{ij}}=\id$ for $m_{ij}=3$ (resp., $m_{ij}=2$). Thus the operators $\rb_i$ define a representation of the \emph{Weyl group $W$ of the Kac-Moody algebra} associated to the \emph{generalized Cartan matrix} $A_G=(a_{ij})_{i,j\in I}$ of $G$ defined by
	
	\begin{equation}\label{eq:cartan}
	a_{ij}= \begin{cases}
       	2,&\text{ if $i=j$};\\
       	-q_{ij},&\text{otherwise,}
       \end{cases}
       \end{equation}
where $q_{ij}$ is the number of edges in $G$ connecting $i$ to $j$. This follows from~\cite[Proposition~3.13]{Kac}. Thus we may assume that the word $\i$ is \emph{reduced}, that is, the element $\r_{i_p}\r_{i_{p-1}}\cdots\r_{i_1}$ cannot be represented as a product of less than $p$ elements in $W$. 

We prove~\eqref{eq:positivity} by induction on $p$. The case $p=0$ is trivial so suppose that $p>0$.

Since $\r_i(0)=0$ for all $i\in I$, we may assume that $i_1=\b$. One easily checks that in this case,
\[\rb_{i_p}\rb_{i_{p-1}}\cdots\rb_{i_1} (0)=\rb_{i_p}\rb_{i_{p-1}}\cdots\rb_{i_2} (0)+\r_{i_p}\r_{i_{p-1}}\cdots\r_{i_2} (\alpha_\b).\]
This is true because for any $i\in I$ and $h_1,h_2\in\W$, we have 
\begin{equation}\label{eq:additivity}
\rb_i(h_1+h_2)=\r_i(h_1)+\rb_i(h_2). 
\end{equation}
Since $\i$ was reduced, the same is true for $\i'=(i_2,\dots,i_p)$, and thus the positivity of the first summand follows by induction. The positivity of the second summand is an immediate application of~\cite[Lemma~3.11(a)]{Kac}.
\end{proof}

\subsection{Twists for arbitrary quivers}\label{sect:twists:general}

\def\twist{\times}
Let $Q$ be a quiver, and let $I:=\VertQ$ be the set of its vertices. Let $G$ be the underlying undirected graph for $Q$ with the same set $I$ of vertices. We let $I'=\{i'\mid i\in I\}$ and $I''=\{i''\mid i\in I\}$. We are going to construct a new quiver $Q\twist Q$ with vertex set $\Vert(Q\twist Q)=I'\cup I''$. For every edge $i\to j$ of $Q$, the quiver $Q\twist Q$ contains edges 
\[i'\to j',\quad i''\to j'',\quad j''\to i',\quad j'\to i''.\]
For example, when $Q$ is a cycle with edges $1\to 2\to 3\to 1$ then $Q\twist Q$ is the well studied del Pezzo $3$ quiver, see Figure~\ref{fig:dp3}.

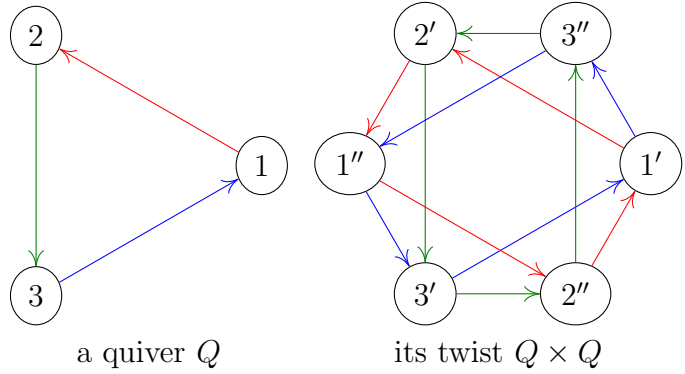
\begin{figure}
\begin{tabular}{cc}
	\begin{tikzpicture}
		\definecolor{green}{rgb}{0,0.5,0}
		\def\ascale{3}
		\def\alength{2}
		\def\awidth{2}
		\node[draw,ellipse] (a1) at (0:2) {$1$};
		\node[draw,ellipse] (a2) at (120:2) {$2$};
		\node[draw,ellipse] (a3) at (240:2) {$3$};
		\draw[-{>[scale=\ascale,length=\alength,width=\awidth]},red] (a1) -- (a2);
		\draw[-{>[scale=\ascale,length=\alength,width=\awidth]},green] (a2) -- (a3);
		\draw[-{>[scale=\ascale,length=\alength,width=\awidth]},blue] (a3) -- (a1);
	\end{tikzpicture}&
	\begin{tikzpicture}
		\definecolor{green}{rgb}{0,0.5,0}
		\def\ascale{3}
		\def\alength{2}
		\def\awidth{2}
		\node[draw,ellipse] (a1) at (0:2) {$1'$};
		\node[draw,ellipse] (a2) at (120:2) {$2'$};
		\node[draw,ellipse] (a3) at (240:2) {$3'$};
		\node[draw,ellipse] (a11) at (180:2) {$1''$};
		\node[draw,ellipse] (a22) at (300:2) {$2''$};
		\node[draw,ellipse] (a33) at (60:2) {$3''$};
		\draw[-{>[scale=\ascale,length=\alength,width=\awidth]},red] (a1) -- (a2);
		\draw[-{>[scale=\ascale,length=\alength,width=\awidth]},green] (a2) -- (a3);
		\draw[-{>[scale=\ascale,length=\alength,width=\awidth]},blue] (a3) -- (a1);
		\draw[-{>[scale=\ascale,length=\alength,width=\awidth]},red] (a11) -- (a22);
		\draw[-{>[scale=\ascale,length=\alength,width=\awidth]},green] (a22) -- (a33);
		\draw[-{>[scale=\ascale,length=\alength,width=\awidth]},blue] (a33) -- (a11);
		\draw[-{>[scale=\ascale,length=\alength,width=\awidth]},blue] (a1) -- (a33);
		\draw[-{>[scale=\ascale,length=\alength,width=\awidth]},green] (a33) -- (a2);
		\draw[-{>[scale=\ascale,length=\alength,width=\awidth]},red] (a2) -- (a11);
		\draw[-{>[scale=\ascale,length=\alength,width=\awidth]},blue] (a11) -- (a3);
		\draw[-{>[scale=\ascale,length=\alength,width=\awidth]},green] (a3) -- (a22);
		\draw[-{>[scale=\ascale,length=\alength,width=\awidth]},red] (a22) -- (a1);
	\end{tikzpicture}\\
	a quiver $Q$ & its twist $Q\twist Q$
\end{tabular}
	
\caption{\label{fig:dp3} The del Pezzo $3$ quiver is a twist.}
\end{figure}

For each $i\in I$, we define a new quiver $\tau_i(Q\twist Q)$ to be $\sigma_i\circ \mu_{i'}\circ \mu_{i''}(Q\twist Q)$, where $\mu_{i'}$ is the usual quiver mutation (see Definition~\ref{dfn:mutations}) and $\sigma_i$ is the operation that swaps the vertices $i'$ and $i''$ in the quiver. We introduced this operation in~\cite{GP1}, however, for the del Pezzo $3$ quiver it already appeared in~\cite{Gregg} under the name \emph{$\tau$-mutation}.

\begin{lemma}
	We have $\tau_i(Q\twist Q)=Q\twist Q$.
\end{lemma}
\begin{proof}
	Suppose that $u\to i'\to w$ is a path of length $2$ in $Q\twist Q$. We then therefore have a path $w\to i''\to u$ of length $2$ as well. The mutation $\mu_{i'}$ introduces an edge $u\to w$, but then the mutation $\mu_{i''}$ introduces an edge $w\to u$ so these edges cancel each other out. 
	
	For any edge $u\to i'$, there is an edge $i''\to u$ so reversing both of these edges and swapping the vertices $i'$ and $i''$ preserves both of these edges. We get that any edge of $Q\twist Q$ is preserved by $\tau_i$ and no new edges are introduced. We are done with the proof.
\end{proof}

\begin{corollary}\label{cor:twist_recurrent}
	If $G$ is bipartite then $Q$ is a bipartite recurrent quiver.
\end{corollary}

For each $i\in I$, we introduce two variables $x_i'$ and $x_i''$ corresponding to the vertices $i'$ and $i''$ of $Q$ respectively, and thus the set of vertex variables for $Q\twist Q$ is $\x'\cup\x''$ where $\x'=\{x_i'\}_{i\in I}$ and $\x''=\{x_i''\}_{i\in I}$. 

Consider a map $T: \Vert(Q\twist Q')\to\Z[(\x')^{\pm1},(\x'')^{\pm1}]$ assigning a Laurent polynomial to each vertex of $Q\twist Q$. The operation $\tau_i$ for $i\in I$ can be lifted to an operation on such maps $T$. More specifically, it is defined as follows: for $j\neq i\in I$, we set 
\[(\tau_iT)(j')=T(j'),\quad (\tau_iT)(j'')=T(j'').\]
For the remaining two vertices, we put
\begin{equation}\label{eq:tau_i}
	\begin{split}
(\tau_iT)(i')=\frac{\prod_{u\to i''} T(u)+\prod_{i''\to v} T(v)}{T(i'')};\\
(\tau_iT)(i'')=\frac{\prod_{u\to i'} T(u)+\prod_{i'\to v} T(v)}{ T(i')}.		
	\end{split}
\end{equation}
Here $u,v\in I'\cup I''$ are the vertices of $Q\twist Q$. Thus the operation $\tau_i$ can be viewed as a composition of two quiver mutations $\tau_{i'}\circ \tau_{i''}$ followed by swapping the values of $T$ at $i'$ and $i''$. 

\subsection{A product formula for any $\tau$-mutation sequence}\label{sect:twists:any_sequence}
We let $T_0:\Vert(G)\to\Z[\x',\x'']$ be the \emph{initial seed}, that is, $T_0(i')=x_i'$ and $T_0(i'')=x_i''$ for all $i\in I$. Consider a sequence $\i=(i_1,i_2,\dots,i_p)$ of elements of $I$ and let 
\begin{equation}\label{eq:T_i_p}
T_{1}=\tau_{i_1}T_0,\quad  T_2=\tau_{i_2}T_1,\quad \dots,\quad  T_p=\tau_{i_p} T_{p-1}.
\end{equation}

We are interested in giving a formula for $T_p(j)$ for any $j\in I$.
\def\X{X}
\def\Id{ \operatorname{Id}}
For each $i\in I$, define $\X_i\in\Z[(\x')^{\pm1},(\x'')^{\pm1}]$ by
\begin{equation}\label{eq:X_i}
\X_i=\frac{\prod_{j\to i} x_j'\prod_{i\to j} x_j''+\prod_{j\to i} x_j''\prod_{i\to j} x_j'}{x_i'x_i''}.
\end{equation}

\newcommand{\rbb}[1]{\r^\parr{#1}}
Just as in Section~\ref{sect:twists:game}, we define the operators $\rbb{j}_i$ for each $i,j\in I$ to be the reflections $\rb_i$ for $G$ with $\b=j$. 
\begin{proposition}\label{prop:factor}
	Let $\i=(i_1,i_2,\dots,i_p)$ and $T_0, T_1,\dots,T_p$ be as in~\eqref{eq:T_i_p}. Then there exists a matrix $A(\i)=(a_{ij}^\i)_{i,j\in I}$ with nonnegative integer entries such that for any $i\in I$ we have
	\begin{equation}\label{eq:formula}
	T_p(i')=x_i'\prod_{j\in I} \X_j^{a_{ij}^\i};\quad T_p(i'')=x_i''\prod_{j\in I} \X_j^{a_{ij}^\i}.
	\end{equation}
	The $j$-th column $(a_{ij}^\i)_{i\in I}$ of $A(\i)$ is equal to 
	\begin{equation}\label{eq:recurrence}
	\rbb{j}_{i_p}\rbb{j}_{i_{p-1}}\cdots\rbb{j}_{i_1}(0).
	\end{equation}
\end{proposition}

\begin{proof}
	We prove~\eqref{eq:formula} and~\eqref{eq:recurrence} by induction on $p$, the case $p=0$ being trivial. Suppose that we already know the result for $p-1$. The induction step is straightforward to check by substituting~\eqref{eq:X_i} into~\eqref{eq:tau_i}. The nonnegativity of the coefficients of $A(\i)$ follows from Proposition~\ref{prop:positivity}

\end{proof}

\begin{corollary}
	If $Q$ is an orientation of a finite (resp., affine) $ADE$ Dynkin diagram $\L$ (resp., $\affL$) then the operators $\tau_i$ define a simply transitive action of the Weyl group $W$ (resp., the affine Weyl group $W_a$) of $\L$ (resp., $\affL$) on the clusters that can be obtained from the initial seed $T_0$ by applying $\tau$-mutations. In particular, such clusters are in bijection with Weyl chambers of $W$ (resp., with alcoves of $W_a$).
\end{corollary}
\begin{proof}
	This follows immediately from Proposition~\ref{prop:factor}. We refer the reader to~\cite[Chapter~V, \S4]{Bourbaki} for the background on alcoves and Weyl chambers.
\end{proof}

\def\Cox{{\mathbf{C}}}
\subsection{Twists of $ADE$ Dynkin diagrams}\label{sect:twists:entropy}
In this section, we return to the case when $Q$ is a finite or affine $ADE$ Dynkin diagram with every edge oriented towards a white vertex (see Definition~\ref{dfn:twist}). Let $I=\{0,1,2,\dots,n\}$ be the set of its vertices and suppose that the vertices $0,1,2,\dots,k-1$ are white while the vertices $k,\dots,n$ are black. We would like to apply the product formula~\eqref{eq:formula} to the $T$-system associated with $Q\twist Q$. Proposition~\ref{prop:factor} implies that we only need to analyze the left hand side of~\eqref{eq:positivity} for the specific mutation sequence $\i=(0,1,2,\dots,n,0,1,2,\dots)$. Let us choose some distinguished vertex, say, $\b=0$. Let $\W$ be a vector space with basis $\alpha_0,\alpha_1,\dots,\alpha_n$ as in Section~\ref{sect:twists:game} and a bilinear form $B$ associated to the generalized Cartan matrix $A_G$ of $G$ from~\eqref{eq:cartan}. In particular, for $h_1,h_2\in\W$, we have
\[B(h_1,h_2)=\<h_1,h_2\>:= h_1^T A_G h_2.\]

 Thus it is well known (see e.g.~\cite[Section~2.17]{Stek}) that $B$ is positive definite (resp., nonnegative definite) if and only if $Q$ is an orientation of a finite (resp., affine) $ADE$ Dynkin diagram. The reflections $\r_i$ can be alternatively defined by
\[\r_i(h)=h-\<h,\alpha_i\>\alpha_i\]
for $i\in I$. Define the \emph{Coxeter transformation} $\Cox$ by
\[\Cox=\omega_2\omega_1,\quad\text{where }\quad \omega_1=\r_{k-1}\r_{k-2}\dots \r_0,\quad \omega_2=\r_n\r_{n-1}\dots \r_{k}.\]

Let $(h_0,h_1,\dots)$ be the sequence of elements of $\W$ associated with $\i=(0,1,2,\dots,n,0,1,2,\dots)$ in the left hand side of~\eqref{eq:positivity}. In other words, $h_0=0$ is the origin and $h_{k+1}=\rbb{\b}_{i_k} h_k$ for $k\geq 0$. 
Here $i_k\in I$ is defined by $i_k\equiv k\pmod n+1$ so that $\i=(i_0,i_1,\dots)$. The following lemma follows immediately from~\eqref{eq:additivity}:
\begin{lemma}\label{lemma:cox}
	For any integer $m\geq 0$, we have
	\begin{equation}\label{eq:h_mn}h_{m(n+1)}=\Cox (h_{mn}+\alpha_0)=\sum_{k=1}^{m} \Cox^k \alpha_0. \end{equation}
\end{lemma}

Since the whole calculation amounts to computing the powers of the Coxeter transformation, it would be nice to find its Jordan normal form $J$ which is actually well studied:
\begin{proposition}[{\cite[Theorems~3.15 and~4.1]{Stek}}]\label{prop:Jordan}
	\leavevmode
	\begin{enumerate}
		\item If $Q$ is an orientation of a finite $ADE$ Dynkin diagram then $J$ is diagonal and $\Cox$ is periodic, and the eigenvalues of $J$ are roots of unity not equal to $1$;
		\item If $Q$ is an orientation of an affine $ADE$ Dynkin diagram then $J$ has one $2\times2$ block corresponding to the eigenvalue $1$, the rest of its blocks are $1\times1$ and all the other eigenvalues of $\Cox$ are roots of unity not equal to $1$;
		\item otherwise there is a simple maximal eigenvalue of $\Cox$ that is greater than one.
	\end{enumerate}
\end{proposition}

\begin{theorem}
	Let $Q$ be a bipartite quiver. 
	\begin{enumerate}
		\item If $Q$ is an orientation of a finite $ADE$ Dynkin diagram then the $T$-system associated with $Q\twist Q$ is periodic;
		\item If $Q$ is an orientation of an affine $ADE$ Dynkin diagram then the $T$-system associated with $Q\twist Q$ grows quadratic exponentially;
		\item otherwise the $T$-system associated with $Q\twist Q$ grows doubly exponentially.
	\end{enumerate}
\end{theorem}
\begin{proof}
Note that the first part follows from~\cite{GP1} while the third part follows from Theorem~\ref{thm:entropy}. However, it is easy to prove all the parts directly using~\eqref{eq:h_mn} and Propositions~\ref{prop:Jordan} and~\ref{prop:factor}. Let $\Cox=P^{-1}JP$ be the Jordan normal form of $\Cox$ and consider the vector $P\alpha_0$. Suppose that $Q$ is an orientation of a finite $ADE$ Dynkin diagram $\L$. Then $\Cox$ is periodic with some period $h$ (the \emph{Coxeter number}) so the sum of $\Cox^k$ over the period will be zero. Indeed, the matrix $J$ is diagonal and by Proposition~\ref{prop:Jordan}, its entries are roots of unity that are not equal to $1$. Thus the sequence $h_k$ is periodic, and by Proposition~\ref{prop:factor}, this sequence describes the degrees of the factors $\X_i$ in the values of the $T$-system associated with $Q\twist Q$. This proves the first claim. 

Suppose now that $Q$ is an orientation of an affine $ADE$ Dynkin diagram. Then the unique $2\times 2$ block of $J^k$ will have the form
\[\begin{pmatrix}
  	1&k\\
  	0&1
  \end{pmatrix}.\]
Since all the other $1\times 1$ blocks correspond to roots of unity that are not equal to $1$, the corresponding entries of $J^1+J^2+\dots+J^m$ will be bounded while the unique $2\times 2$ block of $J^1+J^2+\dots+J^m$ will have the form
\[\begin{pmatrix}
  	m&{m+1\choose 2}\\
  	0&m
  \end{pmatrix}.\]
This shows that the sequence $h_{m(n+1)}$ grows quadratically and thus the values of the $T$-system associated with $Q\twist Q$ grow quadratic exponentially\footnote{it may happen that even though some entries of $J$ grow fast, the vector $P^{-1}JP\alpha_0$ is bounded. But then we can relabel the vertices and choose some other vertex $\b$ for which the growth will be quadratic exponential.} and we are done with the second claim.

Finally, suppose that the underlying graph $G$ of $Q$ is not a finite or affine $ADE$ Dynkin diagram. Then there is a simple maximum eigenvalue $\l$ in $J$ and therefore we will have $\l^k$ in $J^k$ dominating all the other terms. Thus the sequence $h_{m(n+1)}$ grows exponentially which implies that the values of the $T$-system associated with $Q\twist Q$ grow doubly exponentially and we are done with the third claim.
\end{proof}

\section{Conjectures}\label{sect:conj}

In addition to the main Conjecture \ref{conj:master} we make several other conjectures describing the behavior of $T$-systems in our \affaff classification. We prove some of them for twists.

\subsection{Arnold-Liouville integrability}

In this paper we worked with zero algebraic entropy, which is one of the ways to define integrability. An alternative way is to look for the {\it {Arnold-Liouville integrability}}, which means finding a non-degenerate Poisson bracket, and a number of algebraically independent conserved 
quantities in involution with respect to this Poisson bracket. We refer the reader to \cite{A} for a classical account. 

\begin{conjecture}
 $Y$-systems associated with all the \affaff $ADE$ bigraphs in our classification are integrable in  Arnold-Liouville sense.
\end{conjecture}

This conjecture has already been verified in the special case of our $\Toric(\affA_{rd-1},\exp(2\pi i p/r),n)$ family~\bg{toric-A-rotn}. Specifically, it is a special case of a theorem of Goncharov and Kenyon \cite[Theorem 3.7]{GK}. 
It is also present in Ovsienko-Schwartz-Tabachnikov \cite[Theorem 2]{OST} and Gekhtman-Shapiro-Tabachnikov-Vainshtein \cite[Theorem 4.4]{GSTV}. The latter two sources prove it in a somewhat narrower generality than our family~\bg{toric-A-rotn}, 
however their methods extend easily to cover the whole family. 

Note that all three of the above sources prove Arnold-Liouville integrability for the $Y$-variable dynamics. It remains to be understood if a similar claim can be made about the $T$-systems in our classification. 

\subsection{Devron property}

Glick has introduced the {\it {Devron property}} in \cite{G} as a counterpart to {\it {singularity confinement}}, often used to detect integrability. Roughly speaking, Devron property is a property of systems where time flow is reversible. Assume that going backward 
in time one fails due to a really bad singular behavior, i.e. a Devron singularity. Then the system has Devron property if this implies similar failure after a number of steps when going forward in time. 

For a $T$-system associated with an \affaff $ADE$ bigraph from our classification, let us say that the initial values $T_v(t)$ for $t=0,1$ form a {\it {backward Devron singulairty}} if for any $v$ of color $\epsilon_v = 1$ we have $T_v(-1) = 0$.
Let us say that for some time $t_0$, the values of the $T$-system form a {\it {forward Devron singulairty}} if for any $v$ of color $\epsilon_v \not\equiv t_0 \pmod 2$ we have $T_v(t_0+1) = 0$.

\begin{conjecture}
 If the initial values of a $T$-system associated with an \affaff $ADE$ bigraph from our classification form a backward Devron singularity, then after a finite number of steps $t_0$, the $T$-system will reach a state that forms a forward Devron singularity. 
\end{conjecture}

We prove this conjecture for twists of arbitrary bipartite quivers.

\begin{proposition}
	Let $Q$ be any bipartite quiver. Then the twist $Q\times Q$ has the Devron property with $t_0=2$. 
\end{proposition}
\begin{proof}
	This follows immediately from Proposition~\ref{prop:factor}. Indeed, having a backward Devron singularity at $t=0$ means that $X_v=0$ for any $v\in\VertQ$ with $\e_v=1$. Since $X_v$ appears in $T_v(3)$ with exponent equal to $1$ by Proposition~\ref{prop:factor}, we are done. 
\end{proof}

For the case when $Q$ is a tensor product of type $\affA_{2n-1}\otimes\affA_{2m-1}$, our limited computer evidence suggests that we have 
\[t_0\in\{\max(2n,2m),2\max(2n,2m)\},\] 
depending on the parity of $n$ and $m$.

\subsection{Time-dependent conserved quantities}

A notion of time dependent conserved quantities is sometimes used when analysing integrability of dynamical systems, see \cite{Go} for an accessible introduction. Let $A(t)$ be a function of the system parameters evaluated at time $t$. We say that $A$ is a 
{\it {time-dependent conserved quantity}} if there exist integers $m,t_0\geq 1$ such that for any $t\in\Z$, we have $A(t+t_0) = A(t) B^m$, where $B= B(t)$ is a fixed genuine conserved quantity of the system, that is, $B(t+t_0)=B(t)$ for all $t\in\Z$.  

\begin{conjecture}
 The $T$-systems associated with \affaff $ADE$ bigraphs from our classification possess non-trivial time-dependent conserved quantities. 
\end{conjecture}

\def\affCox{h_a}
\def\weird{g}

We again prove this conjecture for twists, but now only of affine $ADE$ Dynkin diagrams. In order to state the result, we need to associate one more integer $\affCox(\affL)$ to each affine $ADE$ Dynkin diagram $\affL$ which is called the \emph{affine Coxeter number} in~\cite{Stek}, not to be confused with the McKay number $\affH(\affL)$ of $\affL$ from Figure~\ref{fig:affADE}. 

\begin{definition}
	The \emph{affine Coxeter number} of an affine $ADE$ Dynkin diagram $\affL$ is the smallest positive integer $m=\affCox(\affL)$ such that $\l^m=1$ for any eigenvalue $\l$ of the Coxeter transformation associated to $\affL$. The values of $\affCox(\affL)$ are given in~\cite[Table~4.1]{Stek}. Moreover, define the \emph{Coxeter-McKay ratio} $\weird(\affL)$ by
	\[\weird(\affL)=\frac{4\affCox(\affL)}{\affH(\affL)}.\]
	The values of $\affCox(\affL)$ and $\weird(\affL)$ are given in Table~\ref{table:weird}. In particular, $\weird(\affL)$ is always equal to either $1$ or $2$.
\end{definition}
The Coxeter-McKay ratio is closely related to the \emph{Dlab-Ringel defect}, see~\cite{DR} or~\cite[Section~6.3.3]{Stek}.

\begin{table}
\begin{tabular}{|c|c|c|c|c|c|c|}\hline
	$\affL$           & $\affA_{2n-1}$ & $\affD_{n}$, $n$ even & $\affD_{n}$, $n$ odd & $\affE_6$ & $\affE_7$ & $\affE_8$ \\\hline
	$\affCox(\affL)$  & $n$            & $n-2$       & $2(n-2)$  & $6$ & $12$ & $30$\\\hline
	$\weird(\affL)$   & $2$            & $1$         & $2$     & $1$ & $1$  & $1$ \\\hline
\end{tabular}
\caption{\label{table:weird} The affine Coxeter number and the Coxeter-McKay ratio for affine $ADE$ Dynkin diagrams.}
\end{table}

\begin{proposition}
	Let $\affL$ be an affine $ADE$ Dynkin diagram with vertex set $I$, edge set $E$, additive function $\l:I\to\Z$, and affine Coxeter number $m=\affCox(\affL)$. Consider the twist $\affL\times\affL$ with vertex set $I'\cup I''$. Then for each vertex $i\in I$, there is a time-dependent conserved quantity $A_i(t)$ defined as follows: for $t\equiv \e_i\pmod2$, we put
	\[A_i(t)=\frac{T_{i'}(t)^2}{\prod_{(i,j)\in \affL} T_{j'}(t-1)}.\]
	These functions satisfy
	\begin{equation}\label{eq:shift_cox}
	A_i(t+2m)=A_i(t) B^{\weird(\affL)\l(i)},
	\end{equation}
	where $B=B(t)$ is defined as follows. For $t$ even, we put
	\[B(t)=\prod_{j} X_j^{\l(j)}(t-\e_j).\]
	For odd $t$, one replaces $\e_j$ by $1-\e_j$. Here $X_j(t)$ is defined analogously to~\eqref{eq:X_i}, namely,
	\[X_j(t)=\frac{\prod_{k\to j} T_{k'}(t)\prod_{j\to k} T_{k''}(t)+\prod_{k\to j} T_{k''}(t)\prod_{j\to k} T_{k'}(t)}{T_{j'}(t-1)T_{j''}(t-1)}.\]
	The function $B(t)$ is a genuine conserved quantity: $B(t+m)=B(t)$ for all $t\in\Z$.
\end{proposition}
\begin{proof}
	Let us fix some $j\in I$ with say $\e_j=0$ and look at the exponent $a_{ij}^t$ of $X_j=X_j(0)$ from~\eqref{eq:X_i} in $T_i(t)$. By Lemma~\ref{lemma:cox}, this value equals to the $i$-th coordinate of 
	\[\delta_j(t)=\sum_{k=1}^{t} \Cox^k \alpha_j.\]
	Now, the degree of $X_j$ in $A_i(t)$ is therefore the value of $\<\alpha_i,\delta_j(t)\>$. We would like to show~\eqref{eq:shift_cox}, and the exponent of $X_j$ in $A_i(t+m)/A_i(t)$ is given by $\<\alpha_i,\delta_j(t+m)-\delta_j(t)\>.$ Note that 
	\[\delta_j(t+m)-\delta_j(t)=\Cox^t\left(\sum_{k=1}^{m} \Cox^k \alpha_j\right).\]
	From Proposition~\ref{prop:Jordan} it follows that the right hand side, written in the Jordan basis for $\Cox$, has nonzero coordinates corresponding only to the two eigenvectors of $\Cox$ associated to the $2\times 2$ Jordan block. One of these vectors is exactly $\l$ and $\<\alpha_i,\l\>=0$. The other vector $\nu$ satisfies $\Cox\nu=\nu+\l$ so its coefficient in $\delta_j(t+m)-\delta_j(t)$ is independent of $t$ and is equal to $m c_j$, where $c_j$ is the coefficient of $\nu$ in the expansion of $\alpha_j$ in the Jordan basis of $\Cox$. Let us calculate $c_j$ explicitly. The vector $\nu$ is orthogonal to the other Jordan basis vectors of $\Cox$ with respect to the scalar product $(\cdot,\cdot)$ defined by $(\alpha_i,\alpha_j)=\delta_{ij}$, see~\cite[(6.47)]{Stek}. Thus $c_j$ equals to $\frac{(\alpha_j,\nu)}{(\nu,\nu)}$. For any $k\in I$, we have $\nu(k)=\frac14(-1)^{\e_k} \l(k)$. Therefore $(\nu,\nu)=\frac1{16} \affH(\affL)$. On the other hand, $(\alpha_j,\nu)$ is equal up to sign to $\frac14 \l(j)$. Thus $c_j$ equals to $\frac{4\l(j)}{\affH(\affL)}$. Multiplying this by $m$ shows that the coefficient of $\nu$ in $\delta_j(t+m)-\delta_j(t)$ equals $\weird(\affL)\l(j)$. It remains to note that up to sign we have $\<\alpha_i,\nu\>=\l(i)$ which yields the result. 
\end{proof}

Note that for the $\Toric(\affA_{rd-1},\exp(2\pi i p/r),n)$ family~\bg{toric-A-rotn} one can use the topology of the torus embedding of the quiver to define conserved quantities, as it was done in a slightly different language by Goncharov and Kenyon in \cite{GK}. 
In \cite{GP2} we performed this construction in our language for cylindric rather than toric quivers, resulting in what we called {\it {Goncharov-Kenyon Hamiltonians}}. However, as evident from the definition, the task of finding time-dependent conserved quantities is strictly harder than that of finding conserved quantities: if one knows $A(t)$, one can find the associated function $B(t) = \left(\frac{A(t+t_0)}{A(t)}\right)^{1/m}$, but there is no simple way to go in the other direction. Thus, even for the $T$-systems from family~\bg{toric-A-rotn} we do not know of a construction of time-dependent conserved quantities in general.

\bibliographystyle{plain}
\bibliography{affine}

\end{document}